\documentclass[12pt]{article}

\usepackage{mathrsfs,amsthm,graphicx,color,verbatim,bbm,amsmath,amsfonts,amssymb,newclude,nicefrac,amsfonts,graphicx,geometry,enumerate,hyperref}
\usepackage[latin1]{inputenc}
\theoremstyle{plain}
\newtheorem{theorem}{Theorem}[section]
\newtheorem{lemma}[theorem]{Lemma}
\newtheorem{corollary}[theorem]{Corollary}
\newtheorem{proposition}[theorem]{Proposition}

\theoremstyle{definition}
\newtheorem{definition} [theorem]{Definition}

\begin{document}

\newcommand{\E}{\mathbb{E}\!}
\newcommand{\ES}{\mathbb{E}}
\renewcommand{\P}{\mathbb{P}}
\newcommand{\R}{\mathbb{R}}
\newcommand{\N}{\mathbb{N}}
\newcommand{\smallsum}{\textstyle\sum}
\newcommand{\tr}{\operatorname{trace}}
\newcommand{\citationand}{\&}
\newcommand{\dt}[1][t]{\, \mathrm{d} #1}
\newcommand{\grid}{\,\mathcal{P}}
\newcommand{\euler}{Z^{\,N}}
\newcommand{\leuler}{\tilde{Z}^{\,N}}
\newcommand{\exteuler}{\bar{Z}^{\,N}}

\newcommand{\floor}[2]{\lfloor #1\rfloor_{#2}}
\newcommand{\ceil}[2]{\lceil #1\rceil_{#2}}
\newcommand{\mesh}[2]{\|#1\|^{#2}_{\grid_T}}
\newcommand{\bigbrack}[2]{\big( #1\big)^{#2}}
\newcommand{\bignorm}[3]{\bnl #1\bnr^{#2}_{#3}}
\newcommand{\bigsharp}[2]{\big[ #1\big]^{#2}}
\newcommand{\lpn}[3]{L^{#1}(#2;#3)}
\newcommand{\eulerm}[1]{Z^{\,\Theta_{#1}}}
\newcommand{\eulerpart}[1]{Z^{#1}}

\newcommand{\expeuler}[1]{Z^{\,\text{exp},#1}}
\newcommand{\impeuler}[1]{Z^{\,\text{imp},#1}}

\newcommand{\embed}[2]{\kappa^I_{#1,\,#2}}

\newcommand{\groupC}{ C }
\newcommand{\diagC}{ C }
\newcommand{\psiC}{ \eta }
\newcommand{\driftC}{{\bf y}}
\newcommand{\diffusionC}{{\bf z}}
\newcommand{\power}{ q }

\newcommand{\set}{{\mathbb{H}}}

\newcommand{\resolvent}[2]{R_{#2}( #1 )}

\title{Weak convergence rates for Euler-type\\
approximations of semilinear stochastic evolution\\
equations with nonlinear diffusion coefficients}

\author{Arnulf Jentzen and 
Ryan Kurniawan\\[2mm]
\emph{ETH Zurich, Switzerland}}

\maketitle

\begin{abstract}
Strong convergence rates for time-discrete numerical approximations 
of semilinear stochastic evolution equations (SEEs) with smooth and regular 
nonlinearities are well understood in the literature. Weak convergence rates for 
time-discrete numerical approximations of such SEEs have been investigated since 
about 12 years and are far away from being well understood: roughly speaking, 
no essentially sharp weak convergence rates are known for time-discrete numerical 
approximations of parabolic SEEs with nonlinear diffusion coefficient functions; see 
Remark~2.3 in [A.\ Debussche, \emph{Weak approximation of stochastic partial differential equations: the 
nonlinear case}, Math.\ Comp.\ {\bf 80} (2011), no.\ 273, 89--117] for details. 
In the recent article [D.\ Conus, A.\ Jentzen \& R.\ Kurniawan, \emph{Weak convergence rates 
of spectral Galerkin approximations for SPDEs with nonlinear 
diffusion coefficients}, arXiv:1408.1108] the weak convergence problem 
emerged from Debussche's article has been solved in the case of spatial spectral Galerkin
approximations for semilinear SEEs with nonlinear diffusion coefficient functions.
In this article we overcome the problem emerged from Debussche's article
in the case of a class of time-discrete Euler-type approximation 
methods (including exponential and linear-implicit Euler approximations as
special cases)
and, in particular, we establish essentially sharp weak convergence rates 
for linear-implicit Euler approximations of semilinear SEEs
with nonlinear diffusion coefficient functions. 
Key ingredients of our approach are applications of a mild It\^{o} type formula
and the use of suitable semilinear integrated counterparts 
of the time-discrete numerical approximation processes.
\end{abstract}

\tableofcontents

\section{Introduction}
\label{sec:intro}

This article studies weak convergence rates for time-discrete numerical approximations of semilinear stochastic evolution equations (SEEs). 
We first review a few weak convergence results from the literature and then present 
the main weak convergence result obtained in this article.
This introductory section is based on Section~1 of Conus et al.~\cite{ConusJentzenKurniawan2014arXiv}.
For finite dimensional stochastic ordinary differential equations (SODEs) with smooth 
and regular nonlinearities both strong and numerically weak convergence rates of time-discrete 
numerical approximations are well understood in the literature; see, e.g., the monographs 
Kloeden \& Platen~\cite{kp92} and Milstein~\cite{m95}.
The situation is different in the case 
of possibly infinite dimensional semilinear stochastic evoluation equations (SEEs).
While strong convergence rates for time-discrete numerical approximations 
of semilinear SEEs with smooth and regular nonlinearities are well understood in 
the literature, 
weak convergence rates for time-discrete numerical approximations of 
such SEEs have been investigated since about 12 years and are far away from being well 
understood: roughly speaking, no essentially sharp weak convergence rates are known 
for time-discrete numerical approximations of parabolic SEEs 
with nonlinear diffusion coefficient functions; see Remark~2.3 in 
Debussche~\cite{Debussche2011} for details.
In this article we overcome the problem emerged from Debussche's article
in the case of a class of time-discrete Euler-type approximation 
methods for SEEs (including exponential and linear-implicit Euler approximations as
special cases)
and, in particular, we establish essentially 
sharp weak convergence rates 
for linear-implicit Euler approximations of semilinear SEEs
with nonlinear diffusion coefficient functions. 
To illustrate the weak convergence problem emerged from 
Debussche's article and our solution to this problem we consider the following setting as a special case of
our general setting in Section~\ref{sec:weak_convergence_irregular} below.
Let 
$ ( H, \left< \cdot, \cdot \right>_H, \left\| \cdot \right\|_H ) $
and
$ ( U, \left< \cdot, \cdot \right>_U, \left\| \cdot \right\|_U ) $
be separable $ \R $-Hilbert spaces,
let 
$ T \in (0,\infty) $, 
let $ ( \Omega, \mathcal{F}, \P, ( \mathcal{F}_t )_{ t \in [0,T] } ) $
be a stochastic basis,
let $ ( W_t )_{ t \in [0,T] } $ be a cylindrical
$ \operatorname{Id}_U $-Wiener process with respect to
$ ( \mathcal{F}_t )_{ t \in [0,T] } $,
let 
$ A \colon D(A) \subseteq H \to H $
be a generator of a strongly continuous analytic semigroup 
with 
$
  \sup\!\big( 
    \text{Re}( \operatorname{spectrum}( A ) ) 
  \big) < 0 
$,
let
$  
  ( 
    H_r 
    ,
    \left< \cdot, \cdot \right>_{ H_r }
    ,$ $
    \left\| \cdot \right\|_{ H_r } 
  )
$,
$ r \in \R $,
be a family of interpolation spaces associated to $ - A $
(cf., e.g., Theorem and Definition~2.5.32 in \cite{Jentzen2014SPDElecturenotes}),
let
$ 
  \iota \in \R
$,
$ 
  \xi \in H_{ \iota } 
$,
$ 
\gamma \in [ 0, \frac{ 1 }{ 2 } ] 
$,
and 
let 
$ 
  F \colon H_{ \iota } \to 
  (
    \cap_{ r < \iota - \gamma }
    H_r
  )
$
and
$
  B \colon H_{ \iota } \to
  \operatorname{Lin}( U, 
    \cap_{ r < \iota - \nicefrac{ \gamma }{ 2 } }
    H_r
  )
$
be functions with the property that
$
  \forall \, r \in ( - \infty ,  \iota - \gamma )
  \colon
  \big[
    \big(
      H_{ \iota } \ni 
      v \mapsto 
      F( v ) \in H_r
    \big)
    \in 
    C^5_b( H_{ \iota } , H_r ) 
  \big]
$,
with the property that 
$
  \forall \, r \in ( - \infty ,  \iota - \frac{ \gamma }{ 2 } ), 
  v \in H_{ \iota }
  \colon
  \big[
    \big(
      U \ni u \mapsto 
      B( v ) u \in H_r
    \big)
    \in 
    HS( U, H_r )
  \big]
$,
and with the property that
$
  \forall \, r \in ( - \infty ,  \iota - \frac{ \gamma }{ 2 } )
  \colon
  \big[
    \big(
      H_{ \iota } \ni 
      v \mapsto 
      [ U \ni u \mapsto B( v ) u \in H_r ]
      \in
      HS( U, H_r )
    \big)
    \in 
    C^5_b( H_{ \iota } , HS( U, H_r ) ) 
  \big]
$,
where for two $ \R $-vector spaces 
$ V_1 $ and $ V_2 $
we denote by 
$
  \operatorname{Lin}( V_1, V_2 )
$
the set of all linear mappings 
from $ V_1 $ to $ V_2 $
and 
where for two $ \R $-Banach spaces 
$ ( V_1, \left\| \cdot \right\|_{ V_1 } ) $ 
and 
$ ( V_2, \left\| \cdot \right\|_{ V_2 } ) $ 
we denote by $ C^5_b( V_1 , V_2 ) $ the set
of all five times continuously Fr\'{e}chet differentiable
functions from $ V_1 $ to $ V_2 $ 
which have globally bounded derivatives 
(see Subsection~\ref{sec:notation} below for more details).
The above 
assumptions 
ensure (cf., e.g., Proposition 3 in Da Prato et al.~\cite{DaPratoJentzenRoeckner2012}, 
Theorem 4.3 in Brze{\'z}niak~\cite{b97b}, 
Theorem 6.2 in Van Neerven et al.~\cite{vvw08}) 
the existence of a continuous mild solution 
process 
$ X \colon [0,T] \times \Omega \to H_{ \iota } $
of the SEE
\begin{equation}
\label{eq:SEE_intro}
dX_t= 
\left[ AX_t 
+ 
F(X_t) 
\right] 
\,dt 
+ 
B(X_t) 
\,dW_t, 
\qquad 
t\in[0,T]
,
\qquad 
X_0=\xi
.
\end{equation}
As an example for \eqref{eq:SEE_intro},
we think of $ H = U = L^2( (0,1) ; \R ) $
being the $ \R $-Hilbert space of equivalence classes
of Lebesgue square integrable functions from $ ( 0,1 ) $ to $ \R $
and $ A $ being a linear differential operator on $ H $.
In particular, in Subsection~\ref{sec:anderson} we formulate the continuous version of 
\emph{the parabolic Anderson model} as an example of \eqref{eq:SEE_intro}
(in that case the parameter $ \gamma $, which controls the regularity of the operators
$ F $ and $ B $, satisfies $ \gamma = \frac{ 1 }{ 2 } $)
and in Subsection~\ref{sec:cahn} we formulate \emph{a fourth-order stochastic partial
differential equation} as an example of \eqref{eq:SEE_intro}
(in that example we have $ \gamma = \frac{ 1 }{ 4 } $).
In this work we are interested in the analysis of numerical 
approximations of \eqref{eq:SEE_intro}.
For example, let $ Y^N \colon \{ 0, 1, \dots, N \} \times \Omega \to H_{ \iota } $,
$ N \in \N $, be stochastic processes with the property that for all $ N \in \N $,
$ n \in \{ 0, 1, \dots, N - 1 \} $
it holds $ \P $-a.s.\ that
\begin{equation}
\label{eq:linear_impl}
  Y^N_0 = \xi,
  \qquad
  Y^N_{ n + 1 }
  =
  \left( \operatorname{Id}_H - \tfrac{ T }{ N } A \right)^{ - 1 }
  \Big(
    Y_n
    +
    F( Y_n ) \tfrac{ T }{ N }
    +
    \smallint\nolimits_{ \frac{ n T }{ N } }^{
      \frac{ ( n + 1 ) T }{ N } 
    }
    B( Y_n )
    \,
    dW_s
  \Big)
  .
\end{equation}
The stochastic processes $ Y^N $, $ N \in \N $,
are referred to as linear-implicit Euler approximations of \eqref{eq:SEE_intro}.

Strong convergence rates for numerical approximations
for SEEs of the form~\eqref{eq:SEE_intro} 
are well understood.
Weak convergence rates for numerical approximations of SEEs 
of the form \eqref{eq:SEE_intro}
have been investigated since about 12 years; cf.
\cite{s03, 
h03b, 
dd06, 
dp09, 
GeissertKovacsLarsson2009, 
h10c, 
Debussche2011, 
kll11, 
Brehier2012b, 
WangGan2013WeakHeatAdditiveNoise,
LindnerSchilling2013, 
KovacsLarssonLindgren2013BIT,
AnderssonLarsson2013,
BrehierKopec2013, 
AnderssonKruseLarsson2013, 
Kruse2014_PhD_Thesis, 
Brehier2014,
KovacsPrintems2014, 
Kopec2014_PhD_Thesis, 
Wang2014Weak, 
ConusJentzenKurniawan2014arXiv, 
Wang2014b, 
AnderssonKovacsLarsson2014}.
Except for Debussche \& De Bouard~\cite{dd06}, 
Debussche~\cite{Debussche2011},
Andersson \& Larsson~\cite{AnderssonLarsson2013},
and Conus et al.~\cite{ConusJentzenKurniawan2014arXiv}, 
all of the above cited weak convergence results assume, 
beside further assumptions, that the considered SEE is driven by additive noise. 
In Debussche \& De Bouard~\cite{dd06} 
weak convergence rates for the nonlinear Schr\"{o}dinger equation,
whose dominant linear operator generates a group (see Section~2 in \cite{dd06}) instead of 
only a semigroup as in the general setting of the SEE~\eqref{eq:SEE_intro}, are analyzed.
The method of proof in Debussche \& De Bouard~\cite{dd06} 
strongly exploits this property of 
the nonlinear Schr\"{o}dinger equation
(see Section~5.2 in \cite{dd06}).
Therefore, 
the method of proof in \cite{dd06}
can, in general, not be used to establish weak convergence
rates for the SEE~\eqref{eq:SEE_intro}.
In Debussche's seminal article \cite{Debussche2011}, 
essentially sharp weak convergence rates for 
\emph{linear-implicit Euler approximations}
(see \eqref{eq:linear_impl}) of 
SEEs of the form \eqref{eq:SEE_intro} are established 
under the hypothesis that $ \iota = 0 $ and that the second derivative of the
diffusion coefficient $ B $ satisfies the smoothing property
that there exists an $ L \in \R $ 
such that for all $ x, v, w \in H $ it holds 
that\footnote{Assumption~\eqref{intro_eq:affine} above slighlty differs from the original assumption in \cite{Debussche2011} 
as we believe that there is a small typo in equation (2.5) in \cite{Debussche2011};
see inequality (4.3) in the proof of Lemma~4.5 in \cite{Debussche2011} for details.}
\begin{equation}
\label{intro_eq:affine}
  \left\|
    B''( x )( v , w )
  \right\|_{ L(H) 
  }
  \leq 
  L
  \left\| v \right\|_{
    H_{ - \nicefrac{ 1 }{ 4 } } 
  }
  \left\| w 
  \right\|_{
    H_{ - \nicefrac{ 1 }{ 4 } }
  }
  .
\end{equation}
The article Andersson \& Larsson~\cite{AnderssonLarsson2013} also assumes \eqref{intro_eq:affine}
but establishes weak convergence rates for spatial approximations.
As pointed out in Remark~2.3 in Debussche~\cite{Debussche2011},
assumption~\eqref{intro_eq:affine}
is a serious restriction for SEEs of the form~\eqref{eq:SEE_intro}.
Roughly speaking, assumption~\eqref{intro_eq:affine}
imposes that the second derivative of the diffusion
coefficient function vanishes and thus that the diffusion coefficient function
is affine linear.
Remark~2.3 
in Debussche~\cite{Debussche2011} 
also asserts that
assumption~\eqref{intro_eq:affine} is crucial in the weak convergence proof 
in \cite{Debussche2011}, 
that assumption~\eqref{intro_eq:affine} is used in 
an essential way
in Lemma~4.5 in \cite{Debussche2011} and that
Lemma~4.5 in \cite{Debussche2011}, in turn, 
is used at many points in the weak convergence proof in \cite{Debussche2011}.
Debussche's article naturally suggests the problem of establishing essentially sharp
weak convergence rates in the case where Debussche's assumption~\eqref{intro_eq:affine}
is not satisfied.
In Conus~\cite{ConusJentzenKurniawan2014arXiv}
essentially sharp weak convergence rates have been
established without imposing Debussche's assumption~\eqref{intro_eq:affine}
in the case of spatial spectral Galerkin approximations.
To the best of our knowledge, it remained an open problem to establish
essentially sharp weak convergence rates for time-discrete numerical approximations
of the SEE~\eqref{eq:SEE_intro} without imposing Debussche's 
assumption~\eqref{intro_eq:affine}.
In this article we overcome this problem 
in the case of a class of time-discrete Euler-type approximation 
methods for SEEs (including exponential and linear-implicit Euler approximations as
special cases)
and, in particular, we establish essentially sharp weak convergence rates 
for linear-implicit Euler approximations of semilinear SEEs
with nonlinear diffusion coefficient functions. 
This is the subject of the following result, Theorem~\ref{intro:theorem}.
Theorem~\ref{intro:theorem} follows\footnote{with $ H = H_{ \iota } $ in the notation
of Corollary~\ref{cor:weak_convergence_irregular}} immediately from Corollary~\ref{cor:weak_convergence_irregular}
and Subsection~\ref{sec:linear_implicit}.

\begin{theorem}
\label{intro:theorem}
Assume the setting in the first paragraph 
of Section~\ref{sec:intro} 
and let 
$ \varphi \in C^5_b( H_\iota, \R ) $.
Then for every 
$ \varepsilon \in ( 0, \infty ) $ 
there exists a 
$ C \in \R $
such that for all $ N \in \N $
it holds that
\begin{equation}
  \left|
    \ES\big[ 
      \varphi( X_T )
    \big]
    -
    \ES\big[ 
      \varphi( Y^N_N )
    \big]
  \right|
\leq
    C \cdot
    N^{
      - ( 1 - \gamma - \varepsilon )
    }
  .
\end{equation}
\end{theorem}

Let us add a few comments regarding Theorem~\ref{intro:theorem}.
First, we would like to emphasize that 
in the general setting of Theorem~\ref{intro:theorem},
the weak convergence rate established
in Theorem~\ref{intro:theorem}
can \emph{essentially not be improved}.
More specifically, in 
Corollary~\ref{cor:weak_error_laplacian_lower_bound}
in Section~\ref{sec:lower_bound} below
we give for every 
$ \iota \in \R $,
$ \gamma \in [ 0 , \frac{ 1 }{ 2 } ] $
examples of
$ A \colon D(A) \subseteq H \to H $,
$ 
  F \colon H_{ \iota } \to 
  (
    \cap_{ r < \iota - \gamma }
    H_r
  )
$, 
$
  B \colon H_{ \iota } \to
  \operatorname{Lin}( U, 
    \cap_{ r < \iota - \nicefrac{ \gamma }{ 2 } }
    H_r
  )
$
and
$ 
  \varphi \in C^5_b( H_{ \iota } , \R ) 
$
with
$
  \forall \, r \in ( - \infty ,  \iota - \gamma )
  \colon
  \big[
    \big(
      H_{ \iota } \ni 
      v \mapsto 
      F( v ) \in H_r
    \big)
    \in 
    C^5_b( H_{ \iota } , H_r ) 
  \big]
$,
$
  \forall \, r \in ( - \infty ,  \iota - \frac{ \gamma }{ 2 } ), 
  v \in H_{ \iota }
  \colon
  \big[
    \big(
      U \ni u \mapsto 
      B( v ) u \in H_r
    \big)
    \in 
    HS( U, H_r )
  \big]
$,
$
  \forall \, r \in ( - \infty ,  \iota - \frac{ \gamma }{ 2 } )
  \colon
  \big[
    \big(
      H_{ \iota } \ni 
      v \mapsto 
      [ U \ni u \mapsto B( v ) u \in H_r ]
      \in
      HS( U, H_r )
    \big)
    \in 
    C^5_b( H_{ \iota } , HS( U, H_r ) ) 
  \big]
$
such that 
there exists a 
$ C \in (0,\infty) $
such that
for all $ N \in \N $
it holds that
\begin{equation}
  \left|
    \ES\big[ 
      \varphi( X_T )
    \big]
    -
    \ES\big[ 
      \varphi( Y^N_N )
    \big]
  \right|
\geq
    C \cdot
    N^{
      - ( 1 - \gamma )
    }
    .
\end{equation}
In addition, we emphasize that in the setting 
of Theorem~\ref{intro:theorem} it is well known 
that for every $ \varepsilon \in ( 0 , \infty ) $
there exists a
$ C \in \R $
such that
for all $ N \in \N $ it holds that
\begin{equation}
\label{eq:strong_rate}
  \big(
    \ES\big[ 
      \| X_T - Y^N_N \|^2_{ H_{ \iota } }
    \big]
  \big)^{ \nicefrac{ 1 }{ 2 } }
\leq
  C
  \cdot
  N^{
    - \left( \frac{ 1 - \gamma }{ 2 } - \varepsilon \right)
  }
  .
\end{equation}
The weak convergence rate 
$ 1 - \gamma - \varepsilon $
established
in Theorem~\ref{intro:theorem}
is thus \emph{twice the well known strong convergence
rate} $ \frac{ 1 - \gamma - \varepsilon }{ 2 } $ in \eqref{eq:strong_rate}.
Next we add that Theorem~\ref{intro:theorem} is, to the best of our knowledge, the first result in the literature
which establishes the essentially sharp weak convergence rate $ 1 - \gamma - \varepsilon $ for time-discrete numerical
approximations of the continuous version of \emph{the parabolic Anderson model} (see Subsection~\ref{sec:anderson} for details).

In the following we briefly outline a few key ideas in the proof 
of Theorem~\ref{intro:theorem} (and Corollary~\ref{cor:weak_convergence_irregular} respectively).
For simplicity we restrict ourself to the case $ \iota = 0 $. 
Our proof of Theorem~\ref{intro:theorem} is partially based on the 
proof of the weak convergence result in Conus et al.~\cite{ConusJentzenKurniawan2014arXiv}.
The first step in the proof 
of Theorem~\ref{intro:theorem} is to rewrite the time-discrete 
stochastic processes $ Y^N $, $ N \in \N $, 
(see \eqref{eq:linear_impl}) as appropriate time-continuous stochastic processes
(see \eqref{eq:tildeY} below).
More formally, 
let 
$ \lfloor \cdot \rfloor_h \colon \R \to \R $, $ h \in (0,\infty) $,
be the mappings with the property that for all 
$ h \in (0,\infty) $, 
$ t \in \R $
it holds that
$
  \lfloor t \rfloor_h = \max\!\big( ( - \infty, t ] \cap \{ 0, h, - h, 2 h, - 2 h, \dots \} \big)
$,
let 
$ 
  S^N \colon \left\{ ( t_1, t_2 ) \in [0,T]^2 \colon t_1 \leq t_2 \right\} \to L( H ) 
$,
$ N \in \N $,
be the mappings with the property that 
for all $ N \in \N $, $ ( t_1, t_2 ) \in [0,T]^2 $ with $ t_1 \leq t_2 $
it holds that
\begin{equation}
  S^N_{ t_1, t_2 }
  =
  \Big( 
    \operatorname{Id}_H - \, ( t_1 - \floor{ t_1 }{ T / N } ) \, A 
  \Big)
  \,
  \Big( 
    \operatorname{Id}_H - \, ( t_2 - \floor{ t_2 }{ T / N } ) \, A 
  \Big)^{ \! - 1 }
  \Big( 
    \operatorname{Id}_H - \, \tfrac{ T }{ N } \, A 
  \Big)^{
    -
    \,
    (
      \floor{ t_2 }{ T / N } - \floor{ t_1 }{ T / N }
    )
    \,
    N
    / 
    T
  }
\end{equation}
(cf., e.g., (142) in Da Prato et al.~\cite{DaPratoJentzenRoeckner2012}),
and
let $ \tilde{Y}^N \colon [0,T] \times \Omega \to H $, $ N \in \N $,
be $ ( \mathcal{F}_t )_{ t \in [0,T] } $-predictable stochastic processes with 
the property that for all $ N \in \N $, $ t \in [0,T] $
it holds $ P $-a.s.\ that
\begin{align}
\label{eq:tildeY}
&
  \tilde{Y}^N_t 
  =
  S^N_{ 0, t } \,
  \xi
\\ & 
\nonumber
  +
  \int_0^t
  S^N_{ s, t } 
  \,
  \big(
    \operatorname{Id}_H 
    - 
    ( s - \lfloor s \rfloor_{ T / N } ) \, A
  \big)^{ - 1 }
  F\big( 
    \tilde{Y}^N_{ \lfloor s \rfloor_{ T / N } } 
  \big) 
  \, ds
  +
  \int_0^t
  S^N_{ s, t } 
  \,
  \big(
    \operatorname{Id}_H 
    - 
    ( s - \lfloor s \rfloor_{ T / N } ) \, A
  \big)^{ - 1 }
  B\big( 
    \tilde{Y}^N_{ \lfloor s \rfloor_{ T / N } } 
  \big) \, dW_s
\end{align}
(cf.\ (143) in Da Prato et al.~\cite{DaPratoJentzenRoeckner2012}).
Note that for all $ N \in \N $, $ n \in \{ 0, \frac{ T }{ N }, \frac{ 2 T }{ N } , \dots, T \} $
it holds $ P $-a.s.\ that
$
  \tilde{Y}^N_{ n T / N } = Y^N_n
$.
Moreover, recall that the solution process $ X $ of the SEE~\eqref{eq:SEE_intro} 
satisfies that for all $ t \in [0,T] $ it holds $ P $-a.s.\ that
\begin{equation}
\label{eq:X_mild_intro}
  X_t 
  =
  e^{ t A } \, \xi
  +
  \int_0^t
  e^{ ( t - s ) A } F( X_s ) \, ds
  +
  \int_0^t
  e^{ ( t - s ) A } B( X_s ) \, dW_s
  .
\end{equation}
The next key step in the proof of Theorem~\ref{intro:theorem} is to use the idea in 
Conus et al.~\cite{ConusJentzenKurniawan2014arXiv} to plug an appropriate process in between
$
    \ES\big[ 
      \varphi( X_T )
    \big]
$
and
$
    \ES\big[ 
      \varphi( Y^N_N )
    \big]
    =
    \ES\big[ 
      \varphi( \tilde{Y}^N_T )
    \big]
$.
More formally, we use the triangle inequality to obtain that for all 
$ N \in \N $ it holds that
\begin{equation}
\label{eq:triangle_inequality_intro}
  \big|
    \ES\big[ 
      \varphi( X_T )
    \big]
    -
    \ES\big[ 
      \varphi( \tilde{Y}^N_T )
    \big]
  \big|
\leq 
  \big|
    \ES\big[ 
      \varphi( X_T )
    \big]
    -
    \ES\big[ 
      \varphi( \bar{Y}^N_T )
    \big]
  \big|
  +
  \big|
    \ES\big[ 
      \varphi( \bar{Y}^N_T )
    \big]
    -
    \ES\big[ 
      \varphi( \tilde{Y}^N_T )
    \big]
  \big|
\end{equation}
where $ \bar{Y}^N \colon [0,T] \times \Omega \to H $, $ N \in \N $,
are appropriate stochastic processes so that it is in some sense not so difficult anymore to 
estimate
$
  \big|
    \ES\big[ 
      \varphi( X_T )
    \big]
    -
    \ES\big[ 
      \varphi( \bar{Y}^N_T )
    \big]
  \big|
$
and
$
  \big|
    \ES\big[ 
      \varphi( \bar{Y}^N_T )
    \big]
    -
    \ES\big[ 
      \varphi( \tilde{Y}^N_T )
    \big]
  \big|
$
for $ N \in \N $.
The main difficulty and also a key difference of the proof of Theorem~\ref{intro:theorem} 
in this article
to the proof of the weak convergence result in Conus et al.~\cite{ConusJentzenKurniawan2014arXiv} 
is the appropriate choice of the processes $ \bar{Y}^N $, $ N \in \N $,
which we put in between.
In the case of Theorem~\ref{intro:theorem}
it turns out to be rather useful to choose $ \bar{Y}^N \colon [0,T] \times \Omega \to H $, $ N \in \N $,
such that for every $ N \in \N $, $ t \in [0,T] $ it holds $ \P $-a.s.\ that
\begin{equation}
\label{eq:barY}
  \bar{Y}^N_t 
  = 
  e^{ t A } 
  \,
  \xi
  +
  \int_0^t
  e^{ ( t - s ) A }
  \,
  F\big( 
    \tilde{Y}^N_{ \lfloor s \rfloor_{ T / N } } 
  \big) 
  \, ds
  +
  \int_0^t
  e^{ ( t - s ) A }
  \,
  B\big( 
    \tilde{Y}^N_{ \lfloor s \rfloor_{ T / N } } 
  \big) \, dW_s
  ,
\end{equation}
cf.\ \eqref{eq:barY} with \eqref{eq:tildeY} and \eqref{eq:X_mild_intro}.
In the remainder of this article we refer to $ \bar{Y}^N $, $ N \in \N $,
as \emph{semilinear integrated counterparts} of \eqref{eq:tildeY}.
In Proposition~\ref{prop:weak_temporal_regularity_2nd} 
(a key result of this article)
in Subsection~\ref{sec:weak_distance_semilinear}
the terms
$
  \big|
    \ES\big[ 
      \varphi( \bar{Y}^N_T )
    \big]
    -
    \ES\big[ 
      \varphi( \tilde{Y}^N_T )
    \big]
  \big|
$,
$ N \in \N $,
in \eqref{eq:triangle_inequality_intro}
are estimated in an appropriate way
by using the mild It\^{o} formula;
see 
Theorem~1 in Da Prato et al.~\cite{DaPratoJentzenRoeckner2012}.
More precisely, 
in Section~\ref{sec:mildcalc}
we generalize the mild It\^{o} formula in
Theorem~1 in Da Prato et al.~\cite{DaPratoJentzenRoeckner2012}
so that it applies also in the case of stopping times
instead of deterministic time points;
see Theorem~\ref{thm:ito} and
Corollary~\ref{cor:itoauto} in Subsection~\ref{sec:mildito}.
We then use Corollary~\ref{cor:itoauto} to derive 
in Proposition~\ref{prop:mild_ito_maximal_ineq}
in Subsection~\ref{sec:weak_estimates_terminal}
an estimate for the expectation of a smooth function composed
with an appropriate type of stochastic process which we call a mild It\^{o} process; 
see Definition~1 in Da Prato et al.~\cite{DaPratoJentzenRoeckner2012}
and, e.g., Definition~\ref{propdef} in Subsection~\ref{sec:mildito} below.
Next we recall that we have rewritten the time-discrete numerical approximation processes $ Y^N $, $ N \in \N $,
(see \eqref{eq:linear_impl}) as the time-continuous stochastic processes $ \tilde{Y}^N $, $ N \in \N $,
(see \eqref{eq:tildeY})
and we emphasize that $ \tilde{Y}^N $, $ N \in \N $, are mild It\^{o} processes;
see (142)--(146) in Da Prato et al.~\cite{DaPratoJentzenRoeckner2012}.
This allows us to apply the mild It\^{o} formula and
so also Proposition~\ref{prop:mild_ito_maximal_ineq} to $ \tilde{Y}^N $, $ N \in \N $.
Thereby we obtain an appropriate estimate
for the terms
$
  \big|
    \ES\big[ 
      \varphi( \bar{Y}^N_T )
    \big]
    -
    \ES\big[ 
      \varphi( \tilde{Y}^N_T )
    \big]
  \big|
$,
$ N \in \N $,
in Proposition~\ref{prop:weak_temporal_regularity_2nd} 
in Subsection~\ref{sec:weak_distance_semilinear}.
The mild It\^{o} formula has also been used in Conus et al.~\cite{ConusJentzenKurniawan2014arXiv}
to establish weak convergence rates for spatial spectral Galerkin approximations.
In this work the analysis is more involved than in Conus et al.~\cite{ConusJentzenKurniawan2014arXiv} as 
the numerical approximation processes $ \tilde{Y}^N $, $ N \in \N $, are not solution processes 
of SEEs of the form \eqref{eq:SEE_intro} but merely mild It\^{o} processes with two-parameter
evolution families (see Subsection~\ref{sec:semigroup_setting} and Section~\ref{sec:weak_temporal_regularity} 
below for more details).
For the estimation of the terms 
$
  \big|
    \ES\big[ 
      \varphi( X_T )
    \big]
    -
    \ES\big[ 
      \varphi( \bar{Y}^N_T )
    \big]
  \big|
$,
$ N \in \N $,
we use (as it is often the case in the case weak convergence analysis; 
see, e.g., R\"{o}\ss ler~\cite{r03}, Debussche~\cite{Debussche2011} and Conus et al.~\cite{ConusJentzenKurniawan2014arXiv})
the Kolmogorov backward equation associated to \eqref{eq:SEE_intro}
and we also use again the mild It\^{o} formula and its consequences respectively
(see Section~\ref{sec:weak_temporal_regularity} 
and Section~\ref{sec:euler_integrated_mollified} for details).
Combining these estimates with the in some sense non-standard
mollification procedure in Conus et al.~\cite{ConusJentzenKurniawan2014arXiv}
will allow us to complete the proof of Theorem~\ref{intro:theorem}
(see Section~\ref{sec:weak_convergence_irregular} for details).

\subsection{Examples}

In this section we illustrate Theorem~\ref{intro:theorem}
by two simple examples. In Subsection~\ref{sec:anderson}
we apply Theorem~\ref{intro:theorem} to
the continuous version of \emph{the parabolic Anderson model}
and in Subsection~\ref{sec:cahn} we apply Theorem~\ref{intro:theorem}
to a linear Cahn-Hilliard-Cook type equation.

\subsubsection{Parabolic Anderson model}
\label{sec:anderson}

Let $ H = L^2( (0,1) ; \R ) $
be the $ \R $-Hilbert space of equivalence classes
of Lebesgue square integrable functions from $ ( 0,1 ) $ to $ \R $, 
let 
$ T, \kappa, \delta, \nu \in (0,\infty) $, 
$ \xi \in H $, 
let 
$
  ( \Omega, \mathcal{F}, \P , ( \mathcal{F}_t )_{ t \in [0,T] } )
$
be a stochastic basis, 
let 
$
  ( W_t )_{ t \in [0,T] }
$
be a cylindrical $ \operatorname{Id}_H $-Wiener process
w.r.t.\ $ ( \mathcal{F}_t )_{ t \in [0,T] } $, 
let $ A \colon D(A) \subseteq H \rightarrow H $ 
be the Laplacian with Dirichlet boundary conditions on $ H $
multiplied by $ \nu $,
let
$  
  ( 
    H_r 
    ,
    \left< \cdot, \cdot \right>_{ H_r }
    ,$ $
    \left\| \cdot \right\|_{ H_r } 
  )
$, 
$ r \in \R $, 
be a family of interpolation spaces associated to $ - A $
(see, e.g., Theorem and Definition~2.5.32 in \cite{Jentzen2014SPDElecturenotes}), 
let 
$
  B \in 
  C( H , HS( H, H_{ - 1 / 4 - \delta } ) )
$
satisfy that for all 
$ v \in H $, 
$ u \in C( [0,1], \R ) $, 
$ x \in ( 0, 1 ) $ 
it holds that 
$
  \big(
    B(v) u
  \big)(x)
  =
  \kappa \cdot v(x) \cdot u(x)
$, 
and let 
$
  Y^N \colon 
  \{ 0, 1, \ldots, N \} \times \Omega
  \rightarrow H
$, $ N \in \N $, 
be stochastic processes which satisfy that
for all $ N \in \N $,
$ n \in \{ 0, 1, \dots, N - 1 \} $
it holds $ \P $-a.s.\ that
$ Y^N_0 = \xi $
and
$
  Y^N_{ n + 1 }
  =
  \big( 
    \operatorname{Id}_H - \frac{ T }{ N } A 
  \big)^{ - 1 }
  \big(
    Y_n
    +
    \int_{ \frac{ n T }{ N } }^{
      \frac{ ( n + 1 ) T }{ N } 
    }
    B( Y_n )
    \,
    dW_s
  \big)
$.
The above 
assumptions 
ensure (cf., e.g., Proposition 3 in Da Prato et al.~\cite{DaPratoJentzenRoeckner2012}, 
Theorem 4.3 in Brze{\'z}niak~\cite{b97b}, 
Theorem 6.2 in Van Neerven et al.~\cite{vvw08}) 
the existence of 
an 
$
  ( \mathcal{F}_t )_{ t \in [0,T] }
$-adapted continuous stochastic process
$ X \colon [0,T] \times \Omega \to H $ 
which satisfies that for all 
$ t \in [0,T] $ it holds $ P $-a.s.\ that
$
  X_t = e^{ t A } \, \xi + \int_0^t e^{ ( t - s ) A } B( X_s ) \, dW_s
$.
The stochastic process $ X $ is thus a solution process (of the 
continuous version) of \emph{the parabolic Anderson model}
\begin{align}
\label{eq:parabolic_Anderson}
&
  dX_t(x)
  = 
  \nu
  \,
  \tfrac{ \partial^2 }{ \partial x^2 } 
  X_t(x)
  \, dt 
  + 
  \kappa \,
  X_t(x)  
  \, dW_t(x)
  , 
\qquad
  X_t(0) = X_t(1) = 0,
\qquad
  X_0( x ) = \xi( x )
\end{align}
for $ x \in ( 0, 1 ) $, $ t \in [ 0, T ] $
(cf., e.g., Carmona \& Molchanov~\cite{CarmonaMolchanov1994}).
Theorem~\ref{intro:theorem} applies here with $ \gamma = \frac{ 1 }{ 2 } $.
More precisely, Theorem~\ref{intro:theorem} proves that 
for all 
$
  \varphi \in C^5_b( H, \R )
$,
$ \varepsilon \in ( 0, \infty ) $ 
there exists a
$ C \in \R $
such that for all $ N \in \N $
it holds that
\begin{equation}
  \left|
    \ES\big[ 
      \varphi( X_T )
    \big]
    -
    \ES\big[ 
      \varphi( Y^N_N )
    \big]
  \right|
\leq
    C
    \cdot
    N^{
      - ( \nicefrac{1}{2} - \varepsilon )
    }
  .
\end{equation}

\subsubsection{A linear Cahn-Hilliard-Cook type equation}
\label{sec:cahn}

Let $ H = L^2( (0,1) ; \R ) $
be the $ \R $-Hilbert space of equivalence classes
of Lebesgue square integrable functions from $ ( 0,1 ) $ to $ \R $, 
let 
$ T, \kappa, \delta \in (0,\infty) $, 
$ \xi \in H $, 
let 
$
  ( \Omega, \mathcal{F}, \P , ( \mathcal{F}_t )_{ t \in [0,T] } )
$
be a stochastic basis, 
let 
$
  ( W_t )_{ t \in [0,T] }
$
be a cylindrical $ \operatorname{Id}_H $-Wiener process
w.r.t.\ $ ( \mathcal{F}_t )_{ t \in [0,T] } $, 
let 
$ \mathcal{A} \colon D(\mathcal{A}) \subseteq H \rightarrow H $ 
be the Laplacian with Neumann boundary conditions on $ H $, 
let $ A \colon D(A) \subseteq H \rightarrow H $ 
be the linear operator with the property that 
$
  D( A ) = 
  D( \mathcal{A}^2 ) 
  =
  \left\{ 
    v \in D( \mathcal{A} ) \colon 
    \mathcal{A} v \in D( \mathcal{A} ) 
  \right\}
$
and with the property that for all 
$ v \in D( A ) $
it holds that
$
  A v
  =
  - \mathcal{A}^2 v - \mathcal{A} v - v
$, 
let
$  
  ( 
    H_r 
    ,
    \left< \cdot, \cdot \right>_{ H_r }
    ,$ $
    \left\| \cdot \right\|_{ H_r } 
  )
$, 
$ r \in \R $, 
be a family of interpolation spaces associated to $ - A $
(see, e.g., Theorem and Definition~2.5.32 in \cite{Jentzen2014SPDElecturenotes}), 
let 
$
  B \in 
  C( H, HS( H, H_{ -1/8 - \delta } ) )
$
satisfy that for all 
$ v \in H $, 
$ u \in C( [0,1], \R ) $, 
$ x \in ( 0, 1 ) $ 
it holds that 
$
  \big(
    B(v) u
  \big)(x)
  =
  \kappa \cdot v(x) \cdot u(x)
$, 
and let 
$
  Y^N \colon 
  \{ 0, 1, \ldots, N \} \times \Omega
  \rightarrow H
$, $ N \in \N $, 
be stochastic processes which satisfy that
for all $ N \in \N $,
$ n \in \{ 0, 1, \dots, N - 1 \} $
it holds $ \P $-a.s.\ that
$ Y^N_0 = \xi $
and
$
  Y^N_{ n + 1 }
  =
  \big( 
    \operatorname{Id}_H - \frac{ T }{ N } A 
  \big)^{ - 1 }
  \big(
    Y_n
    +
    Y_n 
    \,
    \frac{T}{N}
    +
    \int_{ \frac{ n T }{ N } }^{
      \frac{ ( n + 1 ) T }{ N } 
    }
    B( Y_n )
    \,
    dW_s
  \big)
$.
The above assumptions 
ensure (cf., e.g., Proposition 3 in Da Prato et al.~\cite{DaPratoJentzenRoeckner2012}, 
Theorem 4.3 in Brze{\'z}niak~\cite{b97b}, 
Theorem 6.2 in Van Neerven et al.~\cite{vvw08}) 
the existence of an 
$
  ( \mathcal{F}_t )_{ t \in [0,T] }
$-adapted continuous stochastic process
$ X \colon [0,T] \times \Omega \to H $ 
which satisfies that for all 
$ t \in [0,T] $ it holds $ P $-a.s.\ that
$
  X_t = e^{ t A } \, \xi
  +
  \int_0^t e^{ ( t - s ) A } X_s \, ds
  +
  \int_0^t e^{ ( t - s ) A } B( X_s ) \, dW_s
$.
The stochastic process $ X $ is thus a solution process of the 
\textit{linear Cahn-Hilliard-Cook type equation}
\begin{equation}
\label{eq:linear_Cahn-Hilliard-Cook}
\begin{split}
&
dX_t(x)= 
\big[
-\tfrac{\partial^4}{\partial x^4} \,
X_t(x)
-
\tfrac{\partial^2}{\partial x^2} \,
X_t(x)
\big]
\,dt 
+ 
\kappa \,
X_t(x) 
\,dW_t(x), 
\\ &
X'_t(0) = X'_t(1) = X^{(3)}_t(0) = X^{(3)}_t(1) = 0,
\qquad
X_0(x)=\xi(x)
\end{split}
\end{equation}
for $ x \in ( 0, 1 ) $, $ t \in [ 0, T ] $.
Theorem~\ref{intro:theorem} applies here with $ \gamma = \frac{ 1 }{ 4 } $.
More precisely, Theorem~\ref{intro:theorem} proves that 
for all 
$
  \varphi \in C^5_b( H, \R )
$,
$ \varepsilon \in ( 0, \infty ) $ 
there exists a
$ C \in \R $
such that for all $ N \in \N $
it holds that
\begin{equation}
  \left|
    \ES\big[ 
      \varphi( X_T )
    \big]
    -
    \ES\big[ 
      \varphi( Y^N_N )
    \big]
  \right|
\leq
    C
    \cdot
    N^{
      - ( \nicefrac{3}{4} - \varepsilon )
    }
  .
\end{equation}

\subsection{Notation}
\label{sec:notation}

Throughout this article the following notation is used. 
By $ \N = \{ 1, 2, 3, \dots \} $ the set of natural numbers is denoted and 
by $ \N_0 = \N \cup \{ 0 \} $ the union of $ \{ 0 \} $ and the set of natural numbers is denoted.
Moreover, for a set $ A $ we denote by 
$
  \operatorname{Id}_A \colon A \rightarrow A 
$ 
the identity mapping on $ A $, that is, it
holds for all $ a \in A $ that $ \operatorname{Id}_A(a) = a $. 
Furthermore, for a set $ A $ we denote by $ \grid(A) $ the power set of $ A $. 
Let 
$
  \mathcal{E}_r
  \colon
  [0,\infty)
  \rightarrow
  [0,\infty)
$,
$
  r \in (0,\infty)
$,
be the functions with the property that
for all $ x \in [0,\infty) $, $ r \in (0,\infty) $
it holds that
$
  \mathcal{E}_r(x)
  =
  \big[
    \sum^{ \infty }_{ n = 0 }
    \frac{ 
        x^{ 2 n }
        \,
        \Gamma(r)^n
    }{
      \Gamma( n r + 1 ) 
    }
  \big]^{
    \nicefrac{ 1 }{ 2 }
  }
$
(cf.\ Chapter~7 in \cite{h81} and Chapter~3 in \cite{Jentzen2014SPDElecturenotes}). 
In addition, let 
$ ( \cdot )^+, ( \cdot )^- \colon \R \to [ 0, \infty ) $ 
be the functions 
with the property that for all $ a \in \R $ it holds that
$
  a^+ = \max\{ a , 0 \} 
$ 
and 
$
  a^- = \max\{ - a , 0 \} 
$.
Moreover, for a number $ k \in \N_0 $ 
and normed $ \R $-vector spaces 
$
  ( E_i, \left\|\cdot\right\|_{E_i} )
$,
$ i \in \{ 1, 2 \} $, 
let
$ 
  \left| \cdot \right|_{ 
    \operatorname{Lip}^k( E_1, E_2 ) 
  }
  ,
  \left\| \cdot \right\|_{ 
    \operatorname{Lip}^k( E_1, E_2 ) 
  }
  \colon
  C^k( E_1, E_2 )
  \to [0,\infty]
$
be the mappings 
with the property that
for all $ f \in C^k( E_1, E_2 ) $
it holds that
\begin{equation}
  \left| f \right|_{
    \operatorname{Lip}^k( E_1, E_2 ) 
  }
  =
  \sup_{ 
    \substack{
      x, y \in E_1 ,
    \\
      x \neq y
    }
  }
  \tfrac{
    \left\| f^{ (k) }( x ) - f^{ (k) }( y ) \right\|_{ L^{ (k) }( E_1, E_2 ) }
  }{
    \left\| x - y \right\|_{ E_1 }
  }
  ,
\end{equation}
\begin{equation}
  \left\| f \right\|_{
    \operatorname{Lip}^k( E_1, E_2 )
  }
  =
  \left\| f( 0 ) \right\|_{ E_2 }
  +
  \sum_{ l = 0 }^k
  \left|
    f
  \right|_{
    \operatorname{Lip}^l( E_1, E_2 )
  }
  .
\end{equation}
Furthermore, 
for a number $ k \in \N_0 $
and normed $ \R $-vector spaces 
$
  ( E_i, \left\|\cdot\right\|_{E_i} )
$,
$ i \in \{ 1, 2 \} $, 
let 
$
  \operatorname{Lip}^k( E_1, E_2 )
$
be the set given by
$
  \operatorname{Lip}^k( E_1, E_2 )
  =
  \big\{ 
    f \in C^k( E_1, E_2 )
    \colon
    \| f \|_{
      \operatorname{Lip}^k( E_1, E_2 )
    }
    < \infty
  \big\}
$.
In addition, 
for a natural number $ k \in \N $
and normed $ \R $-vector spaces 
$
  ( E_i, \left\|\cdot\right\|_{E_i} )
$,
$ i \in \{ 1, 2 \} $, 
let 
$
  \left|
    \cdot
  \right|_{
    C^k_b(E_1,E_2)
  },
  \left\|
    \cdot
  \right\|_{ C^k_b(E_1,E_2) }
  \colon C^k(E_1,E_2)\rightarrow[0,\infty]
$
be the mappings with the property that
for all $ f \in C^k( E_1, E_2 ) $
it holds that
\begin{equation}
  | f |_{ 
    C^k_b( E_1, E_2)
  }
=
  \sup_{ x \in E_1 }
  \left\|
    f^{ ( k ) }( x ) 
  \right\|_{
    L^{ k }( E_1, E_2 ) 
  }
  ,
  \quad
  \| f \|_{
    C^k_b(E_1,E_2)
  }
=
  \| f(0) \|_{ E_2 } +
  \sum^k_{ l = 1 }
  | f |_{ C^l_b(E_1,E_2) }
\end{equation}
and let 
$ 
  C^k_b( E_1, E_2 ) 
$
be the set given by
$
  C^k_b( E_1, E_2 ) 
  =
  \big\{ 
    f \in C^k( E_1, E_2 )
    \colon
    \left\| f \right\|_{ C^k_b( E_1, E_2 ) } < \infty
  \big\}
$.
Moreover, for a normed $ \R $-vector space 
$
  ( U, \left\|\cdot\right\|_{ U } )
$ 
and a linear operator 
$ 
  A \colon D( A ) \subseteq U \to U 
$ 
we denote by
$ \operatorname{spectrum}( A ) \subseteq \mathbb{ C } $ 
the spectrum of $ A $. 
For sets $ A $ and $ B $ we denote by 
$ \mathbb{M}( A, B ) $ the set of all mappings from $ A $ to $ B $. 
In addition, for measurable spaces 
$
  ( \Omega_i,\mathcal{F}_i ) 
$,
$ i \in \{ 1, 2 \}
$, 
we denote by 
$ 
  \mathcal{M}( \mathcal{F}_1, \mathcal{F}_2 )
$ 
the set of all 
$
  \mathcal{F}_1
$/$
  \mathcal{F}_2 
$-measurable mappings.
For two separable $\R$-Hilbert spaces 
$
   (
     \check{H},
     \left< \cdot , \cdot \right>_{ \check{H} },
     \left\| \cdot \right\|_{ \check{H} }
   )
$ 
and 
$
   (
     \hat{H},
     \left< \cdot , \cdot \right>_{ \hat{H} },
     \left\| \cdot \right\|_{ \hat{H} }
   )
$
let
$
  \mathcal{S}( \hat{H}, \check{H} )
$
be the strong sigma algebra on
$ L( \hat{H}, \check{H} ) $
given by
$
  \mathcal{S}( \hat{H}, \check{H} )
  =
  \sigma_{
    L( \hat{H}, \check{H} )
  }(
    \cup_{ v \in \hat{H} }
    \cup_{ \mathcal{A} \in \mathcal{B}( \check{H} ) }
    \{
      A \in L( \hat{H}, \check{H} )
      \colon
      Av \in \mathcal{A}
    \}
  )
$
(see, e.g., Section~1.2 in Da Prato \citationand\ Zabczyk~\cite{dz92}).
Finally, let 
$ \lfloor \cdot \rfloor_h \colon \R \to \R $, 
$ h \in (0,\infty) $,
and
$ \ceil{ \cdot }{ h } \colon \R \to \R $,
$ h \in ( 0 , \infty) $,
be the mappings 
with the property that for all 
$ t \in \R $, $ h \in (0,\infty) $
it holds that
\begin{equation}
  \lfloor t \rfloor_h
  =
  \max\!\big(
    ( - \infty, t ]
    \cap 
    \{ 
      0, h, -h, 2 h , - 2 h , \dots 
    \}
  \big)
  ,
\end{equation}
\begin{equation}
  \ceil{ t }{ h }
  =
  \min\!\big(
    [ t, \infty )
    \cap 
    \{ 
      0, h, -h, 2 h , - 2 h , \dots 
    \}
  \big)
  .
\end{equation}

\subsection{General setting}
\label{sec:global_setting}
Throughout this article the following
setting is frequently used.
Let 
$ 
  ( H, \left< \cdot, \cdot \right>_H, \left\| \cdot \right\|_H ) 
$, 
$ 
  ( U, 
$
$
    \left< \cdot, \cdot \right>_U, 
$
$
  \left\| \cdot \right\|_U ) 
$, 
and 
$ 
  ( V, \left< \cdot, \cdot \right>_V, \left\| \cdot \right\|_V ) 
$
be separable $ \R $-Hilbert spaces, 
let $ \mathbb{U} \subseteq U $ be an orthonormal basis of $ U $, 
let 
$ 
  A \colon D(A) \subseteq H \to H 
$
be a generator of a strongly continuous analytic semigroup
with the property that
$ 
  \sup\!\big( 
    \text{Re}( \operatorname{spectrum}( A ) ) 
  \big) < 0 
$
(cf., e.g., Theorem~11.31 in Renardy \& Rogers~\cite{rr93}),
let
$  
  ( 
    H_r 
    ,
    \left< \cdot, \cdot \right>_{ H_r }
    ,$ $
    \left\| \cdot \right\|_{ H_r } 
  )
$, 
$ r \in \R $, 
be a family of interpolation spaces associated to $ - A $
(cf., e.g., Theorem and Definition~2.5.32 in \cite{Jentzen2014SPDElecturenotes} 
and Section 11.4.2 in Renardy \& Rogers~\cite{rr93}),
let 
$ T \in (0,\infty) $, 
let
$ 
  \angle 
  =
  \left\{
    ( t_1, t_2 ) \in [0,T]^2 \colon t_1 < t_2
  \right\}
$, 
let 
$
  ( \Omega, \mathcal{F}, \P , ( \mathcal{F}_t )_{ t \in [0,T] } )
$
be a stochastic basis,
and let 
$
  ( W_t )_{ t \in [0,T] }
$
be a cylindrical $ \operatorname{Id}_U $-Wiener process
w.r.t.\ $ ( \mathcal{F}_t )_{ t \in [0,T] } $.

\subsection{Evolution family setting}
\label{sec:semigroup_setting}

In Sections~\ref{sec:weak_temporal_regularity},
\ref{sec:euler_integrated_mollified} and \ref{sec:weak_convergence_irregular} 
below the following setting is also frequently used. 
Assume the setting in Section~\ref{sec:global_setting}, 
let 
$ h \in ( 0, \infty ) $,
$
  ( \groupC_r )_{ r \in \R } \subseteq [ 1, \infty )
$,
$
  ( \groupC_{ r, \rho } )_{ r, \rho \in \R } \subseteq [ 1, \infty )
$,
$
  ( \groupC_{ r, \tilde{r}, \rho } )_{ r, \tilde{r}, \rho \in \R } \subseteq [1,\infty)
$,
$
  R \in 
  \mathcal{M}\big(
    \mathcal{B}
    ( [0,T] )
    ,
    \mathcal{B}( L( H_{ - 1 } ) )
  \big)
$,
$
  S \in 
  \mathcal{M}\big(
    \mathcal{B}( \angle )
    ,
    \mathcal{B}( L( H_{ - 1 } ) )
  \big)
$
satisfy that
for all 
$ t_1, t_2, t_3 \in [0,T] $ 
with 
$ t_1 < t_2 < t_3 $ 
it holds that
$ 
  S_{ t_2, t_3 } S_{ t_1, t_2 } = S_{ t_1, t_3 } 
$
and that for all 
$ ( s, t ) \in \angle $, 
$ r, \rho \in [ 0, 1 ) $, 
$ \tilde{r} \in [ -1, 1 - r ) $
it holds that 
$ S_{ s, t }( H ) \subseteq H $, 
$ S_{ s, t } \, R_s ( H_{ - r } ) \subseteq H_{ \tilde{r} } $, 
$
  \|
    e^{ t A }
  \|_{ 
    L( H, H_\rho ) 
  }
  \leq
  \groupC_\rho \, t^{ -\rho }
$, 
$
  \|
    e^{ t A }
    -
    \operatorname{Id}_H
    \!
  \|_{ 
    L( H, H_{ -\rho } ) 
  }
  \leq
  \groupC_\rho \, t^\rho
$,
$
  \| S_{ s, t } \|_{ L( H ) }
  \leq
  \groupC_0
$, 
$
  \|
  S_{ s, t } \, R_s 
  \|_{ L( H_{ - r }, H ) }
  \leq
  \groupC_{ r } \,
  ( t - s )^{ - r }
$, 
$
  \|
    S_{ s, t }
    - 
    e^{ ( t - s )A } 
  \|_{ L( H, H_{ - r } ) }
  \leq
  \groupC_{ - r , \rho } 
  \, h^\rho \, ( t - s )^{ - ( \rho - r )^+ }
$, 
and 
$
  \|
    S_{ s, t } \, R_s 
    - 
    e^{ ( t - s ) A } 
  \|_{ L( H_{ - r }, H_{ \tilde{r} } ) }
  \leq
  \groupC_{ r, \tilde{r}, \rho } \, 
  h^\rho \, 
  ( t - s )^{ -( \rho + r + \tilde{r} )^+ }
$.

\subsection{Examples of evolution families}
\label{sec:evolution_examples}

In this subsection we provide two examples of evolution families which satisfy 
the assumptions in Subsection~\ref{sec:semigroup_setting}.

\subsubsection{Exponential Euler approximations}
Assume the setting in Section~\ref{sec:global_setting} 
and 
let 
$
  S \in 
  \mathcal{M}\big(
    \mathcal{B}( \angle )
    ,
    \mathcal{B}( L( H_{ - 1 } ) )
  \big)
$
and
$
  R^h \in 
  \mathcal{M}\big(
    \mathcal{B}
    ( [0,T] )
    ,
$
$
    \mathcal{B}( L( H_{ - 1 } ) )
  \big)
$, 
$ h \in ( 0, \infty ) $, 
satisfy that 
for all 
$ h \in ( 0, \infty ) $, 
$ t \in [ 0, T ] $, 
$ ( t_1, t_2 ) \in \angle $ it holds that 
$
  S_{ t_1, t_2 }
  =
  e^{ ( t_2 - t_1 )A }
$
and 
$
  R^h_t
  =
  e^{ ( t - \floor{t}{h} )A }
  .
$
Then it is well-known 
(see, e.g., Lemma~11.36 in Renardy \& Rogers \cite{rr93}) 
that there exist 
real numbers 
$
  ( \groupC_r )_{ r \in \R } \subseteq [ 1, \infty )
$ 
such that 
for all 
$ h \in ( 0, \infty ) $, 
$ t_1, t_2, t_3 \in [0,T] $ 
$ ( s, t ) \in \angle $, 
$ r, \rho \in [ 0, 1 ) $, 
$ \tilde{r} \in [ -1, 1 - r ) $
with 
$ t_1 < t_2 < t_3 $ 
it holds that 
$ 
  S_{ t_2, t_3 } S_{ t_1, t_2 } = S_{ t_1, t_3 }, 
$ 
$ S_{ s, t }( H ) \subseteq H $, 
$ S_{ s, t } \, R^h_s ( H_{ - r } ) \subseteq H_{ \tilde{r} } $, 
$
  \|
    e^{ t A }
  \|_{ 
    L( H, H_\rho ) 
  }
  \leq
  \groupC_\rho \, t^{ -\rho }
$, 
$
  \|
    e^{ t A }
    -
    \operatorname{Id}_H
    \!
  \|_{ 
    L( H, H_{ -\rho } ) 
  }
  \leq
  \groupC_\rho \, t^\rho
$,
$
  \| S_{ s, t } \|_{ L( H ) }
  \leq
  \groupC_0
$, 
$
  \|
  S_{ s, t } \, R^h_s 
  \|_{ L( H_{ - r }, H ) }
  \leq
  \groupC_{ r } \,
  ( t - s )^{ - r }
$, 
$
  \|
    S_{ s, t }
    - 
    e^{ ( t - s )A } 
  \|_{ L( H, H_{ - r } ) }
  = 0
$, 
and 
$
  \|
    S_{ s, t } \, R^h_s 
    - 
    e^{ ( t - s ) A } 
  \|_{ L( H_{ - r }, H_{ \tilde{r} } ) }
  \leq
  \groupC_{ r + \tilde{r} + \rho } \, 
  h^\rho \, 
  ( t - s )^{ -( \rho + r + \tilde{r} )^+ }
$.

\subsubsection{Linear-implicit Euler approximations}
\label{sec:linear_implicit}
Assume the setting in Section~\ref{sec:global_setting} 
and 
let 
$
  S^h \in 
  \mathcal{M}\big(
    \mathcal{B}( \angle )
    ,
    \mathcal{B}( L( H_{ - 1 } ) )
  \big)
$, 
$ h \in ( 0, \infty ) $, 
and
$
  R^h \in 
  \mathcal{M}\big(
    \mathcal{B}
    ( [0,T] )
    ,
    \mathcal{B}( L( H_{ - 1 } ) )
  \big)
$, 
$ h \in ( 0, \infty ) $, 
satisfy that 
for all 
$ h \in ( 0, \infty ) $, 
$ t \in [ 0, T ] $, 
$ ( t_1, t_2 ) \in \angle $ it holds that 
$
  S^h_{ t_1, t_2 }
  =
    \big( \operatorname{Id}_H - ( t_1 - \floor{ t_1 }{ h } ) \, A \big)
    \big( \operatorname{Id}_H - ( t_2 - \floor{ t_2 }{ h } ) \, A \big)^{ -1 }
    \big( \operatorname{Id}_H - hA \big)^{
      -
      (
      \floor{ t_2 }{ h } - \floor{ t_1 }{ h }
      )
      /
      h
    }
$
and 
$
  R^h_t
  =
    \big( \operatorname{Id}_H - ( t - \floor{ t }{ h } ) \, A \big)^{ -1 }
  .
$
Then 
there exist real numbers 
$
  ( \groupC_r )_{ r \in \R } \subseteq [ 1, \infty )
$,
$
  ( \groupC_{ r, \rho } )_{ r, \rho \in \R } \subseteq [ 1, \infty )
$,
$
  ( \groupC_{ r, \tilde{r}, \rho } )_{ r, \tilde{r}, \rho \in \R } \subseteq [1,\infty)
$
such that for all 
$ h \in ( 0, \infty ) $, 
$ t_1, t_2, t_3 \in [0,T] $, 
$ ( s, t ) \in \angle $, 
$ r, \rho \in [ 0, 1 ) $, 
$ \tilde{r} \in [ -1, 1 - r ) $
with 
$ t_1 < t_2 < t_3 $ 
it holds that
$ 
  S^h_{ t_2, t_3 } S^h_{ t_1, t_2 } = S^h_{ t_1, t_3 } 
$, 
$ S^h_{ s, t }( H ) \subseteq H $, 
$ S^h_{ s, t } \, R^h_s ( H_{ - r } ) \subseteq H_{ \tilde{r} } $, 
$
  \|
    e^{ t A }
  \|_{ 
    L( H, H_\rho ) 
  }
  \leq
  \groupC_\rho \, t^{ -\rho }
$, 
$
  \|
    e^{ t A }
    -
    \operatorname{Id}_H
    \!
  \|_{ 
    L( H, H_{ -\rho } ) 
  }
  \leq
  \groupC_\rho \, t^\rho
$,
$
  \| S^h_{ s, t } \|_{ L( H ) }
  \leq
  \groupC_0
$, 
$
  \|
  S^h_{ s, t } \, R^h_s 
  \|_{ L( H_{ - r }, H ) }
  \leq
  \groupC_{ r } \,
  ( t - s )^{ - r }
$, 
$
  \|
    S^h_{ s, t }
    - 
    e^{ ( t - s )A } 
  \|_{ L( H, H_{ - r } ) }
  \leq
  \groupC_{ - r , \rho } 
  \, h^\rho \, ( t - s )^{ - ( \rho - r )^+ }
$, 
and 
$
  \|
    S^h_{ s, t } \, R^h_s 
    - 
    e^{ ( t - s ) A } 
  \|_{ L( H_{ - r }, H_{ \tilde{r} } ) }
  \leq
  \groupC_{ r, \tilde{r}, \rho } \, 
  h^\rho \, 
  ( t - s )^{ -( \rho + r + \tilde{r} )^+ }
$.

\subsection{Acknowledgements}

We are very grateful to Sonja Cox for her considerable help 
in the proof of the statement in Subsection~\ref{sec:linear_implicit}.

\section{Strong a priori estimates for SPDEs}
\label{sec:strong_a_priori}

In this section we establish in Proposition~\ref{prop:strong_apriori_estimate} below
an a priori estimate (see \eqref{eq:apriori_estimate})
for an appropriate class of stochastic processes
(see \eqref{eq:apriori_assumption})
which includes solution processes of certain SEEs as special cases.
The proof of Proposition~\ref{prop:strong_apriori_estimate}
uses the generalized Gronwall lemma in Chapter~7 in Henry~\cite{h81}
(see, e.g., also Corollary 3.4.6 in \cite{Jentzen2014SPDElecturenotes}).
Related estimates can, e.g., be found in Proposition~2.5 in 
Andersson \& Jentzen~\cite{AnderssonJentzen2014}.

\subsection{Setting}
\label{sec:setting_strong_apriori}

Let 
$ 
  ( H, \left< \cdot, \cdot \right>_H, \left\| \cdot \right\|_H ) 
$
and
$ 
  ( U, \left< \cdot, \cdot \right>_U, \left\| \cdot \right\|_U ) 
$
be separable $ \R $-Hilbert spaces, 
let 
$ T \in (0,\infty) $, 
$ p \in [ 2, \infty ) $, 
$ \vartheta \in [ 0, 1 ) $, 
$ \driftC, \diffusionC \in [ 0, \infty ) $, 
let 
$
  ( \Omega, \mathcal{F}, \P , ( \mathcal{F}_t )_{ t \in [0,T] } )
$
be a stochastic basis, 
let 
$
  ( W_t )_{ t \in [0,T] }
$
be a cylindrical $ \operatorname{Id}_U $-Wiener process
w.r.t.\ $ ( \mathcal{F}_t )_{ t \in [0,T] } $, 
let
$
  X
  \colon
  [ 0, T ] \times \Omega
  \to H
$ 
be a stochastic process 
with 
$
  \sup_{ s \in [ 0, T ] } 
  \| X_s \|_{ \lpn{ p }{ \P }{ H } }
  < \infty
$, 
and for every $ t \in ( 0, T ] $ let 
$
  Y^t \colon [ 0, t ] \times \Omega
  \to H
$
and
$
  Z^t \colon [ 0, t ] \times \Omega
  \to HS( U, H )
$ 
be $ ( \mathcal{F}_s )_{ s \in [ 0, t ] } $-predictable stochastic processes 
which satisfy that for all 
$ s \in ( 0, t ) $ 
it holds that 
\begin{equation}
\label{eq:apriori_assumption}
    \|
    Y^t_s
    \|_{ \lpn{p}{\P}{H} }
    \leq
  \tfrac{
    \driftC
    \sup_{ u \in [ 0, s ] }
    \| X_u \|_{ \lpn{ p }{ \P }{ H } }
  }{
    ( t - s )^{ \vartheta }
  }
  \qquad 
  \text{and}
  \qquad
    \|
    Z^t_s
    \|_{ \lpn{p}{\P}{HS(U,H)} }
    \leq
  \tfrac{
    \diffusionC
    \sup_{ u \in [ 0, s ] }
    \| X_u \|_{ \lpn{ p }{ \P }{ H } }
  }{
    ( t - s )^{ \nicefrac{ \vartheta }{ 2 } }
  }
  .
\end{equation}

\subsection{A strong a priori estimate}

\begin{proposition}[A strong a priori estimate]
\label{prop:strong_apriori_estimate}
Assume the setting in Section~\ref{sec:setting_strong_apriori}. 
Then it holds for all $ t \in [0,T] $ that
$
  \P\big(
    \int_0^t
    \| Y^t_s \|_H
    +
    \| Z^t_s \|^2_{ HS( U, H ) }
    \, ds
    < \infty
  \big)
  = 1
$
and it holds that
\begin{equation}
\label{eq:apriori_estimate}
\begin{split}
&
  \sup_{ t \in [ 0, T ] }
  \| X_t \|_{ \lpn{ p }{ \P }{ H } }
\leq
  \sqrt{2} \,
  \mathcal{E}_{ ( 1 - \vartheta ) }
  \!\left[
    \tfrac{ 
      \driftC 
      \sqrt{2} \, 
      T^{ ( 1 - \vartheta ) } 
    }{ 
      \sqrt{ 1 - \vartheta } 
    }
    +
    \diffusionC
    \sqrt{ 
      p \, ( p - 1 ) 
      \,
      T^{ ( 1 - \vartheta ) } 
    }
  \right]
\\ & \cdot
  \sup_{ t \in [ 0, T ] }
  \left\|
    X_t
    -
    \left[
    \int^t_0
    Y^t_s
    \, ds
    +
    \int^t_0
    Z^t_s
    \, dW_s
    \right]
  \right\|_{ \lpn{ p }{ \P }{ H } }
\leq
  \left[
    1
    +
    \tfrac{
      \driftC \, T^{ ( 1 - \vartheta ) }
    }{ 
      ( 1 - \vartheta ) 
    }
    +
    \tfrac{
      \diffusionC 
      \sqrt{
        p \, ( p - 1 ) \, T^{ ( 1 - \vartheta ) }
      }
    }{ 
      \sqrt{ 
        2 \, ( 1 - \vartheta ) 
      }
    }
  \right]
\\ & \cdot
  \sqrt{2} \,
  \mathcal{E}_{ ( 1 - \vartheta ) }\!\left[
    \tfrac{ 
      \driftC 
      \sqrt{2} \, 
      T^{ ( 1 - \vartheta ) } 
    }{ 
      \sqrt{ 1 - \vartheta } 
    }
    +
    \diffusionC
    \sqrt{ 
      p \, ( p - 1 ) 
      \,
      T^{ ( 1 - \vartheta ) } 
    }
  \right]
  \sup_{ t \in [ 0, T ] }
  \|
    X_t
  \|_{ \lpn{ p }{ \P }{ H } }
  < \infty
  .
\end{split}
\end{equation}
\end{proposition}
\begin{proof}
We first observe that 
\eqref{eq:apriori_assumption},
H\"{o}lder's inequality,
and the assumption that
$
  \sup_{ s \in [0,T] } \| X_s \|_{ L^p( \P; H ) }
$
$
  < \infty 
$
imply that for all 
$ t \in [ 0, T ] $ 
it holds that 
\begin{equation}
\label{eq:apriori_F}
\begin{split}
&
  \int^t_0
    \|
      Y^t_s
    \|_{ \lpn{ p }{ \P }{ H } }
    \,
  ds
\leq
  \driftC
  \int^t_0
  \frac{
    \sup_{ v \in [ 0, s ] }
    \|
      X_v
    \|_{ \lpn{ p }{ \P }{ H } }
  }{
    ( t - s )^\vartheta
  }
  \, ds
\\ & 
\leq
  \driftC
  \left[
  \frac{ t^{ ( 1 - \vartheta ) } }{ ( 1 - \vartheta ) }
  \int^{ t }_0
  \frac{
    \sup_{ v \in [ 0, s ] }
    \|
      X_v
    \|^2_{ \lpn{ p }{ \P }{ H } }
  }{
    ( t - s )^\vartheta
  }
  \, ds
  \right]^{ 1 / 2 }
  < \infty .
\end{split}
\end{equation}
In addition, we note that
\eqref{eq:apriori_assumption}
and again the assumption that
$
  \sup_{ s \in [0,T] } \| X_s \|_{ L^p( \P; H ) } < \infty 
$
show that for all 
$ t \in [ 0, T ] $ 
it holds that 
\begin{equation}
\label{eq:apriori_B}
\begin{split}
&
  \left[
    \tfrac{ 
      p \, ( p - 1 ) 
    }{ 
      2 
    }
    \int^t_0
      \|
        Z^t_s
      \|^2_{ 
        \lpn{ p }{ \P }{ HS( U, H ) } 
      }
      \,
    ds
    \,
  \right]^{ 1 / 2 }
\leq
  \diffusionC
  \left[
    \tfrac{
      p \, (p-1)
    }{
      2
    }
    \int^{ t }_0
    \frac{
      \sup_{ v \in [ 0, s ] }
      \|
        X_v
      \|^2_{ \lpn{ p }{ \P }{ H } }
    }{
      ( t - s )^\vartheta
    }
    \, ds
  \right]^{ 1 / 2 }
  < \infty
  .
\end{split}
\end{equation}
Combining~\eqref{eq:apriori_F}--\eqref{eq:apriori_B} and the assumption that 
$ p \geq 2 $ proves that 
for all $ t \in [0,T] $ it holds that
$
  \int_0^t
  \| Y^t_s \|_{ L^1( \P; H ) }
  +
  \| Z^t_s \|^2_{ L^2( \P; HS( U, H ) ) }
  \, ds
  < \infty
$.
This, in turn, shows that for all 
$ t \in [ 0, T ] $ 
it holds $ \P $-a.s.\ that 
\begin{equation}
\label{eq:finite_integral}
    \int_0^t
    \| Y^t_s \|_H
    +
    \| Z^t_s \|^2_{ HS( U, H ) }
    \, ds
    < \infty
    .
\end{equation}
It thus remains to prove \eqref{eq:apriori_estimate} to complete
the proof of Proposition~\ref{prop:strong_apriori_estimate}.
For this observe that \eqref{eq:apriori_F}--\eqref{eq:finite_integral}
and the Burkholder-Davis-Gundy type inequality 
in Lemma~7.7 in Da Prato \& Zabczyk~\cite{dz92} 
imply that for all $ t \in [0,T] $
it holds that
\begin{equation}
\label{eq:apriori_combined}
\begin{split}
&
  \left\|
    \int^t_0
    Y^t_s
    \, ds
  \right\|_{ \lpn{ p }{ \P }{ H } }
+
  \left\|
    \int^t_0
    Z^t_s
    \, dW_s
  \right\|_{ \lpn{ p }{ \P }{ H } }
\\ & \leq
  \left[
    \frac{ \driftC \, t^{ \nicefrac{( 1 - \vartheta )}{2} } }{ \sqrt{ 1 - \vartheta } }
    +
    \diffusionC \, 
    \sqrt{ 
      \frac{
        p \, ( p - 1 )
      }{
        2
      } 
    } 
    \,
  \right]
  \left[
    \int^t_0
    \frac{
      \sup_{ v \in [ 0, s ] }
      \|
        X_v
      \|^2_{ \lpn{ p }{ \P }{ H } }
    }{
      ( t - s )^\vartheta
    }
    \, ds
  \right]^{ 1 / 2 }.
\end{split}
\end{equation}
Next we observe (cf., e.g., \cite{HutzenthalerJentzenKurniawan2014})
that for all 
$ t, u \in [ 0, T ] $ 
with $ t \leq u $ 
it holds that 
\begin{equation}
\label{eq:apriori_pre-gronwall}
\begin{split}
&
  \int^{ t }_0
  \frac{
  \sup_{ v \in [ 0, s ] }
  \big\|
    X_v
  \big\|^2_{ \lpn{ p }{ \P }{ H } }
  }{
    ( t - s )^\vartheta
  }
  \, ds
=
  \int^{ u }_{ u - t }
  \frac{
  \sup_{ v \in [ 0, s - u + t ] }
  \big\|
    X_v
  \big\|^2_{ \lpn{ p }{ \P }{ H } }
  }{
    ( u - s )^\vartheta
  }
  \, ds
\\ & \leq
  \int^{ u }_{ u - t }
  \frac{
  \sup_{ v \in [ 0, s ] }
  \big\|
    X_v
  \big\|^2_{ \lpn{ p }{ \P }{ H } }
  }{
    ( u - s )^\vartheta
  }
  \, ds
\leq
  \int^{ u }_0
  \frac{
  \sup_{ v \in [ 0, s ] }
  \big\|
    X_v
  \big\|^2_{ \lpn{ p }{ \P }{ H } }
  }{
    ( u - s )^\vartheta
  }
  \, ds
  .
\end{split}
\end{equation}
Moreover, we note that the
Minkowski inequality ensures that for all 
$ t \in [ 0, T ] $ 
it holds that 
\begin{equation}
\begin{split}
\label{eq:apriori_decomposition}
&
  \| X_t \|_{ \lpn{ p }{ \P }{ H } }
\\ & \leq
  \Bigg\|
    X_t
    -
    \left[
    \int^t_0
    Y^t_s
    \, ds
    +
    \int^t_0
    Z^t_s
    \, dW_s
    \right]
  \Bigg\|_{ \lpn{ p }{ \P }{ H } }
+
  \Bigg\|
    \int^t_0
    Y^t_s
    \, ds
  \Bigg\|_{ \lpn{ p }{ \P }{ H } }
+
  \Bigg\|
    \int^t_0
    Z^t_s
    \, dW_s
  \Bigg\|_{ \lpn{ p }{ \P }{ H } }
  .
\end{split}
\end{equation}
Combining \eqref{eq:apriori_combined}--\eqref{eq:apriori_decomposition} with the fact that
$
  \forall \, a, b \in \R \colon \left( a + b \right)^2 \leq 2 a^2 + 2 b^2 
$
proves that for all $ u \in [ 0, T ] $ 
it holds that 
\allowdisplaybreaks
\begin{align}
&
  \sup_{ t \in [ 0, u ] }
  \| X_t \|^2_{ \lpn{ p }{ \P }{ H } }
\leq
  2 \, \sup_{ t \in [ 0, T ] }
  \Bigg\|
    X_t
    -
    \left[
    \int^t_0
    Y^t_s
    \, ds
    +
    \int^t_0
    Z^t_s
    \, dW_s
    \right]
  \Bigg\|^2_{ \lpn{ p }{ \P }{ H } }
\\ & +
  2\left[
  \frac{ \driftC \, T^{ \nicefrac{( 1 - \vartheta )}{2} } }{ \sqrt{ 1 - \vartheta } }
  +
  \diffusionC \, 
  \sqrt{ \frac{p\,( p - 1 )}{2} } \,
  \right]^2
  \int^u_0
  \frac{
  \sup_{ t \in [ 0, s ] }
  \big\|
    X_t
  \big\|^2_{ \lpn{ p }{ \P }{ H } }
  }{
    ( u - s )^\vartheta
  }
  \, ds
  .
\end{align}
This and the assumption that
$
  \sup_{ s \in [0,T] }
  \| X_s \|_{ L^p( \P; H ) } < \infty
$
together with 
the generalized Gronwall lemma in Chapter 7 in Henry~\cite{h81}
(see, e.g., also Corollary 3.4.6 in \cite{Jentzen2014SPDElecturenotes})
proves the first inequality in \eqref{eq:apriori_estimate}.
In the next step we note that 
\eqref{eq:apriori_combined}
implies that 
\begin{equation}
\begin{split}
&
  \sup_{ t \in [0,T] }
  \left\|
    X_t
    -
    \left[
    \int^t_0
    Y^t_s
    \, ds
    +
    \int^t_0
    Z^t_s
    \, dW_s
    \right]
  \right\|_{ 
    \lpn{ p }{ \P }{ H } 
  }
\\ & \leq
  \sup_{ t \in [0,T] }
  \left\|
    X_t
  \right\|_{
    \lpn{ p }{ \P }{ H } 
  }
  +
  \sup_{ t \in [0,T] }
  \left[
  \left\|
    \int^t_0
    Y^t_s
    \, ds
  \right\|_{
    \lpn{ p }{ \P }{ H } 
  }
  +
  \left\|
    \int^t_0
    Z^t_s
    \, dW_s
  \right\|_{
    \lpn{ p }{ \P }{ H } 
  }
  \right]
\\ & \leq
  \left[
  1
  +
  \frac{
  \driftC \, T^{ ( 1 - \vartheta ) }
  }{ ( 1 - \vartheta ) }
  +
  \diffusionC \,
  \sqrt{
  \frac{
    p \, ( p - 1 ) \, T^{ ( 1 - \vartheta ) }
  }{ 2 \, ( 1 - \vartheta ) }
  }
  \right]
  \sup_{ t \in [ 0, T ] }
  \big\|
    X_t
  \big\|_{ \lpn{ p }{ \P }{ H } }
  .
\end{split}
\end{equation}
This proves the second inequality 
in \eqref{eq:apriori_estimate}.
The third inequality in \eqref{eq:apriori_estimate}
is an immediate consequence of the 
assumption that
$
  \sup_{ s \in [0,T] }
  \| X_s \|_{
    L^p( \P; H )
  }
  < \infty
$.
The proof of Proposition~\ref{prop:strong_apriori_estimate}
is thus completed.
\end{proof}

\section{Strong perturbations for SPDEs}
\label{sec:strong_perturbation}

In this section we prove in 
Corollary~\ref{prop:general_perturb}
a perturbation estimate 
(see \eqref{eq:perturbation_estimate})
for an appropriate class of stochastic processes
which includes solution processes of certain SEEs as special cases.
Corollary~\ref{prop:general_perturb} 
follows immediately 
from Proposition~\ref{prop:strong_apriori_estimate} in 
Section~\ref{sec:strong_a_priori}.
Corollary~\ref{prop:general_perturb} extends 
the perturbation estimate in Proposition~2.5 in 
Andersson \& Jentzen~\cite{AnderssonJentzen2014}. 
Further related strong perturbation estimates for SEEs can, e.g., be found in Hutzenthaler \& Jentzen~\cite{HutzenthalerJentzen2014}.

\subsection{Setting}
\label{sec:setting_strong_perturbation}

Let 
$ 
  ( H, \left< \cdot, \cdot \right>_H, \left\| \cdot \right\|_H ) 
$
and
$ 
  ( U, \left< \cdot, \cdot \right>_U, \left\| \cdot \right\|_U ) 
$
be separable $ \R $-Hilbert spaces, 
let 
$ T \in (0,\infty) $, 
$ p \in [ 2, \infty ) $, 
$ \vartheta \in [ 0, 1 ) $, 
$ \driftC, \diffusionC \in [ 0, \infty ) $, 
let 
$
  ( \Omega, \mathcal{F}, \P , ( \mathcal{F}_t )_{ t \in [0,T] } )
$
be a stochastic basis, 
let 
$
  ( W_t )_{ t \in [0,T] }
$
be a cylindrical $ \operatorname{Id}_U $-Wiener process
w.r.t.\ $ ( \mathcal{F}_t )_{ t \in [0,T] } $, 
let
$
  X, \bar{X}
  \colon
  [ 0, T ] \times \Omega
  \to H
$ 
be 
stochastic processes 
with 
$
  \sup_{ s \in [ 0, T ] } 
  \| X_s - \bar{X}_s \|_{ \lpn{ p }{ \P }{ H } }
  < \infty
$, 
and for every $ t \in ( 0, T ] $ let 
$
  Y^t, \bar{Y}^t \colon [ 0, t ] \times \Omega
  \to H
$, 
$
  Z^t, \bar{Z}^t \colon [ 0, t ] \times \Omega
  \to HS( U, H )
$ 
be $ ( \mathcal{F}_s )_{ s \in [ 0, t ] } $-predictable stochastic processes 
such that it holds $\P$-a.s.\ that 
$
  \int^t_0
  \| Y^t_s \|_H
  +
  \| \bar{Y}^t_s \|_H
  +
  \| Z^t_s \|^2_{ HS( U, H ) }
  +
  \| \bar{Z}^t_s \|^2_{ HS( U, H ) }
  \, ds
  < \infty
$ 
and such that for all 
$ s \in ( 0, t ) $ 
it holds that 
\begin{equation}
\label{eq:difference_assumption}
    \|
    Y^t_s - \bar{Y}^t_s
    \|_{ \lpn{p}{\P}{H} }
    \leq
  \tfrac{
    \driftC
    \sup_{ u \in [ 0, s ] }
    \| X_u - \bar{X}_u \|_{ \lpn{ p }{ \P }{ H } }
  }{
    ( t - s )^{ \vartheta }
  }
    ,
\quad
    \|
    Z^t_s - \bar{Z}^t_s
    \|_{ \lpn{p}{\P}{HS(U,H)} }
    \leq
  \tfrac{
    \diffusionC
    \sup_{ u \in [ 0, s ] }
    \| X_u - \bar{X}_u \|_{ \lpn{ p }{ \P }{ H } }
  }{
    ( t - s )^{ \nicefrac{ \vartheta }{ 2 } }
  }
  .
\end{equation}

\subsection{Strong perturbation estimates}
\label{sec:strong_perturbation_estimates}

In the next result, Corollary~\ref{prop:general_perturb},
a certain strong perturbation estimate is presented.
Corollary~\ref{prop:general_perturb} is an immediate consequence of 
Proposition~\ref{prop:strong_apriori_estimate}.
Corollary~\ref{prop:general_perturb} is an extension of
the perturbation estimate in Proposition~2.5 in 
Andersson \& Jentzen~\cite{AnderssonJentzen2014}.

\begin{corollary}[A strong perturbation estimate]
\label{prop:general_perturb}
Assume the setting in Section~\ref{sec:setting_strong_perturbation}. 
Then 
\begin{equation}
\label{eq:perturbation_estimate}
\begin{split}
&
  \sup_{ t \in [ 0, T ] }
  \| X_t - \bar{X}_t \|_{ \lpn{ p }{ \P }{ H } }
\leq
  \sqrt{2} \,
  \mathcal{E}_{ ( 1 - \vartheta ) }\!\left[
    \tfrac{ 
      \driftC 
      \sqrt{2} \, 
      T^{ ( 1 - \vartheta ) } 
    }{ 
      \sqrt{ 1 - \vartheta } 
    }
    +
    \diffusionC
    \sqrt{ 
      p \, ( p - 1 ) 
      \,
      T^{ ( 1 - \vartheta ) } 
    }
  \right]
\\ & \cdot
  \sup_{ t \in [ 0, T ] }
  \left\|
    X_t
    -
    \left[
    \int^t_0
    Y^t_s
    \, ds
    +
    \int^t_0
    Z^t_s
    \, dW_s
    \right]
+
    \left[
    \int^t_0
    \bar{Y}^t_s
    \, ds
    +
    \int^t_0
    \bar{Z}^t_s
    \, dW_s
    \right]
    -
    \bar{X}_t
  \right\|_{ \lpn{ p }{ \P }{ H } }
\\ & \leq
  \sqrt{2} \,
  \mathcal{E}_{ ( 1 - \vartheta ) }\!\left[
    \tfrac{ 
      \driftC 
      \sqrt{2} \, 
      T^{ ( 1 - \vartheta ) } 
    }{ 
      \sqrt{ 1 - \vartheta } 
    }
    +
    \diffusionC
    \sqrt{ 
      p \, ( p - 1 ) 
      \,
      T^{ ( 1 - \vartheta ) } 
    }
  \right]
  \left[
  1
  +
  \tfrac{
    \driftC \, T^{ ( 1 - \vartheta ) }
  }{ 
    ( 1 - \vartheta ) 
  }
  +
  \diffusionC \,
  \sqrt{
    \tfrac{
      p \, ( p - 1 ) \, T^{ ( 1 - \vartheta ) }
    }{ 2 \, ( 1 - \vartheta ) }
  }
  \right]
\\ & 
  \cdot
  \sup_{ t \in [ 0, T ] }
  \|
    X_t
    -
    \bar{X}_t
  \|_{ \lpn{ p }{ \P }{ H } }
  < \infty
  .
\end{split}
\end{equation}
\end{corollary}

As an application of Corollary~\ref{prop:general_perturb}
we establish in the next result, Corollary~\ref{cor:initial_perturbation},
an a priori estimate for the difference of two 
numerical approximation processes with possibly 
different initial values.
Corollary~\ref{cor:initial_perturbation} is an extension 
of Corollary~2.6 in Andersson \& Jentzen~\cite{AnderssonJentzen2014}.

\begin{corollary}
\label{cor:initial_perturbation}
Assume the setting in Section~\ref{sec:setting_strong_perturbation}, 
let 
$
  S \in
  \mathbbm{M}( [0,T], L(H) )
$,
and assume that for all 
$ t \in [ 0, T ] $ 
it holds $ \P $-a.s.\ that
\begin{equation}
  X_t
  =
    S_t\, X_0
  + 
    \int_0^t Y^t_s \, ds
  + 
    \int_0^t Z^t_s \, dW_s, 
    \qquad
  \bar{X}_t
  =
    S_t\, \bar{X}_0
  + 
    \int_0^t \bar{Y}^t_s \, ds
  + 
    \int_0^t \bar{Z}^t_s \, dW_s. 
\end{equation}
Then 
\begin{equation}
\begin{split}
&
  \sup_{ t \in [ 0, T ] }
  \| X_t - \bar{X}_t \|_{ \lpn{ p }{ \P }{ H } }
\\ & \leq
  \sqrt{2} \, 
  \left[
  \sup_{ t \in [ 0, T ] }
  \| S_t \|_{ L( H ) }
  \right]
  \|
    X_0
    -
    \bar{X}_0
  \|_{ \lpn{ p }{ \P }{ H } } 
  \,
  \mathcal{E}_{ ( 1 - \vartheta ) }
  \!\left[
    \tfrac{ 
      \driftC 
      \sqrt{2} 
      \, T^{ ( 1 - \vartheta ) } 
    }{ 
      \sqrt{ 1 - \vartheta } 
    }
    +
    \diffusionC
    \sqrt{ 
      p \, ( p - 1 ) 
      \,
      T^{ ( 1 - \vartheta ) } 
    }
  \right]
  .
\end{split}
\end{equation}
\end{corollary}

\section{Strong convergence of mollified solutions for SPDEs}
\label{sec:strong_convergence}

In this section we establish in Proposition~\ref{prop:strong_convergence_numerics} below
an elementary a priori bound on the difference between 
a certain stochastic process 
and 
a mollified version of this process.
In the proof of Proposition~\ref{prop:strong_convergence_numerics} 
we use Corollary~\ref{prop:general_perturb}
from Section~\ref{sec:strong_perturbation} above.
Results related to Proposition~\ref{prop:strong_convergence_numerics} can, e.g., be found 
in Proposition~4.1 in Conus et al.~\cite{ConusJentzenKurniawan2014arXiv}
and 
in Lemma~2.8 in 
Andersson \& Jentzen~\cite{AnderssonJentzen2014}.

\subsection{Setting}
\label{sec:setting_strong_convergence}

Assume the setting in Section~\ref{sec:global_setting}, 
let 
$ p \in [ 2, \infty ) $, $ \vartheta \in [0,1) $, 
$
  \varPi \in \mathcal{M}\big( \mathcal{B}( [ 0, T ] ), \mathcal{B}( [ 0, T ] ) \big)
$,
$ 
  ( \groupC_r )_{ 
    r \in 
    [ 0, 1 ] 
  } 
  \subseteq [ 1 , \infty ) 
$, 
$
  F \in 
  \operatorname{Lip}^0( H , H_{ - \vartheta } ) 
$, 
$
  B \in 
  \operatorname{Lip}^0( 
    H, 
    HS( 
      U, 
      H_{ 
        - \nicefrac{ \vartheta }{ 2 } 
      } 
    ) 
  ) 
$, 
$ 
  L \in 
  \mathcal{M}\big( 
    \mathcal{B}( \angle ) , 
    \mathcal{B}( L( H_{ - 1 } ) )
  \big)
$ 
satisfy that
for all $ t \in [ 0, T ] $ 
it holds that 
$ \varPi( t ) \leq t $
and that for all 
$ ( s, t ) \in \big( \angle \cap (0,T]^2 \big) $, 
$ \rho \in [ 0, 1 ) $ 
it holds that 
$
  L_{ 0, t }( H ) \subseteq H
$, 
$
  L_{ s, t }( H_{ -\rho } ) \subseteq H
$, 
$
  \|
    e^{ t A }
  \|_{ 
    L( H ) 
  }
  \leq
  \groupC_0
$, 
$
  \|
    e^{ t A }
    -
    \operatorname{Id}_H
    \!
  \|_{ 
    L( H_{ \rho } , H ) 
  }
  \leq
  \groupC_\rho \, t^\rho
$, 
and  
$
  \|
    L_{ s, t }
  \|_{ L( H_{ - \rho }, H ) }
  \leq
  \groupC_{ \rho }
  \,
  ( t - s )^{ - \rho }
  ,
$ 
and let 
$ 
  Y^\kappa \colon [0,T] \times \Omega \to H
$, 
$ \kappa \in [ 0, T ] $, 
be $ ( \mathcal{F}_t )_{ t \in [0,T] } $-predictable 
stochastic processes which satisfy that
for all $ \kappa \in [0,T] $
it holds that
$
  \sup_{ t \in [0,T] }
  \| Y^\kappa_{\Pi(t)} \|_{ \lpn{p}{\P}{H} } 
  < \infty
$ 
and which satisfy 
that for all 
$ \kappa \in [0,T] $,
$ t \in (0,T] $ 
it holds $ \P $-a.s.\ that 
$
  Y^{ \kappa }_0 = Y^0_0
$
and 
\begin{equation}
  Y^\kappa_t
  = 
    L_{ 0, t } \, Y^\kappa_0 
  + 
    \int_0^t L_{s,t} \, e^{ \kappa A } F( Y^\kappa_{ \varPi(s) } ) \, ds
  + 
    \int_0^t L_{s,t} \, e^{ \kappa A } B( Y^\kappa_{ \varPi(s) } ) \, dW_s
  .
\end{equation}

\subsection{A priori bounds for the non-mollified process}
\label{secstrong:a_priori}

In this subsection we establish two elementary and essentially well-known a priori bounds 
for the processes $ Y^{ \kappa } $, $ \kappa \in [0,T] $,
from Subsection~\ref{sec:setting_strong_convergence}.
The first a priori bound is presented in Lemma~\ref{lem:strong_apriori_bound}
and the second a priori bound is given in Proposition~\ref{prop:numerics_Lp_bound} below.

\begin{lemma}
\label{lem:strong_apriori_bound}
Assume the setting in Section~\ref{sec:setting_strong_convergence} 
and let $ \kappa \in [ 0, T ] $. 
Then 
$
  \sup_{ t \in [0,T] }
  \| Y^\kappa_t \|_{ \lpn{p}{\P}{H} } 
$
$
  < \infty
  .
$ 
\end{lemma}
\begin{proof}
We observe that the Burkholder-Davis-Gundy type inequality 
in Lemma~7.7 in Da Prato \& Zabczyk~\cite{dz92} 
ensures that for all 
$ t \in [ 0, T ] $ 
it holds that 
\begin{equation}
\begin{split}
&
  \| Y^\kappa_t \|_{ \lpn{p}{\P}{H} }
\leq
  \| L_{ 0, t } \, Y^\kappa_0 \|_{ \lpn{p}{\P}{H} }
  +
    \int^t_0
  \|
    L_{s,t} \, e^{ \kappa A } F( Y^\kappa_{ \varPi(s) } )
  \|_{ \lpn{p}{\P}{H} }
    \, ds
\\ & +
  \left[
    \tfrac{ p \, ( p - 1 ) }{2}
    \int^t_0
  \|
    L_{s,t} \, e^{ \kappa A } B( Y^\kappa_{ \varPi(s) } )
  \|^2_{ \lpn{p}{\P}{ HS( U, H ) } }
    \, ds
  \right]^{ 1/2 }
\\ & \leq
  \groupC_0 \,
  \| Y^\kappa_0 \|_{ \lpn{p}{\P}{H} }
  +
  \int^t_0
  \tfrac{
    \groupC_0 \, \groupC_\vartheta
    \,
    \| F( Y^\kappa_{ \Pi(s) } ) \|_{ \lpn{p}{\P}{ H_{ -\vartheta } } }
  }{ ( t - s )^\vartheta }
  \, ds
\\ & \quad +
  \left[
    \tfrac{ p \, ( p - 1 ) }{2}
    \int^t_0
  \tfrac{
    |\groupC_0|^2 \, |\groupC_{ \nicefrac{\vartheta}{2} }|^2
    \,
    \| B( Y^\kappa_{ \Pi(s) } ) \|^2_{ \lpn{p}{\P}{ HS( U, H_{ -\nicefrac{\vartheta}{2} } ) } }
  }{ ( t - s )^\vartheta }
    \, ds
  \right]^{ 1/2 }
\\ & \leq
  \left[
    \groupC_0
    +
    \tfrac{
      \groupC_0 \, \groupC_\vartheta \, T^{ ( 1 - \vartheta ) } \,
      \| F \|_{ \operatorname{Lip}^0( H, H_{ -\vartheta } ) }
    }{ ( 1 - \vartheta ) }
    +
    \tfrac{
      \groupC_0 \, \groupC_{ \vartheta/2 }
      \sqrt{
        p \, ( p - 1 ) \, T^{ ( 1 - \vartheta ) }
      }
      \,
      \| B \|_{ \operatorname{Lip}^0( H, HS( U, H_{ -\vartheta/2 } ) ) }
    }{
      \sqrt{ 2 - 2 \vartheta }
    }
  \right]
\\ & \quad \cdot
  \sup_{ s \in [0,T] }
  \|
  \max\{
    1
    ,
    \| Y^\kappa_{ \Pi(s) } \|_H
  \}
  \|_{ \lpn{p}{\P}{\R} } 
  .
\end{split}
\end{equation}
This and the fact that 
$
  \sup_{ t \in [0,T] }
  \|
  \max\{
    1
    ,
    \| Y^\kappa_{ \Pi(t) } \|_H
  \}
  \|_{ \lpn{p}{\P}{\R} } 
  \leq
  1
  +
  \sup_{ t \in [0,T] }
  \|
  Y^\kappa_{ \Pi(t) }
  \|_{ \lpn{p}{\P}{H} } 
  < \infty
$
complete the proof of Lemma~\ref{lem:strong_apriori_bound}.
\end{proof}

In the next result, Proposition~\ref{prop:numerics_Lp_bound}, 
an a priori bound for the process $ Y^0 $ is established.
The proof of Proposition~\ref{prop:numerics_Lp_bound}
uses Corollary~\ref{prop:general_perturb} and Lemma~\ref{lem:strong_apriori_bound} above.

\begin{proposition}[An a priori bound for the non-mollified process]
\label{prop:numerics_Lp_bound}
Assume the setting in Section~\ref{sec:setting_strong_convergence}. 
Then 
\begin{equation}
\begin{split}
&
  \sup_{ t \in [0,T] }
  \| Y^0_t \|_{
   \lpn{ p }{ \P }{ H }
  }
\leq
  \sqrt{2} 
  \,
  \bigg[
    \sup_{ t \in ( 0, T ] }
    \max\!\big\{
      1 
      , 
      \| 
        L_{ 0, t } 
      \|_{ L(H) }
    \big\}
    \,
    \| 
      Y^0_0 
    \|_{ \lpn{p}{\P}{H} }
\\ & 
    +
    \tfrac{
      \groupC_\vartheta 
      \,
      T^{ ( 1 - \vartheta ) } \,
      \| F(0) \|_{ H_{ -\vartheta } }
    }{
      ( 1 - \vartheta )
    }
    +
    \tfrac{
      \groupC_{ \vartheta / 2 }
      \sqrt{
        p \, ( p - 1 ) \, T^{ ( 1 - \vartheta ) }
      }
      \,
      \| B(0) \|_{ HS( U, H_{ - \vartheta / 2 } ) }
    }{
      \sqrt{
        2 ( 1 - \vartheta )
      }
    }
  \bigg]
\\ & \cdot
  \mathcal{E}_{ ( 1 - \vartheta ) }\!\left[
    \tfrac{
      \sqrt{ 2 }
      \,
      \groupC_{ \vartheta }
      \,
      T^{ ( 1 - \vartheta ) }
      \,
      |
        F
      |_{
        \operatorname{Lip}^0( H, H_{ - \vartheta } )
      }
    }{
      \sqrt{1 - \vartheta}
    }
    +
    \groupC_{ 
      \vartheta / 2 
    }
    \sqrt{
      p \, (p-1) \, T^{ ( 1 - \vartheta ) }
    } \,
    |
      B
    |_{
      \operatorname{Lip}^0( H, HS( U, H_{ - \vartheta / 2 } ) )
    }
  \right]
  < \infty
  .
\end{split}
\end{equation}
\end{proposition}
\begin{proof}
Throughout this proof let 
$
  \tilde{L}
  \colon
  \{ ( t_1, t_2 ) \in [ 0, T ]^2 \colon t_1 \leq t_2 \}
  \to L( H_{-1} )
$ 
be the mapping with the property that for all 
$ ( t_1, t_2 ) \in \angle $,
$ v \in H_{ - 1 } $
it holds that 
$ \tilde{L}_{ t_1, t_2 } v = L_{ t_1, t_2 } v $ 
and with the property that for all $ t \in [ 0, T ] $ 
it holds that 
$ \tilde{L}_{t,t} = \operatorname{Id}_{ H_{ -1 } } $. 
Combining Corollary~\ref{prop:general_perturb} and Lemma~\ref{lem:strong_apriori_bound}
shows\footnote{with 
$ \bar{X}_t = 0 $,
$ \bar{Y}^t_s = \tilde{L}_{ s, t } F(0) $,
$ \bar{Z}^t_s = \tilde{L}_{ s, t } B(0) $
for $ s \in (0,t) $, $ t \in (0,T] $
in the notation of Corollary~\ref{prop:general_perturb}} 
that 
\begin{equation}
\begin{split}
\label{eq:numerics_lp_bound}
&
  \sup_{ t \in [ 0, T ] }
  \left\|
    Y^0_t
  \right\|_{ \lpn{ p }{ \P }{ H } }
\leq
  \sqrt{2}
  \sup_{ t \in [ 0, T ] }
  \left\|
    \tilde{L}_{ 0, t } \, Y^0_0
    +
    \int^t_0
    \tilde{L}_{ s, t } \,
    F(0)
    \, ds
    +
    \int^t_0
    \tilde{L}_{ s, t } \,
    B(0)
    \, dW_s
  \right\|_{ \lpn{ p }{ \P }{ H } }
\\ & \cdot
  \mathcal{E}_{ ( 1 - \vartheta ) }\!\left[
    \tfrac{ \sqrt{2} \, T^{( 1 - \vartheta )} }{ \sqrt{ 1 - \vartheta } }
    \groupC_\vartheta \,
    | F |_{ \operatorname{Lip}^0( H, H_{ -\vartheta } ) }
    +
    \sqrt{ p \, ( p - 1 ) \, T^{( 1 - \vartheta )} } \,
    \groupC_{ \nicefrac{ \vartheta }{ 2 } } \,
    | B |_{ 
      \operatorname{Lip}^0( H, HS( U, H_{ -\nicefrac{ \vartheta }{ 2 } } ) ) 
    }
  \right]
  .
\end{split}
\end{equation}
Combining \eqref{eq:numerics_lp_bound}
with the triangle inequality 
and the Burkholder-Davis-Gundy type inequality 
in Lemma~7.7 in Da Prato \& Zabczyk~\cite{dz92} 
completes the proof of Proposition~\ref{prop:numerics_Lp_bound}.
\end{proof}

\subsection{A strong convergence result}
\label{secstrong:strong}

\begin{proposition}[A bound on the difference between the mollified and the non-mollified processes]
\label{prop:strong_convergence_numerics}
Assume the setting in Section~\ref{sec:setting_strong_convergence} and let 
$ \kappa \in [ 0, T ] $, 
$ \rho \in [ 0, \frac{ 1 - \vartheta }{ 2 } ) $. 
Then 
\begin{align}
&
\label{eq:strong_convergence_numerics}
  \sup_{ t \in [0,T] }
  \left\|
    Y^0_t - Y^{ \kappa }_t
  \right\|_{ 
    \lpn{ p }{ \P }{ H } 
  }
\leq
  \tfrac{ 
    2 \, \kappa^\rho 
  }{ 
    T^\rho 
  }
  \bigg[
    \sup\nolimits_{ t \in ( 0, T ] }
    \max\!\big\{
      1
      , 
      \| 
        L_{ 0, t } 
      \|_{ L(H) }
    \big\}
    \,
    \max\{ 1, \| Y^0_0 \|_{ \lpn{p}{\P}{H} } \}
\nonumber
\\ & +
    \tfrac{
      \groupC_\rho 
      \, 
      \groupC_\vartheta 
      \, 
      \groupC_{ \rho + \vartheta }
      \,
      T^{ ( 1 - \vartheta ) }
      \,
      \| 
        F 
      \|_{ 
        \operatorname{Lip}^0( H, H_{ - \vartheta } ) 
      }
    }{
      ( 1 - \vartheta - \rho )
    }
    +
    \tfrac{
      \groupC_\rho \, 
      \groupC_{ \vartheta / 2 } \, 
      \groupC_{ \rho + \vartheta / 2 }
      \sqrt{
        p \, 
        ( p - 1 ) \, 
        T^{ ( 1 - \vartheta ) }
      }
      \,
      \| B \|_{ 
        \operatorname{Lip}^0( 
          H, HS( U, H_{ - \vartheta / 2 } ) 
        ) 
      }
    }{
      \sqrt{ 
        2 \, ( 1 - \vartheta - 2 \rho )
      }
    }
  \bigg]^2
\\ & \cdot
\nonumber
  \Big|
    \mathcal{E}_{ 
      ( 1 - \vartheta ) 
    }\Big[
      \tfrac{ 
        \sqrt{2} \, T^{( 1 - \vartheta )} \, 
        \groupC_0 \, \groupC_\vartheta 
      }{ 
        \sqrt{ 1 - \vartheta } 
      }
      | F |_{ 
        \operatorname{Lip}^0( H, H_{ -\vartheta } ) 
      }
      +
      \sqrt{ 
        p \, ( p - 1 ) \,
        T^{ 
          ( 1 - \vartheta ) 
        } 
      }
      \,
      \groupC_0 \, 
      \groupC_{ \vartheta / 2 } \,
      | B |_{ 
        \operatorname{Lip}^0( 
          H, 
          HS( U, H_{ - \vartheta / 2 } ) 
        ) 
      }
    \Big]
  \Big|^2
  .
\end{align}
\end{proposition}
\begin{proof}
First of all, we observe that Lemma~\ref{lem:strong_apriori_bound} 
allows us to apply Corollary~\ref{prop:general_perturb} 
to obtain\footnote{with 
$ \bar{X}_t = Y^\kappa_t $, 
$ \bar{Y}^t_s = L_{ s, t } \, e^{ \kappa A } F( Y^\kappa_{\varPi(s)} ) $, 
$ \bar{Z}^t_s = L_{ s, t } \, e^{ \kappa A } B( Y^\kappa_{\varPi(s)} ) $ 
for $ s \in ( 0, t ) $, $ t \in ( 0, T ] $ 
in the notation of Corollary~\ref{prop:general_perturb}} that 
\begin{equation}\label{eq:strong_converge_num_decompose}
\begin{split}
&
  \sup_{ t \in [0,T] }
  \left\|
    Y^0_t - Y^\kappa_t
  \right\|_{ \lpn{ p }{ \P }{ H } }
\\ & \leq
  \mathcal{E}_{ ( 1 - \vartheta ) }
  \!\left[
    \groupC_\vartheta
    \,
    | 
      e^{ \kappa A } F 
    |_{ 
      \operatorname{Lip}^0( H, H_{ -\vartheta } ) 
    }
    \tfrac{ 
      \sqrt{2} \, 
      T^{
        ( 1 - \vartheta )
      } 
    }{ 
      \sqrt{ 1 - \vartheta } 
    }
    +
    \groupC_{ 
      \nicefrac{ \vartheta }{ 2 } 
    } 
    \,
    | e^{ \kappa A } B |_{ 
      \operatorname{Lip}^0( 
        H, 
        HS( 
          U, 
          H_{ - \vartheta / 2 } 
        ) 
      ) 
    }
    \sqrt{ 
      p \,
      ( p - 1 ) \,
      T^{ 
        ( 1 - \vartheta )
      } 
    }
  \right]
\\ & \cdot
  \sqrt{2}
  \sup_{ t \in [ 0, T ] }
  \left\|
    \int^t_0
    L_{ s, t } \,
    \big( \operatorname{Id}_H - e^{ \kappa A } \big)
    F( Y^0_{ \varPi( s ) } )
    \, ds
    +
    \int^t_0
    L_{ s, t } \,
    \big( \operatorname{Id}_H - e^{ \kappa A } \big)
    B( Y^0_{ \varPi( s ) } )
    \, dW_s
  \right\|_{ \lpn{ p }{ \P }{ H } }
  .
\end{split}
\end{equation}
Moreover, we observe that for all 
$ t \in ( 0, T ] $ 
it holds that 
\begin{equation}\label{eq:strong_converge_num_F}
\begin{split}
&
  \left\|
    \int^t_0
    L_{ s, t } 
    \left( 
      \operatorname{Id}_H - e^{ \kappa A } 
    \right)
    F( Y^0_{ \varPi( s ) } )
    \, ds
  \right\|_{ \lpn{ p }{ \P }{ H } }
\leq
  \int^t_0
  \frac{
    \groupC_\rho \, \groupC_{ \rho + \vartheta } \, \kappa^\rho
  }{ ( t - s )^{ ( \rho + \vartheta ) } } \,
  \| F( Y^0_{ \varPi( s ) } ) \|_{ \lpn{ p }{ \P }{ H_{ -\vartheta } } }
  \, ds
\\ & \leq
  \tfrac{
    \groupC_\rho \, \groupC_{ \rho + \vartheta } \,
    t^{ ( 1 - \vartheta - \rho ) }
  }{
    ( 1 - \vartheta - \rho )
  }
  \,
  \| F \|_{ \operatorname{Lip}^0( H, H_{ -\vartheta } ) }
    \sup_{ s \in [ 0, T ] }
  \max\!\left\{
    1,
    \| Y^0_s \|_{ \lpn{ p }{ \P }{ H } }
  \right\}
  \kappa^\rho
  .
\end{split}
\end{equation}
In addition, the Burkholder-Davis-Gundy type inequality 
in Lemma~7.7 in 
Da Prato \& Zabczyk~\cite{dz92} implies 
that for all $ t \in ( 0, T ] $ it holds that 
\begin{equation}\label{eq:strong_converge_num_B}
\begin{split}
&
  \Bigg\|
    \int^t_0
    L_{ s, t } \,
    \big( 
      \operatorname{Id}_H - e^{ \kappa A } 
    \big)
    B( Y^0_{ \varPi( s ) } )
    \, dW_s
  \Bigg\|_{ \lpn{ p }{ \P }{ H } }
\\ & \leq
  \left[
  \tfrac{ p \, ( p - 1 ) }{ 2 }
  \int^t_0
  \frac{
    |
      \groupC_\rho \, 
      \groupC_{ \rho + \vartheta / 2 } \, 
      \kappa^\rho
    |^2
  }{ 
    ( t - s )^{ 
      ( 2 \rho + \vartheta ) 
    } 
  } \,
  \| 
    B( Y^0_{ \varPi( s ) } ) 
  \|^2_{ 
    \lpn{ p }{ \P }{ 
      HS( U, H_{ - \vartheta / 2 } ) 
    } 
  }
  \, ds
  \right]^{ 1 / 2 }
\\ & \leq
  \tfrac{
    \groupC_\rho \, 
    \groupC_{ \rho + \vartheta / 2 }
    \sqrt{ 
      p \, ( p - 1 ) \,
      t^{ ( 1 - \vartheta - 2 \rho ) } 
    }
  }{
    \sqrt{ 2 - 2 \vartheta - 4 \rho }
  }
  \,
  \| B \|_{ 
    \operatorname{Lip}^0( 
      H, 
      HS( U, H_{ - \vartheta / 2 } ) 
    ) 
  }
  \sup_{ s \in [ 0, T ] }
  \max\!\left\{
    1,
    \| Y^0_s \|_{ \lpn{ p }{ \P }{ H } }
  \right\}
  \kappa^\rho.
\end{split}
\end{equation}
Putting~\eqref{eq:strong_converge_num_F}
and~\eqref{eq:strong_converge_num_B}
into~\eqref{eq:strong_converge_num_decompose} 
yields that
\begin{equation}\label{eq:strong_converge_num}
\begin{split}
&
  \sup_{ t \in [0,T] }
  \left\|
    Y^0_t - Y^\kappa_t
  \right\|_{ \lpn{ p }{ \P }{ H } }
\leq
  \sqrt{2} \,
  \kappa^\rho 
  \sup_{ t \in [ 0, T ] }
  \max\!\left\{
    1,
    \| Y^0_t \|_{ \lpn{ p }{ \P }{ H } }
  \right\}
\\ & \cdot
  \mathcal{E}_{ ( 1 - \vartheta ) }\!\left[
    \tfrac{ 
      \sqrt{2} \, 
      T^{ ( 1 - \vartheta ) } \, 
      \groupC_0 \, 
      \groupC_\vartheta 
    }{ 
      \sqrt{ 1 - \vartheta } 
    }
    | F |_{ 
      \operatorname{Lip}^0( H, H_{ - \vartheta } ) 
    }
    +
    \sqrt{ 
      p \, ( p - 1 ) \,
      T^{ ( 1 - \vartheta ) } 
    }
    \,
    \groupC_0 \, 
    \groupC_{ \vartheta / 2 } \,
    | B |_{ 
      \operatorname{Lip}^0( 
        H, 
        HS( U, H_{ - \vartheta / 2 } ) 
      ) 
    }
  \right]
\\ & \cdot
  \left[
    \tfrac{
      \groupC_\rho \, 
      \groupC_{ \rho + \vartheta } \,
      T^{ ( 1 - \vartheta - \rho ) }
    }{
      ( 1 - \vartheta - \rho )
    }
    \,
    \| F \|_{ \operatorname{Lip}^0( H, H_{ -\vartheta } ) }
    +
    \tfrac{
      \groupC_\rho \, 
      \groupC_{ \rho + \vartheta / 2 } 
      \sqrt{ 
        p \, ( p - 1 ) \,
        T^{ ( 1 - \vartheta - 2 \rho ) } 
      }
    }{
      \sqrt{ 2 - 2 \vartheta - 4 \rho }
    }
    \,
    \| B \|_{ 
      \operatorname{Lip}^0( 
        H, 
        HS( U, H_{ - \vartheta / 2 } ) 
      ) 
    }
  \right]
  .
\end{split}
\end{equation}
Combining 
Proposition~\ref{prop:numerics_Lp_bound} 
and \eqref{eq:strong_converge_num} 
proves that
\begin{equation}
\begin{split}
&
  \left\|
    Y^0_T - Y^{ \kappa }_T
  \right\|_{ \lpn{ p }{ \P }{ H } }
\leq
  2 
  \, 
  \kappa^\rho 
  \,
  \bigg[
    \sup_{ t \in ( 0, T ] }
    \max\!\big\{
      1
      , 
      \| 
        L_{ 0, t } 
      \|_{ L(H) }
    \big\}
    \,
    \max\{ 1, \| Y^0_0 \|_{ \lpn{p}{\P}{H} } \}
\\ & +
    \tfrac{
      \groupC_\vartheta \,
      T^{ ( 1 - \vartheta ) } \,
      \| F(0) \|_{ H_{ -\vartheta } }
    }{
      ( 1 - \vartheta )
    }
    +
    \tfrac{
      \groupC_{ \vartheta / 2 }
      \sqrt{
        p \, (p - 1) \, 
        T^{ ( 1 - \vartheta ) }
      }
      \,
      \| B(0) \|_{ HS( U, H_{ - \vartheta / 2 } ) }
    }{
      \sqrt{
        2 - 2 \vartheta
      }
    }
  \bigg]
\\ & \cdot
  \bigg[
    \tfrac{
      \groupC_\rho \, 
      \groupC_{ \rho + \vartheta } \,
      T^{ ( 1 - \vartheta - \rho ) } \,
      \| F \|_{ \operatorname{Lip}^0( 
        H, H_{ - \vartheta } ) 
      }
    }{
      ( 1 - \vartheta - \rho )
    }
    +
    \tfrac{
      \groupC_\rho 
      \, 
      \groupC_{ 
        \rho 
        + 
        \vartheta / 2 
      }
      \sqrt{ 
        p \, ( p - 1 ) \,
        T^{ ( 1 - \vartheta - 2 \rho ) } 
      }
      \,
      \| B \|_{ 
        \operatorname{Lip}^0( 
          H, HS( U, H_{ - \vartheta / 2 } ) 
        ) 
      }
    }{
      \sqrt{ 2 - 2 \vartheta - 4 \rho }
    }
  \bigg]
\\ & \cdot
  \bigg|
    \mathcal{E}_{ ( 1 - \vartheta ) }\bigg[
      \tfrac{ 
        \sqrt{2} \, 
        T^{ ( 1 - \vartheta ) } \, 
        \groupC_0 \, 
        \groupC_\vartheta \,
        | F |_{ 
          \operatorname{Lip}^0( H, H_{ - \vartheta } ) 
        }
      }{ 
        \sqrt{ 1 - \vartheta } 
      }
      +
      \sqrt{ 
        p \, ( p - 1 ) \, 
        T^{ ( 1 - \vartheta ) } 
      }
      \,
      \groupC_0 \, 
      \groupC_{ \vartheta / 2 } \,
      | B |_{ 
        \operatorname{Lip}^0( 
          H, 
          HS( U, H_{ - \vartheta / 2 } ) 
        ) 
      }
    \bigg]
  \bigg|^2
  .
\end{split}
\end{equation}
Hence, we obtain that
\begin{equation}
\label{eq:strong_convergence_conclude}
\begin{split}
&
  \left\|
    Y^0_T - Y^{ \kappa }_T
  \right\|_{ \lpn{ p }{ \P }{ H } }
\leq
  \tfrac{
    2 
    \, 
    \kappa^\rho 
  }{
    T^{ \rho }
  }
  \,
  \bigg[
    \sup_{ t \in ( 0, T ] }
    \max\!\big\{
      1
      , 
      \| 
        L_{ 0, t } 
      \|_{ L(H) }
    \big\}
    \,
    \max\{ 1, \| Y^0_0 \|_{ \lpn{p}{\P}{H} } \}
\\ & +
    \tfrac{
      \groupC_\vartheta \,
      T^{ ( 1 - \vartheta ) } \,
      \| F(0) \|_{ H_{ -\vartheta } }
    }{
      ( 1 - \vartheta )
    }
    +
    \tfrac{
      \groupC_{ \vartheta / 2 }
      \sqrt{
        p \, (p - 1) \, 
        T^{ ( 1 - \vartheta ) }
      }
      \,
      \| B(0) \|_{ HS( U, H_{ - \vartheta / 2 } ) }
    }{
      \sqrt{
        2 - 2 \vartheta
      }
    }
  \bigg]
\\ & \cdot
  \bigg[
    \tfrac{
      \groupC_\rho \, 
      \groupC_{ \rho + \vartheta } \,
      T^{ ( 1 - \vartheta ) } \,
      \| F \|_{ \operatorname{Lip}^0( 
        H, H_{ - \vartheta } ) 
      }
    }{
      ( 1 - \vartheta - \rho )
    }
    +
    \tfrac{
      \groupC_\rho 
      \, 
      \groupC_{ 
        \rho 
        + 
        \vartheta / 2 
      }
      \sqrt{ 
        p \, ( p - 1 ) \,
        T^{ ( 1 - \vartheta ) } 
      }
      \,
      \| B \|_{ 
        \operatorname{Lip}^0( 
          H, HS( U, H_{ - \vartheta / 2 } ) 
        ) 
      }
    }{
      \sqrt{ 2 - 2 \vartheta - 4 \rho }
    }
  \bigg]
\\ & \cdot
  \bigg|
    \mathcal{E}_{ ( 1 - \vartheta ) }\bigg[
      \tfrac{ 
        \sqrt{2} \, 
        T^{ ( 1 - \vartheta ) } \, 
        \groupC_0 \, 
        \groupC_\vartheta \,
        | F |_{ 
          \operatorname{Lip}^0( H, H_{ - \vartheta } ) 
        }
      }{ 
        \sqrt{ 1 - \vartheta } 
      }
      +
      \sqrt{ 
        p \, ( p - 1 ) \, 
        T^{ ( 1 - \vartheta ) } 
      }
      \,
      \groupC_0 \, 
      \groupC_{ \vartheta / 2 } \,
      | B |_{ 
        \operatorname{Lip}^0( 
          H, 
          HS( U, H_{ - \vartheta / 2 } ) 
        ) 
      }
    \bigg]
  \bigg|^2
  .
\end{split}
\end{equation}
This implies 
\eqref{eq:strong_convergence_numerics}.
The proof of 
Proposition~\ref{prop:strong_convergence_numerics} 
is thus completed.
\end{proof}

\section{Mild stochastic calculus}
\label{sec:mildcalc}

In Theorem~1 in Da Prato et al.~\cite{DaPratoJentzenRoeckner2012} 
a new -- somehow mild -- It\^{o} type formula has been proposed
and this formula has been called mild It\^{o} formula.
The mild It\^{o} formula suggested in Theorem~1 in Da Prato et al.~\cite{DaPratoJentzenRoeckner2012}
has been proved from the deterministic starting time point $ t_0 \in [0,\infty) $
to the deterministic end time point $ t \in [t_0,\infty) $.
In Theorem~\ref{thm:ito} in this section we generalize this 
mild It\^{o} formula by allowing the end time point $ t \in [ t_0, \infty) $ to be a stopping time.
We then use Theorem~\ref{thm:ito} to derive
a mild Dynkin-type formula in Corollary~\ref{cor:mild_dynkin_formula}.
This mild Dynkin-type formula, in turn, is used in Proposition~\ref{prop:mild_ito_maximal_ineq} below
to derive suitable estimates for expectations of compositions of smooth functions 
and mild It\^{o} processes.
Proposition~\ref{prop:mild_ito_maximal_ineq} is used intensively 
in the proof of the weak convergence result in Theorem~\ref{intro:theorem}
(see Section~\ref{sec:weak_temporal_regularity} for details).

\subsection{Setting}
\label{sec:setting_mild_calculus}

Throughout this section we assume the following setting.
Let $ t_0 \in [ 0, \infty ) $, 
$ T \in ( t_0, \infty ) $, 
$
   \angle =
   \left\{
     (t_1, t_2) \in [ t_0, T ]^2 \colon
     t_1 < t_2
   \right\}
$,
let 
$
  ( \Omega, \mathcal{F}, \P , ( \mathcal{F}_t )_{ t \in [ t_0, T ] } )
$
be a stochastic basis,
let
$
  \left( W_t \right)_{ t \in [ t_0, T ] }
$
be a cylindrical $ \operatorname{Id}_U $-Wiener
process w.r.t.\
$
   ( \mathcal{F}_t )_{ t \in [ t_0, T ] }
$,
let
$
   (
     \check{H},
     \left< \cdot , \cdot \right>_{ \check{H} },
     \left\| \cdot \right\|_{ \check{H} }
   )
$,
$
   (
     \tilde{H},
     \left< \cdot , \cdot \right>_{ \tilde{H} },
     \left\| \cdot \right\|_{ \tilde{H} }
   )
$,
$
   (
     \hat{H},
     \left< \cdot , \cdot \right>_{ \hat{H} },
     \left\| \cdot \right\|_{ \hat{H} }
   )
$, 
$
   (
     U,
     \left< \cdot , \cdot \right>_U,
     \left\| \cdot \right\|_U
   )
$, 
and 
$
   ( V,
     \left< \cdot, \cdot \right>_V,
     \left\| \cdot \right\|_V
   )
$
be separable
$  \R $-Hilbert spaces
with
$
   \check{H}
   \subseteq
   \tilde{H}
   \subseteq
   \hat{H}
$
continuously and densely, 
and let
$
   \mathbb{U} \subseteq U
$
be an orthonormal basis of
$ U $.

\subsection{Mild It\^{o} processes}
\label{sec:mildprocesses}

For the convenience of the reader we recall the 
notion of a mild It\^{o} process; see Definition~1 in Da Prato et al.~\cite{DaPratoJentzenRoeckner2012}.

\begin{definition}[Mild It\^{o} process]
\label{propdef}
Assume the setting in Section~\ref{sec:setting_mild_calculus},
let
$
   S \colon
   \angle
   \rightarrow
   L( \hat{H}, \check{H} )
$
be a
$
\mathcal{B}( \angle )
$/$
\mathcal{S}( \hat{H}, \check{H} )
$-measurable
mapping such
that for all $ t_1, t_2, t_3 \in [ t_0, T ] $
with $ t_1 < t_2 < t_3 $
it holds that
$
   S_{ t_2, t_3 }
   S_{ t_1, t_2 }
=
   S_{ t_1, t_3 }
$,
and let
$
   Y \colon [ t_0, T ]
   \times \Omega
   \rightarrow \hat{H}
$,
$
   Z \colon [ t_0, T ]
   \times \Omega
   \rightarrow HS( U, \hat{H} )
$,
and
$
   X \colon [ t_0, T ]
   \times \Omega
   \rightarrow \tilde{H}
$
be
$ ( \mathcal{F}_t )_{ t \in [ t_0, T ] } $-predictable stochastic processes
such that for all $ t \in ( t_0, T ] $
it holds $ \P $-a.s.\ that
$
     \int_{ t_0 }^{ t }
     \|
       S_{ s, t } Y_s
     \|_{ \check{H} }
     +
     \|
       S_{ s, t } Z_s
     \|_{ HS( U, \check{H} ) }^2
     \,
     ds < \infty
$
and
\begin{equation}
\label{eq:mildito}
   X_t
=
   S_{ t_0, t } \,
   X_{ t_0 }
   +
   \int_{ t_0 }^t
     S_{ s, t } \,
     Y_s
   \, ds
   +
   \int_{ t_0 }^t
     S_{ s, t } \,
     Z_s
   \, dW_s
   .
\end{equation}
Then we call
$
   X
$
a {mild It\^{o} process}
(with {evolution family} $S$,
{mild drift} $ Y $, 
and {mild diffusion} $ Z $).
\end{definition}

\begin{lemma}[Regularization of mild It\^{o} processes]
\label{lem:transformation}
Assume the setting in Section~\ref{sec:setting_mild_calculus}
and let
$
   X \colon [ t_0, T ] \times \Omega
   \rightarrow \tilde{H}
$
be a mild It{\^o} process
with evolution family
$
   S \colon \angle
   \rightarrow L( \hat{H}, \check{H} )
$,
mild drift
$
   Y \colon [ t_0, T ] \times
   \Omega \rightarrow \hat{H}
$, 
and mild
diffusion
$
   Z \colon [ t_0, T ] \times
   \Omega \rightarrow
   HS(U,\hat{H})
$.
Then 
\begin{enumerate}[(i)]
\item 
\label{lem:transformation_i}
there exists an up to indistinguishability unique
continuous stochastic process 
$
  \bar{X} \colon [ t_0, T ] \times \Omega \rightarrow \check{H}
$ 
with
$
  \forall \, 
  t \in [ t_0, T ) 
  \colon
  \P\big(
    \bar{X}_t
    =
    S_{ t, T } X_t
  \big) = 1
$

\item
\label{lem:transformation_ii}
and 
for all 
continuous stochastic process 
$
  \bar{X} \colon [ t_0, T ] \times \Omega \rightarrow \check{H}
$ 
with 
$ 
  \forall \, t \in [ t_0, T ) \colon
  \P\big(
    \bar{X}_t
    =
    S_{ t, T } X_t
  \big) = 1
$
and all $ t \in [t_0,T] $
it holds that
$ \bar{X} $ 
is 
$
   ( \mathcal{F}_s )_{ s \in [ t_0, T ] }
$-predictable,
it holds $ \P $-a.s.\ that
$
  \bar{X}_T = X_T
$
and it holds $ \P $-a.s.\ that
\begin{equation}
   \bar{X}_t
=
   S_{ t_0, T } \,
   X_{ t_0 }
   +
   \int_{ t_0 }^t
     S_{ s, T } \,
     Y_s
   \, ds
   +
   \int_{ t_0 }^t
     S_{ s, T } \,
     Z_s
   \, dW_s
   .
\end{equation}
\end{enumerate}
\end{lemma}

\begin{proof}
The assumption that $ X $ is a mild It\^{o} process, in particular, 
ensures that it holds $ \P $-a.s.\ that
$
  \int_{ t_0 }^T
  \left\| S_{ s, T } Y_s \right\|_{
    \check{H}
  }
  +
  \left\| S_{ s, T } Z_s \right\|_{
    HS( U, \check{H} )
  }^2
  ds < \infty
$.
This implies that 
there exists a
continuous stochastic 
process 
$ 
  \bar{X} \colon [ t_0, T ] \times \Omega \to \check{H}
$
such that for all 
$ t \in [ t_0, T ] $
it holds $ \P $-a.s.\ that
\begin{equation}
\label{eq:semiX}
   \bar{X}_t
   =
   S_{ t_0, T } \, X_{ t_0 }
   +
   \int_{ t_0 }^t
   S_{ s, T } \, Y_s \, ds
   +
   \int_{ t_0 }^t
   S_{ s, T } \, Z_s \, dW_s
   .
\end{equation}
Next observe that Definition~\ref{propdef}
ensures that
for all $ t \in ( t_0, T ) $
it holds $ \P $-a.s.\ that
\begin{equation}
\begin{split}
&
   S_{ t_0, T } \, X_{ t_0 }
   +
   \int_{ t_0 }^t
   S_{ s, T } \, Y_s \, ds
   +
   \int_{ t_0 }^t
   S_{ s, T } \, Z_s \, dW_s
\\ & =
   S_{ t, T }
   \left(
     S_{ t_0, t } \, X_{ t_0 }
   +
     \int_{ t_0 }^t
     S_{ s, t } \,
     Y_s \, ds
   +
     \int_{ t_0 }^t
     S_{ s, t } \,
     Z_s \, dW_s
   \right)
=
   S_{ t, T } \,
   X_t
   .
\end{split}
\end{equation}
This establishes that
for all $ t \in [ t_0, T ) $
it holds $ \P $-a.s.\ that
\begin{equation}
\label{eq:fact}
\begin{split}
&
   S_{ t_0, T } \, X_{ t_0 }
   +
   \int_{ t_0 }^t
   S_{ s, T } \, Y_s \, ds
   +
   \int_{ t_0 }^t
   S_{ s, T } \, Z_s \, dW_s
=
   S_{ t, T } \,
   X_t
   .
\end{split}
\end{equation}
Combining \eqref{eq:fact}
with \eqref{eq:semiX}
shows that
there exists a
continuous stochastic 
process 
$ 
  \bar{X} \colon [ t_0, T ] \times \Omega \to \check{H}
$
such that for all 
$ t \in [ t_0, T ) $
it holds $ \P $-a.s.\ that
\begin{equation}
\label{eq:semiX_2}
  \bar{X}_t 
  =
   S_{ t_0, T } \, X_{ t_0 }
   +
   \int_{ t_0 }^t
   S_{ s, T } \, Y_s \, ds
   +
   \int_{ t_0 }^t
   S_{ s, T } \, Z_s \, dW_s
  = S_{ t, T } X_t 
  .
\end{equation}
Moreover, observe that for all 
\emph{continuous} stochastic processes 
$ \bar{X}, \bar{Y} \colon [0,T] \times \Omega \to \check{H} $
with
$
  \forall \, t \in [t_0,T)
  \colon
  \P\big(
    \bar{X}_t = \bar{Y}_t
  \big) = 1
$
it holds that
$
  \P\big(
    \forall \, t \in [ t_0, T ] \colon
    \bar{X}_t = \bar{Y}_t
  \big)
  = 1
$.
Combining this with \eqref{eq:semiX_2} proves Item~\eqref{lem:transformation_i}.
Item~\eqref{lem:transformation_ii} is an immediate
consequence from 
\eqref{eq:semiX_2}
and Item~\eqref{lem:transformation_i}.
The proof of Lemma~\ref{lem:transformation}
is thus completed.
\end{proof}

\subsection{Mild It{\^o} formula for stopping times}
\label{sec:mildito}

\begin{theorem}[Mild It{\^o} formula]
\label{thm:ito}
Assume the setting in Section~\ref{sec:setting_mild_calculus},
let
$
   X \colon [ t_0, T ] \times \Omega
   \rightarrow \tilde{H}
$
be a mild It{\^o} process
with evolution family
$
   S \colon \angle
   \rightarrow L( \hat{H}, \check{H} )
$,
mild drift
$
   Y \colon [ t_0, T ] \times
   \Omega \rightarrow \hat{H}
$, 
and mild
diffusion
$
   Z \colon [ t_0, T ] \times
   \Omega \rightarrow
   HS(U,\hat{H})
$,
let
$
  \bar{X} \colon [ t_0, T ] \times \Omega \rightarrow \check{H}
$ 
be a continuous stochastic process
with 
$
  \forall \, t \in [ t_0, T ) \colon
  \P\big(
    \bar{X}_t
     =
     S_{ t, T } 
     X_t
  \big) = 1
$
(see Lemma~\ref{lem:transformation}),
let
$ r \in [ t_0, T ) $,
$
   \varphi
   =
   (
   \varphi(t,x)
   )_{ t \in [ r, T ],\, x \in \check{H} }
   \in C^{1,2}(
     [ r, T ] \times \check{H}, V
   )
$
and let
$ \tau \colon \Omega \rightarrow [ r, T ] $
be an 
$
   ( \mathcal{F}_t )_{ t \in [ r, T ] }
$-stopping time.
Then it holds $ \P $-a.s.\ that
\begin{equation}
\label{eq:well1}
     \int_r^{ T }
     \left\|
  (
    \tfrac{ \partial }{ \partial x }
    \varphi
  )
    ( s, S_{ s, T } X_s )
       S_{ s, T } Y_s
     \right\|_V
     +
     \left\|
  (
    \tfrac{ \partial }{ \partial x }
    \varphi
  )
    ( s, S_{ s, T } X_s )
       S_{ s, T } Z_s
     \right\|_{ HS(U, V ) }^2
     ds
     < \infty
     ,
\end{equation}
\begin{equation}
\label{eq:well2}
     \int_r^{ T }
     \left\|
  (
    \tfrac{ \partial }{ \partial t }
    \varphi
  )
    ( s, S_{ s, T } X_s )
     \right\|_V
     +
     \|
  (
    \tfrac{ \partial^2 }{ \partial x^2 }
    \varphi
  )
    ( s, S_{ s, T } X_s )
     \|_{ L^{(2)}( \check{H}, V ) } \,
     \|
       S_{ s, T } Z_s
     \|_{ HS(U, \check{H} ) }^2 \,
     ds
     < \infty
     ,
\end{equation}
\begin{equation}
\label{eq:itoformel_start}
\begin{split}
&
   \varphi( \tau, \bar{X}_\tau )
  =
   \varphi( r,
     S_{ r, T }
     X_{ r }
   )
   +
   \int_{ r }^\tau
  (
    \tfrac{ \partial }{ \partial t }
    \varphi
  )
    ( s, S_{ s, T } X_s )
   \, ds
   +
   \int_{ r }^\tau
  (
    \tfrac{ \partial }{ \partial x } \varphi
    )
    ( s, S_{ s, T } X_s )
    \,
   S_{ s, T} \,Y_s
   \, ds
\\&+
   \int_{ r }^\tau
  (
    \tfrac{ \partial }{ \partial x } \varphi
    )
    ( s, S_{ s, T } X_s )
   \, S_{ s, T } \, Z_s
   \, dW_s
   +
   \tfrac{1}{2}
   \sum_{ u \in \mathbb{U} }
   \int_{ r }^\tau
  (
    \tfrac{ \partial^2 }{ \partial x^2 } \varphi
    )
    ( s, S_{ s, T } X_s )
   \left(
     S_{ s, T }
     Z_s u ,
     S_{ s, T }
     Z_s u
   \right) ds
   .
\end{split}
\end{equation}
\end{theorem}
\begin{proof}[Proof
of Theorem~\ref{thm:ito}]
First of all, we note that Theorem 1 in Da Prato et al.~\cite{DaPratoJentzenRoeckner2012} 
establishes \eqref{eq:well1} and \eqref{eq:well2}. 
It thus remains to prove \eqref{eq:itoformel_start}.
For this let 
$
  \varphi_{1,0} \colon [ r, T ] \times \check{H} \to V 
$, 
$
  \varphi_{0,1} \colon [ r, T ] \times \check{H} \to L( \check{H}, V ) 
$, 
$
  \varphi_{0,2} \colon [ r, T ] \times \check{H} \to L^{(2)}( \check{H}, V ) 
$
be the functions with the property that 
for all 
$ t \in [ r, T ] $, 
$ x, v_1, v_2 \in \check{H} $ 
it holds that 
$
  \varphi_{1,0}( t, x )
  =
  \big(\tfrac{ \partial }{ \partial t } \varphi\big)( t, x)
  ,
$
$
  \varphi_{0,1}( t, x ) \, v_1
  =
  \big(\tfrac{ \partial }{ \partial x } \varphi\big)( t, x ) \, v_1
  ,
$
and
$
  \varphi_{0,2}( t, x )( v_1, v_2 )
  =
  \big(
    \tfrac{ \partial^2 }{ \partial x^2 } \varphi
  \big)( t, x )( v_1, v_2 )
$.
Then note that 
Item~\eqref{lem:transformation_ii} of Lemma~\ref{lem:transformation} and the standard It{\^o} formula in Theorem~2.4 in
Brze\'{z}niak, Van Neerven, Veraar \citationand\ Weis~\cite{bvvw08}
show that 
it holds $ \P $-a.s.\ that
\begin{equation}
\label{eq:itoformel2inProof}
\begin{split}
&
   \varphi( \tau, \bar{X}_{ \tau } )
=
   \varphi( r, \bar{X}_{ r } )
   +
   \int_{r}^{ \tau }
   \varphi_{1,0}( s, \bar{X}_{ s } )
   \, ds
   +
   \int_{r}^{ \tau }
   \varphi_{0,1}( s, \bar{X}_{ s } ) \,
   S_{ s, T } \, Y_s \, ds
\\&+
   \int_{r}^{ \tau }
   \varphi_{0,1}( s, \bar{X}_{ s } ) \,
   S_{ s, T } \, Z_s \, dW_s
   +
   \tfrac{1}{2}
   \sum_{ u \in \mathbb{U} }
   \int_{r}^{ \tau }
   \varphi_{0,2}( s,
     \bar{X}_{ s }
   )
   \left(
     S_{ s, T } \, Z_s \, u ,
     S_{ s, T } \, Z_s \, u
   \right) ds
   .
\end{split}
\end{equation}
Combining this with 
Lemma~1 in Da Prato et al.~\cite{DaPratoJentzenRoeckner2012}
and with the fact that 
$ 
  \forall \, t \in [ t_0, T ) 
  \colon
  \P\big(
    \bar{X}_t = S_{ t, T } \, X_t
  \big)
  = 1
$ 
shows that 
it holds $ \P $-a.s.\ that 
\begin{equation}
\label{eq:itoformel_fixed}
\begin{split}
&
   \varphi( \tau , \bar{X}_{ \tau } )
  =
   \varphi( r,
     S_{ r, T }
     X_{ r }
   )
   +
   \int_{ r }^{ \tau }
   \varphi_{1,0}( s,
     S_{ s, T }
     X_s
   ) \, ds
   +
   \int_{ r }^{ \tau }
   \varphi_{0,1}( s,
     S_{ s, T }
     X_s
   ) \,
   S_{ s, T} \,
   Y_s \, ds
\\&+
   \int_{ r }^{ \tau }
   \varphi_{0,1}( s,
     S_{ s, T }
     X_s
   ) \,
   S_{ s, T } \,
   Z_s \, dW_s
   +
   \tfrac{1}{2}
   \sum_{ u \in \mathbb{U} }
   \int_{ r }^{ \tau }
   \varphi_{0,2}( s,
     S_{ s, T }
     X_s
   )
   \left(
     S_{ s, T }
     Z_s u ,
     S_{ s, T }
     Z_s u
   \right) ds
   .
\end{split}
\end{equation}
The proof of Theorem~\ref{thm:ito}
is thus completed.
\end{proof}

\begin{definition}[Extended mild Kolmogorov operators]
Assume the setting in Section~\ref{sec:setting_mild_calculus}, 
let 
$
   S \colon
   \angle
   \rightarrow
   L( \hat{H}, \check{H} )
$
be a
$
\mathcal{B}( \angle )
$/$
\mathcal{S}( \hat{H}, \check{H} )
$-measurable
mapping such
that for all $ t_1, t_2, t_3 \in [ t_0, T ] $
with $ t_1 < t_2 < t_3 $
it holds that
$
   S_{ t_2, t_3 }
   S_{ t_1, t_2 }
=
   S_{ t_1, t_3 }
$, 
and let 
$ (t_1, t_2) \in \angle $.
Then we denote by 
$
  \mathcal{L}^S_{ t_1, t_2 }
  \colon
  C^2( \check{H}, V )
  \rightarrow
  C( \tilde{H} \times \hat{H} \times HS( U, \hat{H} ), V )
$ 
the function with the property that for all 
$ \varphi \in C^2( \check{H}, V ) $, 
$ x \in \tilde{H} $, 
$ y \in \hat{H} $, 
$ z \in HS( U, \hat{H} ) $ 
it holds that 
\begin{equation}
  \big(
   \mathcal{L}^S_{ t_1, t_2 } \varphi
  \big)( x, y, z )
  =
  \varphi'( S_{ t_1, t_2 } \, x )
  \, S_{ t_1, t_2 } \, y
  +
  \tfrac{ 1 }{ 2 }
  \sum_{ u \in \mathbb{U} }
  \varphi''( S_{ t_1, t_2 } \, x )
  ( S_{ t_1, t_2 } \, z \, u, S_{ t_1, t_2 } \, z \, u )
  .
\end{equation}
\end{definition}

The next corollary of 
Theorem~\ref{thm:ito} specialises Theorem~\ref{thm:ito}
to the case 
where
$ r = t_0 $
and where 
the test function
$
   ( \varphi(t, x) )_{
     t \in [ t_0, T ], \,
     x \in \check{H}
   }
   \in
   C^{1,2}(
     [ t_0, T ] \times \check{H}, V
   )
$
depends on $ x \in \check{H} $ only.
\begin{corollary}
\label{cor:itoauto}
Assume the setting in Section~\ref{sec:setting_mild_calculus},
let
$
   X \colon [ t_0, T ] \times \Omega
   \rightarrow \tilde{H}
$
be a mild It{\^o} process
with evolution family
$
   S \colon \angle
   \rightarrow L( \hat{H}, \check{H} )
$,
mild drift
$
   Y \colon [ t_0, T ] \times
   \Omega \rightarrow \hat{H}
$, 
and mild
diffusion
$
   Z \colon [ t_0, T ] \times
   \Omega \rightarrow
   HS(U,\hat{H})
$, 
let
$
  \bar{X} \colon [ t_0, T ] \times \Omega \rightarrow \check{H}
$ 
be a continuous stochastic process
with 
$
  \forall \, t \in [ t_0, T ) \colon
  \P\big(
    \bar{X}_t
    =
    S_{ t, T } X_t
  \big) = 1
$
(see Lemma~\ref{lem:transformation}),
let
$
   \varphi \in C^2(
     \check{H}, V
   )
$, 
and let 
$ \tau \colon \Omega \rightarrow [ t_0, T ] $ 
be an
$
   ( \mathcal{F}_t )_{ t \in [ t_0, T ] }
$-stopping time.
Then it holds
$ \P $-a.s.\ that 
$
     \int_{ t_0 }^{ T }
     \left\|
       \varphi'( S_{ s, T } X_s )
       S_{ s, T } Y_s
     \right\|_V
     +
     \left\|
       \varphi'( S_{ s, T } X_s )
       S_{ s, T } Z_s
     \right\|_{ HS(U, V ) }^2 \,
     +
     \|
       \varphi''( S_{ s, T } X_s )
     \|_{ L^{(2)}( \check{H}, V ) } \,
     \|
       S_{ s, T } Z_s
     \|_{ HS(U, \check{H} ) }^2 \,
$
$
     ds
     < \infty
$
and it holds $\P$-a.s.\ that 
\begin{equation}
\label{eq:mild_ito_type_independent}
\begin{split}
   \varphi( \bar{X}_\tau )
& =
   \varphi(
     S_{ t_0, T }
     X_{ t_0 }
   )
   +
   \int_{ t_0 }^\tau
     (
       \mathcal{L}^S_{ s, T } \varphi
     )( X_s, Y_s, Z_s )
   \, ds
+
   \int_{ t_0 }^\tau
   \varphi'(
     S_{ s, T }
     X_s
   ) \,
   S_{ s, T } \,
   Z_s \, dW_s
   .
\end{split}
\end{equation}
\end{corollary}

\subsection{Mild Dynkin-type formula}

Under suitable additional assumptions (see Corollary~\ref{cor:mild_dynkin_formula} below),
the stochastic integral 
in \eqref{eq:mild_ito_type_independent} is integrable 
and centered.
This is the subject of the following result.

\begin{corollary}[Mild Dynkin-type formula]
\label{cor:mild_dynkin_formula}
Assume the setting in Section~\ref{sec:setting_mild_calculus},
let
$
   X \colon [ t_0, T ] \times \Omega
   \rightarrow \tilde{H}
$
be a mild It{\^o} process
with evolution family
$
   S \colon \angle
   \rightarrow L( \hat{H}, \check{H} )
$,
mild drift
$
   Y \colon [ t_0, T ] \times
   \Omega \rightarrow \hat{H}
$, 
and mild
diffusion
$
   Z \colon [ t_0, T ] \times
   \Omega \rightarrow
   HS(U,\hat{H})
$, 
let
$
  \bar{X} \colon [ t_0, T ] \times \Omega \rightarrow \check{H}
$ 
be a continuous stochastic process
with 
$
  \forall \, t \in [ t_0, T ) \colon
  \P\big(
    \bar{X}_t
    =
    S_{ t, T } X_t
  \big) = 1
$
(see Lemma~\ref{lem:transformation}),
let
$
   \varphi \in C^2(
     \check{H}, V
   )
$,
and let 
$ \tau \colon \Omega \rightarrow [ t_0, T ] $ 
be an
$
   ( \mathcal{F}_t )_{ t \in [ t_0, T ] }
$-stopping time with the property
that 
$
   \ES\big[
   |
   \int_{ t_0 }^\tau
   \|
   \varphi'(
     S_{ s, T }
     X_s
   ) \,
   S_{ s, T } \,
   Z_s
   \|^2_{ HS( U, V ) }
   \, ds
   |^{1/2}
   \big]
   +
  \min\!\big\{
    \ES\big[
    \|
      \varphi( S_{ t_0, T } X_{ t_0 } )
      + 
      \int_{ t_0 }^\tau
      (
        \mathcal{L}^S_{ s, T } \varphi
      )( X_s, Y_s, Z_s )
      \, ds
    \|_V
  \big]
 ,
$
$
    \ES\big[
      \| 
      \varphi(
        \bar{X}_{ \tau }
      )
     \|_V
   \big]
 \big\}
 < \infty
$. 
Then
$
  \ES\big[
    \| \varphi(\bar{X}_\tau) \|_V
    +
    \|
      \varphi( S_{ t_0, T } X_{ t_0 } )
      + 
      \int_{ t_0 }^\tau
      (
        \mathcal{L}^S_{ s, T } \varphi
      )( X_s, Y_s, Z_s )
      \, ds
    \|_V
  \big]
  < \infty
$
and 
\begin{equation}
\label{eq:mild_dynkin_formula}
\begin{split}
&
   \ES\big[
     \varphi( \bar{X}_\tau )
   \big]
=
   \ES\big[
   \varphi(
     S_{ t_0, T }
     X_{ t_0 }
   )
+
   \smallint\nolimits_{ t_0 }^\tau
     (
       \mathcal{L}^S_{ s, T } \varphi
     )( X_s, Y_s, Z_s )
   \, ds
   \big]
   .
\end{split}
\end{equation}
\end{corollary}

Corollary~\ref{cor:mild_dynkin_formula} is an immediate consequence 
of Corollary~\ref{cor:itoauto} 
and, e.g., of the Burkholder-Davis-Gundy inequality 
in Problem 3.29 in Karatzas \& Shreve~\cite{ks91}.

\subsection{Weak estimates for terminal values of mild It\^{o} processes}
\label{sec:weak_estimates_terminal}

\begin{proposition}
\label{prop:mild_ito_stopping_limit}
Assume the setting in Section~\ref{sec:setting_mild_calculus},
let
$
   X \colon [ t_0, T ] \times \Omega
   \rightarrow \tilde{H}
$
be a mild It{\^o} process
with evolution family
$
   S \colon \angle
   \rightarrow L( \hat{H}, \check{H} )
$,
mild drift
$
   Y \colon [ t_0, T ] \times
   \Omega \rightarrow \hat{H}
$, 
and mild
diffusion
$
   Z \colon [ t_0, T ] \times
   \Omega \rightarrow
   HS(U,\hat{H})
   ,
$
let
$
   \varphi \in C^2(
     \check{H}, V
   )
   ,
$
and assume that 
$
  \big\{
  \| 
    \varphi(
    S_{ t_0, T } \, X_{ t_0 }
    +
    \int^\tau_{ t_0 }
    S_{ s, T } \, Y_s 
    \, ds
    +
    \int^\tau_{ t_0 }
    S_{ s, T } \, Z_s 
    \, dW_s
    ) 
  \|_V
  \colon
    (\mathcal{F}_t)_{ t \in [ t_0, T ] }\text{-stopping time }
    \tau \colon \Omega \rightarrow [ t_0, T ]
  \big\}
$ 
is uniformly $ \P $-integrable.
Then it holds that 
$
  \ES\big[
    \|
      \varphi(X_T)
    \|_V
    +
    \|
      \varphi( S_{ t_0, T } X_{ t_0 } )
    \|_V
  \big]
  < \infty
$ 
and 
\begin{equation}
\label{eq:mild_ito_stopping_limit}
\begin{split}
&
  \big\|
    \ES\big[
      \varphi( X_T )
    \big]
  \big\|_V
\leq
  \big\|
    \ES\big[
      \varphi( S_{ t_0, T } X_{ t_0 } )
    \big]
  \big\|_V
  + 
  \smallint_{ t_0 }^T
  \E\left[
    \left\|
      (
        \mathcal{L}^S_{ s, T } \varphi
      )( X_s, Y_s, Z_s )
    \right\|_V
  \right]
  ds
  .
\end{split}
\end{equation}
\end{proposition}

\begin{proof}
First of all, we observe that 
the assumption that the set
$
  \big\{
  \| 
    \varphi(
    S_{ t_0, T } \, X_{ t_0 }
    +
    \int^\tau_{ t_0 }
    S_{ s, T } \, Y_s 
    \, ds
    +
    \int^\tau_{ t_0 }
    S_{ s, T } \, Z_s 
    \, dW_s
    ) 
  \|_V
  \colon
    ( \mathcal{F}_t )_{ t \in [ t_0, T ] }\text{-stopping time }
    \tau \colon \Omega \rightarrow [ t_0, T ]
  \big\}
$ 
is uniformly $ \P $-integrable
ensures that
$
  \ES\big[
    \|
      \varphi(X_T)
    \|_V
    +
    \|
      \varphi( S_{ t_0, T } X_{ t_0 } )
    \|_V
  \big]
  < \infty
$.
It thus remains to prove \eqref{eq:mild_ito_stopping_limit}.
For this let 
$
  \tau_n \colon \Omega \rightarrow [ t_0, T ]
$, 
$ n \in \N $, 
be the functions 
with the property that for all $ n \in \N $ it holds that 
\begin{equation}
\label{eq:taun_def}
  \tau_n
=
  \inf\!\left(
    \{ T \}
    \cup
    \left\{
    t \in [ t_0, T ] \colon
   \smallint_{ t_0 }^t
   \|
   \varphi'(
     S_{ s, T }
     X_s
   ) \,
   S_{ s, T } \,
   Z_s
   \|^2_{ HS( U, V ) }
   \, ds
   \geq n
    \right\}
  \right)
\end{equation}
and let 
$
  \bar{X} \colon [ t_0, T ] \times \Omega \rightarrow \check{H}
$ 
be a continuous stochastic process 
with the property that
$
  \forall \, 
  t \in [ t_0, T ) 
  \colon
  \P\big(
    \bar{X}_t
    =
    S_{ t, T } X_t
  \big) = 1
$.
Note that Item~\eqref{lem:transformation_i} of Lemma~\ref{lem:transformation} ensures that 
$ \bar{X} $ does indeed exist.
Moreover, observe that
for all $ n \in \N $
it holds that
$ \tau_n $ 
is an $ ( \mathcal{F}_t )_{ t \in [ t_0, T ] } $-stopping time.
Next note that 
Corollary~\ref{cor:itoauto} shows that 
it holds $\P$-a.s.\ that 
$
   \smallint_{ t_0 }^T
   \|
   \varphi'(
     S_{ s, T }
     X_s
   ) \,
   S_{ s, T } \,
   Z_s
   \|^2_{ HS( U, V ) }
   \, ds
   < \infty
$.
This, in turn, establishes that it holds $ \P $-a.s.\ that 
$
  \lim_{ n \rightarrow \infty } \tau_n
  =
  T
$.
In addition, note that Item~\eqref{lem:transformation_ii} of 
Lemma~\ref{lem:transformation} together with the assumption that the set 
$
  \big\{ 
    \|
      \varphi( 
    S_{ t_0, T } \, X_{ t_0 }
    +
    \int^\tau_{ t_0 }
    S_{ s, T } \, Y_s 
    \, ds
    +
    \int^\tau_{ t_0 }
    S_{ s, T } \, Z_s 
    \, dW_s
      )
    \|_V
    \colon
    ( \mathcal{F}_t )_{ t \in [ t_0, T ] }\text{-stopping time }
    \tau \colon \Omega \rightarrow [ t_0, T ]
  \big\}
$
is uniformly $ \P $-integrable
ensures that the set
$  
  \{
    \| 
      \varphi( \bar{X}_{ \tau_n } )
    \|_V
    \colon
    n \in \N
  \}
$
is uniformly $ \P $-integrable.
This and \eqref{eq:taun_def}
establish that for all $ n \in \N $ 
it holds that
$
  \ES\big[
    \|
      \varphi( \bar{X}_{ \tau_n } )
    \|_V
  \big]
  +
  \ES\big[
    \int_0^{ \tau_n }
    \|
      \varphi'( S_{ s, T } S_{ s, T } Z_s )
    \|_{ HS( U, V ) }^2
    \,
    ds
  \big]
  < \infty
$.
We can thus apply Corollary~\ref{cor:mild_dynkin_formula}
to obtain that for all $ n \in \N $
it holds that
\begin{equation}
   \ES\big[
     \varphi( \bar{X}_{ \tau_n } )
   \big]
=
   \ES\big[
   \varphi(
     S_{ t_0, T }
     X_{ t_0 }
   )
+
   \smallint\nolimits_{ t_0 }^{ \tau_n }
     (
       \mathcal{L}^S_{ s, T } \varphi
     )( X_s, Y_s, Z_s )
   \, ds
   \big]
   .
\end{equation}
The triangle inequality hence proves that 
\begin{equation}
  \limsup_{ n \to \infty }
  \big\|
    \ES\big[
      \varphi( \bar{X}_{ \tau_n } )
    \big]
  \big\|_V
\leq
  \big\|
    \ES\big[
    \varphi(
      S_{ t_0, T }
      X_{ t_0 }
    )
  \big]\|_V  
  +
  \smallint\nolimits_{ t_0 }^T
  \ES\big[
    \|
      (
        \mathcal{L}^S_{ s, T } \varphi
      )( X_s, Y_s, Z_s )
    \|_V
  \big]
  \, ds
  .
\end{equation}
This together with the 
uniform $ \P $-integrability of 
$
  \{
    \| 
      \varphi( \bar{X}_{ \tau_n } )
    \|_V
    \colon
    n \in \N
  \}
$
proves \eqref{eq:mild_ito_stopping_limit}.
The proof of Proposition~\ref{prop:mild_ito_stopping_limit}
is thus completed.
\end{proof}

\begin{proposition}[Test functions with at most polynomial growth]
\label{prop:mild_ito_maximal_ineq}
Assume the setting in Section~\ref{sec:setting_mild_calculus}, 
let
$
   X \colon [ t_0, T ] \times \Omega
   \rightarrow \tilde{H}
$
be a mild It{\^o} process
with evolution family
$
   S \colon \angle
   \rightarrow L( \hat{H}, \check{H} )
$,
mild drift
$
   Y \colon [ t_0, T ] \times
   \Omega \rightarrow \hat{H}
$, 
and mild
diffusion
$
   Z \colon [ t_0, T ] \times
   \Omega \rightarrow
   HS(U,\hat{H})
$, 
and let 
$ p \in [ 0, \infty ) $, 
$
   \varphi \in C^2(
     \check{H}, V
   )
$ 
satisfy 
$
  \sup_{ x \in \check{H} }
  \big[
    \| \varphi(x) \|_V
    ( 1 + \| x \|^p_{ \check{H} } )^{ - 1 }
  \big]
  < \infty
$ 
and 
$
  \| S_{ t_0, T } X_{ t_0 } \|_{ \check{H} }
  +
  \int^T_{ t_0 }
  \| S_{ s, T } Y_s \|_{ \check{H} }
  \, ds
  +
  \big[
    \int^T_{ t_0 }
    \| S_{ s, T } Z_s \|^2_{ HS( U, \check{H} ) }
    \, ds
  \big]^{ 1 / 2 }
  \in
  \lpn{p}{\P}{\R}
$.
Then 
it holds that
$
  \ES\big[
    \|
      \varphi(X_T)
    \|_V
    +
    \|
      \varphi( S_{ t_0, T } X_{ t_0 } )
    \|_V
  \big]
  < \infty
$ 
and
\begin{equation}
\begin{split}
&
  \big\|
    \ES\big[
      \varphi( X_T )
    \big]
  \big\|_V
\leq
  \big\|
    \ES\big[
      \varphi( S_{ t_0, T } X_{ t_0 } )
    \big]
  \big\|_V
  +
  \smallint_{ t_0 }^T
  \E\left[
    \left\|
      (
        \mathcal{L}^S_{ s, T } \varphi
      )( X_s, Y_s, Z_s )   
    \right\|_V
  \right]
  ds
  .
\end{split}
\end{equation}
\end{proposition}

\begin{proof}
Throughout this proof let 
$
  \bar{X} \colon [ t_0, T ] \times \Omega \rightarrow \check{H}
$ 
be a continuous stochastic process 
with 
$
  \forall \, t \in [ t_0, T ) \colon
  \P\big(
    \bar{X}_t
    =
    S_{ t, T } X_t
  \big) = 1
$.
Item~\eqref{lem:transformation_i} of Lemma~\ref{lem:transformation}
ensures that $ \bar{X} $ does indeed exist.
In addition,
we observe that Item~\eqref{lem:transformation_ii} of Lemma~\ref{lem:transformation} also 
implies that for all $ t \in [ t_0, T ] $ it holds $\P$-a.s.\ that 
\begin{equation}
\label{eq:path_varbar}
\begin{split}
&
  \| \varphi( \bar{X}_t ) \|_V
\leq
  \left[
    \sup_{ x \in \check{H} }
    \frac{
      \| \varphi(x) \|_V
    }{
      ( 1 + \|x\|^p_{ \check{H} } )
    }
  \right]
  \left( 
    1 + \| \bar{X}_t \|^p_{ \check{H} } 
  \right)
\\ & \leq
  3^p 
  \left[
    \sup_{ x \in \check{H} }
    \frac{
      \| \varphi(x) \|_V
    }{
      ( 1 + \|x\|^p_{ \check{H} } )
    }
  \right]
  \left(
    1
    +
    \|
      S_{ t_0, T } X_{ t_0 }
    \|^p_{ \check{H} }
    +
  \left|
    \int^T_{ t_0 }
    \|
      S_{ s, T } Y_s
    \|_{ \check{H} }
    \, ds
  \right|^p
  +
  \left\|
  \int^t_{ t_0 }
    S_{ s, T } \, Z_s
    \, dW_s
  \right\|^p_{ \check{H} }
  \right)
  .
\end{split}
\end{equation}
Moreover, e.g., 
the Burkholder-Davis-Gundy inequality 
in Problem 3.29 in Karatzas \& Shreve~\cite{ks91} 
shows that there exists a real number 
$ C \in [ 0, \infty ) $ such that 
\begin{equation}
\label{eq:maximal_ineq}
\begin{split}
&
  \E\left[
  \sup_{ t \in [ t_0, T ] }
  \left\|
    \int^t_{ t_0 }
    S_{ s, T } \, Z_s
    \, dW_s
  \right\|^p_{ \check{H} }
  \right]
\leq
  C
  \,
  \E\left[
  \left|
    \int^T_{ t_0 }
    \|
    S_{ s, T } \, Z_s
    \|^2_{ HS( U, \check{H} ) }
    \, ds
  \right|^{ p/2 }
  \right]
  .
\end{split}
\end{equation}
Combining \eqref{eq:path_varbar} and \eqref{eq:maximal_ineq} yields that there exists a real number 
$ C \in [ 0, \infty ) $ such that 
\begin{equation}
\label{eq:varphi_bound}
\begin{split}
&
  \E\left[
  \sup_{ t \in [ t_0, T ] }
  \left\|
    \varphi( \bar{X}_t )
  \right\|_V
  \right]
\\ & \leq
  C
  \Bigg(
    1
    +
    \E\left[
    \|
    S_{ t_0, T } \, X_{ t_0 }
    \|^p_{ \check{H} }
    \right]
    +
   \E\left[
    \left|
    \int^T_{ t_0 }
    \|
    S_{ s, T } \, Y_s
    \|_{ \check{H} }
    \, ds
    \right|^p 
    \right]
+
 \E\left[
  \left|
    \int^T_{ t_0 }
    \|
    S_{ s, T } \, Z_s
    \|^2_{ HS( U, \check{H} ) }
    \, ds
  \right|^{ p/2 }
  \right]
  \Bigg)
  .
\end{split}
\end{equation}
In the next step we combine \eqref{eq:varphi_bound} with the assumption that
$
  \| S_{ t_0, T } X_{ t_0 } \|_{ \check{H} }
  +
  \int^T_{ t_0 }
  \| S_{ s, T } Y_s \|_{ \check{H} }
  \, ds
  +
  \big[
    \int^T_{ t_0 }
    \| S_{ s, T } Z_s \|^2_{ HS( U, \check{H} ) }
    \, ds
  \big]^{ 1 / 2 }
  \in
  \lpn{p}{\P}{\R}
$
to obtain that
$
  \E\left[
  \sup_{ t \in [ t_0, T ] }
  \|
    \varphi( \bar{X}_t )
  \|_V
  \right]
  < \infty
$.
Item~\eqref{lem:transformation_ii}
of
Lemma~\ref{lem:transformation}
hence proves that
$
  \ES\big[
  \sup_{ t \in [ t_0, T ] }
  \|
    \varphi( 
      S_{ t_0, T } X_{ t_0 }
      +
      \int_{ t_0 }^t
      S_{ s, t } Y_s \, ds
      +
      \int_{ t_0 }^t
      S_{ s, t } Z_s \, dW_s
    )
  \|_V
  \big]
  < \infty
$.
Combining this with Proposition~\ref{prop:mild_ito_stopping_limit}
completes the proof of 
Proposition~\ref{prop:mild_ito_maximal_ineq}.
\end{proof}

\section{Weak temporal regularity and analysis of the weak distance between Euler-type 
approximations of SPDEs and their semilinear integrated counterparts}
\label{sec:weak_temporal_regularity}

In this section we establish a weak temporal regularity result in
Proposition~\ref{prop:weak_temporal_regularity_1st} below.
In addition, we prove a weak approximation result 
in Proposition~\ref{prop:weak_temporal_regularity_2nd} below.
The proofs of 
Proposition~\ref{prop:weak_temporal_regularity_1st}
and 
Proposition~\ref{prop:weak_temporal_regularity_2nd}
use Proposition~\ref{prop:mild_ito_maximal_ineq}
which, in turn, is established by an application 
of the mild It\^{o} formula.

\subsection{Setting}
\label{sec:setting_weak_temporal_regularity}

Assume the setting in Section~\ref{sec:semigroup_setting},
let 
$ \vartheta \in [0,1) $, 
$
  F \in 
  \operatorname{Lip}^0( H , H_{ - \vartheta } ) 
$, 
$
  B \in 
  \operatorname{Lip}^0( 
    H, 
    HS( 
      U, 
      H_{ 
        - \vartheta / 2 
      } 
    ) 
  ) 
$, 
$ p \in [ 2, \infty ) $, 
let 
$ 
  ( B^b )_{ b \in \mathbb{U} } \subseteq C( H, H_{ -\nicefrac{\vartheta}{2} } ) 
$ 
be the functions with the property that for all 
$ v \in H $, 
$ b \in \mathbb{U} $ 
it holds that 
$
     B^b( v ) 
    = 
      B( 
        v 
      )
      \,
      b
$, 
let $ \varsigma_{ F, B } \in \R $ 
be a real number given by 
$
  \varsigma_{ F, B }
  =
    \max\{
      1,
      \| F \|_{ \operatorname{Lip}^0( H, H_{ -\vartheta } ) },
      \| B \|^2_{ \operatorname{Lip}^0( H, HS( U, H_{ - \vartheta / 2 } ) ) } 
    \}
$, 
let 
$ 
  Y, \bar{Y} \colon [0,T] \times \Omega \to H
$ 
be 
$
  ( \mathcal{F}_t )_{ t \in [0,T] }
$-predictable stochastic processes
such that
$
  \| Y_0 \|_{ \lpn{p}{\P}{H} } 
  < \infty
$,
such that
$
  \bar{Y}_0 = Y_0
$,
and such that 
for all $ t \in (0,T] $ 
it holds $ \P $-a.s.\ that 
\begin{equation} 
  Y_t
  = 
    S_{ 0, t }\, Y_0 
  + 
    \int_0^t S_{ s, t }\, R_s\, F( Y_{ \floor{ s }{ h } } ) \, ds
  + 
    \int_0^t S_{ s, t }\, R_s\, B( Y_{ \floor{ s }{ h } } ) \, dW_s, 
\end{equation} 
\begin{equation} 
  \bar{Y}_t
  = 
    e^{ t A } \, \bar{Y}_0 
  + 
    \int_0^t e^{ ( t - s )A } \, F( Y_{ \floor{ s }{ h } } ) \, ds
  + 
    \int_0^t e^{ ( t - s )A } \, B( Y_{ \floor{ s }{ h } } ) \, dW_s 
    ,
\end{equation}
and let 
$
  ( K_r )_{ r \in [ 0, \infty ) }
  \subseteq
  [ 0, \infty ]
$ 
be extended real numbers
which satisfy that
for all $ r \in [0,\infty) $
it holds that
$
  K_r
  =
  \sup_{ s, t \in [ 0, T ] }
  \ES\big[
    \max\{
    1,
    \| \bar{Y}_s \|^r_H,
    \| Y_t \|^r_H
    \}
  \big]
$.

\subsection{Weak temporal regularity of semilinear integrated Euler-type approximations}

In Proposition~\ref{prop:weak_temporal_regularity_1st} below
we establish a weak temporal regularity result for the process $ \bar{Y} $
in Subsection~\ref{sec:setting_weak_temporal_regularity}.
The proof of Proposition~\ref{prop:weak_temporal_regularity_1st} uses the
following elementary result.

\begin{lemma}
\label{lem:Kp_estimate}
Assume the setting in Section~\ref{sec:setting_weak_temporal_regularity}. Then 
\begin{align}
\label{eq:Kp_estimate}
&
  \sup_{ r \in [ 0, p ] }
  K_r
  =
  K_p
\\ & \leq
\nonumber
  \left[
    \groupC_0 \,
    \max\{
      1
      ,
      \| Y_0 \|_{\lpn{p}{\P}{H}}
    \}
    +
    \tfrac{
    \groupC_\vartheta \,
    \|F\|_{\operatorname{Lip}^0(H, H_{-\vartheta})} \,
    T^{ ( 1 - \vartheta ) }
    }{( 1 - \vartheta )}
  +
  \tfrac{
    \groupC_{\nicefrac{\vartheta}{2}} \,
    \sqrt{p \, (p-1) \, T^{ ( 1 - \vartheta ) }} \,
    \|B\|_{\operatorname{Lip}^0(H, HS( U, H_{-\nicefrac{\vartheta}{2}} ))}
  }{\sqrt{2 - 2 \vartheta}}
  \right]^{2p}
\\ & \cdot
\nonumber
  2^{ ( \frac{p}{2}+1 ) }
  \left|
  \mathcal{E}_{ ( 1 - \vartheta ) }\!\left[
    \tfrac{
      \sqrt{ 2 }
      \,
      \groupC_{ \vartheta }
      \,
      T^{ ( 1 - \vartheta ) }
      \,
      |
        F
      |_{
        \operatorname{Lip}^0( H, H_{ - \vartheta } )
      }
    }{
      \sqrt{1 - \vartheta}
    }
    +
    \groupC_{ 
      \nicefrac{\vartheta}{2}
    }
    \sqrt{
      p \, (p-1) \, T^{ ( 1 - \vartheta ) }
    } \,
    |
      B
    |_{
      \operatorname{Lip}^0( H, HS( U, H_{ - \nicefrac{\vartheta}{2} } ) )
    }
  \right]
  \right|^p
  < \infty
  .
\end{align}
\end{lemma}
\begin{proof}
First of all, we observe that the equality in~\eqref{eq:Kp_estimate} follows from the fact that 
for all $ x \in H $, $ r, s \in [ 0, \infty ) $ with $ r \leq s $ 
it holds that 
$
  \max\{
    1
    ,
    \| x \|^r_H
  \}
  \leq
  \max\{
    1
    ,
    \| x \|^s_H
  \}
$.
Moreover, we note that the second inequality in~\eqref{eq:Kp_estimate} is an immediate consequence 
from the assumption that 
$
  \| Y_0 \|_{ \lpn{p}{\P}{H} }
  < \infty
$. 
It thus remains to prove the first inequality in~\eqref{eq:Kp_estimate}. 
For this, we observe that the Burkholder-Davis-Gundy type inequality 
in Lemma~7.7 in Da Prato \& Zabczyk~\cite{dz92} ensures that for all 
$ k \in \{ 1, 2, \ldots, \nicefrac{\floor{T}{h}}{h} \} $ 
it holds that 
\begin{equation}
\label{eq:numerics_bound_induction}
\begin{split}
&
  \| Y_{kh} \|_{ \lpn{p}{\P}{H} }
\\ & \leq
  \| S_{ 0, kh } \, Y_0 \|_{ \lpn{p}{\P}{H} }
  + 
    \left\|
    \int_0^{kh} 
    S_{ s, kh }\, R_s\, F( Y_{ \floor{ s }{ h } } ) 
    \, ds
    \right\|_{ \lpn{p}{\P}{H} }
  + 
    \left\|
    \int_0^{kh} 
    S_{ s, kh }\, R_s\, B( Y_{ \floor{ s }{ h } } ) 
    \, dW_s
    \right\|_{ \lpn{p}{\P}{H} } 
\\ & \leq
  \groupC_0 \, 
  \| Y_0 \|_{ \lpn{p}{\P}{H} }
  +
  \int^{kh}_0
  \|
  S_{ s, kh }\, R_s\, F( Y_{ \floor{ s }{ h } } )
  \|_{ \lpn{p}{\P}{H} }
  \, ds
\\ & +
  \left[
  \tfrac{p \, (p-1)}{2}
  \int^{kh}_0
  \|
  S_{ s, kh }\, R_s\, B( Y_{ \floor{ s }{ h } } )
  \|^2_{ \lpn{p}{\P}{HS(U,H)} }
  \, ds
  \right]^{1/2}
\\ & \leq
  \groupC_0 \, 
  \| Y_0 \|_{ \lpn{p}{\P}{H} }
  +
  \groupC_\vartheta \,
  \|F\|_{\operatorname{Lip}^0(H, H_{-\vartheta})}
  \left[
  \max_{ j \in \{ 0, 1, \ldots, k -1 \} }
  \big\|
  \max\{
    1
    ,
    \| Y_{ jh } \|_H
  \}
  \big\|_{ \lpn{p}{\P}{\R} }
  \right]
  \int^{kh}_0
  \tfrac{1}{( kh - s )^\vartheta}
  \, ds
\\ & +
  \groupC_{\nicefrac{\vartheta}{2}} \,
  \|B\|_{\operatorname{Lip}^0(H, HS( U, H_{-\nicefrac{\vartheta}{2}} ))}
  \left[
  \max_{ j \in \{ 0, 1, \ldots, k -1 \} }
  \big\|
  \max\{
    1
    ,
    \| Y_{ jh } \|_H
  \}
  \big\|_{ \lpn{p}{\P}{\R} }
  \right]
  \left[
  \tfrac{p \, (p-1)}{2}
  \int^{kh}_0
  \tfrac{1}{( kh - s )^\vartheta}
  \, ds
  \right]^{1/2}
\\ & \leq
  \left[
    \groupC_0
    +
    \tfrac{
    \groupC_\vartheta \,
    \|F\|_{\operatorname{Lip}^0(H, H_{-\vartheta})} \,
    |kh|^{ ( 1 - \vartheta ) }
    }{( 1 - \vartheta )}
  +
  \tfrac{
    \groupC_{\nicefrac{\vartheta}{2}} \,
    \sqrt{p \, (p-1) \, |kh|^{ ( 1 - \vartheta ) }} \,
    \|B\|_{\operatorname{Lip}^0(H, HS( U, H_{-\nicefrac{\vartheta}{2}} ))}
  }{\sqrt{2 - 2 \vartheta}}
  \right]
\\ & \cdot
  \max_{ j \in \{ 0, 1, \ldots, k -1 \} }
  \big\|
  \max\{
    1
    ,
    \| Y_{ jh } \|_H
  \}
  \big\|_{ \lpn{p}{\P}{\R} }
\end{split}
\end{equation}
This and the assumption that 
$
  \| Y_0 \|_{ \lpn{p}{\P}{H} }
  < \infty
$
allow us to conclude inductively that 
\begin{equation*}
  \sup_{ t \in [ 0, T ] }
  \| Y_{\floor{t}{h}} \|_{ \lpn{p}{\P}{H} }
  =
  \max_{ k \in \{ 0, 1, \ldots, \nicefrac{\floor{T}{h}}{h} \} }
  \| Y_{kh} \|_{ \lpn{p}{\P}{H} }
  < \infty
  . 
\end{equation*}
We can hence apply Proposition~\ref{prop:numerics_Lp_bound} 
to obtain\footnote{with 
$ \kappa = 0 $, 
$ L_{0,t} = S_{0,t} $, 
$ L_{s,t} = S_{s,t} \, R_s $, 
$ \Pi(s) = \floor{s}{h} $
for $ (s,t) \in ( \angle \cap (0,T]^2 ) $ 
in the notation of Proposition~\ref{prop:numerics_Lp_bound}} 
that
\begin{equation}
\label{eq:numerics_bound}
\begin{split}
&
  \sup_{t \in [ 0, T ]}
  \| Y_t \|_{ \lpn{p}{\P}{H} }
  \leq
  \sqrt{2}
\\ & \cdot
  \left[
    \groupC_0 \,
    \| Y_0 \|_{ \lpn{p}{\P}{H} }
    +
    \tfrac{
      \groupC_\vartheta \,
      T^{ ( 1 - \vartheta ) }
      \| F(0) \|_{ H_{ -\vartheta } }
    }{
      ( 1 - \vartheta )
    }
      +
      \groupC_{ \nicefrac{ \vartheta }{ 2 } }
      \sqrt{
      \tfrac{
      p \, (p-1) \, T^{ ( 1 - \vartheta ) }
      }{
      ( 2 - 2 \vartheta )
      }
      }
      \| B(0) \|_{ HS( U, H_{ -\nicefrac{\vartheta}{2} } ) }
  \right]
\\ & \cdot
  \mathcal{E}_{ ( 1 - \vartheta ) }\!\left[
    \tfrac{
      \sqrt{ 2 }
      \,
      \groupC_{ \vartheta }
      \,
      T^{ ( 1 - \vartheta ) }
      \,
      |
        F
      |_{
        \operatorname{Lip}^0( H, H_{ - \vartheta } )
      }
    }{
      \sqrt{1 - \vartheta}
    }
    +
    \groupC_{ 
      \nicefrac{\vartheta}{2}
    }
    \sqrt{
      p \, (p-1) \, T^{ ( 1 - \vartheta ) }
    } \,
    |
      B
    |_{
      \operatorname{Lip}^0( H, HS( U, H_{ - \nicefrac{\vartheta}{2} } ) )
    }
  \right]
  .
\end{split}
\end{equation}
Next we note that the Burkholder-Davis-Gundy type inequality 
in Lemma~7.7 in Da Prato \& Zabczyk~\cite{dz92} shows that 
\begin{equation}
\label{eq:integrated_numerics_bound}
\begin{split}
&
  \sup_{ t \in [ 0, T ] }
  \| \bar{Y}_t \|_{\lpn{p}{\P}{H}}
  \leq
  \sup_{ t \in [0,T] }
  \big\|
  \max\{
    1
    ,
    \| Y_t \|_H
  \}
  \big\|_{ \lpn{p}{\P}{\R} }
\\ & \cdot
  \left[
    \groupC_0
    +
    \tfrac{
    \groupC_\vartheta \,
    \|F\|_{\operatorname{Lip}^0(H, H_{-\vartheta})} \,
    T^{ ( 1 - \vartheta ) }
    }{( 1 - \vartheta )}
  +
  \tfrac{
    \groupC_{\nicefrac{\vartheta}{2}} \,
    \sqrt{p \, (p-1) \, T^{ ( 1 - \vartheta ) }} \,
    \|B\|_{\operatorname{Lip}^0(H, HS( U, H_{-\nicefrac{\vartheta}{2}} ))}
  }{\sqrt{2 - 2 \vartheta}}
  \right]
  .
\end{split}
\end{equation}
Moreover, we observe that for all $ s, t \in [ 0, T ] $ 
it holds that 
\begin{equation}
\begin{split}
  \ES\big[
    \max\{
    1,
    \| \bar{Y}_s \|^p_H,
    \| Y_t \|^p_H
    \}
  \big]
& \leq
  \ES\big[
  \| \bar{Y}_s \|^p_H
  \big]
  +
  \ES\big[
  \max\{
    1
    ,
    \| Y_t \|^p_H
  \}
  \big]
\\ & \leq 
  \sup_{ u \in [0,T] }
  \| \bar{Y}_u \|^p_{ \lpn{p}{\P}{H} }
  +
  \sup_{ u \in [0,T] }
  \big\|
  \max\{
    1
    ,
    \| Y_u \|_H
  \}
  \big\|^p_{ \lpn{p}{\P}{\R} }
  .
\end{split}
\end{equation}
This together with \eqref{eq:numerics_bound} and \eqref{eq:integrated_numerics_bound} 
proves the first inequality in~\eqref{eq:Kp_estimate}. 
The proof of Lemma~\ref{lem:Kp_estimate} is thus completed.
\end{proof}

\begin{proposition}
\label{prop:weak_temporal_regularity_1st}
Assume the setting in Section~\ref{sec:setting_weak_temporal_regularity} and let 
$ \psiC \in [ 0, \infty ) $, 
$ \power \in [ 0, \infty ) \cap ( -\infty, p - 3 ] $, 
$ \rho \in [ 0, 1 - \vartheta ) $, 
$ \psi = ( \psi(x,y) )_{ x, y \in H } \in C^2( H \times H, V ) $ 
satisfy that for all 
$ x_1, x_2, y \in H $, 
$ i, j \in \{ 0, 1, 2 \} $
with 
$ i + j \leq 2 $
it holds that 
\begin{equation*}
\begin{split}
  \big\|
    \big(
     \tfrac{ \partial^{ (i + j) } }{ \partial x^i \partial y^j  }
     \psi
    \big)
    ( x_1, y )
  -
    \big(
     \tfrac{ \partial^{ (i + j) } }{ \partial x^i \partial y^j  }
     \psi
    \big)
    ( x_2, y )
  \big\|_{ L^{ (i + j) }( H, V ) }
& \leq
  \psiC
  \max\{ 1, \| x_1 \|^\power_H, \| x_2 \|^\power_H, \| y \|^\power_H \}
  \left\| 
    x_1 - x_2 
  \right\|_H
  .
\end{split}
\end{equation*}
Then for all 
$ ( s, t ) \in \angle $ 
it holds that 
$
  \ES\big[
    \| 
      \psi( \bar{Y}_t, Y_s )
      -
      \psi( \bar{Y}_s, Y_s ) 
    \|_V
  \big]
  < \infty
$ 
and 
\begin{equation}
\label{eq:prop_temporal_reg}
\begin{split}
&
  \left\|
  \E\left[
    \psi(
      \bar{Y}_{ t } ,
      Y_{ s }
    )
  -
    \psi(
      \bar{Y}_{ s } ,
      Y_{ s }
    )
  \right]
  \right\|_V
\leq
  \psiC \,
  | \groupC_0 |^{ ( \power + 1 ) } \, 
  | \groupC_\rho |^2 \,
  \varsigma_{ F, B } \,
  K_{ \power + 3 } 
  \,
  ( t - s )^\rho
\\ & \cdot
  \bigg[
    \tfrac{
      2^{ \rho }
    }{
      t^{ \rho }
    }
    +
    \tfrac{
      \left(
        2 \, 
        \groupC_\vartheta 
        +
        \groupC_{ \rho + \vartheta }
        +
        2 \, 
        | \groupC_{ \vartheta / 2 } |^2
        + 
        2 \,
        \groupC_{ \rho + \vartheta / 2 }
        \,
        \groupC_{ \vartheta / 2 }
      \right)
      \,
      s^{ ( 1 - \vartheta - \rho ) }
      +
      \left( 
        \groupC_\vartheta 
        +
        \frac{ 1 }{ 2 }
        | \groupC_{ \nicefrac{ \vartheta }{ 2 } } |^2 
      \right)
      \,
      \left| t - s \right|^{ ( 1 - \vartheta - \rho ) }
    }{ 
      ( 1 - \vartheta - \rho ) }
  \bigg]
  .
\end{split}
\end{equation}
\end{proposition}

\begin{proof}
Throughout this proof 
let 
$
  ( g_r )_{ r \in [ 0, \infty ) }
  \subseteq
  C( H, \R )
$ 
be the functions
with the property that
for all $ r \in [0,\infty) $, $ x \in H $
it holds that
$
  g_r( x )
  =
  \max\{
  1, \| x \|^r_H
  \}
$ 
and let 
$
  \psi_{1,0} \colon H \times H \to L( H, V ) 
$, 
$
  \psi_{0,1} \colon H \times H \to L( H, V ) 
$, 
$
  \psi_{2,0} \colon H \times H \to L^{(2)}( H, V ) 
$, 
$
  \psi_{0,2} \colon H \times H \to L^{(2)}( H, V ) 
$, 
$
  \psi_{1,1} \colon H \times H \to L^{(2)}( H, V ) 
$
be the functions with the property that 
for all 
$ x, y, v_1, v_2 \in H $ 
it holds that 
$
  \psi_{1,0}( x, y ) \, v_1
  =
  \big(\tfrac{ \partial }{ \partial x } \psi( x, y )\big)( v_1 )
$
and
\begin{equation}
  \psi_{0,1}( x, y ) \, v_1
  =
  \big(\tfrac{ \partial }{ \partial y } \psi( x, y )\big)( v_1 )
  ,
  \qquad
  \psi_{2,0}( x, y )( v_1, v_2 )
  =
  \big(
    \tfrac{ \partial^2 }{ \partial x^2 } \psi( x, y )
  \big)( v_1, v_2 )
  ,
\end{equation}
\begin{equation}
  \psi_{0,2}( x, y )( v_1, v_2 )
  =
  \big(
    \tfrac{ \partial^2 }{ \partial y^2 } \psi( x, y )
  \big)( v_1, v_2 )
  ,
  \qquad
  \psi_{ 1, 1 }( x, y )( v_1, v_2 )
  =
  \big(
    \tfrac{ \partial }{ \partial y } 
    \tfrac{ \partial }{ \partial x } 
    \psi( x, y )
  \big)( v_1, v_2 )
  .
\end{equation} 
Next we observe that Lemma~\ref{lem:Kp_estimate}
and the assumption that $ q \leq p - 3 $ ensure that $ K_{ q + 1 } \leq K_{ q + 3 } < \infty $.
Combining this with the fact that
\begin{equation}
  \forall \, x_1, x_2, y \in H \colon
  \| \psi( x_1, y ) - \psi( x_2, y ) \|_V
  \leq 
  2 \, \eta
  \max\!\big\{ 
    1, \| x_1 \|^{ q + 1 }_H , \| x_2 \|^{ q + 1 }_H , \| y \|^{ q + 1 }_H 
  \big\}
\end{equation}
shows that 
for all $ ( s, t ) \in \angle $ 
it holds that
$
  \ES\big[
    \| 
      \psi( \bar{Y}_t, Y_s )
      -
      \psi( \bar{Y}_s, Y_s ) 
    \|_V
  \big]
  < \infty
$.
It thus remains to prove \eqref{eq:prop_temporal_reg}.
To do so, we make use of a consequence of
the mild It\^{o} formula 
in 
Corollary~\ref{cor:itoauto} above.
More formally,
an application of 
Proposition~\ref{prop:mild_ito_maximal_ineq}
shows\footnote{with 
$ t_0 = s $, 
$ T = t $, 
$ \check{H} = H \times H \times H $, 
$ p = q + 1 $, 
and
$ \varphi( x, y, z ) = \psi( x, y ) - \psi( z, y ) $ 
for $ (x,y,z) \in \check{H} $
in the notation of 
Proposition~\ref{prop:mild_ito_maximal_ineq}} that for all 
$ ( s, t ) \in \angle $ 
it holds that 
$
  \ES\big[
    \| 
      \psi( e^{ ( t - s )A } \, \bar{Y}_s, Y_s )
      -
      \psi( \bar{Y}_s, Y_s ) 
    \|_V
  \big]
  < \infty
$ 
and 
\begin{equation}
\label{eq:mild_ito_outer_1st}
\begin{split}
&
  \left\|
  \E\left[
    \psi(
      \bar{Y}_{ t } ,
      Y_{ s }
    )
  -
    \psi(
      \bar{Y}_{ s } ,
      Y_{ s }
    )
  \right]
  \right\|_V
  \leq
  \left\|\E\left[
    \psi(
      e^{ ( t - s )A } \, \bar{Y}_{ s } ,
      Y_{ s }
    \big)
    -
    \psi(
      \bar{Y}_{ s } ,
      Y_{ s }
    )
  \right]\right\|_V
\\ & +
  \int^{ t }_{ s }
  \E\left[
  \left\|
    \psi_{1,0}(
      e^{ ( t - r )A } \, \bar{Y}_r,
      Y_{ s }
    )
    \,
    e^{ ( t - r )A } 
    F(
      Y_{ \floor{ r }{ h } }
    )
  \right\|_V
  \right]
  dr
\\ & +
  \int^{ t }_{ s }
  \E\left[\left\|
   \tfrac{ 1 }{ 2 }
  {\smallsum\limits_{ b \in \mathbb{ U } }}
    \psi_{2,0}(
      e^{ ( t - r )A } \, \bar{Y}_r ,
      Y_{ s }
    )\big(
      e^{ ( t - r ) A } \,
      B^b( Y_{ \floor{ r }{ h } } )
    ,
      e^{ ( t - r ) A } \,
      B^b( Y_{ \floor{ r }{ h } } )
  \big)
  \right\|_V\right]
  dr .
\end{split}
\end{equation}
In the following we establish suitable estimates 
for the three summands appearing on right hand side of \eqref{eq:mild_ito_outer_1st}.
Combining these estimates with \eqref{eq:mild_ito_outer_1st} will then
allow us to establish \eqref{eq:prop_temporal_reg}.
We begin with the second and the third summands on 
the right hand side of \eqref{eq:mild_ito_outer_1st}.
We note that the assumption that
$
  \forall \, x_1, x_2 , y \in H \colon
$
\begin{equation}
  \left\|
    \psi( x_1 , y )
    -
    \psi( x_2, y )
  \right\|_V
  \leq
  \psiC
  \max\{
    1,
    \| x_1 \|^\power_H
    ,
    \| x_2 \|^\power_H
    ,
    \| y \|^\power_H
  \}
  \,
  \| x_1 - x_2 \|_H
\end{equation}
implies that
$
  \forall \, x , y \in H \colon
  \|
    \psi_{ 1, 0 }( x, y )
  \|_{ L( H, V ) }
  \leq
  \psiC
  \max\{
    1,
    \| x \|^\power_H,
    \| y \|^\power_H
  \}
$.
This, in turn, proves that for all 
$ ( r, t ) \in \angle $, 
$ u, v, w \in H $ 
it holds that 
\begin{equation}
\begin{split}
\label{eq:drift_1st_mild_ito}
&
  \left\|
    \psi_{ 1 , 0 }(
      e^{ ( t - r ) A } \, u ,
      v
    )
    \, e^{ ( t - r ) A } \,
    F( w )
  \right\|_V
\\ & \leq
  \eta
  \left| C_0 \right|^{ \power }
  \max\!\left\{ 
    1 ,
    \| u \|_H^{ \power }
    ,
    \| v \|_H^{ \power }
  \right\}
  \| 
    e^{ ( t - r ) A }
  \|_{
    L( H, H_{ \vartheta } )
  }
  \left\| F( w ) \right\|_{ H_{ - \vartheta } }
\\ & \leq
  \frac{
    \psiC \,
    | \groupC_0 |^\power \, 
    \groupC_\vartheta
    \max\{
      1, \| u \|^\power_H, \| v \|^\power_H
    \}
    \| F \|_{ \operatorname{Lip}^0( H, H_{ -\vartheta } ) }   
    \,
    g_1( w )
  }{
    ( t - r )^\vartheta
  }
  .
\end{split}
\end{equation}
Next we observe that the assumption that
$
  \forall \, x_1, x_2 , y \in H \colon
$
\begin{equation}
  \left\|
    \psi_{ 1,0 }( x_1 , y )
    -
    \psi_{ 1,0 }( x_2 , y )
  \right\|_{ 
    L( H, V ) 
  }
  \leq
  \psiC
  \max\{
    1,
    \| x_1 \|^\power_H
    ,
    \| x_2 \|^\power_H
    ,
    \| y \|^\power_H
  \}
  \| x_1 - x_2 \|_H
\end{equation}
shows that
$
  \forall \, x , y \in H \colon
  \|
    \psi_{ 2, 0 }( x, y )
  \|_{ L^{ (2) }( H, V ) }
  \leq
  \psiC
  \max\{
    1,
    \| x \|^\power_H,
    \| y \|^\power_H
  \}
$.
This, in turn, proves that for all 
$ ( r, t ) \in \angle $, 
$ u, v, w \in H $ 
it holds that 
\begin{equation}
\begin{split}
\label{eq:diffusion_1st_mild_ito}
&
  \tfrac{ 1 }{ 2 }
  \smallsum\limits_{ b \in \mathbb{ U } }
  \displaystyle
  \left\|
    \psi_{2,0}\big(
     e^{ ( t - r ) A } \, u,
    v
    \big)
  \big(
    e^{ ( t - r ) A } \,
    B^b( w )
    ,
    e^{ ( t - r ) A } \,
    B^b( w )
  \big)
  \right\|_V
\\ & \leq 
  \psiC \,
  | \groupC_0 |^{ \power }
  \,
  \max\{
    1, \| u \|^\power_H, \| v \|^\power_H
  \}
  \,
  \|
    e^{ ( t - r ) A }
    B( w )
  \|^2_{
    HS( U, H )
  }
\\ & \leq
  \frac{
    \psiC \,
    | \groupC_0 |^\power
    \,
    \frac{ 1 }{ 2 } \,
    | \groupC_{ \nicefrac{ \vartheta }{ 2 } } |^2
    \max\{
      1, \| u \|^\power_H, \| v \|^\power_H
    \}
    \,
    \| B \|^2_{ \operatorname{Lip}^0( H, HS( U, H_{ - \vartheta / 2 } ) ) } \,
    g_2( w )
  }{
    ( t - r )^\vartheta
  }
  .
\end{split}
\end{equation}
Furthermore, we note
that H\"{o}lder's inequality implies that 
for all 
$ r, l \in ( 0, \infty ) $, 
$ s, t \in [ 0, T ] $ 
it holds that 
\begin{equation}
\begin{split}
&
  \ES\!\left[
    \max\!\left\{
      1,
      \| \bar{Y}_s \|^r_H,
      \| Y_t \|^r_H
    \right\}
    g_l( Y_{ \floor{ s }{ h } } )
  \right]
\\ & \leq
  \left(
    \sup_{ u, v \in [ 0, T ] }
    \left\|
    \max\!\left\{
      1,
      \| \bar{Y}_u \|^r_H,
      \| Y_v \|^r_H
    \right\}
    \right\|_{ \lpn{ 1 + \nicefrac{l}{r} }{ \P }{ \R } }
  \right)
  \left(
    \sup_{ u \in [ 0, T ] }
    \left\|
    \max\!\left\{
      1,
      \| Y_u \|^l_H
    \right\}
    \right\|_{ \lpn{ 1 + \nicefrac{r}{l} }{ \P }{ \R } }
  \right)
\\ & \leq
  | K_{ r + l } |^{ \frac{ 1 }{ 1 + \nicefrac{l}{r} } }
  \,
  | K_{ r + l } |^{ \frac{ 1 }{ 1 + \nicefrac{r}{l} } }
  =
  K_{ r + l }
  .
\end{split}
\end{equation}
This and the fact that for all 
$ l \in [0,\infty) $
it holds that 
$
  \sup_{ s \in [0,T] }
  \ES\big[
    g_l( Y_{ \floor{ s }{h} } )
  \big]
  \leq
  K_l
$ 
prove that
for all 
$ r, l \in [ 0, \infty ) $, 
$ s, t \in [ 0, T ] $ 
it holds that 
\begin{equation}
\label{eq:optimal_moment_1}
\begin{split}
&
  \ES\!\left[
    \max\!\left\{
      1,
      \| \bar{Y}_s \|^r_H,
      \| Y_t \|^r_H
    \right\}
    g_l( Y_{ \floor{ s }{ h } } )
  \right]
\leq
  K_{ r + l }
  .
\end{split}
\end{equation}
Combining \eqref{eq:drift_1st_mild_ito}, \eqref{eq:diffusion_1st_mild_ito}, 
and \eqref{eq:optimal_moment_1} 
implies that for all 
$ ( s, t ) \in \angle $ 
it holds that 
\begin{equation}
\label{eq:outer_summands}
\begin{split}
&
  \int^{ t }_{ s }
  \E\left[\left\|
    \psi_{1,0}(
      e^{ ( t - r )A } \, \bar{Y}_r,
      Y_{ s }
    )
    \,
    e^{ ( t - r )A } 
    F(
      Y_{ \floor{ r }{ h } }
    )
  \right\|_V\right]
  dr
\\ & +
  {\displaystyle
  \int^{ t }_{ s }}
  \E\left[\left\|
   \tfrac{ 1 }{ 2 }
  {\smallsum\limits_{ b \in \mathbb{ U } }}
    \psi_{2,0}(
      e^{ ( t - r )A } \, \bar{Y}_r ,
      Y_{ s }
    )\big(
      e^{ ( t - r ) A } \,
      B^b( Y_{ \floor{ r }{ h } } )
    ,
      e^{ ( t - r ) A } \,
      B^b( Y_{ \floor{ r }{ h } } )
  \big)
  \right\|_V\right]
  dr
\\ & \leq
    \psiC \,
    | \groupC_0 |^\power
    \left(
      \groupC_\vartheta
      \,
      \| F \|_{
        \operatorname{Lip}^0( H, H_{ - \vartheta } )
      }
      +
      \tfrac{ 1 }{ 2 }
      \,
      | \groupC_{ \nicefrac{ \vartheta }{ 2 } } |^2 
      \,
      \| B \|^2_{ \operatorname{Lip}^0( H, HS( U, H_{ - \vartheta / 2 } ) ) } 
    \right)
    K_{ \power + 2 }
    \int_s^t
    \frac{
      1
    }{
      ( t - r )^{ \vartheta }
    }
    \,
    dr
\\ & \leq
  \frac{
    \psiC \,
    | \groupC_0 |^\power
    \left(
      \groupC_\vartheta
      +
      \frac{ 1 }{ 2 }
      | \groupC_{ \nicefrac{ \vartheta }{ 2 } } |^2 
    \right)
    \varsigma_{ F, B } 
    \,
    K_{ \power + 2 }
    \left( t - s \right)^{ ( 1 - \vartheta ) }
  }{
    ( 1 - \vartheta )
  }
  .
\end{split}
\end{equation}
Inequality~\eqref{eq:outer_summands} provides us an
appropriate estimate for the second and the third summand on 
the right hand side of \eqref{eq:mild_ito_outer_1st}.
It thus remains to provide a suitable estimate 
for the first summand on the right hand side of 
\eqref{eq:mild_ito_outer_1st}.
For this we will employ Proposition~\ref{prop:mild_ito_maximal_ineq} again
and this will allow us to obtain an appropriate upper bound for  
$
  \big\|
  \ES\big[
    \psi(
      e^{ ( t - s ) A } \, \bar{Y}_s ,
      Y_s
    )
    -
    \psi(
      \bar{Y}_{ s } ,
      Y_{ s }
    )
  \big]
  \big\|_V
$ 
for 
$ ( s, t ) \in \angle $.
More formally, let 
$ \tilde{F}_{ r, s, t } \colon H \times H \times H \to V $, 
$ r \in [ 0, s ) $, 
$ s \in ( 0, t ) $, 
$ t \in ( 0, T ] $, 
be the functions with the property that 
for all 
$ t \in ( 0, T ] $, 
$ s \in ( 0, t ) $, 
$ r \in [ 0, s ) $, 
$ u,v, w \in H $ 
it holds that 
\begin{equation}
\begin{split}
\label{eq:tilde_F_def_1st}
  \tilde{F}_{ r,s,t }(u,v,w)
& =
  \psi_{1,0}\big(
    e^{ ( t - r )A } \, u,
    S_{ r, s } \, v
  \big) \,
  e^{ ( t - r )A } \, F( w )
  -
  \psi_{1,0}\big(
    e^{ ( s - r )A } \, u,
    S_{ r, s } \, v
  \big) \,
  e^{ ( s - r )A } \, F( w )
\\ & 
  +
  \left[
    \psi_{0,1}\big(
      e^{ ( t - r )A } \, u,
      S_{ r, s } \, v
    \big)
    -
    \psi_{0,1}\big(
      e^{ ( s - r )A } \, u,
      S_{ r, s } \, v
    \big)
  \right]
  S_{ r, s } \, R_r \, F( w )
\end{split}
\end{equation} 
and let 
$ \tilde{B}_{ r, s, t } \colon H \times H \times H \to V $, 
$ r \in [ 0, s ) $, 
$ s \in ( 0, t ) $, 
$ t \in ( 0, T ] $, 
be the functions with the property that 
for all 
$ t \in ( 0, T ] $, 
$ s \in ( 0, t ) $, 
$ r \in [ 0, s ) $, 
$ u,v, w \in H $ 
it holds that 
\begin{equation}
\begin{split}
&
  \tilde{B}_{r,s,t}( u, v, w )
=
  \tfrac{1}{2}
  \smallsum\limits_{ b \in \mathbb{U} }
  \psi_{2,0}\big(
     e^{ ( t - r )A } \, u,
     S_{ r, s } \, v
  \big)
  \big(
  e^{ ( t - r )A } \, B^b( w ),
  e^{ ( t - r )A } \, B^b( w )
  \big)
\\ &
-
  \tfrac{1}{2}
  \smallsum\limits_{ b \in \mathbb{U} }
  \psi_{2,0}\big(
     e^{ ( s - r )A } \, u,
     S_{ r, s } \, v
  \big)
  \big(
  e^{ ( s - r )A } \, B^b( w ),
  e^{ ( s - r )A } \, B^b( w )
  \big)
\\ & 
  +
  \tfrac{1}{2}
  \smallsum\limits_{ b \in \mathbb{U} }
  \left[
    \psi_{0,2}\big(
      e^{ ( t - r )A } \, u,
      S_{ r, s } \, v
    \big)
    -
    \psi_{0,2}\big(
      e^{ ( s - r )A } \, u,
      S_{ r, s } \, v
    \big)
  \right]\!\big(
    S_{ r, s } \, R_r \, B^b( w )
    ,
    S_{ r, s } \, R_r \, B^b( w )
  \big)
\\ & +
  \smallsum\limits_{ b \in \mathbb{U} }
  \psi_{1,1}\big(
     e^{ ( t - r )A } \, u,
     S_{ r, s } \, v
  \big)
  \big(
  e^{ ( t - r )A } \, B^b( w ),
  S_{ r, s } \, R_r \, B^b( w )
  \big)
\\ & 
-
  \smallsum\limits_{ b \in \mathbb{U} }
  \psi_{1,1}\big(
     e^{ ( s - r )A } \, u,
     S_{ r, s } \, v
  \big)
  \big(
  e^{ ( s - r )A } \, B^b( w ),
  S_{ r, s } \, R_r \, B^b( w )
  \big)
  .
\end{split}
\end{equation}
An application of Proposition~\ref{prop:mild_ito_maximal_ineq}
then shows\footnote{with 
$ t_0 = 0 $, 
$ T = s $, 
$ \check{H} = H \times H $, 
$ p = q + 1 $,
and
$
  \varphi(x,y)
  =
  \psi( e^{ ( t - s )A } x, y )
  -
  \psi( x, y )
$ 
for $ ( x, y ) \in \check{H} $
in the notation of Proposition~\ref{prop:mild_ito_maximal_ineq}} 
that for all 
$ t \in (0,T] $, $ s \in (0,t) $
it holds that 
\begin{equation}
\label{eq:mild_ito_inner_1st}
\begin{split}
&
  \left\|\E\left[
    \psi(
      e^{ ( t - s )A } \, \bar{Y}_{ s } ,
      Y_{ s }
    )
    -
    \psi(
      \bar{Y}_{ s } ,
      Y_{ s }
    )
  \right]\right\|_V
\leq
  \E\left[\left\|
    \psi\big(
      e^{ t A } \, Y_0
      ,
      S_{ 0, s } \, Y_0
    \big)
    -
    \psi\big(
      e^{ s A } \, Y_0
      ,
      S_{ 0, s } \, Y_0
    \big)
  \right\|_V\right]
\\ & +
  \int^{ s }_0
  \ES\big[\|
    \tilde{F}_{r,s,t}\big(
      \bar{Y}_r , 
      Y_r , 
      Y_{ \floor{ r }{ h } }
    \big)
  \|_V\big]
  \,
  dr
  +
  \int^{ s }_0
  \ES\big[\|
    \tilde{B}_{r,s,t}\big(
      \bar{Y}_r , 
      Y_r , 
      Y_{ \floor{ r }{ h } }
    \big)
  \|_V\big]
  \,
  dr
  .
\end{split}
\end{equation}
In the next step we estimate the summands on the right hand side of~\eqref{eq:mild_ito_inner_1st}. 
We observe that for all 
$ t \in (0,T] $, $ s \in (0,t) $
it holds that 
\begin{equation}
\begin{split}
&
  \left\|
    \psi\big(
    e^{ t A } \, Y_0,
    S_{ 0, s } \, Y_0
    \big)
    -
    \psi\big(
    e^{ s A } \, Y_0,
    S_{ 0, s } \, Y_0
    \big)
  \right\|_V
\\ & \leq 
  \eta
  \max\!\left\{ 
    1 ,
    \| e^{ t A } \, Y_0 \|_H^{ \power }
    ,
    \| S_{ 0, s } \, Y_0 \|_H^{ \power }
  \right\}
  \| e^{ t A } \, Y_0 - e^{ s A } \, Y_0 \|_H
\\ & \leq
  \psiC \,
  |\groupC_0|^{ \power }
  \,
  g_\power( Y_0 ) \,
  \| e^{ t A } - e^{ s A } \|_{ L( H ) }
  \,
  \| Y_0 \|_H
\leq
  \psiC
  \,
  | \groupC_0 |^{ \power }
  \,
  g_{ \power + 1 }( Y_0 ) \,
  \tfrac{
    | \groupC_\rho |^2 \,
    ( t - s )^\rho
  }{
    s^{ \rho }
  }
  .
\end{split}
\end{equation}
This and the fact that 
$
  \ES\big[ 
    g_{ \power + 1 }( Y_0 )
  \big]
  \leq
  K_{ \power + 1 }
$ 
imply that for all 
$ t \in (0,T] $, $ s \in (0,t) $
it holds that 
\begin{equation}
\label{eq:1st_summand_inner}
\begin{split}
&
  \E\left[
  \|
    \psi\big(
      e^{ t A } \, Y_0
      ,
      S_{ 0, s } \, Y_0
    \big)
    -
    \psi\big(
      e^{ s A } \, Y_0
      ,
      S_{ 0, s } \, Y_0
    \big)
  \|_V
  \right]
\leq
  \psiC
  \,
  | \groupC_0 |^{ \power }
  \,
  K_{ \power + 1 } \,
  \tfrac{
    |\groupC_\rho|^2
  }{
    s^{ \rho }
  }
  \left( t - s \right)^{ \rho }
  .
\end{split}
\end{equation}
Inequality \eqref{eq:1st_summand_inner} provides us an appropriate estimate for 
the first summand on the right hand side of~\eqref{eq:mild_ito_inner_1st}. 
In the next step we establish a suitable bound for the second summand on the right hand side of~\eqref{eq:mild_ito_inner_1st}. 
Note that for all 
$ t \in ( 0, T ] $, 
$ s \in ( 0, t ) $, 
$ r \in [ 0, s ) $, 
$ u,v, w \in H $ 
it holds that 
\begin{equation}
\label{eq:mild_ito_2nd_F_A}
\begin{split}
&
  \left\|
    \psi_{1,0}\big(
     e^{ ( t - r )A } \, u,
     S_{ r, s } \, v
    \big) 
    \,
    e^{ ( t - r )A } \, F( w )
    -
    \psi_{1,0}\big(
     e^{ ( s - r )A } \, u,
     S_{ r, s } \, v
    \big) 
    \,
    e^{ ( s - r )A } \, F( w )
  \right\|_V
\\ & \leq
  \left\|
    \big[
      \psi_{1,0}\big(
         e^{ ( t - r )A } \, u,
         S_{ r, s } \, v
      \big)
      -
      \psi_{1,0}\big(
        e^{ ( s - r )A } \, u,
        S_{ r, s } \, v
      \big)
    \big]
    \,
    e^{ ( t - r )A } \, F( w )
  \right\|_V
\\ & +
  \left\|
    \psi_{1,0}\big(
     e^{ ( s - r )A } \, u,
     S_{ r, s } \, v
    \big) \,
    e^{ ( s - r )A } 
    \left(
      e^{ ( t - s )A } - \operatorname{Id}_H 
    \right)
    F( w )
  \right\|_V
\\ & \leq
  \psiC
  \,
  \max\!\big\{
    1
    , 
    \| e^{ ( t - r ) A } u \|^\power_H
    , 
    \| e^{ ( s - r ) A } u \|^\power_H
    , 
    \| S_{ r, s } v \|^\power_H
  \big\}
  \left\|
    \left[ 
      e^{ ( t - r ) A } - e^{ ( s - r ) A } 
    \right] 
    u
  \right\|_H
  \| 
    e^{ ( t - r ) A } F( w )
  \|_{
    H
  }
\\ & +
    \psiC 
    \,
    \max\!\big\{
      1, \| e^{ ( s - r ) A } u \|^\power_H, \| S_{ r, s } v \|^\power_H
    \big\}
    \,
    \left\| 
      e^{ ( s - r ) A }
      \left(
        e^{ ( t - s ) A }
        -
        \operatorname{Id}_H
      \right)
      F( w )
    \right\|_H
\\ & \leq
  \frac{
    \psiC
    \,
    | \groupC_0 |^\power
    \max\!\big\{
      1, \| u \|^\power_H, \| v \|^\power_H
    \big\}
    \,
    | \groupC_\rho |^2
    \,
    ( t - s )^\rho
    \,
    \| u \|_H 
    \,
    \groupC_\vartheta
    \,
    \| F \|_{ \operatorname{Lip}^0( H, H_{ -\vartheta } ) }
    \, 
    g_1( w )
  }{
    ( s - r )^{ \rho }
    \,
    ( t - r )^{ \vartheta }
  }
\\ & +
  \frac{
    \psiC \,
    | \groupC_0 |^\power
    \max\!\big\{
      1, \| u \|^\power_H, \| v \|^\power_H
    \big\}
    \,
    \groupC_{ \rho + \vartheta } 
    \,
    \groupC_\rho 
    \,
    ( t - s )^\rho
    \,
    \| F \|_{ \operatorname{Lip}^0( H, H_{ -\vartheta } ) } 
    \,
    g_1( w )
  }{
    ( s - r )^{ ( \rho + \vartheta ) }
  }
  ,
\end{split}
\end{equation}
\begin{equation}
\label{eq:mild_ito_2nd_F_B}
\begin{split}
&
  \left\|
    \left[
      \psi_{0,1}\big(
        e^{ ( t - r )A } \, u
        ,
        S_{ r, s } \, v
      \big)
      -
      \psi_{0,1}\big(
        e^{ ( s - r )A } \, u
        ,
        S_{ r, s } \, v
      \big)
    \right]
    S_{ r, s } \, R_r \, F( w )
  \right\|_V
\\ & \leq
    \psiC
    \,
    | \groupC_0 |^\power
    \max\!\big\{
      1, \| u \|^\power_H, \| v \|^\power_H
    \big\}
    \,
    \|
      e^{ ( t - r ) A } \, u 
      - 
      e^{ ( s - r ) A } \, u
    \|_H
    \,
    \| S_{ r, s } R_r \|_{ 
      L(
        H_{ - \vartheta }
        ,
        H
      )
    }
    \,
    \| F( w ) \|_{
      H_{ - \vartheta }
    }
\\ & \leq
  \frac{
    \psiC
    \,
    | \groupC_0 |^\power
    \max\!\big\{
      1, \| u \|^\power_H, \| v \|^\power_H
    \big\}
    \,
    | \groupC_\rho |^2
    \,
    ( t - s )^\rho
    \,
    \| u \|_H 
    \,
    \groupC_\vartheta
    \,
    \| F \|_{ \operatorname{Lip}^0( H, H_{ -\vartheta } ) } 
    \,
    g_1( w )
  }{
    ( s - r )^{ ( \rho + \vartheta ) }
  }
\\ & \leq
  \frac{
    \psiC
    \,
    | \groupC_0 |^\power
    \,
    | \groupC_\rho |^2
    \,
    \groupC_\vartheta
    \max\!\big\{
      1, \| u \|^{ \power + 1 }_H, \| v \|^{ \power + 1 }_H
    \big\}
    \,
    \| F \|_{ \operatorname{Lip}^0( H, H_{ -\vartheta } ) } 
    \,
    g_1( w )
    \,
    ( t - s )^\rho
  }{
    ( s - r )^{ ( \rho + \vartheta ) }
  }
  .
\end{split}
\end{equation}
Inequalities~\eqref{eq:mild_ito_2nd_F_A} and \eqref{eq:mild_ito_2nd_F_B} prove that for all 
$ t \in ( 0, T ] $, 
$ s \in ( 0, t ) $, 
$ r \in [ 0, s ) $, 
$ u,v, w \in H $ 
it holds that 
\begin{equation}
\begin{split}
&
  \| 
    \tilde{F}_{r,s,t}( u, v, w ) 
  \|_V
\\ & \leq
  \psiC
  \,
  | \groupC_0 |^{ \power }
  \left[
    \frac{
      | \groupC_\rho |^2
      \,
      \groupC_\vartheta
    }{
      ( s - r )^{ \rho }
      \,
      ( t - r )^{ \vartheta }
    }
    +
    \frac{ 
      \groupC_{ \rho + \vartheta } 
      \,
      \groupC_\rho 
    }{
      ( s - r )^{ ( \rho + \vartheta ) }
    }
    +
    \frac{
      | \groupC_\rho |^2
      \,
      \groupC_\vartheta
    }{
      ( s - r )^{ ( \rho + \vartheta ) }
    }
  \right]
\\ & \quad \cdot
  \max\!\big\{
    1, \| u \|^{\power+1}_H, \| v \|^{\power+1}_H
  \big\}
  \,
  \| F \|_{ \operatorname{Lip}^0( H, H_{ -\vartheta } ) } 
  \,
  g_1( w )
  \,
  ( t - s )^\rho
\\ & \leq
  \psiC
  \,
  | \groupC_0 |^{ \power }
  \left[
    \frac{
      \groupC_\rho
      \,
      (
        2 \, \groupC_\rho \groupC_\vartheta
        +
        \groupC_{ \rho + \vartheta }
      )
    }{
      ( s - r )^{ ( \rho + \vartheta ) }
    } 
  \right]
  \max\!\big\{
    1, \| u \|^{\power+1}_H, \| v \|^{\power+1}_H
  \big\}
  \,
  \| F \|_{ \operatorname{Lip}^0( H, H_{ -\vartheta } ) } 
  \,
  g_1( w )
  \,
  ( t - s )^\rho
  .
\end{split}
\end{equation}
This and \eqref{eq:optimal_moment_1} 
prove that for all 
$ t \in (0,T] $, $ s \in (0,t) $
it holds that 
\begin{equation}
\label{eq:F_summand_inner}
\begin{split}
&
  \int^{ s }_0
  \ES\big[\|
    \tilde{F}_{r,s,t}\big(
      \bar{Y}_r , 
      Y_r , 
      Y_{ \floor{ r }{ h } }
    \big)
  \|_V\big]
  \,
  dr
\\ & \leq
  \frac{
    \psiC
    \,
    | \groupC_0 |^q
    \,
    \groupC_\rho
    \,
    \big(
      2 \, \groupC_\rho \, \groupC_\vartheta 
      +
      \groupC_{ \rho + \vartheta }
    \big)
  }{
    ( 1 - \vartheta - \rho )
  }
  \,
  \| F \|_{ \operatorname{Lip}^0( H, H_{ -\vartheta } ) } 
  \,
  K_{ \power + 2 }
  \,
  ( t - s )^\rho 
  \,
  s^{ ( 1 - \vartheta - \rho ) }
  .
\end{split}
\end{equation}
Next we provide an appropriate bound for the third summand on the right hand side of~\eqref{eq:mild_ito_inner_1st}. 
Observe that 
for all 
$ t \in ( 0, T ] $, 
$ s \in ( 0, t ) $, 
$ r \in [ 0, s ) $, 
$ u,v, w \in H $ 
it holds that 
\allowdisplaybreaks
\begin{align}
\label{eq:prop_temporal_reg_termB1}
&
\nonumber
  \Big\|
  \smallsum\limits_{ b \in \mathbb{U} }
  \psi_{2,0}\big(
     e^{ ( t - r )A } \, u,
     S_{ r, s } \, v
  \big)
  \big(
  e^{ ( t - r )A } \, B^b( w ),
  e^{ ( t - r )A } \, B^b( w )
  \big)
\\ &
\nonumber
  -
  \smallsum\limits_{ b \in \mathbb{U} }
  \psi_{2,0}\big(
     e^{ ( s - r )A } \, u,
     S_{ r, s } \, v
  \big)
  \big(
  e^{ ( s - r )A } \, B^b( w ),
  e^{ ( s - r )A } \, B^b( w )
  \big)
  \Big\|_V
\\ & \leq
\nonumber
  \smallsum\limits_{ b \in \mathbb{U} }
  \left\|
    \big[
      \psi_{2,0}\big(
         e^{ ( t - r )A } \, u,
         S_{ r, s } \, v
      \big)
      -
      \psi_{2,0}\big(
        e^{ ( s - r )A } \, u,
        S_{ r, s } \, v
      \big)
    \big]
    \big(
      e^{ ( t - r )A } \, B^b( w )
      ,
      e^{ ( t - r )A } \, B^b( w )
    \big)
  \right\|_V
\\ & +
  \smallsum\limits_{ b \in \mathbb{U} }
  \left\|
    \psi_{2,0}\big(
      e^{ ( s - r )A } \, u
      ,
      S_{ r, s } \, v
    \big)
    \big(
      (
        e^{ ( t - r )A } + e^{ ( s - r )A }
      )
      \, 
      B^b( w )
      ,
      e^{ ( s - r ) A } 
      \,
      (
        e^{ ( t - s ) A } - \operatorname{Id}_H 
      )
      \,
      B^b( w )
    \big)
  \right\|_V
\\ & \leq
\nonumber
  \frac{
    \psiC
    \,
    | \groupC_0 |^\power
    \max\{
      1, \| u \|^\power_H, \| v \|^\power_H
    \}
    \,
    | \groupC_\rho |^2
    \,
    ( t - s )^\rho
    \,
    \| u \|_H 
    \,
    | \groupC_{ \nicefrac{ \vartheta }{ 2 } } |^2
    \,
    \| B \|^2_{ \operatorname{Lip}^0( H, HS( U, H_{ - \vartheta / 2 } ) ) }
    \, 
    g_2( w )
  }{
    ( s - r )^{ \rho }
    \,
    ( t - r )^{ \vartheta }
  }
\\ & +
\nonumber
  \frac{
    \psiC 
    \,
    | \groupC_0 |^\power
    \max\{
      1, \| u \|^\power_H, \| v \|^\power_H
    \}
    \,
    2 
    \,
    \groupC_{ \nicefrac{ \vartheta }{ 2 } } 
    \, 
    \groupC_{ \rho + \nicefrac{ \vartheta }{ 2 } } 
    \,
    \groupC_\rho
    \,
    ( t - s )^\rho
    \,
    \| B \|^2_{ 
      \operatorname{Lip}^0( 
        H, 
        HS( U, H_{ - \vartheta / 2 } ) 
      ) 
    }
    \,
    g_2( w )
  }{
    ( s - r )^{ ( \rho + \vartheta ) }
  }
  ,
\end{align}
\begin{equation}
\begin{split}
&
  \Big\|
    \smallsum\limits_{ b \in \mathbb{U} }
    \left[
      \psi_{0,2}\big(
        e^{ ( t - r )A } \, u,
        S_{ r, s } \, v
      \big)
      -
      \psi_{0,2}\big(
        e^{ ( s - r )A } \, u,
        S_{ r, s } \, v
      \big)
    \right]
    \!
    \big(
      S_{ r, s } \, R_r \, B^b( w )
      ,
      S_{ r, s } \, R_r \, B^b( w )
    \big)
  \Big\|_V
\\ & \leq
  \frac{
    \psiC
    \,
    | \groupC_0 |^\power
    \max\!\big\{
      1, \| u \|^\power_H, \| v \|^\power_H
    \big\}
    \,
    | \groupC_\rho |^2
    \,
    ( t - s )^\rho
    \,
    \| u \|_H 
    \,
    | \groupC_{ \nicefrac{ \vartheta }{ 2 } } |^2
    \,
    \| B \|^2_{ \operatorname{Lip}^0( H, HS( U, H_{ -\nicefrac{ \vartheta }{ 2 } } ) ) }\,
    g_2( w )
  }{
    ( s - r )^{ ( \rho + \vartheta ) }
  }
  ,
\end{split}
\end{equation}
\allowdisplaybreaks
\begin{align}
\label{eq:prop_temporal_reg_termB3}
&
\nonumber
  \Big\|
  \smallsum\limits_{ b \in \mathbb{U} }
  \psi_{1,1}\big(
     e^{ ( t - r )A } \, u,
     S_{ r, s } \, v
  \big)
  \big(
  e^{ ( t - r )A } \, B^b( w ),
  S_{ r, s } \, R_r \, B^b( w )
  \big)
\\ & 
\nonumber
-
  \smallsum\limits_{ b \in \mathbb{U} }
  \psi_{1,1}\big(
     e^{ ( s - r )A } \, u,
     S_{ r, s } \, v
  \big)
  \big(
    e^{ ( s - r )A } \, B^b( w )
    ,
    S_{ r, s } \, R_r \, B^b( w )
  \big)
  \Big\|_V
\\ & \leq
\nonumber
  \smallsum\limits_{ b \in \mathbb{U} }
  \left|
  \big[
    \psi_{1,1}\big(
      e^{ ( t - r )A } \, u
      ,
      S_{ r, s } \, v
    \big)
    -
    \psi_{1,1}\big(
      e^{ ( s - r )A } \, u
      ,
      S_{ r, s } \, v
    \big)
  \big]
  \big(
    e^{ ( t - r )A } \, B^b( w ),
    S_{ r, s } \, R_r \, B^b( w )
  \big)
  \right|
\\ & \quad +
  \smallsum\limits_{ b \in \mathbb{U} }
  \left|
    \psi_{1,1}\big(
      e^{ ( s - r )A } \, u,
      S_{ r, s } \, v
    \big)
    \big(
      e^{ ( s - r )A } 
      \,
      (
        e^{ ( t - s ) A } - \operatorname{Id}_H
      ) 
      \,
      B^b( w )
      , 
      S_{ r, s } \, 
      R_r \, 
      B^b( w )
    \big)
  \right|
\\ & \leq
\nonumber
  \frac{
    \psiC
    \,
    | \groupC_0 |^\power
    \max\{
      1, \| u \|^\power_H, \| v \|^\power_H
    \}
    \,
    | \groupC_\rho |^2
    \,
    ( t - s )^\rho
    \,
    \| u \|_H 
    \,
    | \groupC_{ \nicefrac{ \vartheta }{ 2 } } |^2
    \,
    \| B \|^2_{ 
      \operatorname{Lip}^0( H, HS( U, H_{ -\nicefrac{ \vartheta }{ 2 } } ) ) 
    }
    \, 
    g_2( w )
  }{
    ( s - r )^{ 
      ( \rho + \nicefrac{ \vartheta }{ 2 } ) 
    }
    \,
    ( t - r )^{ 
      \nicefrac{ \vartheta }{ 2 } 
    }
  }
\\ & \quad +
\nonumber
  \frac{
    \psiC 
    \,
    | \groupC_0 |^\power
    \max\{
      1, \| u \|^\power_H, \| v \|^\power_H
    \}
    \,
    \groupC_{ \rho + \nicefrac{ \vartheta }{ 2 } } 
    \,
    \groupC_\rho
    \,
    ( t - s )^\rho
    \,
    \groupC_{ \nicefrac{ \vartheta }{ 2 } } 
    \,
    \| B \|^2_{ 
      \operatorname{Lip}^0( H, HS( U, H_{ - \vartheta / 2 } ) ) 
    }
    \,
    g_2( w )
  }{
    ( s - r )^{ ( \rho + \vartheta ) }
  }
  .
\end{align}
Inequalities~\eqref{eq:prop_temporal_reg_termB1}--\eqref{eq:prop_temporal_reg_termB3} 
imply that 
for all 
$ t \in ( 0, T ] $, 
$ s \in ( 0, t ) $, 
$ r \in [ 0, s ) $, 
$ u,v, w \in H $ 
it holds that 
\begin{equation}
\begin{split}
&
  \| 
    \tilde{B}_{r,s,t}( u, v, w ) 
  \|_V
\leq
  \psiC 
  \,
  | \groupC_0 |^\power
  \max\!\big\{
    1, \| u \|^{ \power + 1 }_H, \| v \|^{ \power + 1 }_H
  \big\}
  \left(
    t - s 
  \right)^{ \rho }
  \| B \|^2_{ 
    \operatorname{Lip}^0( H, HS( U, H_{ - \vartheta / 2 } ) ) 
  }
  \,
  g_2( w )
\\ & 
  \cdot
  \left[
    \tfrac{
      \frac{ 1 }{ 2 } 
      \,
      | C_{ \rho } |^2
      \,
      | \groupC_{ \nicefrac{ \vartheta }{ 2 } } |^2
    }{
      ( s - r )^{ \rho }
      \,
      ( t - r )^{ \vartheta }
    }
    +
    \tfrac{
      \frac{ 1 }{ 2 } 
      \,
      2 
      \,
      \groupC_{ \nicefrac{ \vartheta }{ 2 } } 
      \, 
      \groupC_{ \rho + \nicefrac{ \vartheta }{ 2 } } 
      \,
      \groupC_\rho
    }{
      ( s - r )^{ ( \rho + \vartheta ) }
    }
    +
    \tfrac{
      \frac{ 1 }{ 2 } 
      \,
      | C_{ \rho } |^2
      \,
      | \groupC_{ \nicefrac{ \vartheta }{ 2 } } |^2
    }{
      ( s - r )^{ ( \rho + \vartheta ) }
    }
    +
    \tfrac{
      | C_{ \rho } |^2
      \,
      | \groupC_{ \nicefrac{ \vartheta }{ 2 } } |^2
    }{
      ( s - r )^{ 
        ( 
          \rho + 
          \nicefrac{ \vartheta }{ 2 }
        ) 
      }
      \,
      ( t - r )^{
        \nicefrac{ \vartheta }{ 2 }
      }
    }
    +
    \tfrac{
      \groupC_{ \rho + \nicefrac{ \vartheta }{ 2 } } 
      \, 
      \groupC_\rho
      \,
      \groupC_{ \nicefrac{ \vartheta }{ 2 } } 
    }{
      ( s - r )^{ ( \rho + \vartheta ) }
    }
  \right]
\\ & \leq
  \psiC 
  \,
  | \groupC_0 |^\power
  \max\!\big\{
    1, \| u \|^{ \power + 1 }_H, \| v \|^{ \power + 1 }_H
  \big\}
  \left(
    t - s 
  \right)^{ \rho }
  \| B \|^2_{ 
    \operatorname{Lip}^0( H, HS( U, H_{ - \vartheta / 2 } ) ) 
  }
  \,
  g_2( w )
\\ & \cdot
  \frac{
    2 
    \, 
    \groupC_\rho 
    \,
    \groupC_{ 
      \nicefrac{ \vartheta }{ 2 } 
    }
    \,
    \big(
      \groupC_\rho 
      \, 
      \groupC_{ \nicefrac{ \vartheta }{ 2 } } 
      + 
      \groupC_{ 
        \rho + \nicefrac{ \vartheta }{ 2 } 
      }
    \big)
  }{ 
    ( s - r )^{ ( \rho + \vartheta ) } 
  }
  .
\end{split}
\end{equation}
This and \eqref{eq:optimal_moment_1} prove that for all 
$ t \in (0,T] $, $ s \in (0,t) $
it holds that 
\begin{equation}\label{eq:B_summand_inner}
\begin{split}
&
  \int^{ s }_0
  \big\|
  \ES\big[
    \tilde{B}_{r,s,t}(
      \bar{Y}_r, 
      Y_r, 
      Y_{ \floor{ r }{ h } }
    )
  \big]
  \big\|_V
  \,
  dr
\\ & \leq
    \psiC 
    \,
    | \groupC_0 |^\power
    \,
    K_{ \power + 3 }
    \,
  ( t - s )^\rho 
  \, 
  \| B \|^2_{ 
    \operatorname{Lip}^0( H, HS( U, H_{ -\nicefrac{ \vartheta }{ 2 } } ) ) 
  }
  \,
  \frac{
    2 
    \, 
    \groupC_\rho 
    \,
    \groupC_{ \nicefrac{ \vartheta }{ 2 } }
    \,
    \big(
      \groupC_\rho \, \groupC_{ \nicefrac{ \vartheta }{ 2 } } 
      + 
      \groupC_{ \rho + \nicefrac{ \vartheta }{ 2 } }
    \big)
    \,
    s^{ ( 1 - \vartheta - \rho ) }
  }{ 
    ( 1 - \vartheta - \rho ) 
  }
  .
\end{split}
\end{equation}
Combining \eqref{eq:mild_ito_inner_1st}
with the estimates 
\eqref{eq:1st_summand_inner}, \eqref{eq:F_summand_inner}, 
and \eqref{eq:B_summand_inner} 
yields that 
for all $ t \in (0,T] $, $ s \in (0,t) $ it holds that 
\begin{equation}
\label{eq:prop_temporal_reg_term1A}
\begin{split}
&
  \left\|
  \E\left[
    \psi(
      e^{ ( t - s )A } \, \bar{Y}_{ s } ,
      Y_{ s }
    )
    -
    \psi(
      \bar{Y}_{ s } ,
      Y_{ s }
    )
  \right]
  \right\|_V
\leq
  \psiC
  \,
  | \groupC_0 |^{ \power }
  \,
  K_{ \power + 1 } \,
  \tfrac{
    |\groupC_\rho|^2
  }{
    s^{ \rho }
  }
  \left( t - s \right)^{ \rho }
\\ & 
  +
  \frac{
    \psiC
    \,
    | \groupC_0 |^q
    \,
    \groupC_\rho
    \left(
      2 \, \groupC_\rho \groupC_\vartheta 
      +
      \groupC_{ \rho + \vartheta }
    \right)
  }{
    ( 1 - \vartheta - \rho )
  }
  \,
  \| F \|_{ \operatorname{Lip}^0( H, H_{ -\vartheta } ) } 
  \,
  K_{ \power + 2 }
  \,
  ( t - s )^\rho 
  \,
  s^{ ( 1 - \vartheta - \rho ) }
\\ & +
  \frac{
    \psiC 
    \,
    | \groupC_0 |^\power
    \,
    2 
    \, 
    \groupC_\rho 
    \,
    \groupC_{ \nicefrac{ \vartheta }{ 2 } }
    \,
    \big(
      \groupC_\rho \, \groupC_{ \nicefrac{ \vartheta }{ 2 } } 
      + 
      \groupC_{ \rho + \nicefrac{ \vartheta }{ 2 } }
    \big)
  }{ 
    ( 1 - \vartheta - \rho ) 
  }
  \, 
  \| B \|^2_{ 
    \operatorname{Lip}^0( H, HS( U, H_{ -\nicefrac{ \vartheta }{ 2 } } ) ) 
  }
  \,
    K_{ \power + 3 }
    \,
  ( t - s )^\rho 
    \,
    s^{ ( 1 - \vartheta - \rho ) }
\\ & \leq
  \psiC
  \,
  | \groupC_0 |^q
  \,
  \varsigma_{ F, B }
  \,
  K_{ q + 3 }
  \left( t - s \right)^{ \rho }
  \left[ 
    \tfrac{ 
      | C_{ \rho } |^2 
    }{
      s^{ \rho }
    }
    +
    \tfrac{
      \groupC_\rho
      \,
      \left(
        2 \, \groupC_\rho \groupC_\vartheta 
        +
        \groupC_{ \rho + \vartheta }
        +
        2 \, 
        \groupC_\rho \, | \groupC_{ \vartheta / 2 } |^2
        + 
        2 \,
        \groupC_{ \rho + \vartheta / 2 }
        \,
        \groupC_{ \vartheta / 2 }
      \right)
      \,
      s^{
        ( 1 - \vartheta - \rho )
      }
    }{
      ( 1 - \vartheta - \rho )
    }
  \right]
\\ & \leq
  \psiC
  \,
  | \groupC_{ \rho } |^2 
  \,
  | \groupC_0 |^q
  \,
  \varsigma_{ F, B }
  \,
  K_{ q + 3 }
  \left( t - s \right)^{ \rho }
  \left[ 
    \tfrac{ 
      1
    }{
      s^{ \rho }
    }
    +
    \tfrac{
      \left(
        2 \, 
        \groupC_\vartheta 
        +
        \groupC_{ \rho + \vartheta }
        +
        2 \, 
        | \groupC_{ \vartheta / 2 } |^2
        + 
        2 \,
        \groupC_{ \rho + \vartheta / 2 }
        \,
        \groupC_{ \vartheta / 2 }
      \right)
      \,
      s^{
        ( 1 - \vartheta - \rho )
      }
    }{
      ( 1 - \vartheta - \rho )
    }
  \right]
  .
\end{split}
\end{equation}
In addition, we note that 
for all $ (s,t) \in \angle $ 
it holds that 
\begin{equation}
\label{eq:initial_value_t=0_1st}
\begin{split}
&
  \left\|
  \E\left[
    \psi(
      e^{ (t - s) A } \, \bar{Y}_s ,
      Y_s
    )
    -
    \psi(
      \bar{Y}_s ,
      Y_s
    )
  \right]
  \right\|_V
\\ & \leq
  \eta
  \,
  \E\left[ 
    \max\!\left\{ 
      1 ,
      \| e^{ ( t - s ) A } \bar{Y}_s \|^{ \power }_H
      ,
      \| \bar{Y}_s \|^{ \power }_H
      ,
      \| Y_s \|^{ \power }_H
    \right\}
    \|
      e^{ ( t - s ) A } - \operatorname{Id}_H 
    \|_{ L(H) }
    \,
    \| 
      \bar{Y}_s
    \|_H
  \right]
\\ & \leq
  \eta
  \left|
  C_0
  \right|^{ ( \power + 1 ) }
  \E\left[ 
    \max\!\left\{ 
      1 ,
      \| \bar{Y}_s \|^{ \power }_H
      ,
      \| Y_s \|^{ \power }_H
    \right\}
    \| 
      \bar{Y}_s 
    \|_H
  \right]
\\ & \leq
  \eta
  \left|
  C_0
  \right|^{ ( \power + 1 ) }
  \E\left[ 
    \max\!\left\{ 
      1 ,
      \| \bar{Y}_s \|^{ \power + 1 }_H
      ,
      \| Y_s \|^{ \power + 1 }_H
    \right\}
  \right]
\leq
  \psiC \, 
  | \groupC_0 |^{ ( \power + 1 ) }
  K_{ \power + 1 }
  .
\end{split}
\end{equation}
Combining this with 
\eqref{eq:prop_temporal_reg_term1A}
proves that for all
$ ( s, t ) \in \angle $
it holds that
\begin{equation}
\begin{split}
&
  \left\|
  \E\left[
    \psi(
      e^{ (t - s) A } \, \bar{Y}_s ,
      Y_s
    )
    -
    \psi(
      \bar{Y}_s ,
      Y_s
    )
  \right]
  \right\|_V
\\ & \leq
  \psiC
  \,
  | \groupC_{ \rho } |^2 
  \,
  | \groupC_0 |^{ ( \power + 1 ) }
  \,
  \varsigma_{ F, B }
  \,
  K_{ q + 3 }
  \left[
    \min\!\left\{ 
      1 ,
      \tfrac{ 
        \left( t - s \right)^{ \rho }
      }{
        s^{ \rho }
      }
    \right\}
    +
    \tfrac{
      \left( t - s \right)^{ \rho }
      \,
      \left(
        2 \, 
        \groupC_\vartheta 
        +
        \groupC_{ \rho + \vartheta }
        +
        2 \, 
        | \groupC_{ \vartheta / 2 } |^2
        + 
        2 \,
        \groupC_{ \rho + \vartheta / 2 }
        \,
        \groupC_{ \vartheta / 2 }
      \right)
      \,
      s^{
        ( 1 - \vartheta - \rho )
      }
    }{
      ( 1 - \vartheta - \rho )
    }
  \right]
\\ & =
  \psiC
  \,
  | \groupC_{ \rho } |^2 
  \,
  | \groupC_0 |^{ ( \power + 1 ) }
  \,
  \varsigma_{ F, B }
  \,
  K_{ q + 3 }
\\ & \quad \cdot
  \left[
    \mathbbm{1}_{
      [ \frac{ t }{ 2 } , T ]
    }( s )
    \cdot
      \tfrac{ 
        \left( t - s \right)^{ \rho }
      }{
        s^{ \rho }
      }
    +
    \mathbbm{1}_{
      [ 0, \frac{ t }{ 2 } )
    }( s )
    \cdot
    \tfrac{
      \left( t - s \right)^{ \rho }
    }{
      \left( t - s \right)^{ \rho }
    }
    +
    \tfrac{
        \left( t - s \right)^{ \rho }
        \,
      \left(
        2 \, 
        \groupC_\vartheta 
        +
        \groupC_{ \rho + \vartheta }
        +
        2 \, 
        | \groupC_{ \vartheta / 2 } |^2
        + 
        2 \,
        \groupC_{ \rho + \vartheta / 2 }
        \,
        \groupC_{ \vartheta / 2 }
      \right)
      \,
      s^{
        ( 1 - \vartheta - \rho )
      }
    }{
      ( 1 - \vartheta - \rho )
    }
  \right]
\\ & \leq
  \psiC
  \,
  | \groupC_{ \rho } |^2 
  \,
  | \groupC_0 |^{ ( \power + 1 ) }
  \,
  \varsigma_{ F, B }
  \,
  K_{ q + 3 }
  \left[
      \tfrac{ 
        \left( t - s \right)^{ \rho }
      }{
        \left( t / 2 \right)^{ \rho }
      }
    +
    \tfrac{
        \left( t - s \right)^{ \rho }
        \,
      \left(
        2 \, 
        \groupC_\vartheta 
        +
        \groupC_{ \rho + \vartheta }
        +
        2 \, 
        | \groupC_{ \vartheta / 2 } |^2
        + 
        2 \,
        \groupC_{ \rho + \vartheta / 2 }
        \,
        \groupC_{ \vartheta / 2 }
      \right)
      \,
      s^{
        ( 1 - \vartheta - \rho )
      }
    }{
      ( 1 - \vartheta - \rho )
    }
  \right]
  .
\end{split}
\end{equation}
Combining this, \eqref{eq:outer_summands}, 
and \eqref{eq:mild_ito_outer_1st} establishes that
for all $ ( s, t ) \in \angle $
it holds that
\begin{equation}
\begin{split}
\label{eq:mild_ito_2nd_case_1st}
&
  \left\|
  \E\left[
    \psi\big(
    \bar{Y}_{ t },
    Y_{ s }
    \big)
    -
    \psi\big(
    \bar{Y}_{ s },
    Y_{ s }
    \big)
  \right]
  \right\|_V
\leq
  \psiC \,
  | \groupC_0 |^{ ( \power + 1 ) } \, 
  | \groupC_\rho |^2
  \,
  \varsigma_{ F, B } \,
  K_{ \power + 3 } 
  \,
  ( t - s )^\rho
\\ & \cdot
  \bigg[
    \big|
      \tfrac{
        2
      }{
        t
      }
    \big|^\rho
    +
    \tfrac{
      \left(
        2 
        \, 
        \groupC_\vartheta 
        +
        \groupC_{ \rho + \vartheta }
        +
        2 
        \, 
        | \groupC_{ \vartheta / 2 } |^2
        + 
        2 
        \,
        \groupC_{ \rho + \vartheta / 2 }
        \,
        \groupC_{ \vartheta / 2 }
      \right)
      \,
      s^{ ( 1 - \vartheta - \rho ) }
      +
      \left( 
        \groupC_\vartheta 
        +
        \frac{ 1 }{ 2 }
        | \groupC_{ \nicefrac{ \vartheta }{ 2 } } |^2 
      \right)
      \,
      \left| 
        t - s 
      \right|^{ 
        ( 1 - \vartheta - \rho ) 
      }
    }{ 
      ( 1 - \vartheta - \rho ) 
    }
  \bigg]
  .
\end{split}
\end{equation}
The proof of Proposition~\ref{prop:weak_temporal_regularity_1st} is thus completed.
\end{proof}

\subsection{Analysis of the weak distance between Euler-type approximations and their semilinear integrated counterparts}
\label{sec:weak_distance_semilinear}

\begin{lemma}[Analysis of the analytically weak distance between Euler-type approximations and their semilinear integrated counterparts]
\label{lem:euler_integrated_strong}
Assume the setting in Section~\ref{sec:setting_weak_temporal_regularity} 
and let 
$ \rho \in [ 0, 1 ) $, 
$ 
  \varrho \in 
  \big[ 
    0, 
    1 - [ \frac{ 1 + \vartheta }{ 2 } - \rho ]^+ 
  \big) 
$, 
$ t \in (0,T] $.
Then 
\begin{align}
&
  \| Y_t - \bar{ Y }_t \|_{ \lpn{ p }{ \P }{ H_{ -\rho } } }
\leq
  \left| 
    K_p 
  \right|^{ 
    \frac{ 1 }{ p } 
  } 
  h^\varrho
\\ & \cdot
\nonumber
  \left[
    \tfrac{
      \groupC_{ -\rho, \varrho } 
    }{ 
      t^{ ( \varrho - \rho )^+ } 
    }
    +
    \tfrac{
      \groupC_{ \vartheta, -\rho, \varrho } 
      \, 
      t^{ 
        1 - ( \vartheta + \varrho - \rho )^+ 
      }
      \|
        F
      \|_{
        \operatorname{Lip}^0( H, H_{ -\vartheta } )
      } 
    }{ 
      ( 1 - ( \vartheta + \varrho - \rho )^+ ) 
    }
    +
    \tfrac{
      \groupC_{ \vartheta / 2 , - \rho, \varrho } 
      \,
      \sqrt{ 
        p \, (p - 1) 
        \, 
        t^{ 
          1 - ( \vartheta + 2 \varrho - 2 \rho )^+ 
        } 
      }
      \,
      \|
        B
      \|_{
        \operatorname{Lip}^0(
          H, HS( U, H_{ - \vartheta / 2 } )
        )
      } 
  }{ \sqrt{ 2 - 2( \vartheta + 2\varrho - 2\rho )^+ } }
  \right]
  .
\end{align}
\end{lemma}
\begin{proof}
First of all, we observe that 
\begin{equation}\label{eq:strong_integral_decomposition}
\begin{split}
&
  \left\|
    Y_t - \bar{Y}_t
  \right\|_{
    \lpn{p}{\P}{H_{ -\rho }}
  }
\leq
  \left\|
    \int^t_0
    \left(
      S_{ s, t } \, R_s - e^{ ( t - s ) A }
    \right)
    F( Y_{ \floor{ s }{ h } } )
    \, ds
  \right\|_{
    \lpn{p}{\P}{H_{ -\rho }}
  }
\\ & +
  \left\|
    \int^t_0
    \left(
      S_{ s, t } \, R_s 
      - 
      e^{ ( t - s ) A }
    \right)
    B( Y_{ \floor{ s }{ h } } )
    \, dW_s
  \right\|_{
    \lpn{p}{\P}{ H_{ - \rho } }
  }
+
  \left\|
    \left( 
      S_{ 0, t } - e^{ t A } 
    \right)
    Y_0
  \right\|_{
    \lpn{p}{\P}{H_{ -\rho }}
  }
  .
\end{split}
\end{equation}
Next we note that 
\begin{equation}\label{eq:analytically_weak_initial}
\begin{split}
&
  \left\|
    ( 
      S_{ 0, t } - e^{ t A } 
    )
    Y_0
  \right\|_{
    \lpn{p}{\P}{ H_{ - \rho } }
  }
\leq
  \| S_{ 0, t } - e^{ t A } \|_{ L( H, H_{-\rho} ) } \,
  \| Y_0 \|_{ \lpn{ p }{ \P }{ H } }
\leq
  \frac{
    \groupC_{ -\rho, \varrho } 
  }{ 
    t^{ ( \varrho - \rho )^+ } 
  }
  \, \| Y_0 \|_{ \lpn{ p }{ \P }{ H } } \, h^\varrho
\\ & \leq
  \frac{
  \groupC_{ -\rho, \varrho } 
  }{ t^{ ( \varrho - \rho )^+ } } \,
  | K_p |^{ \frac{ 1 }{ p } } \, h^\varrho
\end{split}
\end{equation}
and 
\begin{equation}
\label{eq:strong_integral_decomposition_2nd}
\begin{split}
&
  \left\|
    \int^t_0
    \left(
      S_{ s, t } \, R_s - e^{ ( t - s ) A }
    \right)
    F( Y_{ \floor{ s }{ h } } )
    \, ds
  \right\|_{
    \lpn{p}{\P}{ H_{ - \rho } }
  }
\leq
  \int^t_0
  \left\|
    (
      S_{ s, t } \, R_s - e^{ ( t - s ) A }
    )
    F( Y_{ \floor{ s }{ h } } )
  \right\|_{
    \lpn{p}{\P}{H_{ -\rho }}
  }
  ds
\\ & \leq
  \int^t_0
  \frac{
    \groupC_{ \vartheta, -\rho, \varrho } 
    \, h^\varrho
    \,
    \|
      F( 
        Y_{ \floor{ s }{ h } } 
      )
    \|_{
      \lpn{p}{\P}{H_{ -\vartheta }}
    }
  }{ 
    ( t - s )^{ ( \vartheta + \varrho - \rho )^+ } 
  }
  \, ds
\leq
  \frac{
    \groupC_{ \vartheta, -\rho, \varrho } 
    \, 
    t^{ 1 - ( \vartheta + \varrho - \rho )^+ }
  }{ 
    ( 
      1 - ( \vartheta + \varrho - \rho )^+ 
    ) 
  }
  \left\|
    F
  \right\|_{
    \operatorname{Lip}^0( H, H_{ -\vartheta } )
  }
  | K_p |^{ \frac{ 1 }{ p } } \,
  h^\varrho
  .
\end{split}
\end{equation}
Moreover, observe that the Burkholder-Davis-Gundy type inequality 
in Lemma~7.7 in Da Prato \& Zabczyk~\cite{dz92} proves that 
\begin{equation}
\label{eq:strong_integral_decomposition_3rd}
\begin{split}
  &\left\|
  \int^t_0
  (
  S_{ s, t } \, R_s - e^{ ( t - s ) A }
  )
  B\big( Y_{ \floor{ s }{ h } } \big)
  \, dW_s
  \right\|_{
  \lpn{p}{\P}{H_{ -\rho }}
  }
\\ & \leq
  \left[
    \frac{ p \, ( p - 1 ) }{ 2 }
    \int^t_0
    \left\|
      \left(
        S_{ s, t } \, R_s - e^{ ( t - s ) A }
      \right)
      B( Y_{ \floor{ s }{ h } } )
    \right\|_{
      \lpn{p}{\P}{
        HS( U, H_{ -\rho } )
      }
    }^2
    ds
  \right]^{ 1 / 2 }
\\ & \leq
  \left[
    \frac{ p \, ( p - 1 ) }{ 2 }
    \int^t_0
    \frac{
    |\groupC_{ \nicefrac{ \vartheta }{ 2 }, -\rho, \varrho }|^2 \, h^{2\varrho} \,
    \|
      B( Y_{ \floor{ s }{ h } } )
    \|_{
      \lpn{p}{\P}{
        HS( U, H_{ -\nicefrac{ \vartheta }{ 2 } } )
      }
    }^2
    }{ ( t - s )^{ ( \vartheta + 2\varrho - 2\rho )^+ } }
    \, ds
  \right]^{ 1 / 2 }
\\ & \leq
  \frac{
  \groupC_{ \nicefrac{ \vartheta }{ 2 }, -\rho, \varrho } 
  \sqrt{ p \, (p-1) \, t^{ 1 - ( \vartheta + 2\varrho - 2\rho )^+ } }
  }{ \sqrt{ 2 - 2( \vartheta + 2\varrho - 2\rho )^+ } }
  \left\|
  B
  \right\|_{
  \operatorname{Lip}^0(
  H, HS( U, H_{ -\nicefrac{ \vartheta }{ 2 } } )
  )
  }
  | K_p |^{ \frac{ 1 }{ p } } \, h^\varrho
  .
\end{split}
\end{equation}
Combining \eqref{eq:strong_integral_decomposition}--\eqref{eq:strong_integral_decomposition_3rd} 
completes the proof of Lemma~\ref{lem:euler_integrated_strong}.
\end{proof}

\begin{proposition}[Weak distance between Euler-type approximations and 
their semilinear integrated counterparts]
\label{prop:weak_temporal_regularity_2nd}
Assume the setting in Section~\ref{sec:setting_weak_temporal_regularity} and let 
$ \psiC \in [ 0, \infty ) $, 
$ \power \in [ 0, \infty ) \cap ( -\infty, p - 3 ] $, 
$ \rho \in [ 0, 1 - \vartheta ) $, 
$ \psi = ( \psi(x,y) )_{ x, y \in H } \in C^2( H \times H, V ) $ 
satisfy that for all 
$ x, y_1, y_2 \in H $, 
$ i, j \in \N_0 $ 
with 
$ i + j \leq 2 $
it holds that 
\begin{equation*}
  \big\|
    \big(
     \tfrac{ \partial^{ ( i + j ) } }{ \partial x^i \partial y^j  }
     \psi
    \big)
    ( x, y_1 )
  -
    \big(
     \tfrac{ \partial^{ ( i + j ) } }{ \partial x^i \partial y^j  }
     \psi
    \big)
    ( x, y_2 )
  \big\|_{ L^{ ( i + j ) }( H, V ) }
\leq
  \psiC
  \,
  \max\{ 1, \| x \|^\power_H, \| y_1 \|^\power_H, \| y_2 \|^\power_H \}
  \,
  \| y_1 - y_2 \|_H
  .
\end{equation*}
Then for all $ t \in ( 0, T ] $ it holds that 
$
  \ES\big[
    \| 
      \psi( \bar{Y}_t, Y_t )
      -
      \psi( \bar{Y}_t, \bar{Y}_t ) 
    \|_V
  \big]
  < \infty
$ 
and 
\begin{equation}
\label{eq:weak_distance}
\begin{split}
&
  \left\|
  \E\left[
    \psi\big(
    \bar{Y}_{ t },
    Y_{ t }
    \big)
    -
    \psi\big(
    \bar{Y}_{ t },
    \bar{Y}_{ t }
    \big)
  \right]\right\|_V
\leq
  \psiC
  \,
  | \groupC_0 |^\power
  \,
  \varsigma_{ F, B } \,
  K_{ \power + 3 } \, h^\rho
\\ & \cdot
  \Bigg[
    \tfrac{
      \groupC_{ 0, \rho }
    }{
      t^\rho
    }
    +
    \tfrac{ t^{ ( 1 - \vartheta - \rho ) } }{ ( 1 - \vartheta - \rho ) }
  \bigg(
      \groupC_{ \vartheta, 0, \rho } 
    + 
      2 \, \groupC_{ \nicefrac{ \vartheta }{ 2 } } \, 
      \groupC_{ \nicefrac{ \vartheta }{ 2 }, 0, \rho }
    +
      2 \, ( | \groupC_{ \nicefrac{ \vartheta }{ 2 } } |^2 + \groupC_\vartheta ) \, \groupC_{ 0, \rho }
    +
      2 \, ( | \groupC_{ \nicefrac{ \vartheta }{ 2 } } |^2 + \groupC_\vartheta ) 
      \, \groupC_\rho
\\ & \cdot
  \Big[
    \groupC_{ -\rho, \rho }
    +
    \tfrac{
    \groupC_{ \vartheta, -\rho, \rho } \, t^{ ( 1 - \vartheta ) } \,
    \| F \|_{ \operatorname{Lip}^0( H, H_{ -\vartheta } ) }
    }{ ( 1 - \vartheta ) }
    +
    \tfrac{
    \groupC_{ \nicefrac{\vartheta}{2}, -\rho, \rho } \, 
    \sqrt{ (\power+3)(\power+2) \, t^{ ( 1 - \vartheta ) } }
    \| B \|_{ \operatorname{Lip}^0( H, HS( U, H_{ -\nicefrac{\vartheta}{2} } ) ) }
    }{
    \sqrt{ 2 - 2 \vartheta }
    }
  \Big]
  \bigg)
  \Bigg]
  .
\end{split}
\end{equation}
\end{proposition}
\begin{proof}
Throughout this proof 
let 
$
  ( g_r )_{ r \in [ 0, \infty ) }
  \subseteq
  C( H, \R )
$ 
be the functions
with the property that
for all $ r \in [0,\infty) $, $ x \in H $
it holds that
$
  g_r( x )
  =
  \max\{
  1, \| x \|^r_H
  \}
$
and let 
$
  \psi_{1,0} \colon H \times H \to L( H, V ) 
$, 
$
  \psi_{0,1} \colon H \times H \to L( H, V ) 
$, 
$
  \psi_{2,0} \colon H \times H \to L^{(2)}( H, V ) 
$, 
$
  \psi_{0,2} \colon H \times H \to L^{(2)}( H, V ) 
$, 
$
  \psi_{1,1} \colon H \times H \to L^{(2)}( H, V ) 
$
be the functions with the property that 
for all 
$ x, y, v_1, v_2 \in H $ 
it holds that 
$
  \psi_{1,0}( x, y ) \, v_1
  =
  \big(\tfrac{ \partial }{ \partial x } \psi( x, y )\big)( v_1 )
$
and
\begin{equation}
  \psi_{0,1}( x, y ) \, v_1
  =
  \big(\tfrac{ \partial }{ \partial y } \psi( x, y )\big)( v_1 )
  ,
  \qquad
  \psi_{2,0}( x, y )( v_1, v_2 )
  =
  \big(
    \tfrac{ \partial^2 }{ \partial x^2 } \psi( x, y )
  \big)( v_1, v_2 )
  ,
\end{equation}
\begin{equation}
  \psi_{0,2}( x, y )( v_1, v_2 )
  =
  \big(
    \tfrac{ \partial^2 }{ \partial y^2 } \psi( x, y )
  \big)( v_1, v_2 )
  ,
  \qquad
  \psi_{ 1, 1 }( x, y )( v_1, v_2 )
  =
  \big(
    \tfrac{ \partial }{ \partial y } 
    \tfrac{ \partial }{ \partial x } 
    \psi( x, y )
  \big)( v_1, v_2 )
  .
\end{equation} 
Next we observe that Lemma~\ref{lem:Kp_estimate}
and the assumption that $ q \leq p - 3 $ ensure that $ K_{ q + 1 } \leq K_{ q + 3 } < \infty $.
Combining this with the fact that
\begin{equation}
  \forall \, x, y_1, y_2 \in H \colon
  \| 
    \psi( x, y_1 ) 
    - 
    \psi( x, y_2 ) 
  \|_V
  \leq 
  2 \, \psiC
  \max\!\big\{ 
    1, \| x \|^{ q + 1 }_H , \| y_1 \|^{ q + 1 }_H , \| y_2 \|^{ q + 1 }_H 
  \big\}
\end{equation}
shows that for all $ t \in ( 0, T ] $ it holds that 
$
  \ES\big[
    \| 
      \psi( \bar{Y}_t, Y_t )
      -
      \psi( \bar{Y}_t, \bar{Y}_t ) 
    \|_V
  \big]
  < \infty
$.
It thus remains to prove \eqref{eq:weak_distance}.
To do so, we make use of a consequence of the mild It\^{o} formula in Corollary~\ref{cor:itoauto},
that is, we will apply Proposition~\ref{prop:mild_ito_maximal_ineq} above. 
For this let 
$ \tilde{ F }_{ s, t } \colon H \times H \times H \to V $, 
$ ( s, t ) \in \angle $, 
be the functions with the property that 
for all 
$ ( s, t ) \in \angle $, 
$ u, v, w \in H $ 
it holds that 
\begin{equation}
\begin{split}
  \tilde{F}_{ s, t }( u, v, w )
& =
  \big[
  \psi_{1,0}
  \big(
  e^{ ( t - s )A } \, u,
  S_{ s, t } \, v
  \big)
  -
  \psi_{1,0}
  \big(
  e^{ ( t - s )A } \, u,
  e^{ ( t - s )A } \, u
  \big)
  \big]
  \,
  e^{ ( t - s )A } \, F( w )
\\ & +
  \psi_{0,1}
  \big(
  e^{ ( t - s )A } \, u,
  S_{ s, t } \, v
  \big)
  \,
  S_{ s, t } \, R_s \, F( w )
  -
  \psi_{0,1}
  \big(
  e^{ ( t - s )A } \, u,
  e^{ ( t - s )A } \, u
  \big)
  \,
  e^{ ( t - s )A } \, F( w )
\end{split}
\end{equation}
and let 
$ \tilde{ B }_{ s, t } \colon H \times H \times H \to V $, 
$ ( s, t ) \in \angle $, 
be the functions with the property that 
for all 
$ ( s, t ) \in \angle $, 
$ u, v, w \in H $ 
it holds that 
\begin{equation}
\begin{split}
&
  \tilde{B}_{ s, t }( u, v, w )
\\ & =
  \tfrac{ 1 }{ 2 }
  \smallsum\limits_{ b \in \mathbb{ U 	} }
  \big[
  \psi_{ 2,0 }
  \big(
  e^{ ( t - s )A } \, u,
  S_{ s, t } \, v
  \big)
  -
  \psi_{ 2,0 }
  \big(
  e^{ ( t - s )A } \, u,
  e^{ ( t - s )A } \, u
  \big)
  \big]
  \big(
    e^{ ( t - s )A } \, B^b( w ),
    e^{ ( t - s )A } \, B^b( w )
  \big)
\\ & +
  \tfrac{ 1 }{ 2 }
  \smallsum\limits_{ b \in \mathbb{ U 	} }
  \psi_{ 0,2 }
  \big(
  e^{ ( t - s )A } \, u,
  S_{ s, t } \, v
  \big)
  \big(
    S_{ s, t } \, R_s \, B^b( w ),
    S_{ s, t } \, R_s \, B^b( w )
  \big)
\\ & 
-
  \tfrac{ 1 }{ 2 }
  \smallsum\limits_{ b \in \mathbb{ U 	} }
  \psi_{ 0,2 }
  \big(
  e^{ ( t - s )A } \, u,
  e^{ ( t - s )A } \, u
  \big)
  \big(
    e^{ ( t - s )A } \, B^b( w ),
    e^{ ( t - s )A } \, B^b( w )
  \big)
\\ & +
  \smallsum\limits_{ b \in \mathbb{ U } }
  \psi_{ 1,1 }
  \big(
  e^{ ( t - s )A } \, u,
  S_{ s, t } \, v
  \big)
  \big(
    e^{ ( t - s )A } \, B^b( w ),
    S_{ s, t } \, R_s \, B^b( w )
  \big)
\\ & 
  -
  \smallsum\limits_{ b \in \mathbb{ U } }
  \psi_{ 1,1 }
  \big(
  e^{ ( t - s )A } \, u,
  e^{ ( t - s )A } \, u
  \big)
  \big(
    e^{ ( t - s )A } \, B^b( w ),
    e^{ ( t - s )A } \, B^b( w )
  \big)
  .
\end{split}
\end{equation}
An application of Proposition~\ref{prop:mild_ito_maximal_ineq}
then shows\footnote{with 
$ t_0 = 0 $, 
$ T = t $, 
$ \check{H} = H \times H $, 
$ p = q + 1 $, 
and
$ \varphi(x,y) = \psi(x,y) - \psi(x,x) $ 
for $ (x,y) \in \check{H} $ 
in the notation of Proposition~\ref{prop:mild_ito_maximal_ineq}} 
that for all $ t \in ( 0, T ] $ it holds that 
\begin{equation}
\label{eq:mild_ito_2nd}
\begin{split}
&
  \left\|
  \E\left[
    \psi\big(
      \bar{Y}_{ t } ,
      Y_{ t }
    \big)
    -
    \psi\big(
      \bar{Y}_{ t } ,
      \bar{Y}_{ t }
    \big)
  \right]\right\|_V
\leq
  \E\left[
  \|
    \psi\big(
    e^{ tA } \, Y_0,
    S_{ 0, t } \, Y_0
    \big)
    -
    \psi\big(
    e^{ tA } \, Y_0,
    e^{ tA } \, Y_0
    \big)
  \|_V
  \right]
\\ & +
  \int^{ t }_0
  \E\left[
    \big\|
      \tilde{F}_{ s, t }\big(
        \bar{Y}_s, Y_s, Y_{ \floor{ s }{ h } }
      \big)
    \big\|_V
  \right]
  +
  \E\left[
    \big\|
      \tilde{B}_{ s, t }\big(
        \bar{Y}_s, Y_s, Y_{ \floor{ s }{ h } }
      \big)
    \big\|_V
  \right]
  ds
  .
\end{split}
\end{equation}
In the following we establish suitable estimates for the two summands on the right hand side of \eqref{eq:mild_ito_2nd}. 
We begin with the first summand on the right hand side of \eqref{eq:mild_ito_2nd}. 
Observe that for all $ t \in ( 0, T ] $ it holds that 
\begin{equation}
\begin{split}
&
  \left\|
    \psi\big(
    e^{ tA } \, Y_0,
    S_{ 0, t } \, Y_0
    \big)
    -
    \psi\big(
    e^{ tA } \, Y_0,
    e^{ tA } \, Y_0
    \big)
  \right\|_V
\\ & \leq
  \psiC
    \max\{ 1, \| S_{ 0, t } \, Y_0 \|^\power_H, \| e^{ tA } \, Y_0 \|^\power_H \}
    \| ( S_{ 0, t } - e^{ tA } ) Y_0 \|_H
\\ & \leq
  \psiC \,
  | \groupC_0 |^\power \,
  g_q( Y_0 ) \,
  \| S_{ 0, t } - e^{ tA } \|_{ L(H) } \,
  \| Y_0 \|_H
\leq
  \frac{
    \psiC
    \,
    |\groupC_0|^\power
    \,
    \groupC_{ 0, \rho }
    \,
    g_{ \power + 1 }( Y_0 ) \,
    h^\rho
  }{ t^\rho }
  .
\end{split}
\end{equation}
This and the fact that 
$
  \ES\big[ 
    g_{ \power + 1 }( Y_0 )
  \big]
  \leq
  K_{ \power + 1 }
$ 
imply that for all $ t \in ( 0, T ] $ it holds that 
\begin{equation}\label{eq:initial_value_part_2nd}
\begin{split}
&
  \E\left[
  \|
    \psi\big(
    e^{ tA } \, Y_0,
    S_{ 0, t } \, Y_0
    \big)
    -
    \psi\big(
    e^{ tA } \, Y_0,
    e^{ tA } \, Y_0
    \big)
  \|_V
  \right]
\leq
  \frac{
    \psiC
    \,
    |\groupC_0|^\power
    \,
    \groupC_{ 0, \rho }
    \,
    K_{ \power + 1 } 
    \,
    h^\rho
  }{ t^\rho }
  .
\end{split}
\end{equation}
Now we will estimate the second summand on the right hand side of~\eqref{eq:mild_ito_2nd}. 
Observe that for all 
$ ( s, t ) \in \angle $, 
$ u, v, w \in H $ 
it holds that 
\begin{equation}
\label{eq:tilde_F_1st_decompose}
\begin{split}
&
  \left\|
  \big[
    \psi_{1,0}\big(
      e^{ ( t - s )A } \, u,
      S_{ s, t } \, v
    \big)
    -
    \psi_{1,0}\big(
      e^{ ( t - s )A } \, u,
      e^{ ( t - s )A } \, u
    \big)
  \big]
  \, e^{ ( t - s ) A } \, F( w )
  \right\|_V
\\ & \leq
  \psiC
  \max\!\big\{
    1,
    \| S_{ s, t } \, v \|^\power_H,
    \| e^{ ( t - s )A } \, u \|^\power_H
  \big\}
  \|
  S_{ s, t } \, v
  -
  e^{ ( t - s )A } \, u
  \|_H
  \,
  \|
  e^{ ( t - s )A } \, F( w )
  \|_H
\\ & \leq
  \frac{
    \psiC
    \,
    | \groupC_0 |^\power
    \max\{
      1,
      \| u \|^\power_H,
      \| v \|^\power_H
    \}
    \big[
      \| ( S_{ s, t }  - e^{ ( t - s )A } ) v \|_H
      +
      \| e^{ ( t - s )A } ( v - u ) \|_H
    \big]
    \,
    \groupC_\vartheta 
    \,
    \| F( w ) \|_{ H_{ -\vartheta } }
  }{
    ( t - s )^\vartheta
  }
\\ & \leq
  \frac{
    \psiC
    \,
    | \groupC_0 |^\power
    \,
    \groupC_\vartheta
    \max\!\big\{
      1,
      \| u \|^\power_H,
      \| v \|^\power_H
    \big\}
    \,
    \| F \|_{ \operatorname{Lip}^0( H, H_{ -\vartheta } ) }
    \, g_1( w )
  }{
    ( t - s )^\vartheta
  }
\\ & \quad \cdot
  \Bigg[
    \| S_{ s, t } - e^{ ( t - s )A } \|_{ L( H ) }
    \,
    \| v \|_H
  +
  \frac{
    \groupC_\rho
    \| v - u \|_{ H_{ -\rho } }
  }{
    ( t - s )^\rho
  }
  \Bigg]
  .
\end{split}
\end{equation}
Moreover, we note that the assumption that 
$
  \forall x, y_1, y_2 \in H \colon 
$
\begin{equation}
  \|
    \psi
    ( x, y_1 )
  -
    \psi
    ( x, y_2 )
  \|_V
\leq
  \psiC
  \max\!\big\{ 
    1, \| x \|^\power_H, \| y_1 \|^\power_H, \| y_2 \|^\power_H 
  \big\}
  \,
  \| y_1 - y_2 \|_H
\end{equation}
implies that 
for all $ x, y \in H $ 
it holds that
$
  \| \psi_{ 0, 1 }( x, y ) \|_{ L( H, V ) }
  \leq
  \psiC \,
  \max\{
    1
    ,
    \| x \|^\power_H
    ,
    \| y \|^\power_H
  \}
$.
This, in turn, proves that for all 
$ ( s, t ) \in \angle $, 
$ u, v, w \in H $ 
it holds that 
\allowdisplaybreaks
\begin{align}
\label{eq:tilde_F_2nd_decompose}
&
  \left\|
  \psi_{0,1}
  \big(
  e^{ ( t - s )A } \, u,
  S_{ s, t } \, v
  \big)
  \,
  S_{ s, t } \, R_s \, F( w )
  -
  \psi_{0,1}
  \big(
  e^{ ( t - s )A } \, u,
  e^{ ( t - s )A } \, u
  \big)
  \,
  e^{ ( t - s )A } \, F( w )
  \right\|_V
\\ & \leq
\nonumber
  \left\|
  \psi_{0,1}
  \big(
  e^{ ( t - s )A } \, u,
  S_{ s, t } \, v
  \big)
  \big[
    S_{ s, t } \, R_s
    -
    e^{ ( t - s )A }
  \big]
  F( w )
  \right\|_V
\\ & +
\nonumber
  \left\|
  \big[
    \psi_{0,1}
    \big(
    e^{ ( t - s )A } \, u,
    S_{ s, t } \, v
    \big)
    -
    \psi_{0,1}
    \big(
    e^{ ( t - s )A } \, u,
    e^{ ( t - s )A } \, u
    \big)
  \big]
  \,
  e^{ ( t - s )A } \, F( w )
  \right\|_V
\\ & \leq
\nonumber
  \psiC \,
  \max\!\big\{
    1,
    \| S_{ s, t } \, v \|^\power_H,
    \| e^{ ( t - s )A } \, u \|^\power_H
  \big\} 
  \,
  \big\|
  \big[
    S_{ s, t } \, R_s
    -
    e^{ ( t - s )A }
  \big]
  F( w )
  \big\|_H
\\ & +
\nonumber
  \psiC \,
  \max\!\big\{
    1,
    \| S_{ s, t } \, v \|^\power_H,
    \| e^{ ( t - s )A } \, u \|^\power_H
  \big\}
  \,
  \|
    S_{ s, t } \, v
    -
    e^{ ( t - s )A } \, u
  \|_H 
  \,
  \| e^{ ( t - s )A } \, F(w) \|_H
\\ & \leq
\nonumber
  \psiC \,
  | \groupC_0 |^\power
  \max\{
    1,
    \| u \|^\power_H,
    \| v \|^\power_H
  \}
  \,
  \| F \|_{ \operatorname{Lip}^0( H, H_{ -\vartheta } ) }
  \, g_1( w )
\\ & \cdot
\nonumber
  \Bigg[
    \|
    S_{ s, t } \, R_s
    -
    e^{ ( t - s )A }
    \|_{ L( H_{ -\vartheta }, H ) }
    +
    \frac{
    \groupC_\vartheta
    }{
    ( t - s )^\vartheta
    }
    \left[
      \| S_{ s, t } - e^{ ( t - s )A } \|_{ L( H ) }
      \,
      \| v \|_H
      +
      \frac{
        \groupC_\rho
        \,
        \| v - u \|_{ H_{ -\rho } }
      }{
        ( t - s )^\rho
      }
    \right]
  \Bigg]
  .
\end{align}
Inequalities \eqref{eq:tilde_F_1st_decompose} and \eqref{eq:tilde_F_2nd_decompose} imply that for all 
$ ( s, t ) \in \angle $, 
$ u, v, w \in H $ 
it holds that 
\begin{equation}\label{eq:absolute_tilde_F_2nd}
\begin{split}
  \| 
    \tilde{F}_{ s, t }( u, v, w )
  \|_V
& \leq
  \psiC
  \,
  | \groupC_0 |^\power
  \max\!\big\{
    1,
    \| u \|^\power_H ,
    \| v \|^\power_H
  \big\}
  \,
  \| F \|_{ \operatorname{Lip}^0( H, H_{ -\vartheta } ) }
  \;
  g_1( w )
\\ & \quad \cdot
  \Bigg[
    \left[
    \frac{
      \groupC_{ \vartheta, 0, \rho } + 2 \, \groupC_\vartheta \, \groupC_{ 0, \rho } \, \| v \|_H
    }{ 
      ( t - s )^{ ( \rho + \vartheta ) } 
    }
    \right]
    h^\rho
    +
    \frac{
      2 \, \groupC_\vartheta \, \groupC_\rho \,
      \| v - u \|_{ H_{ -\rho } }
    }{
      { ( t - s )^{ ( \rho + \vartheta ) } }
    }
  \Bigg]
  .
\end{split}
\end{equation}
Next we observe that for all 
$ ( s, t ) \in \angle $, 
$ u, v, w \in H $ 
it holds that 
\begin{equation}
\label{eq:weak_distance_B_1}
\begin{split}
&
  \smallsum\limits_{ b \in \mathbb{ U 	} }
  \big\|
  \big[
  \psi_{2,0}
  \big(
  e^{ ( t - s )A } \, u,
  S_{ s, t } \, v
  \big)
  -
  \psi_{2,0}
  \big(
  e^{ ( t - s )A } \, u,
  e^{ ( t - s )A } \, u
  \big)
  \big]
  \big(
    e^{ ( t - s )A } \, B^b( w ),
    e^{ ( t - s )A } \, B^b( w )
  \big)
  \big\|_V
\\ & \leq
  \frac{
    \psiC
    \,
    | \groupC_0 |^\power
    \,
    |
      \groupC_{ \nicefrac{ \vartheta }{ 2 } }
    |^2
    \max\{
      1,
      \| u \|^\power_H,
      \| v \|^\power_H
    \}
    \,
    \| B \|^2_{ \operatorname{Lip}^0( H, HS( U, H_{ -\nicefrac{ \vartheta }{ 2 } } ) ) }
    \; 
    g_2( w )
  }{
    ( t - s )^\vartheta
  }
\\ & \quad \cdot
  \Bigg[
    \| S_{ s, t } - e^{ ( t - s )A } \|_{ L( H ) }
    \,
    \| v \|_H
    +
    \frac{
      \groupC_\rho \,
      \| v - u \|_{ H_{ -\rho } }
    }{
      ( t - s )^\rho
    }
  \Bigg]
  ,
\end{split}
\end{equation}
\allowdisplaybreaks
\begin{align}
&
  \smallsum\limits_{ b \in \mathbb{ U 	} }
  \big\|
  \psi_{0,2}
  \big(
  e^{ ( t - s )A } \, u,
  S_{ s, t } \, v
  \big)
  \big(
    S_{ s, t } \, R_s \, B^b( w ),
    S_{ s, t } \, R_s \, B^b( w )
  \big)
\\ & 
\nonumber
  -
  \psi_{0,2}
  \big(
  e^{ ( t - s )A } \, u,
  e^{ ( t - s )A } \, u
  \big)
  \big(
    e^{ ( t - s )A } \, B^b( w ),
    e^{ ( t - s )A } \, B^b( w )
  \big)
  \big\|_V
\\ & \leq
\nonumber
  \smallsum\limits_{ b \in \mathbb{ U 	} }
  \big\|
    \psi_{0,2}\big(
      e^{ ( t - s )A } \, u,
      S_{ s, t } \, v
    \big)
    \big(
      [S_{ s, t } \, R_s
      +
      e^{ ( t - s )A }]
      \, B^b( w ),
      [S_{ s, t } \, R_s
      -
      e^{ ( t - s )A }]
      \, B^b( w )
    \big)
  \big\|_V
\\ & +
\nonumber
  \smallsum\limits_{ b \in \mathbb{ U 	} }
  \big\|
  \big[
    \psi_{0,2}\big(
      e^{ ( t - s )A } \, u,
      S_{ s, t } \, v
    \big)
    -
    \psi_{0,2}\big(
      e^{ ( t - s )A } \, u,
      e^{ ( t - s )A } \, u
    \big)
  \big]
  \big(
    e^{ ( t - s )A } \, B^b( w ),
    e^{ ( t - s )A } \, B^b( w )
  \big)
  \big\|_V
\\ & \leq
\nonumber
  \psiC
  \,
  | \groupC_0 |^\power
    \max\{
      1,
      \| u \|^\power_H,
      \| v \|^\power_H
    \}
  \Big[
  \|
    S_{ s, t } \, v
    -
    e^{ ( t - s )A } \, u
  \|_H
  \,
  \|
    e^{ ( t - s )A } \, B( w )
  \|^2_{ HS( U, H ) }
\\ & +
\nonumber
  \|
    (S_{ s, t } \, R_s
    +
    e^{ ( t - s )A })
    \, B( w )
  \|_{ HS( U, H ) }
  \,
  \|
    (S_{ s, t } \, R_s
    -
    e^{ ( t - s )A })
    \, B( w )
  \|_{ HS( U, H ) }
  \Big]
\\ & \leq
\nonumber
  \psiC
  \,
  | \groupC_0 |^\power
  \max\{
    1,
    \| u \|^\power_H,
    \| v \|^\power_H
  \}
  \,
  \| B \|^2_{ \operatorname{Lip}^0( H, HS( U, H_{ -\nicefrac{ \vartheta }{ 2 } } ) ) } \,
  g_2( w )
\\ & \quad \cdot
\nonumber
  \Bigg[
  \|
    S_{ s, t } \, R_s
    +
    e^{ ( t - s )A }
  \|_{ L( H_{ -\nicefrac{ \vartheta }{ 2 } }, H ) }
  \,
  \|
    S_{ s, t } \, R_s
    -
    e^{ ( t - s )A }
  \|_{ L( H_{ -\nicefrac{ \vartheta }{ 2 } }, H ) }
\\ & \quad +
\nonumber
  \frac{
    | \groupC_{ \nicefrac{ \vartheta }{ 2 } } |^2
  }{
    ( t - s )^\vartheta
  }
    \left[
      \| S_{ s, t } - e^{ ( t - s )A } \|_{ L( H ) }
      \,
      \| v \|_H
      +
      \frac{
        \groupC_\rho
        \,
        \| v - u \|_{ H_{ -\rho } }
      }{
        ( t - s )^\rho
      }
    \right]
  \Bigg],
\end{align}
and 
\allowdisplaybreaks
\begin{align}
\label{eq:weak_distance_B_3}
&
  \smallsum\limits_{ b \in \mathbb{ U } }
  \big\|
  \psi_{1,1}
  \big(
  e^{ ( t - s )A } \, u,
  S_{ s, t } \, v
  \big)
  \big(
    e^{ ( t - s )A } \, B^b( w ),
    S_{ s, t } \, R_s \, B^b( w )
  \big)
\\ & 
\nonumber
  -
  \psi_{1,1}
  \big(
  e^{ ( t - s )A } \, u,
  e^{ ( t - s )A } \, u
  \big)
  \big(
    e^{ ( t - s )A } \, B^b( w ),
    e^{ ( t - s )A } \, B^b( w )
  \big)
  \big\|_V
\\ & \leq
\nonumber
  \smallsum\limits_{ b \in \mathbb{ U } }
  \big\|
  \psi_{1,1}\big(
    e^{ ( t - s )A } \, u,
    S_{ s, t } \, v
  \big)
  \big(
    e^{ ( t - s )A } \, B^b( w ),
    [S_{ s, t } \, R_s
    -
    e^{ ( t - s )A }
    ] \, B^b( w )
  \big)
  \big\|_V
\\ & +
\nonumber
  \smallsum\limits_{ b \in \mathbb{ U } }
  \big\|
  \big[
  \psi_{1,1}
  \big(
  e^{ ( t - s )A } \, u,
  S_{ s, t } \, v
  \big)
  -
  \psi_{1,1}
  \big(
  e^{ ( t - s )A } \, u,
  e^{ ( t - s )A } \, u
  \big)\big]
  \big(
    e^{ ( t - s )A } \, B^b( w ),
    e^{ ( t - s )A } \, B^b( w )
  \big)
  \big\|_V
\\ & \leq
\nonumber
  \psiC
  \,
  | \groupC_0 |^\power
    \max\{
      1,
      \| u \|^\power_H,
      \| v \|^\power_H
    \}
  \Big[
  \|
    S_{ s, t } \, v
    -
    e^{ ( t - s )A } \, u
  \|_H
  \,
  \|
    e^{ ( t - s )A } \, B( w )
  \|^2_{ HS( U, H ) }
\\ & +
\nonumber
  \|
  e^{ ( t - s )A } \, B( w )
  \|_{ HS( U,H ) }
  \,
  \|
    [S_{ s, t } \, R_s
    -
    e^{ ( t - s )A }
    ] \, B( w )
  \|_{ HS( U,H ) }
  \Big]
\\ & \leq
\nonumber
  \psiC
  \,
  | \groupC_0 |^\power
    \max\{
      1,
      \| u \|^\power_H,
      \| v \|^\power_H
    \}
    \,
  \| B \|^2_{ \operatorname{Lip}^0( H, HS( U, H_{ -\nicefrac{ \vartheta }{ 2 } } ) ) } 
  \,
  g_2( w )
\\ & \cdot
\nonumber
  \Bigg[
  \frac{
    \groupC_{ \nicefrac{ \vartheta }{ 2 } }
    \,
    \|
      S_{ s, t } \, R_s
      -
      e^{ ( t - s )A }
    \|_{ L( H_{ -\nicefrac{ \vartheta }{ 2 } }, H ) }
  }{
    ( t - s )^{ \nicefrac{ \vartheta }{ 2 } }
  }
  +
  \frac{
    | \groupC_{ \nicefrac{ \vartheta }{ 2 } } |^2
  }{
    ( t - s )^\vartheta
  }
    \left[
    \| S_{ s, t } - e^{ ( t - s )A } \|_{ L( H ) }
    \,
    \| v \|_H
    +
    \frac{
      \groupC_\rho
      \,
      \| v - u \|_{ H_{ -\rho } }
    }{
      ( t - s )^\rho
    }
    \right]
  \Bigg]
  .
\end{align}
Combining \eqref{eq:weak_distance_B_1}--\eqref{eq:weak_distance_B_3} implies 
that for all $ ( s, t ) \in \angle $, 
$ u, v, w \in H $ 
it holds that 
\begin{equation}\label{eq:absolute_tilde_B_2nd}
\begin{split}
&
  \|\tilde{B}_{ s, t }( u, v, w )\|_V
\\ & \leq
  2 
  \, 
  \psiC
  \,
  | \groupC_0 |^\power
  \,
  \groupC_{ \nicefrac{ \vartheta }{ 2 } }
  \max\{
    1,
    \| u \|^\power_H,
    \| v \|^\power_H
  \}
  \,
  \| B \|^2_{ \operatorname{Lip}^0( H, HS( U, H_{ -\nicefrac{ \vartheta }{ 2 } } ) ) } 
  \,
  g_2( w )
\\ & \cdot
  \Bigg[
    \left[
      \frac{
       \groupC_{ \nicefrac{ \vartheta }{ 2 }, 0, \rho } + \groupC_{ \nicefrac{ \vartheta }{ 2 } } \, \groupC_{ 0, \rho } \, \| v \|_H
      }{
       ( t - s )^{ ( \rho + \vartheta ) }
      }
    \right]
    h^\rho
+
    \frac{
      \groupC_{ \nicefrac{ \vartheta }{ 2 } } \, \groupC_\rho
      \,
      \| v - u \|_{ H_{ -\rho } }
    }{
      ( t - s )^{ ( \rho + \vartheta ) }
    }
  \Bigg].
\end{split}
\end{equation}
Next observe that \eqref{eq:absolute_tilde_F_2nd} and \eqref{eq:absolute_tilde_B_2nd} show that for all 
$ ( s, t ) \in \angle $, 
$ u, v, w \in H $ 
it holds that 
\begin{equation}\label{eq:absolute_sum_FB_tilde_2nd}
\begin{split}
&
  \|
    \tilde{F}_{ s, t }( u, v, w )
  \|_V
  +
  \|
    \tilde{B}_{ s, t }( u, v, w )
  \|_V
\leq
  \psiC
  \,
  | \groupC_0 |^\power
    \max\{
      1,
      \| u \|^\power_H,
      \| v \|^\power_H
    \}
    \,
  \varsigma_{ F, B } 
  \,
  g_2( w )
\\ & \cdot
  \Bigg[
    \left[
      \frac{
        \groupC_{ \vartheta, 0, \rho } + 2 \, \groupC_{ \nicefrac{ \vartheta }{ 2 } } \, \groupC_{ \nicefrac{ \vartheta }{ 2 }, 0, \rho }
        +
        2 \, ( | \groupC_{ \nicefrac{ \vartheta }{ 2 } } |^2 + \groupC_\vartheta ) \, \groupC_{ 0, \rho }
      }{
        ( t - s )^{ ( \rho + \vartheta ) }
      } 
    \right] 
    g_1( v ) \, h^\rho
    +
    \frac{
      2 
      \, 
      ( 
        | \groupC_{ \nicefrac{ \vartheta }{ 2 } } |^2 + \groupC_\vartheta 
      ) 
      \, 
      \groupC_\rho \, 
      \| v - u \|_{ 
        H_{ -\rho } 
      }
    }{
      ( t - s )^{ ( \rho + \vartheta ) }
    }
  \Bigg]
  .
\end{split}
\end{equation}
In addition, note that H\"{o}lder's inequality ensures that 
for all 
$ r \in ( 0, \infty ) $, 
$ s \in [ 0, T ] $ 
it holds that 
\begin{equation}
\label{eq:optimal_moment_2}
\begin{split}
&
  \ES\big[
  \max\!\big\{
    1,
    \| \bar{Y}_s \|^r_H,
    \| Y_s \|^r_H
  \big\} \,
  g_2\big( Y_{ \floor{ s }{ h } } \big)
  \big]
\\ & \leq
  \left(
  \sup_{ u, v \in [ 0, T ] }
  \big\|\!
  \max\!\big\{
    1,
    \| \bar{Y}_u \|^r_H,
    \| Y_v \|^r_H
  \big\}
  \big\|_{ \lpn{ 1 + \nicefrac{2}{r} }{ \P }{ \R } }
  \right)
  \left(
  \sup_{ u \in [ 0, T ] }
  \big\|
  \!
  \max\!\big\{
    1,
    \| Y_u \|^2_H
  \big\}
  \big\|_{ \lpn{ 1 + \nicefrac{r}{2} }{ \P }{ \R } }
  \right)
\\ & \leq
  | K_{ r + 2 } |^{ \frac{ 1 }{ 1 + \nicefrac{2}{r} } }
  \,
  | K_{ r + 2 } |^{ \frac{ 1 }{ 1 + \nicefrac{r}{2} } }
  = K_{ r + 2 },
\end{split}
\end{equation}
\begin{equation}
\begin{split}
  \ES\big[
    g_2(
      Y_{ \floor{ s }{ h } }
    )
    \,
    \|
      Y_s - \bar{Y}_s
    \|_{ H_{ -\rho } }
  \big]
& \leq
  \|
    g_2(
      Y_{ \floor{ s }{ h } }
    )
  \|_{ \lpn{ \nicefrac{ 3 }{ 2 } }{ \P }{ \R } } 
  \,
  \|
    Y_s - \bar{Y}_s
  \|_{ \lpn{ 3 }{ \P }{ H_{ -\rho } } }
\\ & \leq
  | K_{ 3 } |^{ \nicefrac{ 2 }{ 3 } } 
  \left(
  \sup_{ u \in [ 0, T ] }
  \big\|
      Y_u - \bar{Y}_u
  \big\|_{ \lpn{ 3 }{ \P }{ H_{ -\rho } } }
  \right),
\end{split}
\end{equation}
and 
\begin{equation}\label{eq:optimal_moment_3}
\begin{split}
&
    \ES\big[
      \max\!\big\{
        1 ,
        \| \bar{Y}_s \|^r_H ,
        \| Y_s \|^r_H
      \big\}
      \,
      g_2(
        Y_{ \floor{ s }{ h } }
      )
      \,
      \|
        Y_s - \bar{Y}_s
      \|_{ H_{ -\rho } }
    \big]
\\ & \leq
  \|
    \max\{
      1
      ,
      \| \bar{Y}_s \|^r_H
      ,
      \| Y_s \|^r_H
    \}
  \|_{ \lpn{ 1 + \nicefrac{ 3 }{ r } }{ \P }{ \R } } 
  \,
  \|
    g_2(
      Y_{ \floor{ s }{ h } }
    )
  \|_{ 
    \lpn{ \nicefrac{ ( r + 3 ) }{ 2 } }{ \P }{ \R } 
  } 
  \,
  \|
      Y_s - \bar{Y}_s
  \|_{ \lpn{ r + 3 }{ \P }{ H_{ -\rho } } }
\\ & \leq
  | K_{ r + 3 } |^{ \frac{ r + 2 }{ r + 3 } } 
  \left(
    \sup_{ u \in [ 0, T ] }
    \big\|
      Y_u - \bar{Y}_u
    \big\|_{ \lpn{ r + 3 }{ \P }{ H_{ -\rho } } }
  \right)
  .
\end{split}
\end{equation}
Combining 
\eqref{eq:absolute_sum_FB_tilde_2nd}--\eqref{eq:optimal_moment_3} 
with
Lemma~\ref{lem:euler_integrated_strong} 
and the fact that 
$ 
  1 - ( \frac{ 1 + \vartheta }{ 2 } - \rho )^+ > \rho 
$ 
yields that for all $ t \in ( 0, T ] $ it holds that 
\begin{equation}
\label{eq:integrated_absolute_sum_FB_tilde_2nd_conclude}
\begin{split}
&
  \int^{ t }_0
  \E\left[
    \big\|
      \tilde{F}_{ s, t }\big(
        \bar{Y}_s, Y_s, Y_{ \floor{ s }{ h } }
    \big)
    \big\|_V
  \right]
  +
  \E\left[
    \big\|
      \tilde{B}_{ s, t }\big(
        \bar{Y}_s, Y_s, Y_{ \floor{ s }{ h } }
      \big)
    \big\|_V
  \right]
  ds
\leq
  \tfrac{
    \psiC
    \,
    | \groupC_0 |^\power
    \,
    \varsigma_{ F, B } 
    \,
    K_{ \power + 3 } 
    \, 
    h^\rho
    \,
    t^{ ( 1 - \vartheta - \rho ) }
  }{ 
    ( 1 - \vartheta - \rho ) 
  }
\\ & \cdot
  \Bigg[
    \groupC_{ \vartheta, 0, \rho } + 2 \, \groupC_{ \nicefrac{ \vartheta }{ 2 } } \, \groupC_{ \nicefrac{ \vartheta }{ 2 }, 0, \rho }
    +
    2 \, ( | \groupC_{ \nicefrac{ \vartheta }{ 2 } } |^2 + \groupC_\vartheta ) \, \groupC_{ 0, \rho }
    +
    2 \, 
    ( | \groupC_{ \nicefrac{ \vartheta }{ 2 } } |^2 + \groupC_\vartheta ) \, \groupC_\rho
\\ & 
    \cdot
    \bigg(
      \groupC_{ -\rho, \rho }
      +
      \tfrac{
        \groupC_{ \vartheta, -\rho, \rho } \, t^{ ( 1 - \vartheta ) } \,
        \| F \|_{ \operatorname{Lip}^0( H, H_{ -\vartheta } ) }
      }{ 
        ( 1 - \vartheta ) 
      }
      +
      \tfrac{
        \groupC_{ \nicefrac{\vartheta}{2}, -\rho, \rho } 
        \, 
        \sqrt{ 
          ( \power + 3 ) \, 
          (\power+2) \, 
          t^{ ( 1 - \vartheta ) } 
        }
        \,
        \| B \|_{ \operatorname{Lip}^0( H, HS( U, H_{ - \vartheta / 2 } ) ) }
      }{
        \sqrt{ 2 - 2 \vartheta }
      }
    \bigg)
  \Bigg]
  .
\end{split}
\end{equation}
Putting 
\eqref{eq:initial_value_part_2nd} and \eqref{eq:integrated_absolute_sum_FB_tilde_2nd_conclude} 
into 
\eqref{eq:mild_ito_2nd} 
proves \eqref{eq:weak_distance}.
This finishes
the proof of Proposition~\ref{prop:weak_temporal_regularity_2nd}.
\end{proof}

\section{Weak convergence rates for Euler-type approximations of 
SPDEs with mollified nonlinearities}
\label{sec:euler_integrated_mollified}

In this section we use the results of Section~\ref{sec:weak_temporal_regularity}
and the Kolmogorov backward equation associated to an SEE 
to establish weak convergence rates for temporal numerical approximations 
of SEEs with mollified nonlinearities; see Corollary~\ref{cor:mollified_weak_solution-num} 
and Corollary~\ref{cor:mollified_weak_solution-num_komplete} below.
Some of the arguments in this section
are similar to some of the arguments in Section~3
in Conus et al.~\cite{ConusJentzenKurniawan2014arXiv}.

\subsection{Setting}
\label{sec:setting_euler_integrated_mollified}

Assume the setting in Section~\ref{sec:semigroup_setting},
let 
$ \vartheta \in [0,\frac{1}{2}) $,
$
  F \in 
  C^5_b( H , H_1 ) 
$, 
$
  B \in 
  C^5_b( 
    H, 
    HS( 
      U, 
$
$
      H_1
    ) 
  ) 
$, 
$
  \varphi \in C^5_b( H, V)
$, 
let 
$ 
  ( B^b )_{ b \in \mathbb{U} } \subseteq C( H, H ) 
$ 
be the functions with the property that for all 
$ v \in H $, 
$ b \in \mathbb{U} $ 
it holds that 
$
     B^b( v ) 
    = 
      B( 
        v 
      )
      \,
      b
$, 
let
$ \varsigma_{ F, B } \in \R $
be a real number given by
$  
  \varsigma_{ F, B } 
  =
    \max\!\big\{
    1
    ,
    \|
      F
    \|_{ 
      C_b^3( H, H_{-\vartheta} )
    }^3
    ,
    \|
      B
    \|_{ 
      C_b^3( H, HS( U, H_{ - \vartheta / 2 } ) )
    }^6
    \big\}
$, 
let
$
  X, Y \colon [0,T] \times \Omega \to H
$, 
$
  \bar{Y} \colon [0,T] \times \Omega \to H_1
$, 
and 
$
  X^x \colon [0,T] \times \Omega \to H
$, 
$ x \in H $, 
be 
$
  ( \mathcal{F}_t )_{ t \in [0,T] }
$-predictable stochastic processes 
such that for all 
$ x \in H $ 
it holds that 
$
  \sup_{ t \in [0,T] }
  \big[
  \| X_t \|_{ \lpn{5}{\P}{ H } } 
  +
  \| X^x_t \|_{ \lpn{5}{\P}{H} } 
  \big]
  < \infty
$, 
$ X^x_0 = x $, 
$
  \bar{Y}_0 \in \lpn{5}{\P}{H_1}
$,
and 
$
  Y_0 = X_0 = \bar{Y}_0 
$ 
and such that
for all 
$ x \in H $, 
$ t \in (0,T] $ 
it holds $ \P $-a.s.\ that 
\begin{equation}
  X_t
  = 
    e^{ t A }\, X_0 
  + 
    \int_0^t e^{ ( t - s )A }\, F( X_s ) \, ds
  + 
    \int_0^t e^{ ( t - s )A }\, B( X_s ) \, dW_s
    ,
\end{equation}
\begin{equation}
  X^x_t
  = 
    e^{ t A }\, x 
  + 
    \int_0^t e^{ ( t - s )A }\, F( X^x_s ) \, ds
  + 
    \int_0^t e^{ ( t - s )A }\, B( X^x_s ) \, dW_s
    ,
\end{equation}
\begin{equation}
  Y_t
  = 
    S_{ 0, t }\, Y_0 
  + 
    \int_0^t S_{ s, t }\, R_s\, F( Y_{ \floor{ s }{ h } } ) \, ds
  + 
    \int_0^t S_{ s, t }\, R_s\, B( Y_{ \floor{ s }{ h } } ) \, dW_s 
    ,
\end{equation}
\begin{equation}
  \bar{Y}_t
  = 
    e^{ t A } \, \bar{Y}_0 
  + 
    \int_0^t e^{ ( t - s )A } \, F( Y_{ \floor{ s }{ h } } ) \, ds
  + 
    \int_0^t e^{ ( t - s )A } \, B( Y_{ \floor{ s }{ h } } ) \, dW_s 
    ,
\end{equation}
let 
$
  u \colon [0,T] \times H \to V  
$
be the function 
with the property that for all $ x \in H $, $ t \in [0,T] $
it holds that
$
  u( t, x ) = 
  \ES\big[ 
    \varphi( X^x_{ T - t } )
  \big]
$,
let 
$
  c_{ \delta_1, \dots, \delta_k }
  \in [0,\infty]	
$,
$ \delta_1, \dots, \delta_k \in \R $,
$ k \in \{ 1, 2, 3, 4 \} $,
be the extended real numbers 
with the property that
for all $ k \in \{ 1, 2, 3, 4 \} $,
$ \delta_1, \dots, \delta_k \in \R $
it holds that
\begin{equation}
\begin{split} 
&
  c_{ \delta_1, \delta_2, \dots, \delta_k }
=
  \sup_{
    t \in [0,T)
  }
  \sup_{ 
    x \in H
  }
  \sup_{ 
    v_1, \dots, v_k \in H \backslash \{ 0 \}
  }
  \left[
  \frac{
    \big\|
      ( 
        \frac{ 
          \partial^k
        }{
          \partial x^k
        }
        u
      )( t, x )( v_1, \dots, v_k )
    \big\|_V
  }{
    ( T - t )^{ 
      (
        \delta_1 + \ldots + \delta_k
      ) 
    }
    \left\| v_1 \right\|_{ 
      H_{ \delta_1 } 
    }
    \cdot
    \ldots
    \cdot
    \left\| v_k \right\|_{ 
      H_{ \delta_k } 
    }
  }
  \right]
  ,
\end{split}
\end{equation}
let 
$
  \tilde{c}_{ \delta_1, \delta_2, \delta_3, \delta_4 }
  \in [0,\infty]	
$,
$ \delta_1, \delta_2, \delta_3, \delta_4 \in \R $, 
be the extended real numbers 
with the property that for all 
$ \delta_1, \dots, \delta_4 \in \R $
it holds that
\begin{equation}
\begin{split} 
&
  \tilde{c}_{ \delta_1, \delta_2, \delta_3, \delta_4 }
\\ & =
  \sup_{ t \in [ 0, T ) }\,
  \sup_{\substack{
    x_1, x_2 \in H, \\ x_1 \neq x_2
  }}\,
  \sup_{ 
    v_1, \ldots, v_4 \in H \backslash \{ 0 \}
  }
  \left[
  \frac{
    \big\|
    \big(
      \big( 
        \frac{ 
          \partial^4
        }{
          \partial x^4
        }
        u
      \big)( t, x_1 )
      -
      \big( 
        \frac{ 
          \partial^4
        }{
          \partial x^4
        }
        u
      \big)( t, x_2 )
      \big)
      ( v_1, \dots, v_4 )
    \big\|_V
  }{
    ( T - t )^{ 
      (
        \delta_1 + \ldots + \delta_4
      ) 
    }
    \left\| x_1 - x_2 \right\|_H
    \left\| v_1 \right\|_{ H_{ \delta_1 } }
     \cdot
     \ldots
    \cdot
    \left\| v_4 \right\|_{ H_{ \delta_4 } }
  }
  \right]
  ,
\end{split}
\end{equation}
let 
$
  ( K_r )_{ r \in [ 0, \infty ) }
  \subseteq
  [ 0, \infty ]
$ 
be the extended real numbers which satisfy 
that for all 
$ r \in [ 0, \infty ) $ 
it holds that 
$
  K_r
  =
  \sup_{ s, t \in [ 0, T ] }
  \ES\big[
    \max\{
    1,
    \| \bar{Y}_s \|^r_H,
    \| Y_t \|^r_H
    \}
  \big]
$, 
and let
$ 
  u_{1,0} \colon [0,T] \times H \to V 
$
and
$
  u_{0,k} \colon [0,T] \times H \to L^{ (k) }( H, V ) 
$, 
$ k \in \{ 1, 2, 3, 4 \} $, 
be the functions 
with the property that
for all $ t \in [0,T] $, 
$ x \in H $, 
$ k \in \{ 1, 2, 3, 4 \} $, 
$ v_1, \ldots, v_k \in H $ 
it holds that 
$
  u_{1,0}( t, x )
  =
  ( \frac{ \partial }{ \partial t } u )( t, x )
$
and 
$
  u_{0,k}( t, x )( v_1, \ldots, v_k )
  =
  \big( ( \frac{ \partial^k }{ \partial x^k } u )( t, x ) \big)( v_1, \ldots, v_k )
$.

\subsection{Weak convergence rates for semilinear integrated Euler-type 
approximations of SPDEs with mollified nonlinearities}
\label{sec:euler_integrated_result_mollified}

\begin{lemma}\label{lem:weak_regularity_F_2nd_term}
Assume the setting in Section~\ref{sec:setting_euler_integrated_mollified} and let 
$ t \in [ 0, T ) $, 
$
  \psi = ( \psi(x,y) )_{ x, y \in H } \in \mathbb{M}( H \times H, V )
$, 
$
  \phi \in \mathbb{M}( H, V )
$
satisfy that for all 
$ x, y \in H $ 
it holds that 
$
  \psi( x, y )
  =
  u_{0,1}( t, x )
  F( y )
$
and 
$
  \phi(x)
  =
  \psi(x,x)
$.
Then it holds that 
$ \psi \in C^3( H \times H, V ) $, 
$ \phi \in C^3(H, V) $
and for all 
$ x, x_1, x_2, y, y_1, y_2 \in H $ 
it holds that 
\begin{equation}
\label{eq:psiF_Lipschitzx}
\begin{split}
&
    \max_{
    i, j \in \N_0
    ,\,
    i + j \leq 2
    }
    \big\|
    \big(
    \tfrac{ \partial^{(i+j)} }{ \partial x^i \partial y^j }
    \psi
    \big)
    ( x_1, y )
    -
    \big(
    \tfrac{ \partial^{(i+j)} }{ \partial x^i \partial y^j }
    \psi
    \big)
    ( x_2, y )
    \big\|_{ L^{ (i+j) }( H, V ) }
\\ & \leq
  \tfrac{
    \| x_1 - x_2 \|_H
  }{
  ( T - t )^{ \vartheta }
  } \,
  \| F \|_{ C^2_b( H, H_{ -\vartheta } ) } \,
  \big[ c_{ -\vartheta, 0 } + c_{ -\vartheta, 0, 0 } + c_{ -\vartheta, 0, 0, 0 } \big]
  \max\{ 1, \| y \|_H \}
  ,
\end{split}
\end{equation}
\begin{equation}
\label{eq:psiF_Lipschitzy}
\begin{split}
&
    \max_{
    i, j \in \N_0
    ,\,
    i + j \leq 2
    }
    \big\|
    \big(
    \tfrac{ \partial^{(i+j)} }{ \partial x^i \partial y^j }
    \psi
    \big)
    ( x, y_1 )
    -
    \big(
    \tfrac{ \partial^{(i+j)} }{ \partial x^i \partial y^j }
    \psi
    \big)
    ( x, y_2 )
    \big\|_{ L^{ (i+j) }( H, V ) }
\\ & \leq
  \tfrac{
    \| y_1 - y_2 \|_H
  }{
  ( T - t )^{ \vartheta }
  } \,
  \| F \|_{ C^3_b( H, H_{ -\vartheta } ) } \,
  \big[ c_{ -\vartheta } + c_{ -\vartheta, 0 } + c_{ -\vartheta, 0, 0 } \big]
  ,
\end{split}
\end{equation}
\begin{equation}
\label{eq:psiF_Lipschitzxx}
\begin{split}
&
    \max_{
    i \in \{ 0, 1, 2 \}
    }
    \big\|
    \phi^{ (i) }
    ( x_1 )
    -
    \phi^{ (i) }
    ( x_2 )
    \big\|_{ L^{ (i) }( H, V ) }
\\ & \leq
  \tfrac{
    3 \,
    \| x_1 - x_2 \|_H
  }{
  ( T - t )^{ \vartheta }
  } \,
  \| F \|_{ C^3_b( H, H_{ -\vartheta } ) } \,
  \big[ c_{ -\vartheta } + c_{ -\vartheta, 0 } + c_{ -\vartheta, 0, 0 } + c_{ -\vartheta, 0, 0, 0 } \big]
  \max\{ 1, \| x_1 \|_H, \| x_2 \|_H \}
  .
\end{split}
\end{equation}
\end{lemma}
\begin{proof}
First, we note that 
the assumption that
$
  F \in \operatorname{Lip}^4( H, H_1 )
$
and the fact that
$
  \big(
    H \ni x \mapsto
    u_{ 0, 1 }( t, x ) \in L( H, V )
  \big)
  \in
  C^3( H, L( H, V ) ) 
$
ensure that
$ \psi \in C^3( H \times H, V ) $ 
and 
$ \phi \in C^3( H, V ) $. 
Next we observe that for all 
$ x, y, v_1, v_2, v_3 \in H $ 
with 
$ \|v_1\|_H, \|v_2\|_H, \|v_3\|_H \leq 1 $
it holds that 
\begin{equation}
\label{eq:psiF_xderivatives_begin}
\begin{split}
&
  \big\|
   \big(\tfrac{ \partial }{ \partial x }
   \psi\big)( x, y ) \, v_1
  \big\|_V
=
  \left\|
    u_{0,2}( t, x )
    \big(
     F( y ), v_1
    \big)
  \right\|_V
\leq
  \tfrac{
    c_{ -\vartheta, 0 }
  }{
  ( T - t )^{ \vartheta }
  } \,
  \| F( y ) \|_{ H_{ -\vartheta } },
\end{split}
\end{equation}
\begin{equation}
\begin{split}
&
  \big\|
   \big(\tfrac{ \partial^2 }{ \partial x^2 }
   \psi\big)( x, y ) \, ( v_1, v_2 )
  \big\|_V
=
  \left\|
    u_{0,3}( t, x )
    \big(
     F( y ), v_1, v_2
    \big)
  \right\|_V
\leq
  \tfrac{
    c_{ -\vartheta, 0, 0 }
  }{
  ( T - t )^{ \vartheta }
  } \,
  \| F( y ) \|_{ H_{ -\vartheta } },
\end{split}
\end{equation}
\begin{equation}
\label{eq:psiF_xderivatives_end}
\begin{split}
&
  \big\|
   \big(\tfrac{ \partial^3 }{ \partial x^3 }
   \psi\big)( x, y ) \, ( v_1, v_2, v_3 )
  \big\|_V
=
  \left\|
    u_{0,4}( t, x )
    \big(
     F( y ), v_1, v_2, v_3
    \big)
  \right\|_V
\leq
  \tfrac{
    c_{ -\vartheta, 0, 0, 0 }
  }{
  ( T - t )^{ \vartheta }
  } \,
  \| F( y ) \|_{ H_{ -\vartheta } },
\end{split}
\end{equation}
\begin{equation}
\label{eq:psiF_yderivatives_begin}
\begin{split}
&
  \big\|
   \big(\tfrac{ \partial }{ \partial y }
   \psi\big)( x, y ) \, v_1
  \big\|_V
=
  \left\|
    u_{0,1}( t, x )
    \,
     F'( y ) \, v_1
  \right\|_V
\leq
  \tfrac{
    c_{ -\vartheta }
  }{
  ( T - t )^{ \vartheta }
  } \,
  \| F'( y ) \|_{ L( H, H_{ -\vartheta } ) },
\end{split}
\end{equation}
\begin{equation}
\begin{split}
&
  \big\|
   \big(\tfrac{ \partial^2 }{ \partial y^2 }
   \psi\big)( x, y )( v_1, v_2 )
  \big\|_V
=
  \left\|
    u_{0,1}( t, x )
    \big(
     F''( y ) ( v_1, v_2 )
    \big)
  \right\|_V
\leq
  \tfrac{
    c_{ -\vartheta }
  }{
  ( T - t )^{ \vartheta }
  } \,
  \| F''( y ) \|_{ L^{ (2) }( H, H_{ -\vartheta } ) },
\end{split}
\end{equation}
\begin{equation}
\label{eq:psiF_yderivatives_end}
\begin{split}
&
  \big\|
   \big(\tfrac{ \partial^3 }{ \partial y^3 }
   \psi\big)( x, y )( v_1, v_2, v_3 )
  \big\|_V
=
  \left\|
    u_{0,1}( t, x )
    \big(
     F^{(3)}( y ) ( v_1, v_2, v_3 )
    \big)
  \right\|_V
\\ & \leq
  \tfrac{
    c_{ -\vartheta }
  }{
  ( T - t )^{ \vartheta }
  } \,
  \| F^{(3)}( y ) \|_{ L^{ (3) }( H, H_{ -\vartheta } ) },
\end{split}
\end{equation}
\begin{equation}
\label{eq:psiF_mixedderivatives_begin}
\begin{split}
&
  \big\|
   \big(\tfrac{ \partial^2 }{ \partial x \partial y }
   \psi
   \big)( x, y )( v_1, v_2 )
  \big\|_V
=
  \left\|
    u_{0,2}( t, x )
    \big(
     F'( y ) \, v_1,
     v_2
    \big)
  \right\|_V
\leq
  \tfrac{
    c_{ -\vartheta, 0 }
  }{
  ( T - t )^{ \vartheta }
  } \,
  \| F'( y ) \|_{ L( H, H_{ -\vartheta } ) },
\end{split}
\end{equation}
\begin{equation}
\begin{split}
&
  \big\|
   \big(\tfrac{ \partial^3 }{ \partial x^2 \partial y }
   \psi\big)( x, y )( v_1, v_2, v_3 )
  \big\|_V
=
  \left\|
    u_{0,3}( t, x )
    \big(
     F'( y ) \, v_1,
     v_2, v_3
    \big)
  \right\|_V
\\ & \leq
  \tfrac{
  c_{ -\vartheta, 0, 0 }
  }{
  ( T - t )^{ \vartheta }
  } \,
  \| F'( y ) \|_{ L( H, H_{ -\vartheta } ) },
\end{split}
\end{equation}
\begin{equation}
\label{eq:psiF_mixedderivatives_end}
\begin{split}
&
  \big\|
   \big(\tfrac{ \partial^3 }{ \partial x \partial y^2 }
   \psi\big)( x, y )( v_1, v_2, v_3 )
  \big\|_V
=
  \left\|
    u_{0,2}( t, x )
    \big(
     F''( y ) ( v_1, v_2 ),
     v_3
    \big)
  \right\|_V
\\ & \leq
  \tfrac{
    c_{ -\vartheta, 0 }
  }{
  ( T - t )^{ \vartheta }
  } \,
  \| F''( y ) \|_{ L^{ ( 2 ) }( H, H_{ -\vartheta } ) }.
\end{split}
\end{equation}
Combining \eqref{eq:psiF_xderivatives_begin}--\eqref{eq:psiF_xderivatives_end} and 
\eqref{eq:psiF_mixedderivatives_begin}--\eqref{eq:psiF_mixedderivatives_end} 
with the fundamental theorem of calculus in Banach spaces proves~\eqref{eq:psiF_Lipschitzx}.
Moreover, combining \eqref{eq:psiF_yderivatives_begin}--\eqref{eq:psiF_mixedderivatives_end} 
with the fundamental theorem of calculus in Banach spaces shows~\eqref{eq:psiF_Lipschitzy}. 
It thus remains to prove~\eqref{eq:psiF_Lipschitzxx}. 
For this we observe that \eqref{eq:psiF_xderivatives_begin}--\eqref{eq:psiF_mixedderivatives_end} ensure that 
for all $ x, v_1, v_2, v_3 \in H $ with 
$ \|v_1\|_H, \|v_2\|_H, \|v_3\|_H \leq 1 $
it holds that 
\begin{equation}
\label{eq:psiF_xxderivatives_begin}
\begin{split}
&
  \big\|
   \phi'( x ) \, v_1
  \big\|_V
  \leq
  \big\|
   \big(\tfrac{ \partial }{ \partial x }
   \psi\big)( x, x ) \, v_1
  \big\|_V
  +
  \big\|
   \big(\tfrac{ \partial }{ \partial y }
   \psi\big)( x, x ) \, v_1
  \big\|_V
\\ & \leq
  \tfrac{
    c_{ -\vartheta, 0 } \,
    \| F( x ) \|_{ H_{ -\vartheta } }
    +
    c_{ -\vartheta } \,
    \| F'( x ) \|_{ L( H, H_{ -\vartheta } ) }
  }{
  ( T - t )^{ \vartheta }
  }
  \leq
  \tfrac{
    [ c_{ -\vartheta }
    +
    c_{ -\vartheta, 0 } ]
  }{
  ( T - t )^{ \vartheta }
  }
  \,
  \| F \|_{ C^1_b( H, H_{ -\vartheta } ) } \,
  \max\{ 1, \| x \|_H \}
  ,
\end{split}
\end{equation}
\begin{equation}
\begin{split}
&
  \big\|
   \phi''( x ) \, ( v_1, v_2 )
  \big\|_V
\\ & \leq
  \big\|
   \big(\tfrac{ \partial^2 }{ \partial x^2 }
   \psi\big)( x, x ) \, ( v_1, v_2 )
  \big\|_V
  +
  2 \,
  \big\|
   \big(\tfrac{ \partial^2 }{ \partial x \partial y }
   \psi\big)( x, x ) \, ( v_1, v_2 )
  \big\|_V
  +
  \big\|
   \big(\tfrac{ \partial^2 }{ \partial y^2 }
   \psi\big)( x, x ) \, ( v_1, v_2 )
  \big\|_V
\\ & \leq
  \tfrac{
    c_{ -\vartheta, 0, 0 } \,
    \| F( x ) \|_{ H_{ -\vartheta } }
    +
    2 \,
    c_{ -\vartheta, 0 } \,
    \| F'( x ) \|_{ L( H, H_{ -\vartheta } ) }
    +
    c_{ -\vartheta } \,
    \| F''( x ) \|_{ L^{(2)}( H, H_{ -\vartheta } ) }
  }{
  ( T - t )^{ \vartheta }
  }
\\ & \leq
  \tfrac{
    2 \,
    [ c_{ -\vartheta }
    +
    c_{ -\vartheta, 0 } 
    + c_{ -\vartheta, 0, 0 } ]
  }{
  ( T - t )^{ \vartheta }
  }
  \,
  \| F \|_{ C^2_b( H, H_{ -\vartheta } ) } \,
  \max\{ 1, \| x \|_H \}
  ,
\end{split}
\end{equation}
\begin{equation}
\label{eq:psiF_xxderivatives_end}
\begin{split}
&
  \big\|
   \phi^{(3)}( x ) \, ( v_1, v_2, v_3 )
  \big\|_V
\leq
  \big\|
   \big(\tfrac{ \partial^3 }{ \partial x^3 }
   \psi\big)( x, x ) \, ( v_1, v_2, v_3 )
  \big\|_V
+
  3 \,
  \big\|
   \big(\tfrac{ \partial^3 }{ \partial x^2 \partial y }
   \psi\big)( x, x ) \, ( v_1, v_2, v_3 )
  \big\|_V
\\ & +
  3 \,
  \big\|
   \big(\tfrac{ \partial^3 }{ \partial x \partial y^2 }
   \psi\big)( x, x ) \, ( v_1, v_2, v_3 )
  \big\|_V
  +
  \big\|
   \big(\tfrac{ \partial^3 }{ \partial y^3 }
   \psi\big)( x, x ) \, ( v_1, v_2, v_3 )
  \big\|_V
\\ & \leq
  \tfrac{
    c_{ -\vartheta, 0, 0, 0 } \,
    \| F( x ) \|_{ H_{ -\vartheta } }
    +
    3 \,
    c_{ -\vartheta, 0, 0 } \,
    \| F'( x ) \|_{ L( H, H_{ -\vartheta } ) }
    +
    3 \,
    c_{ -\vartheta, 0 } \,
    \| F''( x ) \|_{ L^{(2)}( H, H_{ -\vartheta } ) }
    +
    c_{ -\vartheta } \,
    \| F^{(3)}( x ) \|_{ L^{(3)}( H, H_{ -\vartheta } ) }
  }{
  ( T - t )^{ \vartheta }
  }
\\ & \leq
  \tfrac{
    3 \,
    [ c_{ -\vartheta }
    +
    c_{ -\vartheta, 0 } 
    + c_{ -\vartheta, 0, 0 } 
    + c_{ -\vartheta, 0, 0, 0 } ]
  }{
  ( T - t )^{ \vartheta }
  }
  \,
  \| F \|_{ C^3_b( H, H_{ -\vartheta } ) } \,
  \max\{ 1, \| x \|_H \}
  .
\end{split}
\end{equation}
Combining \eqref{eq:psiF_xxderivatives_begin}--\eqref{eq:psiF_xxderivatives_end} 
with the fundamental theorem of calculus in Banach spaces establishes~\eqref{eq:psiF_Lipschitzxx}. 
The proof of Lemma~\ref{lem:weak_regularity_F_2nd_term} is thus completed.
\end{proof}

\begin{lemma}
\label{lem:weak_regularity_B_2nd_term}
Assume the setting in Section~\ref{sec:setting_euler_integrated_mollified} and let 
$ t \in [ 0, T ) $, 
$
  \psi = ( \psi(x,y) )_{ x, y \in H } \in \mathbb{M}( H \times H, V )
$, 
$
  \phi \in \mathbb{M}( H, V )
$ 
satisfy that for all 
$ x, y \in H $ 
it holds that
$
  \psi( x, y )
  =
  \smallsum_{ b \in \mathbb{U} }
  u_{0,2}( t, x )
  \big(
    B^b( y ),
    B^b( y )
  \big)
$ 
and 
$
  \phi(x)
  =
  \psi(x,x)
$. 
Then it holds that 
$ \psi \in C^2( H \times H, V ) $, 
$ \phi \in C^2( H, V ) $ 
and for all 
$ x, x_1, x_2, y, y_1, y_2 \in H $ 
it holds that 
\begin{equation}
\label{eq:psiB_Lipschitzx}
\begin{split}
&
    \max_{
    i, j \in \N_0
    ,\,
    i + j \leq 2
    }
    \big\|
    \big(
    \tfrac{ \partial^{(i+j)} }{ \partial x^i \partial y^j }
    \psi
    \big)
    ( x_1, y )
    -
    \big(
    \tfrac{ \partial^{(i+j)} }{ \partial x^i \partial y^j }
    \psi
    \big)
    ( x_2, y )
    \big\|_{ L^{ (i+j) }( H, V ) }
\leq
  \tfrac{
    2 \, \| x_1 - x_2 \|_H
  }{
    ( T - t )^{ \vartheta }
  }
\\ & \cdot
  \| B \|^2_{ C^2_b( H, HS( U, H_{ -\nicefrac{ \vartheta }{ 2 } } ) ) } \,
  \big[ 
    c_{ - \nicefrac{ \vartheta }{ 2 }, -\nicefrac{ \vartheta }{ 2 }, 0 } 
    + c_{ - \nicefrac{ \vartheta }{ 2 }, -\nicefrac{ \vartheta }{ 2 }, 0, 0 }
    + \tilde{c}_{ - \nicefrac{ \vartheta }{ 2 }, -\nicefrac{ \vartheta }{ 2 }, 0, 0 }
  \big]
  \max\{ 1, \| y \|_{ H }^2 \}
  ,
\end{split}
\end{equation}
\begin{equation}
\label{eq:psiB_Lipschitzy}
\begin{split}
&
    \max_{
    i, j \in \N_0
    ,\,
    i + j \leq 2
    }
    \big\|
    \big(
    \tfrac{ \partial^{(i+j)} }{ \partial x^i \partial y^j }
    \psi
    \big)
    ( x, y_1 )
    -
    \big(
    \tfrac{ \partial^{(i+j)} }{ \partial x^i \partial y^j }
    \psi
    \big)
    ( x, y_2 )
    \big\|_{ L^{ (i+j) }( H, V ) }
\leq
  \tfrac{
    6 \, \| y_1 - y_2 \|_H
  }{
    ( T - t )^{ \vartheta }
  }
\\ & \cdot
  \| B \|^2_{ C^3_b( H, HS( U, H_{ -\nicefrac{ \vartheta }{ 2 } } ) ) } \,
  \big[ c_{ -\nicefrac{ \vartheta }{ 2 }, -\nicefrac{ \vartheta }{ 2 } } + c_{ -\nicefrac{ \vartheta }{ 2 }, -\nicefrac{ \vartheta }{ 2 }, 0 } + c_{ -\nicefrac{ \vartheta }{ 2 }, -\nicefrac{ \vartheta }{ 2 }, 0, 0 } \big]
  \max\{ 1, \| y_1 \|_{ H }, \| y_2 \|_{ H } \}
  ,
\end{split}
\end{equation}
\begin{equation}
\label{eq:psiB_Lipschitzxx}
\begin{split}
&
    \max_{
    i \in \{ 0, 1, 2 \}
    }
    \big\|
    \phi^{(i)}
    ( x_1 )
    -
    \phi^{(i)}
    ( x_2 )
    \big\|_{ L^{ (i) }( H, V ) }
\leq
  \tfrac{
    8 \,
    \| x_1 - x_2 \|_H
  }{
  ( T - t )^{ \vartheta }
  }
  \,
  \| B \|^2_{ C^3_b( H, HS( U, H_{ -\nicefrac{ \vartheta }{ 2 } } ) ) } 
\\ & \cdot
  \big[ 
    c_{ 
      - \nicefrac{ \vartheta }{ 2 }, 
      - \nicefrac{ \vartheta }{ 2 } 
    } 
    + 
    c_{ 
      - \nicefrac{ \vartheta }{ 2 }, 
      - \nicefrac{ \vartheta }{ 2 }, 0 
    } 
    + 
    c_{ -\nicefrac{ \vartheta }{ 2 }, -\nicefrac{ \vartheta }{ 2 }, 0, 0 } 
    + 
    \tilde{c}_{ -\nicefrac{ \vartheta }{ 2 }, -\nicefrac{ \vartheta }{ 2 }, 0, 0 } 
  \big]
  \max\{ 1, \| x_1 \|^2_H, \| x_2 \|^2_H \}
  .
\end{split}
\end{equation}
\end{lemma}
\begin{proof}
First of all, we note that the assumption that 
$ B \in \operatorname{Lip}^4( H, HS( U, H_1 ) ) $ 
and the fact that 
$
  \big(
  H \ni x \mapsto
  u_{0,2}( t, x ) \in L^{(2)}(H,V)
  \big)
  \in
  C^2( H, L^{(2)}(H,V) ) 
$
ensure that 
$
  \psi \in 
  C^2( H \times H, V )
$,
$
  \phi \in C^2( H, V )
$,
and
$
  \big(
    H \times H \ni (x,y) 
    \mapsto
    ( \frac{ \partial^2 }{ \partial y^2 } \psi)( x, y ) 
    \in L^{(2)}(H,V)
  \big)
  \in
  C^2( H, L^{(2)}(H,V) ) 
$. 
Next we observe that for all 
$ x, x_1, x_2, y, v_1, v_2, v_3 \in H $ with 
$ \|v_1\|_H, \|v_2\|_H, \|v_3\|_H \leq 1 $
it holds that 
\begin{equation}
\label{eq:psiB_xderivatives_begin}
\begin{split}
&
  \big\|
   \big(\tfrac{ \partial }{ \partial x }
   \psi\big)( x, y ) \, v_1
  \big\|_V
\leq
  \smallsum\limits_{ b \in \mathbb{U} }
  \big\|
  u_{0,3}( t, x )
  \big(
  B^b( y ),
  B^b( y ), v_1
  \big)
  \big\|_V
\\ & \leq
  \tfrac{
    c_{ -\nicefrac{ \vartheta }{ 2 }, -\nicefrac{ \vartheta }{ 2 }, 0 }
  }{
  ( T - t )^{ \vartheta }
  } \,
  \| B( y ) \|^2_{ HS( U, H_{ -\nicefrac{ \vartheta }{ 2 } } ) },
\end{split}
\end{equation}
\begin{equation}
\begin{split}
&
  \big\|
   \big(\tfrac{ \partial^2 }{ \partial x^2 }
   \psi\big)( x, y ) \, ( v_1, v_2 )
  \big\|_V
\leq
  \smallsum\limits_{ b \in \mathbb{U} }
  \big\|
  u_{0,4}( t, x )
  \big(
  B^b( y ),
  B^b( y ), v_1, v_2
  \big)
  \big\|_V
\\ & \leq
  \tfrac{
    c_{ -\nicefrac{ \vartheta }{ 2 }, -\nicefrac{ \vartheta }{ 2 }, 0, 0 }
  }{
  ( T - t )^{ \vartheta }
  } \,
  \| B( y ) \|^2_{ HS( U, H_{ -\nicefrac{ \vartheta }{ 2 } } ) },
\end{split}
\end{equation}
\begin{equation}
\label{eq:psiB_xderivatives_end}
\begin{split}
&
  \big\|
   \big(\tfrac{ \partial^2 }{ \partial x^2 }
   \psi\big)( x_1, y ) \, ( v_1, v_2 )
   -
   \big(\tfrac{ \partial^2 }{ \partial x^2 }
   \psi\big)( x_2, y ) \, ( v_1, v_2 )
  \big\|_V
\\ & \leq
  \smallsum\limits_{ b \in \mathbb{U} }
  \big\|
  \big(
    u_{0,4}( t, x_1 ) - u_{0,4}( t, x_2 )
  \big)
  \big(
  B^b( y ),
  B^b( y ), v_1, v_2
  \big)
  \big\|_V
\\ & \leq
  \tfrac{
    \tilde{c}_{ -\nicefrac{ \vartheta }{ 2 }, -\nicefrac{ \vartheta }{ 2 }, 0, 0 } \,
    \| x_1 - x_2 \|_H
  }{
  ( T - t )^{ \vartheta }
  } \,
  \| B( y ) \|^2_{ HS( U, H_{ -\nicefrac{ \vartheta }{ 2 } } ) },
\end{split}
\end{equation}
\begin{equation}
\label{eq:psiB_yderivatives_begin}
\begin{split}
&
  \big\|
   \big(\tfrac{ \partial }{ \partial y }
   \psi\big)( x, y ) \, v_1
  \big\|_V
\leq
  2 
  \smallsum\limits_{ b \in \mathbb{U} }
  \big\|
  u_{0,2}( t, x )
  \big(
  B^b( y ),
  ( B^b )'( y ) \, v_1
  \big)
  \big\|_V
\\ & \leq
  \tfrac{
    2 \, c_{ -\nicefrac{ \vartheta }{ 2 }, -\nicefrac{ \vartheta }{ 2 } }
  }{
  ( T - t )^{ \vartheta }
  } \,
  \| B( y ) \|_{ HS( U, H_{ -\nicefrac{ \vartheta }{ 2 } } ) } \,
  \| B'( y ) \|_{ L( H, HS( U, H_{ -\nicefrac{ \vartheta }{ 2 } } ) ) },
\end{split}
\end{equation}
\begin{equation}
\begin{split}
&
  \big\|
   \big(\tfrac{ \partial^2 }{ \partial y^2 }
   \psi\big)( x, y )( v_1, v_2 )
  \big\|_V
\\ & \leq
  2 
  \smallsum\limits_{ b \in \mathbb{U} }
  \big\|
  u_{0,2}( t, x )
  \big(
  ( B^b )'( y ) \, v_1,
  ( B^b )'( y ) \, v_2
  \big)
  +
    u_{0,2}( t, x )
    \big(
    B^b( y ),
    ( B^b )''( y )( v_1, v_2 )
    \big)
  \big\|_V
\\ & \leq
  \tfrac{
    2 \, c_{ -\nicefrac{ \vartheta }{ 2 }, -\nicefrac{ \vartheta }{ 2 } }
  }{
  ( T - t )^{ \vartheta }
  }
  \,
  \big(
  \| B'( y ) \|^2_{ L( H, HS( U, H_{ -\nicefrac{ \vartheta }{ 2 } } ) ) }
  +
  \| B( y ) \|_{ HS( U, H_{ -\nicefrac{ \vartheta }{ 2 } } ) } \,
  \| B''( y ) \|_{ L^{(2)}( H, HS( U, H_{ -\nicefrac{ \vartheta }{ 2 } } ) ) }
  \big)
  ,
\end{split}
\end{equation}
\begin{equation}
\label{eq:psiB_yderivatives_end}
\begin{split}
&
  \big\|
   \big(\tfrac{ \partial^3 }{ \partial y^3 }
   \psi\big)( x, y )( v_1, v_2, v_3 )
  \big\|_V
\\ & \leq
  2
  \smallsum\limits_{ b \in \mathbb{U} }
  \big\|
  u_{0,2}( t, x )
  \big(
  ( B^b )'( y ) \, v_2,
  ( B^b )''( y ) ( v_1, v_3 )
  \big)
\\ & \quad
  +
  u_{0,2}( t, x )
  \big(
  ( B^b )'( y ) \, v_1,
  ( B^b )''( y ) ( v_2, v_3 )
  \big)
\\ & \quad
    +
    u_{0,2}( t, x )
    \big(
    ( B^b )'( y ) \, v_3,
    ( B^b )''( y )( v_1, v_2 )
    \big)
\\ & \quad
    +
    u_{0,2}( t, x )
    \big(
    B^b( y ),
    ( B^b )^{(3)}( y )( v_1, v_2, v_3 )
    \big)
  \big\|_V
\\ & \leq
  \tfrac{
    2 \, c_{ -\nicefrac{ \vartheta }{ 2 }, -\nicefrac{ \vartheta }{ 2 } }
  }{
  ( T - t )^{ \vartheta }
  }
  \,
  \big(
  3 \, \| B'( y ) \|_{ L( H, HS( U, H_{ -\nicefrac{ \vartheta }{ 2 } } ) ) } \,
  \| B''( y ) \|_{ L^{(2)}( H, HS( U, H_{ -\nicefrac{ \vartheta }{ 2 } } ) ) }
\\ & \quad
  +
  \| B( y ) \|_{ HS( U, H_{ -\nicefrac{ \vartheta }{ 2 } } ) } \,
  \| B^{(3)}( y ) \|_{ L^{(3)}( H, HS( U, H_{ -\nicefrac{ \vartheta }{ 2 } } ) ) }
  \big)
  ,
\end{split}
\end{equation}
\begin{equation}
\label{eq:psiB_mixedderivatives_begin}
\begin{split}
&
  \big\|
   \big(\tfrac{ \partial^2 }{ \partial x \partial y }
   \psi\big)( x, y )( v_1, v_2 )
  \big\|_V
\leq
  2
  \smallsum\limits_{ b \in \mathbb{U} }
  \big\|
  u_{0,3}( t, x )
  \big(
  B^b( y ),
  ( B^b )'( y ) \, v_1,
  v_2
  \big)
  \big\|_V
\\ & \leq
  \tfrac{
    2 \, c_{ -\nicefrac{ \vartheta }{ 2 }, -\nicefrac{ \vartheta }{ 2 }, 0 }
  }{
  ( T - t )^{ \vartheta }
  }
  \,
  \| B( y ) \|_{ HS( U, H_{ -\nicefrac{ \vartheta }{ 2 } } ) } \,
  \| B'( y ) \|_{ L( H, HS( U, H_{ -\nicefrac{ \vartheta }{ 2 } } ) ) },
\end{split}
\end{equation}
\begin{equation}
\label{eq:psiB_mixedderivatives_middle}
\begin{split}
&
  \big\|
   \big(\tfrac{ \partial^3 }{ \partial x^2 \partial y }
   \psi\big)( x, y )( v_1, v_2, v_3 )
  \big\|_V
\\ & \leq
  2 
  \smallsum\limits_{ b \in \mathbb{U} }
  \big\|
  u_{0,4}( t, x )
  \big(
  B^b( y ),
  ( B^b )'( y ) \, v_1,
  v_2, v_3
  \big)
  \big\|_V
\\ & \leq
  \tfrac{
    2 \, c_{ -\nicefrac{ \vartheta }{ 2 }, -\nicefrac{ \vartheta }{ 2 }, 0, 0 }
  }{
  ( T - t )^{ \vartheta }
  } \,
  \,
  \| B( y ) \|_{ HS( U, H_{ -\nicefrac{ \vartheta }{ 2 } } ) } \,
  \| B'( y ) \|_{ L( H, HS( U, H_{ -\nicefrac{ \vartheta }{ 2 } } ) ) },
\end{split}
\end{equation}
\begin{align}
\label{eq:psiB_mixedderivatives_end}
&
  \big\|
   \big(\tfrac{ \partial^3 }{ \partial x \partial y^2 }
   \psi\big)( x, y )( v_1, v_2, v_3 )
  \big\|_V
\\ & \leq
\nonumber
  2
  \smallsum\limits_{ b \in \mathbb{U} }
  \big\|
  u_{0,3}( t, x )
  \big(
  ( B^b )'( y ) \, v_1,
  ( B^b )'( y ) \, v_2,
  v_3
  \big)
+
  u_{0,3}( t, x )
  \big(
  B^b( y ),
  ( B^b )''( y )( v_1, v_2 ),
  v_3
  \big)
  \big\|_V
\\ & \leq
\nonumber
  \tfrac{
    2 \, c_{ -\nicefrac{ \vartheta }{ 2 }, -\nicefrac{ \vartheta }{ 2 }, 0 }
  }{
  ( T - t )^{ \vartheta }
  }
  \,
  \big(
  \| B'( y ) \|^2_{ L( H, HS( U, H_{ -\nicefrac{ \vartheta }{ 2 } } ) ) }
  +
  \| B( y ) \|_{ HS( U, H_{ -\nicefrac{ \vartheta }{ 2 } } ) }
  \,
  \| B''( y ) \|_{ L^{(2)}( H, HS( U, H_{ -\nicefrac{ \vartheta }{ 2 } } ) ) }
  \big)
  .
\end{align}
Combining \eqref{eq:psiB_xderivatives_begin}--\eqref{eq:psiB_xderivatives_end} 
and \eqref{eq:psiB_mixedderivatives_begin}--\eqref{eq:psiB_mixedderivatives_end}
with the fundamental theorem of calculus in Banach spaces proves~\eqref{eq:psiB_Lipschitzx}. 
Moreover, combining \eqref{eq:psiB_yderivatives_begin}--\eqref{eq:psiB_mixedderivatives_end} with the fundamental theorem of calculus in Banach spaces establishes~\eqref{eq:psiB_Lipschitzy}. 
It thus remains to prove~\eqref{eq:psiB_Lipschitzxx}.
For this we observe that 
\eqref{eq:psiB_xderivatives_begin}--\eqref{eq:psiB_mixedderivatives_end} ensure that 
for all 
$ x, v_1, v_2, v_3 \in H $ with 
$ \|v_1\|_H, \|v_2\|_H, \|v_3\|_H \leq 1 $
it holds that 
\begin{equation}
\label{eq:psiB_xxderivatives_begin}
\begin{split}
&
  \big\|
   \phi'( x ) \, v_1
  \big\|_V
  \leq
  \big\|
   \big(\tfrac{ \partial }{ \partial x }
   \psi\big)( x, x ) \, v_1
  \big\|_V
  +
  \big\|
   \big(\tfrac{ \partial }{ \partial y }
   \psi\big)( x, x ) \, v_1
  \big\|_V
\\ & \leq
  \tfrac{
    c_{ -\nicefrac{ \vartheta }{ 2 }, -\nicefrac{ \vartheta }{ 2 }, 0 } \,
    \| B( x ) \|^2_{ HS( U, H_{ -\nicefrac{ \vartheta }{ 2 } } ) }
    +
    2 \,
    c_{ -\nicefrac{ \vartheta }{ 2 }, -\nicefrac{ \vartheta }{ 2 } } \,
    \| B( x ) \|_{ HS( U, H_{ -\nicefrac{ \vartheta }{ 2 } } ) } \,
    \| B'( x ) \|_{ L( H, HS( U, H_{ -\nicefrac{ \vartheta }{ 2 } } ) ) }
  }{
  ( T - t )^{ \vartheta }
  }
\\ & \leq
  \tfrac{
    2 \,
    [ c_{ -\nicefrac{ \vartheta }{ 2 }, -\nicefrac{ \vartheta }{ 2 } }
    +
    c_{ -\nicefrac{ \vartheta }{ 2 }, -\nicefrac{ \vartheta }{ 2 }, 0 } ]
  }{
  ( T - t )^{ \vartheta }
  }
  \,
  \| B \|^2_{ C^1_b( H, HS( U, H_{ -\nicefrac{ \vartheta }{ 2 } } ) ) } \,
  \max\{ 1, \| x \|^2_H \}
  ,
\end{split}
\end{equation}
\begin{equation}
\label{eq:psiB_xxderivatives_end}
\begin{split}
&
  \big\|
   \phi''( x ) \, ( v_1, v_2 )
  \big\|_V
\\ & \leq
  \big\|
   \big(\tfrac{ \partial^2 }{ \partial x^2 }
   \psi\big)( x, x ) \, ( v_1, v_2 )
  \big\|_V
  +
  2 \,
  \big\|
   \big(\tfrac{ \partial^2 }{ \partial x \partial y }
   \psi\big)( x, x ) \, ( v_1, v_2 )
  \big\|_V
  +
  \big\|
   \big(\tfrac{ \partial^2 }{ \partial y^2 }
   \psi\big)( x, x ) \, ( v_1, v_2 )
  \big\|_V
\\ & \leq
  \tfrac{
    c_{ -\nicefrac{ \vartheta }{ 2 }, -\nicefrac{ \vartheta }{ 2 }, 0, 0 } \,
    \| B( x ) \|^2_{ HS( U, H_{ -\nicefrac{ \vartheta }{ 2 } } ) }
    +
    4 \,
    c_{ -\nicefrac{ \vartheta }{ 2 }, -\nicefrac{ \vartheta }{ 2 }, 0 } \,
    \| B( x ) \|_{ HS( U, H_{ -\nicefrac{ \vartheta }{ 2 } } ) } \,
    \| B'( x ) \|_{ L( H, HS( U, H_{ -\nicefrac{ \vartheta }{ 2 } } ) ) }
  }{
  ( T - t )^{ \vartheta }
  }
\\ & \quad +
  \tfrac{
    2 \,
    c_{ -\nicefrac{ \vartheta }{ 2 }, -\nicefrac{ \vartheta }{ 2 } } \,
  \big(
  \| B'( x ) \|^2_{ L( H, HS( U, H_{ -\nicefrac{ \vartheta }{ 2 } } ) ) }
  +
  \| B( x ) \|_{ HS( U, H_{ -\nicefrac{ \vartheta }{ 2 } } ) } \,
  \| B''( x ) \|_{ L^{(2)}( H, HS( U, H_{ -\nicefrac{ \vartheta }{ 2 } } ) ) }
  \big)
  }{
  ( T - t )^{ \vartheta }
  }
\\ & \leq
  \tfrac{
    4 \,
    [
    c_{ -\nicefrac{ \vartheta }{ 2 }, -\nicefrac{ \vartheta }{ 2 } }
    +
    c_{ -\nicefrac{ \vartheta }{ 2 }, -\nicefrac{ \vartheta }{ 2 }, 0 } 
    +
    c_{ -\nicefrac{ \vartheta }{ 2 }, -\nicefrac{ \vartheta }{ 2 }, 0, 0 }
    ]
  }{
  ( T - t )^{ \vartheta }
  }
  \,
  \| B \|^2_{ C^2_b( H, HS( U, H_{ -\nicefrac{ \vartheta }{ 2 } } ) ) } \,
  \max\{ 1, \| x \|^2_H \}
  .
\end{split}
\end{equation}
In the next step we observe that 
\eqref{eq:psiB_yderivatives_end}, \eqref{eq:psiB_mixedderivatives_middle}, 
\eqref{eq:psiB_mixedderivatives_end},
and the fact that
$
  \big(
    H \ni x
    \mapsto
   \phi''( x )
   -
   \big(\tfrac{ \partial^2 }{ \partial x^2 }
   \psi\big)( x, x )
   \in L^{(2)}( H, V )
  \big)
  \in C^1( H, L^{(2)}( H, V ) )
$ 
show that for all $ x, x_1, x_2, v_1, v_2, v_3 \in H $ with 
$
  \|v_1\|_H, \|v_2\|_H,
  \|v_3\|_H
$
$
  \leq 1
$
it holds that 
\begin{align}
\label{eq:psiB_2ndxxLipschitz_begin}
&
  \big\|
   \tfrac{ \partial }{ \partial x }
   \big(
   \phi''( x )
   -
   \big(\tfrac{ \partial^2 }{ \partial x^2 }
   \psi\big)( x, x )
   \big)
   ( v_1, v_2, v_3 )
  \big\|_V
\leq
  2 \,
  \big\|
   \big(\tfrac{ \partial^3 }{ \partial x^2 \partial y }
   \psi\big)( x, x ) \, ( v_1, v_2, v_3 )
  \big\|_V
\\ & +
\nonumber
  3 \,
  \big\|
   \big(\tfrac{ \partial^3 }{ \partial x \partial y^2 }
   \psi\big)( x, x ) \, ( v_1, v_2, v_3 )
  \big\|_V
  +
  \big\|
   \big(\tfrac{ \partial^3 }{ \partial y^3 }
   \psi\big)( x, x ) \, ( v_1, v_2, v_3 )
  \big\|_V
\\ & \leq
\nonumber
  \tfrac{
    4 \, c_{ -\nicefrac{ \vartheta }{ 2 }, -\nicefrac{ \vartheta }{ 2 }, 0, 0 } \,
    \| B( x ) \|_{ HS( U, H_{ -\nicefrac{ \vartheta }{ 2 } } ) } \,
    \| B'( x ) \|_{ L( H, HS( U, H_{ -\nicefrac{ \vartheta }{ 2 } } ) ) }
  }{
  ( T - t )^{ \vartheta }
  }
\\ & +
\nonumber
  \tfrac{
    6 \,
    c_{ -\nicefrac{ \vartheta }{ 2 }, -\nicefrac{ \vartheta }{ 2 }, 0 } \,
  \big(
  \| B'( x ) \|^2_{ L( H, HS( U, H_{ -\nicefrac{ \vartheta }{ 2 } } ) ) }
  +
  \| B( x ) \|_{ HS( U, H_{ -\nicefrac{ \vartheta }{ 2 } } ) }
  \,
  \| B''( x ) \|_{ L^{(2)}( H, HS( U, H_{ -\nicefrac{ \vartheta }{ 2 } } ) ) }
  \big)
  }{
  ( T - t )^{ \vartheta }
  }
\\ & +
\nonumber
  \tfrac{
    6 \, c_{ -\nicefrac{ \vartheta }{ 2 }, -\nicefrac{ \vartheta }{ 2 } }
    \,
    \big(
    \| B'( x ) \|_{ L( H, HS( U, H_{ -\nicefrac{ \vartheta }{ 2 } } ) ) } \,
    \| B''( x ) \|_{ L^{(2)}( H, HS( U, H_{ -\nicefrac{ \vartheta }{ 2 } } ) ) }
    +
    \| B( x ) \|_{ HS( U, H_{ -\nicefrac{ \vartheta }{ 2 } } ) } \,
    \| B^{(3)}( x ) \|_{ L^{(3)}( H, HS( U, H_{ -\nicefrac{ \vartheta }{ 2 } } ) ) }
    \big)
  }{
  ( T - t )^{ \vartheta }
  }
\\ & \leq
\nonumber
  \tfrac{
    6 \,
    [
    c_{ -\nicefrac{ \vartheta }{ 2 }, -\nicefrac{ \vartheta }{ 2 } }
    +
    c_{ -\nicefrac{ \vartheta }{ 2 }, -\nicefrac{ \vartheta }{ 2 }, 0 } 
    +
    c_{ -\nicefrac{ \vartheta }{ 2 }, -\nicefrac{ \vartheta }{ 2 }, 0, 0 }
    ]
  }{
  ( T - t )^{ \vartheta }
  }
  \,
  \| B \|^2_{ C^3_b( H, HS( U, H_{ -\nicefrac{ \vartheta }{ 2 } } ) ) } \,
  \max\{ 1, \| x \|_H \}
  .
\end{align}
In addition, we combine \eqref{eq:psiB_xderivatives_end} 
and \eqref{eq:psiB_mixedderivatives_middle} 
with the fundamental theorem of calculus in Banach spaces 
to obtain that for all $ x_1, x_2, v_1, v_2 \in H $ with 
$
  \|v_1\|_H, \|v_2\|_H
  \leq 1
$
it holds that 
\begin{equation}
\label{eq:psiB_2ndxxLipschitz_end}
\begin{split}
&
  \big\|
   \big(
   \big(\tfrac{ \partial^2 }{ \partial x^2 }
   \psi\big)( x_1, x_1 )
   -
   \big(\tfrac{ \partial^2 }{ \partial x^2 }
   \psi\big)( x_2, x_2 )
   \big)
   ( v_1, v_2 )
  \big\|_V
\\ & \leq
  \big\|
   \big(
   \big(\tfrac{ \partial^2 }{ \partial x^2 }
   \psi\big)( x_1, x_1 )
   -
   \big(\tfrac{ \partial^2 }{ \partial x^2 }
   \psi\big)( x_2, x_1 )
   \big)
   ( v_1, v_2 )
  \big\|_V
\\ & 
  +
  \big\|
   \big(
   \big(\tfrac{ \partial^2 }{ \partial x^2 }
   \psi\big)( x_2, x_1 )
   -
   \big(\tfrac{ \partial^2 }{ \partial x^2 }
   \psi\big)( x_2, x_2 )
   \big)
   ( v_1, v_2 )
  \big\|_V
\\ & \leq
  \tfrac{
    \tilde{c}_{ -\nicefrac{ \vartheta }{ 2 }, -\nicefrac{ \vartheta }{ 2 }, 0, 0 } \,
    \| x_1 - x_2 \|_H
  }{
  ( T - t )^{ \vartheta }
  } \,
  \| B( x_1 ) \|^2_{ HS( U, H_{ -\nicefrac{ \vartheta }{ 2 } } ) }
\\ & 
  +
  \tfrac{
    2 \,
    c_{ -\nicefrac{ \vartheta }{ 2 }, -\nicefrac{ \vartheta }{ 2 }, 0, 0 } \,
    \| x_1 - x_2 \|_H
  }{
    ( T - t )^{ \vartheta }
  } 
  \,
  \| B \|^2_{ 
    C^1_b( H, HS( U, H_{ -\nicefrac{ \vartheta }{ 2 } } ) ) 
  } 
  \,
  \max\{ 1, \| x_1 \|_H, \| x_2 \|_H \}
\\ & \leq
  \tfrac{
    2 \,
    \| x_1 - x_2 \|_H
  }{
  ( T - t )^{ \vartheta }
  }
  \,
  \| B \|^2_{ C^1_b( H, HS( U, H_{ -\nicefrac{ \vartheta }{ 2 } } ) ) } \,
  \big[
    c_{ -\nicefrac{ \vartheta }{ 2 }, -\nicefrac{ \vartheta }{ 2 }, 0, 0 } 
    +
    \tilde{c}_{ -\nicefrac{ \vartheta }{ 2 }, -\nicefrac{ \vartheta }{ 2 }, 0, 0 }
  \big]
  \max\{ 1, \| x_1 \|^2_H, \| x_2 \|^2_H \}
  .
\end{split}
\end{equation}
Combining~\eqref{eq:psiB_xxderivatives_begin}--\eqref{eq:psiB_2ndxxLipschitz_end} with the fundamental theorem of calculus in Banach spaces 
finally yields~\eqref{eq:psiB_Lipschitzxx}. 
The proof of Lemma~\ref{lem:weak_regularity_B_2nd_term} is thus completed.
\end{proof}

\begin{lemma}[Weak convergence of semilinear integrated Euler-type 
approximations of SPDEs with mollified nonlinearities]
\label{lem:mollified_weak_solution-integrated_num}
Assume the setting in Section~\ref{sec:setting_euler_integrated_mollified} 
and let $ \rho \in [ 0, 1 - \vartheta ) $. 
Then it holds that 
$
  \ES\big[
  \| \varphi(X_T) \|_V
  +
  \| \varphi(\bar{Y}_T) \|_V
  \big]
  < \infty
$ 
and 
\begin{equation}
\label{eq:mollified_weak_solution-integrated_num}
\begin{split}
&
  \left\|
    \ES\big[
      \varphi( X_T )
    \big]
    -
    \ES\big[
      \varphi( \bar{Y}_T )
    \big]
  \right\|_V
\leq
  \tfrac{
    5 \, |\groupC_0|^3 \, | \groupC_\rho |^2 \, 
    \groupC_{ 0, \rho } \, T^{ ( 1 - \vartheta - \rho ) }
  }{ 
    ( 1 - \vartheta - \rho ) 
  } \, 
  \varsigma_{ F, B } \, K_5 \, h^\rho \,
\\ & \cdot
  \Big[
  c_{ -\vartheta } + c_{ -\vartheta, 0 }
  + c_{ -\vartheta, 0, 0 } + c_{ -\vartheta, 0, 0, 0 }
  +
  c_{ -\nicefrac{ \vartheta }{ 2 }, -\nicefrac{ \vartheta }{ 2 } }
  +
  c_{ -\nicefrac{ \vartheta }{ 2 }, -\nicefrac{ \vartheta }{ 2 }, 0 }
  +
  c_{ -\nicefrac{ \vartheta }{ 2 }, -\nicefrac{ \vartheta }{ 2 }, 0, 0 }
  +
  \tilde{c}_{ -\nicefrac{ \vartheta }{ 2 }, -\nicefrac{ \vartheta }{ 2 }, 0, 0 }
  \Big]
\\ & \cdot
  \Bigg[
    2^{ ( \rho + 1 ) }
    +
    \tfrac{ T^{ ( 1 - \vartheta ) } }{ ( 1 - \vartheta - \rho ) }
  \bigg(
       2 \, \groupC_\vartheta + \groupC_{ \rho + \vartheta } + 2 \, | \groupC_{ \nicefrac{\vartheta}{2} } |^2
       +
       2 \, \groupC_{ \rho + \nicefrac{\vartheta}{2} } \, \groupC_{ \nicefrac{\vartheta}{2} }
       +
    \groupC_{ \vartheta, 0, \rho } + 2 \, \groupC_{ \nicefrac{ \vartheta }{ 2 } } \, \groupC_{ \nicefrac{ \vartheta }{ 2 }, 0, \rho }
\\ & +
    3 \, ( | \groupC_{ \nicefrac{ \vartheta }{ 2 } } |^2 + \groupC_\vartheta ) 
+
  2 \, ( | \groupC_{ \nicefrac{ \vartheta }{ 2 } } |^2 + \groupC_\vartheta ) \, \groupC_\rho
  \Big[
    \groupC_{ -\rho, \rho }
    +
    \tfrac{
    \groupC_{ \vartheta, -\rho, \rho } \, T^{ ( 1 - \vartheta ) } \,
    }{ ( 1 - \vartheta ) }
    +
    \tfrac{
    \sqrt{6} \, \groupC_{ \nicefrac{\vartheta}{2}, -\rho, \rho } \, 
    T^{ ( 1 - \vartheta )/2 }
    }{
    \sqrt{ 1 - \vartheta }
    }
  \Big]
  \bigg)
  \Bigg]
  .
\end{split}
\end{equation}
\end{lemma}
\begin{proof}
First of all, we observe that the assumption that 
$
  \sup_{ t \in [ 0, T ] }
  \| X_t \|_{ \lpn{5}{\P}{H} }
  < \infty
$ 
implies that 
$
  \ES\big[
    \| \varphi(X_T) \|_V
  \big]
  < \infty
$. 
Moreover, combining the assumption that 
$
  Y_0
  \in \lpn{5}{\P}{H_1}
$ 
with Lemma~\ref{lem:Kp_estimate} 
proves that $ K_5 < \infty $. 
This shows, in particular, that we have that
$
  \sup_{ s \in [ 0, T ] } 
  \ES\big[
    \| \varphi(\bar{Y}_T) \|_V
    +
    \| \bar{Y}_s \|_{ H_1 } 
    +
      \int^T_0
      \|
      u_{ 0, 1 }( t, \bar{Y}_t )
      \,
      B( Y_{ \floor{t}{h} } )
      \|^2_{ HS( U, V ) }
      \, dt
  \big]
  < \infty
$.
This and the standard It{\^o} formula in Theorem~2.4 in
Brze\'{z}niak, Van Neerven, Veraar \citationand\ Weis~\cite{bvvw08} 
prove that
\begin{equation}
\begin{split}
&
    \ES\big[ 
      \varphi( \bar{Y}_T )
    \big]
    -
    \ES\big[ 
      \varphi( X_T )
    \big]
    =
  \E\left[ 
    u( T, \bar{Y}_T )
    -
    u( 0, \bar{Y}_0 )
  \right]
\\ & =
  \int_0^T
  \E\left[
    u_{1,0}( t, \bar{Y}_t )
    +
    u_{0,1}( t, \bar{Y}_t )
    \!\left(
      A
      \bar{Y}_t
      +
      F(
        Y_{ \floor{t}{h} } 
      )
    \right)
  \right]
  dt
\\ & 
  +
  \frac{ 1 }{ 2 }
  \sum_{ b\in\mathbb{U} }
  \int_0^T
  \E\left[
    u_{0,2}( t, \bar{Y}_t )\!\left(
      B^b(
        Y_{ \floor{t}{h} }
      )
      ,
      B^b(
        Y_{ \floor{t}{h} }
      )
    \right)
  \right]
  dt
  .
\end{split}
\end{equation}
Exploiting the fact that $ u $ is a solution of the Kolmogorov backward equation associated 
to 
$
  X^x \colon [0,T] \times \Omega \to H
$, 
$ x \in H $, 
and 
$ \varphi $
(cf., e.g., Theorem 7.5.1 in Da Prato \& Zabczyk~\cite{dz02b})
hence shows that
\begin{equation}
\label{eq:Kolmogorov}
\begin{split}
&
    \ES\big[ 
      \varphi( \bar{Y}_T )
    \big]
    -
    \ES\big[ 
      \varphi( X_T )
    \big]
\\ & =
  \int_0^T
  \E\left[
    u_{0,1}( t, \bar{Y}_t )
    \,
      F(
        Y_{ \floor{ t }{ h } } 
      )
  -
    u_{0,1}( t, \bar{Y}_t )
    \,
      F(
        \bar{Y}_t 
      )
  \right]
  dt
\\ & +
  \frac{ 1 }{ 2 }
  \sum_{ b\in\mathbb{U} }
  \int_0^T
  \E\left[
    u_{0,2}( t, \bar{Y}_t )\!\left(
      B^b\big(
        Y_{ \floor{ t }{ h } }
      \big)
      ,
      B^b\big(
        Y_{ \floor{ t }{ h } }
      \big)
    \right)
  -
    u_{0,2}( t, \bar{Y}_t )\!\left(
      B^b\big(
        \bar{Y}_t 
      \big)
      ,
      B^b\big(
        \bar{Y}_t 
      \big)
    \right)
  \right]
  dt
  .
\end{split}
\end{equation}
The triangle inequality hence shows that 
\begin{align}
\label{eq:mild_ito}
&
  \left\|
  \ES
  \big[
  \varphi( X_T )
  \big]
  -
  \ES
  \big[
  \varphi( \bar{Y}_T )
  \big]
  \right\|_V
\nonumber
\\ & \leq
\nonumber
  \smallint_0^T
  \left\|\E\left[
    u_{0,1}( t, \bar{Y}_t )
    \,
      F(
        Y_{ \floor{ t }{ h } } 
      )
    -
    u_{0,1}( t, \bar{Y}_{ \floor{ t }{ h } } )
    \,
      F(
        Y_{ \floor{ t }{ h } } 
      )
  \right]\right\|_V
  dt
\\ & +
\nonumber
  \smallint_0^T
  \left\|\E\left[
    u_{0,1}( t, \bar{Y}_{ \floor{ t }{ h } } )
    \,
      F(
        Y_{ \floor{ t }{ h } } 
      )
  -
    u_{0,1}( t, \bar{Y}_{ \floor{ t }{ h } } )
    \,
      F(
        \bar{Y}_{ \floor{ t }{ h } } 
      )
  \right]\right\|_V
  dt
\\ & +
  \smallint_0^T
  \left\|\E\left[
    u_{0,1}( t, \bar{Y}_{ \floor{ t }{ h } } )
    \,
      F(
        \bar{Y}_{ \floor{ t }{ h } } 
      )
    -
    u_{0,1}( t, \bar{Y}_t )
    \,
      F(
        \bar{Y}_t 
      )
  \right]\right\|_V
  dt
\\ & +
\nonumber
  \tfrac{ 1 }{ 2 } 
  \smallint_0^T
  \left\|\E\left[
  \,{\smallsum\limits_{ b\in\mathbb{U} }}
    u_{0,2}( t, \bar{Y}_t )\!\left(
      B^b\big(
        Y_{ \floor{ t }{ h } }
      \big)
      ,
      B^b\big(
        Y_{ \floor{ t }{ h } }
      \big)
    \right)
  -
  {\smallsum\limits_{ b\in\mathbb{U} }}
    u_{0,2}( t, \bar{Y}_{ \floor{ t }{ h } } )\!\left(
      B^b\big(
        Y_{ \floor{ t }{ h } }
      \big)
      ,
      B^b\big(
        Y_{ \floor{ t }{ h } }
      \big)
    \right)
  \right]\right\|_V
  dt
\\ & +
\nonumber
  \tfrac{ 1 }{ 2 } 
  \smallint_0^T
  \left\|\E\left[
  \,{\smallsum\limits_{ b\in\mathbb{U} }}
    u_{0,2}( t, \bar{Y}_{ \floor{ t }{ h } } )\!\left(
      B^b\big(
        Y_{ \floor{ t }{ h } }
      \big)
      ,
      B^b\big(
        Y_{ \floor{ t }{ h } }
      \big)
    \right)
  -
  {\smallsum\limits_{ b\in\mathbb{U} }}
    u_{0,2}( t, \bar{Y}_{ \floor{ t }{ h } } )\!\left(
      B^b\big(
        \bar{Y}_{ \floor{ t }{ h } } 
      \big)
      ,
      B^b\big(
        \bar{Y}_{ \floor{ t }{ h } } 
      \big)
    \right)
  \right]\right\|_V
  dt
\\ & +
\nonumber
  \tfrac{ 1 }{ 2 } 
  \smallint_0^T
  \left\|\E\left[
  \,{\smallsum\limits_{ b\in\mathbb{U} }}
    u_{0,2}( t, \bar{Y}_{ \floor{ t }{ h } } )\!\left(
      B^b\big(
        \bar{Y}_{ \floor{ t }{ h } }
      \big)
      ,
      B^b\big(
        \bar{Y}_{ \floor{ t }{ h } }
      \big)
    \right)
-
  {\smallsum\limits_{ b\in\mathbb{U} }}
    u_{0,2}( t, \bar{Y}_t )\!\left(
      B^b\big(
        \bar{Y}_t
      \big)
      ,
      B^b\big(
        \bar{Y}_t
      \big)
    \right)
  \right]\right\|_V
  dt
  .
\end{align}
In the next step we combine 
Lemma~\ref{lem:weak_regularity_F_2nd_term} and Lemma~\ref{lem:weak_regularity_B_2nd_term} 
with Proposition~\ref{prop:weak_temporal_regularity_1st} 
to obtain that for all 
$ t \in ( 0, T ) $ 
it holds that 
\allowdisplaybreaks
\begin{align}
\label{eq:mild_ito_1st_term}
&
  \left\|\E\left[
    u_{0,1}( t, \bar{Y}_t )
    \,
      F(
        Y_{ \floor{ t }{ h } } 
      )
    -
    u_{0,1}( t, \bar{Y}_{ \floor{ t }{ h } } )
    \,
      F(
        Y_{ \floor{ t }{ h } } 
      )
  \right]\right\|_V
\nonumber
\\ & +
\nonumber
  \left\|\E\left[
    u_{0,1}( t, \bar{Y}_{ \floor{ t }{ h } } )
    \,
      F(
        \bar{Y}_{ \floor{ t }{ h } } 
      )
    -
    u_{0,1}( t, \bar{Y}_t )
    \,
      F(
        \bar{Y}_t 
      )
  \right]\right\|_V
\\ & +
\nonumber
  \tfrac{ 1 }{ 2 } 
  \left\|\E\left[
  \,{\smallsum\limits_{ b\in\mathbb{U} }}
    u_{0,2}( t, \bar{Y}_t )\!\left(
      B^b\big(
        Y_{ \floor{ t }{ h } }
      \big)
      ,
      B^b\big(
        Y_{ \floor{ t }{ h } }
      \big)
    \right)
  -
  {\smallsum\limits_{ b\in\mathbb{U} }}
    u_{0,2}( t, \bar{Y}_{ \floor{ t }{ h } } )\!\left(
      B^b\big(
        Y_{ \floor{ t }{ h } }
      \big)
      ,
      B^b\big(
        Y_{ \floor{ t }{ h } }
      \big)
    \right)
  \right]\right\|_V
\\ & +
\nonumber
  \tfrac{ 1 }{ 2 } 
  \left\|\E\left[
  \,{\smallsum\limits_{ b\in\mathbb{U} }}
    u_{0,2}( t, \bar{Y}_{ \floor{ t }{ h } } )\!\left(
      B^b\big(
        \bar{Y}_{ \floor{ t }{ h } }
      \big)
      ,
      B^b\big(
        \bar{Y}_{ \floor{ t }{ h } }
      \big)
    \right)
  -
  {\smallsum\limits_{ b\in\mathbb{U} }}
    u_{0,2}( t, \bar{Y}_t )\!\left(
      B^b\big(
        \bar{Y}_t
      \big)
      ,
      B^b\big(
        \bar{Y}_t
      \big)
    \right)
  \right]\right\|_V
\\ & \leq
  \tfrac{
    | \groupC_0 |^3 \, 
    | \groupC_\rho |^2 \,}{ ( T - t )^\vartheta }
    \,
    K_5 
    \,
    h^\rho
    \,
  \max\!\big\{
    1,
    \|
      F
    \|_{ 
      \operatorname{Lip}^0( H, H_{-\vartheta} )
    }
    ,
    \|
      B
    \|^2_{ 
      \operatorname{Lip}^0( H, HS( U, H_{-\nicefrac{\vartheta}{2}} ) )
    }
  \big\}
\\ & \cdot
\nonumber
  \bigg[
  4 \,
  \big[
  c_{ -\vartheta } + c_{ -\vartheta, 0 }
  + c_{ -\vartheta, 0, 0 } + c_{ -\vartheta, 0, 0, 0 }
  \big] \,
  \| F \|_{ C^3_b( H, H_{ -\vartheta } ) }
\\ & +
\nonumber
  5 \,
  \big[
  c_{ -\nicefrac{ \vartheta }{ 2 }, -\nicefrac{ \vartheta }{ 2 } }
  +
  c_{ -\nicefrac{ \vartheta }{ 2 }, -\nicefrac{ \vartheta }{ 2 }, 0 }
  +
  c_{ -\nicefrac{ \vartheta }{ 2 }, -\nicefrac{ \vartheta }{ 2 }, 0, 0 }
  +
  \tilde{c}_{ -\nicefrac{ \vartheta }{ 2 }, -\nicefrac{ \vartheta }{ 2 }, 0, 0 }
  \big] \,
  \| B \|^2_{ C^3_b( H, HS( U, H_{ -\nicefrac{ \vartheta }{ 2 } } ) ) }
  \bigg]
\\ & \cdot
\nonumber
  \bigg[
    \tfrac{
      2^{ \rho }
    }{
      t^{ \rho }
    }
    +
    \tfrac{
      \left(
        2 \, 
        \groupC_\vartheta 
        +
        \groupC_{ \rho + \vartheta }
        +
        2 \, 
        | \groupC_{ \vartheta / 2 } |^2
        + 
        2 \,
        \groupC_{ \rho + \vartheta / 2 }
        \,
        \groupC_{ \vartheta / 2 }
      \right)
      \,
      |\floor{t}{h}|^{ ( 1 - \vartheta - \rho ) }
      +
      \left( 
        \groupC_\vartheta 
        +
        \frac{ 1 }{ 2 }
        | \groupC_{ \nicefrac{ \vartheta }{ 2 } } |^2 
      \right)
      \,
      ( t - \floor{t}{h} )^{ ( 1 - \vartheta - \rho ) }
    }{ 
      ( 1 - \vartheta - \rho ) }
  \bigg]
  .
\end{align}
In addition, we combine 
Lemma~\ref{lem:weak_regularity_F_2nd_term} and 
Lemma~\ref{lem:weak_regularity_B_2nd_term} 
with Proposition~\ref{prop:weak_temporal_regularity_2nd} 
and the fact that 
$
  \forall \,
  t \in [ h, T ]
  \colon
  \floor{t}{h}
  >
  t / 2
$
to obtain that for all 
$ t \in ( 0, T ) $ 
it holds that 
\allowdisplaybreaks
\begin{align}
\label{eq:mild_ito_2nd_term}
\nonumber
&
  \left\|\E\left[
    u_{0,1}( t, \bar{Y}_{ \floor{ t }{ h } } )
    \,
      F(
        Y_{ \floor{ t }{ h } } 
      )
  -
    u_{0,1}( t, \bar{Y}_{ \floor{ t }{ h } } )
    \,
      F(
        \bar{Y}_{ \floor{ t }{ h } } 
      )
  \right]\right\|_V
\\ & +
\nonumber
  \tfrac{ 1 }{ 2 } 
  \left\|\E\left[
  \,{\smallsum\limits_{ b\in\mathbb{U} }}
    u_{0,2}( t, \bar{Y}_{ \floor{ t }{ h } } )\!\left(
      B^b\big(
        Y_{ \floor{ t }{ h } }
      \big)
      ,
      B^b\big(
        Y_{ \floor{ t }{ h } }
      \big)
    \right)
  -
  {\smallsum\limits_{ b\in\mathbb{U} }}
    u_{0,2}( t, \bar{Y}_{ \floor{ t }{ h } } )\!\left(
      B^b\big(
        \bar{Y}_{ \floor{ t }{ h } } 
      \big)
      ,
      B^b\big(
        \bar{Y}_{ \floor{ t }{ h } } 
      \big)
    \right)
  \right]\right\|_V
\\ & \leq
\nonumber
  \tfrac{
    \groupC_0 \, \groupC_{ 0, \rho }
  }{
    ( T - t )^{ \vartheta }
  }
  \,
    K_4 \, h^\rho \,
  \max\!\big\{
    1,
    \|
      F
    \|_{ 
      \operatorname{Lip}^0( H, H_{-\vartheta} )
    }
    ,
    \|
      B
    \|^2_{ 
      \operatorname{Lip}^0( H, HS( U, H_{-\nicefrac{\vartheta}{2}} ) )
    }
  \big\}
\\ & \cdot
  \max\!\big\{
    1,
    \|
      F
    \|_{ 
      \operatorname{Lip}^0( H, H_{-\vartheta} )
    }
    ,
    \|
      B
    \|_{ 
      \operatorname{Lip}^0( H, HS( U, H_{-\nicefrac{\vartheta}{2}} ) )
    }
  \big\}
  \,
  \Big(
  \big[
  c_{ -\vartheta } + c_{ -\vartheta, 0 }
  + c_{ -\vartheta, 0, 0 }
  \big] \,
  \| F \|_{ C^3_b( H, H_{ -\vartheta } ) }
\\ & +
\nonumber
  3 \,
  \big[
  c_{ -\nicefrac{ \vartheta }{ 2 }, -\nicefrac{ \vartheta }{ 2 } }
  +
  c_{ -\nicefrac{ \vartheta }{ 2 }, -\nicefrac{ \vartheta }{ 2 }, 0 }
  +
  c_{ -\nicefrac{ \vartheta }{ 2 }, -\nicefrac{ \vartheta }{ 2 }, 0, 0 }
  \big] \,
  \| B \|^2_{ C^3_b( H, HS( U, H_{ -\nicefrac{ \vartheta }{ 2 } } ) ) }
  \Big)
\\ & \cdot
\nonumber
  \Bigg[
    \tfrac{
      2^\rho 
    }{
      t^\rho
    }
    +
    \tfrac{ | \floor{t}{h} |^{ ( 1 - \vartheta - \rho ) } }{ ( 1 - \vartheta - \rho ) }
  \bigg(
    \groupC_{ \vartheta, 0, \rho } + 2 \, \groupC_{ \nicefrac{ \vartheta }{ 2 } } \, \groupC_{ \nicefrac{ \vartheta }{ 2 }, 0, \rho }
    +
    2 \, ( | \groupC_{ \nicefrac{ \vartheta }{ 2 } } |^2 + \groupC_\vartheta ) 
+
  2 \, ( | \groupC_{ \nicefrac{ \vartheta }{ 2 } } |^2 + \groupC_\vartheta ) \, \groupC_\rho
\\ & \cdot
\nonumber
  \Big[
    \groupC_{ -\rho, \rho }
    +
    \tfrac{
    \groupC_{ \vartheta, -\rho, \rho } \, T^{ ( 1 - \vartheta ) } \,
    }{ ( 1 - \vartheta ) }
    +
    \tfrac{
    \sqrt{6} \, \groupC_{ \nicefrac{\vartheta}{2}, -\rho, \rho } \, 
    T^{ ( 1 - \vartheta )/2 }
    }{
    \sqrt{ 1 - \vartheta }
    }
  \Big]
  \bigg)
  \Bigg]
  .
\end{align}
Combining \eqref{eq:mild_ito}--\eqref{eq:mild_ito_2nd_term} proves that 
\begin{align}
\label{eq:mild_ito_final}
&
  \left\|
  \ES
  \big[
  \varphi( X_T )
  \big]
  -
  \ES
  \big[
  \varphi( \bar{Y}_T )
  \big]
  \right\|_V
\leq
  5 \, |\groupC_0|^3 \, | \groupC_\rho |^2 \, \groupC_{ 0, \rho } \, \varsigma_{ F, B } \, K_5 \, h^\rho \,
  \smallint^T_0
  \tfrac{1}{
    ( T - t )^\vartheta \, t^\rho
  }
  \, dt
\\ & \cdot
\nonumber
  \Big[
  c_{ -\vartheta } + c_{ -\vartheta, 0 }
  + c_{ -\vartheta, 0, 0 } + c_{ -\vartheta, 0, 0, 0 }
  +
  c_{ -\nicefrac{ \vartheta }{ 2 }, -\nicefrac{ \vartheta }{ 2 } }
  +
  c_{ -\nicefrac{ \vartheta }{ 2 }, -\nicefrac{ \vartheta }{ 2 }, 0 }
  +
  c_{ -\nicefrac{ \vartheta }{ 2 }, -\nicefrac{ \vartheta }{ 2 }, 0, 0 }
  +
  \tilde{c}_{ -\nicefrac{ \vartheta }{ 2 }, -\nicefrac{ \vartheta }{ 2 }, 0, 0 }
  \Big]
\\ & \cdot
\nonumber
  \Bigg[
    2^{ ( \rho + 1 ) }
    +
    \tfrac{ T^{ ( 1 - \vartheta ) } }{ ( 1 - \vartheta - \rho ) }
  \bigg(
       2 \, \groupC_\vartheta + \groupC_{ \rho + \vartheta } + 2 \, | \groupC_{ \nicefrac{\vartheta}{2} } |^2
       +
       2 \, \groupC_{ \rho + \nicefrac{\vartheta}{2} } \, \groupC_{ \nicefrac{\vartheta}{2} }
       +
    \groupC_{ \vartheta, 0, \rho } + 2 \, \groupC_{ \nicefrac{ \vartheta }{ 2 } } \, \groupC_{ \nicefrac{ \vartheta }{ 2 }, 0, \rho }
\\ & +
\nonumber
    3 \, ( | \groupC_{ \nicefrac{ \vartheta }{ 2 } } |^2 + \groupC_\vartheta ) 
+
  2 \, ( | \groupC_{ \nicefrac{ \vartheta }{ 2 } } |^2 + \groupC_\vartheta ) \, \groupC_\rho
  \Big[
    \groupC_{ -\rho, \rho }
    +
    \tfrac{
    \groupC_{ \vartheta, -\rho, \rho } \, T^{ ( 1 - \vartheta ) } \,
    }{ ( 1 - \vartheta ) }
    +
    \tfrac{
    \sqrt{6} \, \groupC_{ \nicefrac{\vartheta}{2}, -\rho, \rho } \, 
    T^{ ( 1 - \vartheta )/2 }
    }{
    \sqrt{ 1 - \vartheta }
    }
  \Big]
  \bigg)
  \Bigg]
  .
\end{align}
This and 
Lemma~3.1.6 in~\cite{Jentzen2014SPDElecturenotes} 
show\footnote{with $ x = 1 - \vartheta $ and $ y = 1 - \rho $ in the notation of 
Lemma~3.1.6 in~\cite{Jentzen2014SPDElecturenotes}}~\eqref{eq:mollified_weak_solution-integrated_num}. 
The proof of Lemma~\ref{lem:mollified_weak_solution-integrated_num}
is thus completed.
\end{proof}

\subsection{Weak convergence rates for Euler-type approximations of SPDEs with mollified nonlinearities}

The next result, Corollary~\ref{cor:weak_error_numerics-integrated},
provides a bound for the weak distance of the numerical approximation
and its semilinear integrated counterpart.
Corollary~\ref{cor:weak_error_numerics-integrated}
is an immediate consequence 
of Proposition~\ref{prop:weak_temporal_regularity_2nd}
and of Lemma~\ref{lem:mollified_weak_solution-integrated_num}.

\begin{corollary}[Weak distance between Euler-type approximations of SPDEs 
with mollified nonlinearities and their semilinear integrated counterparts]
\label{cor:weak_error_numerics-integrated}
Assume the setting in Section~\ref{sec:setting_euler_integrated_mollified} 
and let $ \rho \in [ 0, 1 - \vartheta ) $. Then it holds that 
$
  \ES\big[
  \| \varphi(\bar{Y}_T) \|_V
  +
  \| \varphi(Y_T) \|_V
  \big]
  < \infty
$ 
and 
\begin{align}
&
  \left\|
    \ES\big[ \varphi( \bar{Y}_T ) \big]
    -
    \ES\big[ \varphi( Y_T ) \big]
  \right\|_V
  \leq
  \tfrac{ \groupC_{ 0, \rho } }{ T^\rho } \,
  \| \varphi \|_{ \operatorname{Lip}^2( H, V ) }
  \,
  K_3 \, h^\rho \, \varsigma_{ F, B }
\\ & \cdot
\nonumber
  \Bigg[
    1
    +
    \tfrac{ T^{ ( 1 - \vartheta ) } }{ ( 1 - \vartheta - \rho ) }
  \bigg(
    \groupC_{ \vartheta, 0, \rho } + 2 \, \groupC_{ \nicefrac{ \vartheta }{ 2 } } \, \groupC_{ \nicefrac{ \vartheta }{ 2 }, 0, \rho }
    +
    2 \, ( | \groupC_{ \nicefrac{ \vartheta }{ 2 } } |^2 + \groupC_\vartheta )
\\ & +
  2 \, ( | \groupC_{ \nicefrac{ \vartheta }{ 2 } } |^2 + \groupC_\vartheta ) \, \groupC_\rho
\nonumber
  \Big[
    \groupC_{ -\rho, \rho }
    +
    \tfrac{
    \groupC_{ \vartheta, -\rho, \rho } \, T^{ ( 1 - \vartheta ) } \,
    }{ ( 1 - \vartheta ) }
    +
    \tfrac{
    \groupC_{ \nicefrac{\vartheta}{2}, -\rho, \rho } \, 
    \sqrt{ 3 \, T^{ ( 1 - \vartheta ) } }
    }{
    \sqrt{ 1 - \vartheta }
    }
  \Big]
  \bigg)
  \Bigg]
  .
\end{align}
\end{corollary}

The next result is a direct consequence of the triangle inequality,
of Corollary~\ref{cor:weak_error_numerics-integrated} 
and of Lemma~\ref{lem:mollified_weak_solution-integrated_num}.

\begin{corollary}[Weak convergence of Euler-type approximations of SPDEs with mollified nonlinearities]
\label{cor:mollified_weak_solution-num}
Assume the setting in Section~\ref{sec:setting_euler_integrated_mollified} and let 
$ \rho \in [ 0, 1 - \vartheta ) $. 
Then it holds that 
$
  \ES\big[
  \| \varphi(X_T) \|_V
  +
  \| \varphi(Y_T) \|_V
  \big]
  < \infty
$ 
and 
\begin{equation}\label{eq:mollified_weak_solution-num}
\begin{split}
&
  \left\|
  \ES
  \big[
  \varphi( X_T )
  \big]
  -
  \ES
  \big[
  \varphi( Y_T )
  \big]
  \right\|_V
\leq
  \tfrac{
    5 \, 
    | \groupC_0 |^3 \, 
    | \groupC_\rho |^2 \, 
    \groupC_{ 0, \rho } 
    \max\{ 1, T^{ ( 1 - \vartheta ) } \}
  }{ 
    ( 1 - \vartheta - \rho ) \, T^\rho 
  } 
  \, \varsigma_{ F, B } \, K_5 \, h^\rho 
\\ & \cdot
  \Bigg[
    2^{ ( \rho + 1 ) }
    +
    \tfrac{ T^{ ( 1 - \vartheta ) } }{ ( 1 - \vartheta - \rho ) }
  \bigg(
       2 \, \groupC_\vartheta + \groupC_{ \rho + \vartheta } + 2 \, | \groupC_{ \nicefrac{\vartheta}{2} } |^2
       +
       2 \, \groupC_{ \rho + \nicefrac{\vartheta}{2} } \, \groupC_{ \nicefrac{\vartheta}{2} }
       +
    \groupC_{ \vartheta, 0, \rho } + 2 \, \groupC_{ \nicefrac{ \vartheta }{ 2 } } \, \groupC_{ \nicefrac{ \vartheta }{ 2 }, 0, \rho }
\\ & +
    3 \, ( | \groupC_{ \nicefrac{ \vartheta }{ 2 } } |^2 + \groupC_\vartheta ) 
+
  2 \, ( | \groupC_{ \nicefrac{ \vartheta }{ 2 } } |^2 + \groupC_\vartheta ) \, \groupC_\rho
  \Big[
    \groupC_{ -\rho, \rho }
    +
    \tfrac{
    \groupC_{ \vartheta, -\rho, \rho } \, T^{ ( 1 - \vartheta ) } \,
    }{ ( 1 - \vartheta ) }
    +
    \tfrac{
    \sqrt{6} \, \groupC_{ \nicefrac{\vartheta}{2}, -\rho, \rho } \, 
    T^{ ( 1 - \vartheta )/2 }
    }{
    \sqrt{ 1 - \vartheta }
    }
  \Big]
  \bigg)
  \Bigg]
\\ & \cdot
  \Big[
  \| \varphi \|_{ \operatorname{Lip}^2( H, V ) }
  +
  c_{ -\vartheta } + c_{ -\vartheta, 0 }
  + c_{ -\vartheta, 0, 0 } + c_{ -\vartheta, 0, 0, 0 }
  +
  c_{ -\nicefrac{ \vartheta }{ 2 }, -\nicefrac{ \vartheta }{ 2 } }
  +
  c_{ -\nicefrac{ \vartheta }{ 2 }, -\nicefrac{ \vartheta }{ 2 }, 0 }
\\ & +
  c_{ -\nicefrac{ \vartheta }{ 2 }, -\nicefrac{ \vartheta }{ 2 }, 0, 0 }
+
  \tilde{c}_{ -\nicefrac{ \vartheta }{ 2 }, -\nicefrac{ \vartheta }{ 2 }, 0, 0 }
  \Big]
  .
\end{split}
\end{equation}
\end{corollary}

In the next result, 
Corollary~\ref{cor:mollified_weak_solution-num_komplete}, we use 
Lemma~\ref{lem:Kp_estimate}
to estimate the real number 
$ K_5 $ on the right hand side of~\eqref{eq:mollified_weak_solution-num}. 
For the formulation of Corollary~\ref{cor:mollified_weak_solution-num_komplete} 
we recall that for all 
$ x \in [ 0, \infty ) $, 
$ \theta \in [ 0, 1 ) $ 
it holds that 
$
  \mathcal{E}_{ 1 - \theta }( x )
  =
  \big[
    \sum^{ \infty }_{ n = 0 }
    \frac{ 
        x^{ 2 n }
        \,
        \Gamma( 1 - \theta )^n
    }{
      \Gamma( n ( 1 - \theta ) + 1 ) 
    }
  \big]^{
    \nicefrac{ 1 }{ 2 }
  }
$ 
(see Section~\ref{sec:notation}).

\begin{corollary}[Weak convergence of Euler-type approximations of SPDEs with mollified nonlinearities]
\label{cor:mollified_weak_solution-num_komplete}
Assume the setting in Section~\ref{sec:setting_euler_integrated_mollified} and let 
$ \theta \in [ 0, 1 ) $, 
$ \rho \in [ 0, 1 - \vartheta ) $. 
Then it holds that 
$
  \ES\big[
  \| \varphi(X_T) \|_V
  +
  \| \varphi(Y_T) \|_V
  \big]
  < \infty
$ 
and 
\begin{equation}
\begin{split}
&
  \left\|
  \ES
  \big[
  \varphi( X_T )
  \big]
  -
  \ES
  \big[
  \varphi( Y_T )
  \big]
  \right\|_V
\leq
  \tfrac{
    57 \, |\groupC_0|^3 \, | \groupC_\rho |^2 \, 
    \groupC_{ 0, \rho } 
    \max\{ 
      1 , T^{ ( 1 - \vartheta ) } 
    \}
  }{ 
    ( 1 - \vartheta - \rho ) \, T^\rho 
  } 
  \, 
  \varsigma_{ F, B } 
  \, 
  h^\rho 
\\ & \cdot
  \left[
    \groupC_0 
    \max\{ 1, \| X_0 \|_{ \lpn{5}{\P}{H} } \}
    +
    \tfrac{
      \groupC_\theta \,
      T^{ ( 1 - \theta ) } \,
      \| F \|_{ \operatorname{Lip}^0( H, H_{ -\theta } ) }
    }{
      ( 1 - \theta )
    }
      +
      \groupC_{ \nicefrac{ \theta }{ 2 } }
      \sqrt{
      \tfrac{
      10 \, T^{ ( 1 - \theta ) }
      }{
      ( 1 - \theta )
      }
      } \,
      \| B \|_{ \operatorname{Lip}^0( H, HS( U, H_{ -\nicefrac{\theta}{2} } ) ) }
  \right]^{ 10 }
\\ & \cdot
  \left|\mathcal{E}_{ ( 1 - \theta ) }\!\left[
    \tfrac{
      \sqrt{ 2 }
      \,
      \groupC_{ \theta }
      \,
      T^{ ( 1 - \theta ) }
      \,
      |
        F
      |_{
        \operatorname{Lip}^0( H, H_{ - \theta } )
      }
    }{
      \sqrt{1 - \theta}
    }
    +
    2 \, \groupC_{ 
      \theta / 2 
    }
    \sqrt{
      5 \, T^{ ( 1 - \theta ) }
    } \,
    |
      B
    |_{
      \operatorname{Lip}^0( H, HS( U, H_{ - \nicefrac{\theta}{2} } ) )
    }
  \right]\right|^5
\\ & \cdot
  \Bigg[
    2^{ (\rho+1) }
    +
    \tfrac{ T^{ ( 1 - \vartheta ) } }{ ( 1 - \vartheta - \rho ) }
  \bigg(
       2 \, \groupC_\vartheta + \groupC_{ \rho + \vartheta } + 2 \, | \groupC_{ \nicefrac{\vartheta}{2} } |^2
       +
       2 \, \groupC_{ \rho + \nicefrac{\vartheta}{2} } \, \groupC_{ \nicefrac{\vartheta}{2} }
       +
    \groupC_{ \vartheta, 0, \rho } + 2 \, \groupC_{ \nicefrac{ \vartheta }{ 2 } } \, \groupC_{ \nicefrac{ \vartheta }{ 2 }, 0, \rho }
\\ & +
    3 \, ( | \groupC_{ \nicefrac{ \vartheta }{ 2 } } |^2 + \groupC_\vartheta ) 
+
  2 \, ( | \groupC_{ \nicefrac{ \vartheta }{ 2 } } |^2 + \groupC_\vartheta ) \, \groupC_\rho
  \Big[
    \groupC_{ -\rho, \rho }
    +
    \tfrac{
    \groupC_{ \vartheta, -\rho, \rho } \, T^{ ( 1 - \vartheta ) } \,
    }{ ( 1 - \vartheta ) }
    +
    \tfrac{
    \sqrt{6} \, \groupC_{ \nicefrac{\vartheta}{2}, -\rho, \rho } \, 
    T^{ ( 1 - \vartheta )/2 }
    }{
    \sqrt{ 1 - \vartheta }
    }
  \Big]
  \bigg)
  \Bigg]
\\ & \cdot
  \Big[
  \| \varphi \|_{ \operatorname{Lip}^2( H, V ) }
  +
  c_{ -\vartheta } + c_{ -\vartheta, 0 }
  + c_{ -\vartheta, 0, 0 } + c_{ -\vartheta, 0, 0, 0 }
  +
  c_{ -\nicefrac{ \vartheta }{ 2 }, -\nicefrac{ \vartheta }{ 2 } }
  +
  c_{ -\nicefrac{ \vartheta }{ 2 }, -\nicefrac{ \vartheta }{ 2 }, 0 }
\\ & +
  c_{ -\nicefrac{ \vartheta }{ 2 }, -\nicefrac{ \vartheta }{ 2 }, 0, 0 }
+
  \tilde{c}_{ -\nicefrac{ \vartheta }{ 2 }, -\nicefrac{ \vartheta }{ 2 }, 0, 0 }
  \Big]
  .
\end{split}
\end{equation}
\end{corollary}

\section{Weak convergence rates for Euler-type approximations of SPDEs}
\label{sec:weak_convergence_irregular}

In this section we use Corollary~\ref{cor:mollified_weak_solution-num_komplete} 
in Section~\ref{sec:euler_integrated_mollified}
and the somehow non-standard mollification procedure 
in Conus et al.~\cite{ConusJentzenKurniawan2014arXiv}
to establish in Corollary~\ref{cor:weak_convergence_irregular}
weak convergence rates for temporal numerical 
approximations of a certain class of SEEs.
Corollary~\ref{cor:weak_convergence_irregular}, in turn, 
implies Theorem~\ref{intro:theorem} in the introduction.
The arguments in this section are quite similar to 
the arguments in Section~5 in Conus et al.~\cite{ConusJentzenKurniawan2014arXiv}.

\subsection{Setting}
\label{sec:setting_weak_convergence_irregular}

Assume the setting in Section~\ref{sec:semigroup_setting},
assume that 
$ h \leq T $, 
let 
$ \theta \in [0,1) $, 
$ \vartheta \in [ 0, \nicefrac{ 1 }{ 2 } ) \cap [ 0, \theta ] $, 
$
  F \in 
  C^5_b( H , H_{ - \theta } ) 
$, 
$
  B \in 
  C^5_b( 
    H, 
    HS( 
      U, 
      H_{ - \theta / 2 }
    ) 
  ) 
$, 
$
  \varphi \in 
  C^5_b( H, V )
$, 
let
$ \varsigma_{ F, B } \in \R $
be a real number given by 
$  
  \varsigma_{ F, B } 
  =
    \max\!\big\{
    1
    ,
    \|
      F
    \|^3_{ 
      C_b^3( H, H_{-\theta} )
    }
    ,
    \|
      B
    \|_{ 
      C_b^3( H, HS( U, H_{-\nicefrac{\theta}{2}} ) )
    }^6
    \big\}
$, 
let 
$
  X, Y \colon [0,T] \times \Omega \to H
$
and
$
  X^{ \kappa, x } \colon [0,T] \times \Omega \to H
$, 
$ \kappa \in [0,T] $, 
$ x \in H $, 
be 
$
  ( \mathcal{F}_t )_{ t \in [0,T] }
$-predictable stochastic processes 
which satisfy that for all 
$ \kappa \in [0,T] $, 
$ x \in H $ 
it holds that 
$
  \sup_{ t \in [0,T] }
  \big[
  \| X_t \|_{ \lpn{5}{\P}{H} } 
  +
  \| X^{ \kappa, x }_t \|_{ \lpn{5}{\P}{H} } 
  \big]
  < \infty
$, 
$
  X^{ \kappa, x }_0 = x
$, 
and 
$ Y_0 = X_0 $ 
and which satisfy that
for all 
$ \kappa \in [0,T] $, 
$ x \in H $, 
$ t \in (0,T] $ 
it holds $ \P $-a.s.\ that
\begin{equation}
  X_t
  = 
    e^{ t A } X_0 
  + 
    \int_0^t e^{ ( t - s )A }\, F( X_s ) \, ds
  + 
    \int_0^t e^{ ( t - s )A }\, B( X_s ) \, dW_s
    ,
\end{equation}
\begin{equation}
  X^{ \kappa, x }_t
  = 
    e^{ t A } \, x 
  + 
    \int_0^t e^{ ( \kappa + t - s )A }\, F( X^{ \kappa, x }_s ) \, ds
  + 
    \int_0^t e^{ ( \kappa + t - s )A }\, B( X^{ \kappa, x }_s ) \, dW_s
    ,
\end{equation}
\begin{equation}
  Y_t
  = 
    S_{ 0, t }\, Y_0 
  + 
    \int_0^t S_{ s, t }\, R_s\, F( Y_{ \floor{ s }{ h } } ) \, ds
  + 
    \int_0^t S_{ s, t }\, R_s\, B( Y_{ \floor{ s }{ h } } ) \, dW_s 
    ,
\end{equation}
let 
$
  u^{ ( \kappa ) } \colon [0,T] \times H \to V, 
$
$
  \kappa \in [ 0, T ],
$
be the functions 
with the property that
for all $ \kappa, t \in [ 0, T ] $, $ x \in H $ 
it holds that
$
  u^{ ( \kappa ) }( t, x ) = 
  \ES\big[ 
    \varphi( X^{ \kappa, x }_{ T - t } )
  \big]
$,
let 
$
  c^{ ( \kappa ) }_{ \delta_1, \dots, \delta_k }
  \in [0,\infty]	
$,
$ \delta_1, \dots, \delta_k \in \R $,
$ k \in \{ 1, 2, 3, 4 \} $, 
$ \kappa \in [ 0, T ] $, 
be the extended real numbers 
with the property that for all 
$ \kappa \in [ 0, T ] $, 
$ k \in \{ 1, 2, 3, 4 \} $,
$ \delta_1, \dots, \delta_k \in \R $
it holds that
\begin{equation}
\begin{split} 
&
  c^{ ( \kappa ) }_{ \delta_1, \delta_2, \dots, \delta_k }
=
  \sup_{
    t \in [0,T)
  }
  \sup_{ 
    x \in H
  }
  \sup_{ 
    v_1, \dots, v_k \in H \backslash \{ 0 \}
  }
  \left[
  \frac{
    \big\|
      \big( 
        \frac{ 
          \partial^k
        }{
          \partial x^k
        }
        u^{ ( \kappa ) }
      \big)( t, x )( v_1, \dots, v_k )
    \big\|_V
  }{
    ( T - t )^{ 
      (
        \delta_1 + \ldots + \delta_k
      ) 
    }
    \left\| v_1 \right\|_{ H_{ \delta_1 } }
    \cdot
    \ldots
    \cdot
    \left\| v_k \right\|_{ H_{ \delta_k } }
  }
  \right]
  ,
\end{split}
\end{equation}
and let 
$
  \tilde{c}^{ ( \kappa ) }_{ \delta_1, \delta_2, \delta_3, \delta_4 }
  \in [0,\infty]	
$,
$ \delta_1, \delta_2, \delta_3, \delta_4 \in \R $, 
$ \kappa \in [ 0, T ] $, 
be the extended real numbers 
with the property that
for all 
$ \kappa \in [ 0, T ] $, 
$ \delta_1, \delta_2, \delta_3, \delta_4 \in \R $
it holds that
\begin{equation}
\begin{split} 
&
  \tilde{c}^{ ( \kappa ) }_{ \delta_1, \delta_2, \delta_3, \delta_4 }
\\ & =
  \sup_{ t \in [ 0, T ) } \,
  \sup_{\substack{
    x_1, x_2 \in H, \\ x_1 \neq x_2
  }} \,
  \sup_{ 
    v_1, \ldots, v_4 \in H \backslash \{ 0 \}
  }
  \left[
  \frac{
    \big\|
    \big(
      \big( 
        \frac{ 
          \partial^4
        }{
          \partial x^4
        }
        u^{ ( \kappa ) }
      \big)( t, x_1 )
      -
      \big( 
        \frac{ 
          \partial^4
        }{
          \partial x^4
        }
        u^{ ( \kappa ) }
      \big)( t, x_2 )
      \big)
      ( v_1, \dots, v_4 )
    \big\|_V
  }{
    ( T - t )^{ 
      (
        \delta_1 + \ldots + \delta_4
      ) 
    }
    \left\| x_1 - x_2 \right\|_H
    \left\| v_1 \right\|_{ H_{ \delta_1 } }
     \cdot
     \ldots
    \cdot
    \left\| v_4 \right\|_{ H_{ \delta_4 } }
  }
  \right]
  .
\end{split}
\end{equation}

\subsection{Weak convergence result}
\label{sec:weak_convergence_irregular_result}

\begin{proposition}\label{prop:weak_convergence_irregular}
Assume the setting in Section~\ref{sec:setting_weak_convergence_irregular} 
and let 
$ r \in [ 0, 1 - \vartheta ) $, 
$ \rho \in ( 0, 1 - \theta ) $. 
Then it holds that 
$
  \ES\big[
  \| \varphi(X_T) \|_V
  +
  \| \varphi(Y_T) \|_V
  \big]
  < \infty
$
and 
\allowdisplaybreaks
\begin{align}
\label{eq:weak_convergence}
&
  \big\|
    \ES\big[ 
      \varphi( X_T )
    \big]
    -
    \ES\big[ 
      \varphi( Y_T )
    \big]
  \big\|_V
\leq
  \left[
    57 
    \left|
      \max\{ T, \tfrac{ 1 }{ T } \} 
    \right|^{
      ( r + 3 ( \theta - \vartheta ) )
    }
    | \groupC_0 |^{20} 
  \right]
  h^{
    \frac{ \rho \, r }{ ( \rho + 6 ( \theta - \vartheta ) ) }
  }
\nonumber
\\ & \cdot
\nonumber
  \left[
    \max\{
    1
    ,
    \| X_0 \|_{ \lpn{5}{\P}{H} }
    \}
+
  \tfrac{
    \groupC_\theta \, 
    \groupC_{ \nicefrac{\rho}{2} + \theta } 
    \,
    T^{ ( 1 - \theta ) }
    \,
    \| F \|_{ C^1_b( H, H_{ -\theta } ) }
  }{
    ( 1 - \theta - \nicefrac{\rho}{2} )
  }
  +
  \tfrac{
    \groupC_{ \nicefrac{\theta}{2} } \, 
    \groupC_{ \nicefrac{( \rho + \theta )}{2} } \,
    \sqrt{ 10 \, T^{ ( 1 - \theta ) } }
    \,
    \| B \|_{ C^1_b( H, HS( U, H_{ - \theta / 2 } ) ) }
  }{
    \sqrt{ 1 - \theta - \rho }
  }
  \right]^{ 10 }
\\ & \cdot
\nonumber
  \left|\mathcal{E}_{ ( 1 - \theta ) }\!\left[
    \tfrac{
      \sqrt{ 2 }
      \,
      \groupC_0 \, \groupC_{ \theta }
      \,
      T^{ ( 1 - \theta ) }
      \,
      |
        F
      |_{
        C^1_b( H, H_{ - \theta } )
      }
    }{
      \sqrt{1 - \theta}
    }
    +
    2 \, 
    \groupC_0 
    \, 
    \groupC_{ 
      \theta / 2 
    }
    \sqrt{
      5 \, T^{ ( 1 - \theta ) }
    } 
    \,
    |
      B
    |_{
      C^1_b( H, HS( U, H_{ - \nicefrac{\theta}{2} } ) )
    }
  \right]
  \right|^5
\\ & \cdot
  \Bigg[
    2^{(r+1)}
    +
    \tfrac{ T^{ ( 1 - \vartheta ) } }{ ( 1 - \vartheta - r ) }
  \bigg(
       2 \, \groupC_\vartheta + \groupC_{ r + \vartheta } + 2 \, | \groupC_{ \nicefrac{\vartheta}{2} } |^2
       +
       2 \, \groupC_{ r + \nicefrac{\vartheta}{2} } \, \groupC_{ \nicefrac{\vartheta}{2} }
       +
    \groupC_{ \vartheta, 0, r } + 2 \, \groupC_{ \nicefrac{ \vartheta }{ 2 } } \, \groupC_{ \nicefrac{ \vartheta }{ 2 }, 0, r }
\\ & +
\nonumber
  3 \, ( | \groupC_{ \nicefrac{ \vartheta }{ 2 } } |^2 + \groupC_\vartheta ) 
  +
  2 \, ( | \groupC_{ \nicefrac{ \vartheta }{ 2 } } |^2 + \groupC_\vartheta ) \, \groupC_r
  \Big[
    \groupC_{ -r, r }
    +
    \tfrac{
      \groupC_{ \vartheta, -r, r } \, 
      T^{ ( 1 - \vartheta ) } 
    }{ ( 1 - \vartheta ) }
    +
    \tfrac{
      \sqrt{6} \, \groupC_{ \nicefrac{\vartheta}{2}, -r, r } \, 
      T^{ ( 1 - \vartheta )/2 }
    }{
      \sqrt{ 1 - \vartheta }
    }
  \Big]
  \bigg)
  \Bigg]
\\ & \cdot
\nonumber
  \bigg[
  \tfrac{
  |\groupC_{ \nicefrac{\rho}{2} }|^2
  }{
  T^{ \rho/2 }
  } \,
  | \varphi |_{ C^1_b( H, V ) }
  +
  \tfrac{
    | \groupC_0 |^3 \, 
    | \groupC_r |^2 \, 
    \groupC_{ 0, r } \,
    | \groupC_{ \theta - \vartheta } |^3 
    \,
    | \groupC_{ \nicefrac{ ( \theta - \vartheta ) }{ 2 } } |^6 
    \max\{ 1, T^{ ( 1 - \vartheta ) } \}
    \,
    \varsigma_{ F, B } 
  }{ 
    ( 1 - \vartheta - r ) \, T^r 
  } 
  \,
  \Big(
  \| \varphi \|_{ C^3_b( H, V ) }
  +
  \sup_{ \kappa \in ( 0, T ] }
  \big[
  c^{ ( \kappa ) }_{ -\vartheta } 
\\ & +
\nonumber
  c^{ ( \kappa ) }_{ -\vartheta, 0 }
+ 
  c^{ ( \kappa ) }_{ -\vartheta, 0, 0 } + c^{ ( \kappa ) }_{ -\vartheta, 0, 0, 0 } 
  +
  c^{ ( \kappa ) }_{ -\nicefrac{ \vartheta }{ 2 }, -\nicefrac{ \vartheta }{ 2 } }
  +
  c^{ ( \kappa ) }_{ -\nicefrac{ \vartheta }{ 2 }, -\nicefrac{ \vartheta }{ 2 }, 0 }
+
  c^{ ( \kappa ) }_{ -\nicefrac{ \vartheta }{ 2 }, -\nicefrac{ \vartheta }{ 2 }, 0, 0 }
  +
  \tilde{c}^{ ( \kappa ) }_{ -\nicefrac{ \vartheta }{ 2 }, -\nicefrac{ \vartheta }{ 2 }, 0, 0 }
  \big]
  \Big)
  \bigg]
  < \infty
  .
\end{align}
\end{proposition}
\begin{proof}
First of all, we note that there exist 
up to modifications unique 
$
  ( \mathcal{F}_t )_{ t \in [0,T] }
$-predictable stochastic processes 
$
  \hat{Y}^{ \kappa, \delta } \colon [0,T] \times \Omega \to H
$, 
$ \kappa,\delta \in [0,T] $, 
and 
$
  \hat{X}^{ \kappa, \delta } \colon [0,T] \times \Omega \to H
$, 
$ \kappa,\delta \in [0,T] $, 
which satisfy that for all 
$ \kappa, \delta \in [0,T] $ 
it holds that 
$
  \sup_{ t \in [0,T] }
  \| \hat{X}^{ \kappa, \delta }_t \|_{ \lpn{5}{\P}{H} } 
  < \infty
$ 
and 
$
  \hat{X}^{ \kappa, \delta }_0
  =
  \hat{Y}^{ \kappa, \delta }_0
  =
  e^{ \delta A } X_0
$ 
and which satisfy that
for all 
$ \kappa, \delta \in [0,T] $, 
$ t \in (0,T] $ 
it holds $ \P $-a.s.\ that 
\begin{equation}
\label{eq:mollified_solution}
  \hat{X}_t^{ \kappa, \delta } 
= 
    e^{ t A } \hat{X}^{ \kappa, \delta }_0
  + 
    \int_0^t e^{ ( \kappa + t - s ) A } 
  F( \hat{X}_s^{ \kappa, \delta } ) \, ds
  +
    \int_0^t e^{ ( \kappa + t - s ) A } 
  B( \hat{X}_s^{ \kappa, \delta } ) \, dW_s 
  ,
\end{equation}
\begin{equation}
\label{eq:mollified_numerics}
  \hat{Y}_t^{ \kappa, \delta } 
= 
    S_{ 0, t } \, \hat{Y}^{ \kappa, \delta }_0
  + 
    \int_0^t S_{ s, t } \, R_s \, e^{ \kappa A } 
  F( \hat{Y}_{ \floor{ s }{ h } }^{ \kappa, \delta } ) \, ds
  +
    \int_0^t S_{ s, t } \, R_s \, e^{ \kappa A } 
  B( \hat{Y}_{ \floor{ s }{ h } }^{ \kappa, \delta } ) \, dW_s 
\end{equation}
(see, e.g., Proposition 3 in Da Prato et 
al.~\cite{DaPratoJentzenRoeckner2012}, Theorem 4.3 in Brze{\'z}niak \cite{b97b}, 
Theorem 6.2 in Van Neerven et al.~\cite{vvw08}).
In the next step we combine 
Lemma~\ref{lem:Kp_estimate} with the fact that 
$
  \forall \,
  \kappa, \delta \in [0,T]
  \colon
  \| \hat{Y}^{ \kappa, \delta }_0 \|_{ \lpn{5}{\P}{H} } 
  < \infty
$ 
to obtain that for all 
$ \kappa, \delta \in [ 0, T ] $ 
it holds that 
$
  \sup_{ t \in [0,T] }
  \| \hat{Y}^{ \kappa, \delta }_t \|_{ \lpn{5}{\P}{H} } 
  < \infty
$. 
This, the fact that
$
  \forall \, \kappa, \delta \in [0,T]
  \colon
  \sup_{ t \in [0,T] }
  \| 
    \hat{X}^{ \kappa, \delta }_t 
  \|_{ \lpn{5}{\P}{H} } 
  < \infty
$ 
and the assumption that 
$ \varphi \in \operatorname{Lip}^4( H, V ) $ 
ensure that 
for all $ \kappa, \delta \in [ 0, T ] $ it holds that 
\begin{equation}
  \ES\big[
    \| \varphi(\hat{X}^{ \kappa, \delta }_T) \|_V
    +
    \| \varphi(\hat{Y}^{ \kappa, \delta }_T) \|_V
  \big]
  < \infty
  .
\end{equation}
This proves, in particular, that
$
  \ES\big[
  \| \varphi(X_T) \|_V
  +
  \| \varphi(Y_T) \|_V
  \big]
  < \infty
$. 
It thus remains to show \eqref{eq:weak_convergence}. 
For this we observe that the triangle inequality ensures that for all 
$ \kappa, \delta \in [ 0, T ] $ 
it holds that 
\begin{equation}\label{eq:mollified_decompose_solution}
\begin{split}
&
  \big\|
    \ES\big[ 
      \varphi( \hat{X}^{ 0, \delta }_T )
    \big]
    -
    \ES\big[ 
      \varphi( \hat{Y}_T^{ 0, \delta } )
    \big]
  \big\|_V
\leq 
  \big\|
    \ES\big[ 
      \varphi( \hat{X}^{ 0, \delta }_T )
    \big]
    -
    \ES\big[ 
      \varphi( \hat{X}_T^{ \kappa, \delta } )
    \big]
  \big\|_V
\\ & +
  \big\|
    \ES\big[ 
      \varphi( \hat{X}_T^{ \kappa, \delta } )
    \big]
    -
    \ES\big[ 
      \varphi( \hat{Y}_T^{ \kappa, \delta } )
    \big]
  \big\|_V
  +
  \big\|
    \ES\big[ 
      \varphi( \hat{Y}_T^{ \kappa, \delta } )
    \big]
    -
    \ES\big[ 
      \varphi( \hat{Y}_T^{ 0, \delta } )
    \big]
  \big\|_V.
\end{split}
\end{equation}
In the following we provide suitable bounds for the three summands 
on the right hand side of \eqref{eq:mollified_decompose_solution}.
For the first and the third summand on the right hand side of \eqref{eq:mollified_decompose_solution}
we observe that  
Proposition~\ref{prop:strong_convergence_numerics} 
together with the fact that 
$
  \forall \, \kappa, \delta \in [ 0, T ]
  \colon
  \sup_{ t \in [0,T] }
  \| \hat{Y}^{ \kappa, \delta }_{ \floor{t}{h} } \|_{ \lpn{2}{\P}{H} } 
  \leq
  \sup_{ t \in [0,T] }
  \| \hat{Y}^{ \kappa, \delta }_t \|_{ \lpn{2}{\P}{H} } 
  < \infty
$
shows
that for all 
$ \kappa, \delta \in [0,T] $ 
it holds that
\begin{align}
\label{eq:mollified_strong_convergence}
&
  \big\|
    \ES\big[ 
      \varphi( \hat{X}^{ 0, \delta }_T )
    \big]
    -
    \ES\big[ 
      \varphi( \hat{X}_T^{ \kappa, \delta } )
    \big]
  \big\|_V
  +
  \big\|
    \ES\big[ 
      \varphi( \hat{Y}_T^{ \kappa,\delta} )
    \big]
    -
    \ES\big[ 
      \varphi( \hat{Y}_T^{ 0, \delta } )
    \big]
  \big\|_V
  \leq
  \tfrac{
    4 \, |\groupC_{ \nicefrac{\rho}{2} }|^2
  }{
    T^{ \rho/2 }
  } 
  \,
  | \varphi |_{ C^1_b( H, V ) } \,
  \kappa^{ \frac{ \rho }{ 2 } }
\\ & \cdot
\nonumber
  \left[
    \groupC_0 
    \max\{
      1
      ,
      \| e^{ \delta A } X_0 \|_{ \lpn{2}{\P}{H} }
    \}
+
  \tfrac{
    \groupC_\theta \, \groupC_{ \nicefrac{\rho}{2} + \theta } \,
    T^{ ( 1 - \theta ) }
    \,
    \| F \|_{ C^1_b( H, H_{ -\theta } ) }
  }{
    ( 1 - \theta - \nicefrac{\rho}{2} )
  }
  +
  \tfrac{
    \groupC_{ \nicefrac{\theta}{2} } \, \groupC_{ \nicefrac{( \rho + \theta )}{2} }
    \sqrt{ T^{ ( 1 - \theta ) } }
    \,
    \| B \|_{ C^1_b( H, HS( U, H_{ -\nicefrac{\theta}{2} } ) ) }
  }{
    \sqrt{ 1 - \theta - \rho }
  }
  \right]^2
\\ & \cdot
\nonumber
  \left|
  \mathcal{E}_{ ( 1 - \theta ) }\!\left[
    \tfrac{ 
      \sqrt{2} 
      \, 
      T^{ ( 1 - \theta ) } 
      \, 
      \groupC_0 
      \, 
      \groupC_\theta 
    }{ 
      \sqrt{ 1 - \theta } 
    }
    | F |_{ C^1_b( H, H_{ -\theta } ) }
    +
    \sqrt{ 
      2 \, T^{ ( 1 - \theta ) } 
    }
    \,
    \groupC_0 \, 
    \groupC_{ \nicefrac{ \theta }{ 2 } } \,
    | B |_{ 
      C^1_b( H, HS( U, H_{ -\nicefrac{ \theta }{ 2 } } ) ) 
    }
  \right]
  \right|^2
  .
\end{align}
Next we bound the second summand on the right hand side 
of \eqref{eq:mollified_decompose_solution}. For this we note that 
for all $ \kappa \in (0,T] $ it holds that
\begin{equation}
\begin{split}
&
  \max\!\big\{
    1
    ,
    \|
      e^{ \kappa A } F
    \|_{ 
      C_b^3( H, H_{-\vartheta} )
    }^3
    ,
    \|
      e^{ \kappa A } B
    \|_{ 
      C_b^3( H, HS( U, H_{-\nicefrac{\vartheta}{2}} ) )
    }^6
  \big\}
\\ & \leq
  | \groupC_{ \theta - \vartheta } |^3 
  \,
  | \groupC_{ \nicefrac{ ( \theta - \vartheta ) }{ 2 } } |^6 
  \,
  \varsigma_{ F, B } 
  \max\!\big\{ 
    1, 
    \kappa^{ - 3 ( \theta - \vartheta ) } 
  \big\}
  .
\end{split}
\end{equation}
This and Corollary~\ref{cor:mollified_weak_solution-num_komplete} show that for all 
$ \kappa, \delta \in ( 0, T ] $ 
it holds that 
\allowdisplaybreaks
\begin{align}
\label{eq:mollified_weak_convergence}
&
  \big\|
    \ES\big[ 
      \varphi( \hat{X}_T^{ \kappa, \delta } )
    \big]
    -
    \ES\big[ 
      \varphi( \hat{Y}_T^{ \kappa, \delta } )
    \big]
  \big\|_V
\nonumber
\\ & \leq
\nonumber
  \tfrac{
    57 \, 
    |\groupC_0|^3 \, 
    | \groupC_r |^2 \, 
    \groupC_{ 0, r } \,
    | \groupC_{ \theta - \vartheta } |^3 \,
    | \groupC_{ \nicefrac{ ( \theta - \vartheta ) }{ 2 } } |^6 
    \max\{ 1, T^{ ( 1 - \vartheta ) } \}
  }{ 
    ( 1 - \vartheta - r ) \; T^r 
  } 
  \,
  \varsigma_{ F, B } 
  \max\{ 1, \kappa^{ - 3 ( \theta - \vartheta ) } \} 
  \, 
  h^r
\\ & \cdot
\nonumber
  \left[
    \groupC_0 
    \max\{ 1, \| e^{ \delta A } X_0 \|_{ \lpn{5}{\P}{H} } \}
    +
    \tfrac{
      \groupC_\theta 
      \,
      T^{ ( 1 - \theta ) }
      \,
      \| e^{ \kappa A } F \|_{ C^1_b( H, H_{ -\theta } ) }
    }{
      ( 1 - \theta )
    }
      +
    \tfrac{
      \groupC_{ \nicefrac{ \theta }{ 2 } }
      \sqrt{ 10 \, T^{ ( 1 - \theta ) } } 
      \,
      \| e^{ \kappa A } B \|_{ C^1_b( H, HS( U, H_{ - \theta / 2 } ) ) }
    }{
      \sqrt{1 - \theta}
    }
  \right]^{ 10 }
\\ & \cdot
\nonumber
  \left|\mathcal{E}_{ ( 1 - \theta ) }\!\left[
    \tfrac{
      \sqrt{ 2 }
      \,
      \groupC_{ \theta }
      \,
      T^{ ( 1 - \theta ) }
      \,
      |
        e^{ \kappa A } F
      |_{
        C^1_b( H, H_{ - \theta } )
      }
    }{
      \sqrt{1 - \theta}
    }
    +
    2 \, \groupC_{ 
      \theta / 2 
    }
    \sqrt{
      5 \, T^{ ( 1 - \theta ) }
    } \,
    |
      e^{ \kappa A } B
    |_{
      C^1_b( H, HS( U, H_{ - \nicefrac{\theta}{2} } ) )
    }
  \right]\right|^5
\\ & \cdot
  \Bigg[
    2^{(r+1)}
    +
    \tfrac{ T^{ ( 1 - \vartheta ) } }{ ( 1 - \vartheta - r ) }
  \bigg(
       2 \, \groupC_\vartheta + \groupC_{ r + \vartheta } + 2 \, | \groupC_{ \nicefrac{\vartheta}{2} } |^2
       +
       2 \, \groupC_{ r + \nicefrac{\vartheta}{2} } \, \groupC_{ \nicefrac{\vartheta}{2} }
       +
    \groupC_{ \vartheta, 0, r } + 2 \, \groupC_{ \nicefrac{ \vartheta }{ 2 } } \, \groupC_{ \nicefrac{ \vartheta }{ 2 }, 0, r }
\\ & +
\nonumber
    3 \, ( | \groupC_{ \nicefrac{ \vartheta }{ 2 } } |^2 + \groupC_\vartheta ) 
+
  2 \, 
  ( | \groupC_{ \nicefrac{ \vartheta }{ 2 } } |^2 + \groupC_\vartheta ) \, 
  \groupC_r \,
  \Big[
    \groupC_{ -r, r }
    +
    \tfrac{
    \groupC_{ \vartheta, -r, r } \, T^{ ( 1 - \vartheta ) } \,
    }{ ( 1 - \vartheta ) }
    +
    \tfrac{
    \sqrt{6} \, \groupC_{ \nicefrac{\vartheta}{2}, -r, r } \, 
    T^{ ( 1 - \vartheta )/2 }
    }{
    \sqrt{ 1 - \vartheta }
    }
  \Big]
  \bigg)
  \Bigg]
\\ & \cdot
\nonumber
  \Big[
  \| \varphi \|_{ C^3_b( H, V ) }
  +
  c^{ ( \kappa ) }_{ -\vartheta } + c^{ ( \kappa ) }_{ -\vartheta, 0 }
  + c^{ ( \kappa ) }_{ -\vartheta, 0, 0 } + c^{ ( \kappa ) }_{ -\vartheta, 0, 0, 0 }
+
  c^{ ( \kappa ) }_{ -\nicefrac{ \vartheta }{ 2 }, -\nicefrac{ \vartheta }{ 2 } }
  +
  c^{ ( \kappa ) }_{ -\nicefrac{ \vartheta }{ 2 }, -\nicefrac{ \vartheta }{ 2 }, 0 }
\\ & +
\nonumber
  c^{ ( \kappa ) }_{ -\nicefrac{ \vartheta }{ 2 }, -\nicefrac{ \vartheta }{ 2 }, 0, 0 }
+
  \tilde{c}^{ ( \kappa ) }_{ -\nicefrac{ \vartheta }{ 2 }, -\nicefrac{ \vartheta }{ 2 }, 0, 0 }
  \,
  \Big]
  .
\end{align}
Plugging \eqref{eq:mollified_strong_convergence} and \eqref{eq:mollified_weak_convergence} 
into \eqref{eq:mollified_decompose_solution} 
then shows that for all 
$ \kappa, \delta \in ( 0, T ] $ 
it holds that 
\begin{align}
\label{eq:mollified_combined}
&
  \big\|
    \ES\big[ 
      \varphi( \hat{X}^{ 0, \delta }_T )
    \big]
    -
    \ES\big[ 
      \varphi( \hat{Y}_T^{ 0, \delta } )
    \big]
  \big\|_V
\leq
  \max\!\left\{
    4 \, \kappa^{ \frac{\rho}{2} }
    ,
    57 
    \max\!\big\{ 
      1 , 
      \kappa^{ - 3 ( \theta - \vartheta ) } 
    \big\} 
    \, h^r
  \right\}
  | \groupC_0 |^{ 20 } 
\nonumber
\\ & \cdot
  \left[
    \max\{
      1
      ,
      \| X_0 \|_{ \lpn{5}{\P}{H} }
    \}
    +
  \tfrac{
    \groupC_\theta \, \groupC_{ \nicefrac{\rho}{2} + \theta } \,
    T^{ ( 1 - \theta ) }
    \,
    \| F \|_{ C^1_b( H, H_{ -\theta } ) }
  }{
    ( 1 - \theta - \nicefrac{\rho}{2} )
  }
  +
  \tfrac{
    \groupC_{ \nicefrac{\theta}{2} } \, \groupC_{ \nicefrac{( \rho + \theta )}{2} }
    \sqrt{ 10 \, T^{ ( 1 - \theta ) } }
    \,
    \| B \|_{ C^1_b( H, HS( U, H_{ -\nicefrac{\theta}{2} } ) ) }
  }{
    \sqrt{ 1 - \theta - \rho }
  }
  \right]^{ 10 }
\nonumber
\\ & \cdot
\nonumber
  \left|
  \mathcal{E}_{ ( 1 - \theta ) }\!\left[
    \tfrac{
      \sqrt{ 2 }
      \,
      \groupC_0 \, \groupC_{ \theta }
      \,
      T^{ ( 1 - \theta ) }
      \,
      |
        F
      |_{
        C^1_b( H, H_{ - \theta } )
      }
    }{
      \sqrt{1 - \theta}
    }
    +
    2 \, \groupC_0 \, \groupC_{ 
      \theta / 2 
    }
    \sqrt{
      5 \, T^{ ( 1 - \theta ) }
    } \,
    |
      B
    |_{
      C^1_b( H, HS( U, H_{ - \nicefrac{\theta}{2} } ) )
    }
  \right]
  \right|^5
\\ & \cdot
  \Bigg[
    2^{(r+1)}
    +
    \tfrac{ T^{ ( 1 - \vartheta ) } }{ ( 1 - \vartheta - r ) }
  \bigg(
       2 \, \groupC_\vartheta + \groupC_{ r + \vartheta } + 2 \, | \groupC_{ \nicefrac{\vartheta}{2} } |^2
       +
       2 \, \groupC_{ r + \nicefrac{\vartheta}{2} } \, \groupC_{ \nicefrac{\vartheta}{2} }
       +
    \groupC_{ \vartheta, 0, r } + 2 \, \groupC_{ \nicefrac{ \vartheta }{ 2 } } \, \groupC_{ \nicefrac{ \vartheta }{ 2 }, 0, r }
\\ & +
\nonumber
    3 \, ( | \groupC_{ \nicefrac{ \vartheta }{ 2 } } |^2 + \groupC_\vartheta ) 
+
  2 \, ( | \groupC_{ \nicefrac{ \vartheta }{ 2 } } |^2 + \groupC_\vartheta ) \, \groupC_r
  \Big[
    \groupC_{ -r, r }
    +
    \tfrac{
    \groupC_{ \vartheta, -r, r } \, T^{ ( 1 - \vartheta ) } \,
    }{ ( 1 - \vartheta ) }
    +
    \tfrac{
    \sqrt{6} \, \groupC_{ \nicefrac{\vartheta}{2}, -r, r } \, 
    T^{ ( 1 - \vartheta )/2 }
    }{
    \sqrt{ 1 - \vartheta }
    }
  \Big]
  \bigg)
  \Bigg]
\\ & \cdot
\nonumber
  \bigg[
  \tfrac{
  |\groupC_{ \nicefrac{\rho}{2} }|^2
  }{
  T^{ \rho/2 }
  } \,
  | \varphi |_{ C^1_b( H, V ) }
  +
  \tfrac{
    |\groupC_0|^3 \, | \groupC_r |^2 \, \groupC_{ 0, r } \,
    | \groupC_{ \theta - \vartheta } |^3 \,
    | \groupC_{ \nicefrac{ ( \theta - \vartheta ) }{ 2 } } |^6 
    \max\{ 1, T^{ ( 1 - \vartheta ) } \}
  }{ 
    ( 1 - \vartheta - r ) \, T^r 
  } 
  \, 
  \varsigma_{ F, B } 
  \,
  \Big[
  \| \varphi \|_{ C^3_b( H, V ) }
  +
  c^{ ( \kappa ) }_{ -\vartheta } + c^{ ( \kappa ) }_{ -\vartheta, 0 }
\\ & + 
\nonumber
  c^{ ( \kappa ) }_{ -\vartheta, 0, 0 } 
+
  c^{ ( \kappa ) }_{ -\vartheta, 0, 0, 0 }
+
  c^{ ( \kappa ) }_{ -\nicefrac{ \vartheta }{ 2 }, -\nicefrac{ \vartheta }{ 2 } }
  +
  c^{ ( \kappa ) }_{ -\nicefrac{ \vartheta }{ 2 }, -\nicefrac{ \vartheta }{ 2 }, 0 }
+
  c^{ ( \kappa ) }_{ -\nicefrac{ \vartheta }{ 2 }, -\nicefrac{ \vartheta }{ 2 }, 0, 0 }
+
  \tilde{c}^{ ( \kappa ) }_{ -\nicefrac{ \vartheta }{ 2 }, -\nicefrac{ \vartheta }{ 2 }, 0, 0 }
  \,
  \Big]
  \bigg]
  .
\end{align}
In addition, we observe that
\begin{equation}\label{eq:optimal_rate}
\begin{split}
&
  \inf_{ \kappa \in ( 0, T ] }
  \max\!\left\{
    4 \, 
    \kappa^{ \frac{ \rho }{ 2 } }
    ,
    57 
    \max\!\big\{
      1,
      \kappa^{ -3( \theta - \vartheta ) }
    \big\} 
    \,
    h^r
  \right\}
\\ & \leq
  \max\!\left\{
  4 
  \left[
    \min\{ 1, T \}
    \left|
      \tfrac{h}{T}
    \right|^{ \frac{ 2r }{ ( \rho + 6 ( \theta - \vartheta ) ) } }
  \right]^{ \frac{ \rho }{ 2 } }
  ,
    57  
    \max\!\left\{
    1,
    \left[
  \min\{ 1, T \}
  \left|
  \tfrac{h}{T}
  \right|^{ \frac{ 2r }{ ( \rho + 6 ( \theta - \vartheta ) ) } }
    \right]^{ -3( \theta - \vartheta ) }
    \right\} 
    h^r
  \right\}
\\ & =
  \max\!\left\{
  4 
  \left[
    \min\{ 1, T \}
    \left|
      \tfrac{h}{T}
    \right|^{ \frac{ 2r }{ ( \rho + 6 ( \theta - \vartheta ) ) } }
  \right]^{ \frac{ \rho }{ 2 } }
  ,
    57  
    \,
    h^r
    \left[
  \min\{ 1, T \}
  \left|
  \tfrac{h}{T}
  \right|^{ \frac{ 2r }{ ( \rho + 6 ( \theta - \vartheta ) ) } }
    \right]^{ -3( \theta - \vartheta ) }
  \right\}
\\ & =
  \max\!\left\{
  \tfrac{
    4 \,
    | \min\{ 1, T \} |^{ \frac{ \rho }{ 2 } }
  }{
    T^{
      \frac{
      r \rho
      }{
      ( \rho + 6 ( \theta - \vartheta ) )
      }
    }
  }
  ,
  \tfrac{
    57 \,
    T^{
      \frac{
      6 ( \theta - \vartheta ) r
      }{
      ( \rho + 6 ( \theta - \vartheta ) )
      }
    }
  }{
    |\min\{ 1, T \}|^{ 3( \theta - \vartheta ) }
  }
  \right\}
  h^{
    \frac{ \rho \, r }{ ( \rho + 6 ( \theta - \vartheta ) ) }
  }
\\ & 
\leq
  57
  \max\!\left\{
  \tfrac{
    1
  }{
    \left| \min\{ 1, T \} \right|^r
  }
  ,
  \tfrac{
    \left| \max\{ 1, T \} \right|^r
  }{
    |\min\{ 1, T \}|^{ 3( \theta - \vartheta ) }
  }
  \right\}
  h^{
    \frac{ \rho \, r }{ ( \rho + 6 ( \theta - \vartheta ) ) }
  }
\leq
  \frac{
    57 
    \,
    h^{
      \frac{ \rho \, r }{ ( \rho + 6 ( \theta - \vartheta ) ) }
    }
  }{
    \left| 
      \min\{ T, \frac{ 1 }{ T } \} 
    \right|^{ ( r + 3( \theta - \vartheta ) ) }
  } 
  .
\end{split}
\end{equation}
Combining \eqref{eq:mollified_combined} and \eqref{eq:optimal_rate} yields that for all 
$ \delta \in (0,T] $ it holds that 
\allowdisplaybreaks
\begin{align}
\label{eq:weak_convergence_conclude}
&
  \big\|
    \ES\big[ 
      \varphi( \hat{X}^{ 0, \delta }_T )
    \big]
    -
    \ES\big[ 
      \varphi( \hat{Y}_T^{ 0, \delta } )
    \big]
  \big\|_V
\leq
  \left[
    57 
    \left| 
      \max\{ T, \tfrac{ 1 }{ T } \} 
    \right|^{ ( r + 3( \theta - \vartheta ) ) }
    | \groupC_0 |^{ 20 } 
  \right]
  h^{
    \frac{ \rho \, r }{ ( \rho + 6 ( \theta - \vartheta ) ) } 
  }
\\ & \cdot
\nonumber
  \left[
    \max\{
      1
      ,
      \| X_0 \|_{ \lpn{5}{\P}{H} }
    \}
    +
    \tfrac{
      \groupC_\theta \, \groupC_{ \nicefrac{\rho}{2} + \theta } \,
      T^{ ( 1 - \theta ) }
      \,
      \| F \|_{ C^1_b( H, H_{ -\theta } ) }
    }{
      ( 1 - \theta - \nicefrac{\rho}{2} )
    }
    +
    \tfrac{
      \groupC_{ \nicefrac{\theta}{2} } \, \groupC_{ \nicefrac{( \rho + \theta )}{2} }
      \sqrt{ 10 \, T^{ ( 1 - \theta ) } }
      \,
      \| B \|_{ C^1_b( H, HS( U, H_{ -\nicefrac{\theta}{2} } ) ) }
    }{
      \sqrt{ 1 - \theta - \rho }
    }
  \right]^{ 10 }
\\ & \cdot
\nonumber
  \left|\mathcal{E}_{ ( 1 - \theta ) }\!\left[
    \tfrac{
      \sqrt{ 2 }
      \,
      \groupC_0 \, \groupC_{ \theta }
      \,
      T^{ ( 1 - \theta ) }
      \,
      |
        F
      |_{
        C^1_b( H, H_{ - \theta } )
      }
    }{
      \sqrt{1 - \theta}
    }
    +
    2 \, \groupC_0 \, \groupC_{ 
      \theta / 2 
    }
    \sqrt{
      5 \, T^{ ( 1 - \theta ) }
    } \,
    |
      B
    |_{
      C^1_b( H, HS( U, H_{ - \nicefrac{\theta}{2} } ) )
    }
  \right]\right|^5
\\ & \cdot
\nonumber
  \bigg[
    2^{(r+1)}
    +
    \tfrac{ T^{ ( 1 - \vartheta ) } }{ ( 1 - \vartheta - r ) }
  \bigg(
       2 \, \groupC_\vartheta + \groupC_{ r + \vartheta } + 2 \, | \groupC_{ \nicefrac{\vartheta}{2} } |^2
       +
       2 \, \groupC_{ r + \nicefrac{\vartheta}{2} } \, \groupC_{ \nicefrac{\vartheta}{2} }
       +
    \groupC_{ \vartheta, 0, r } + 2 \, \groupC_{ \nicefrac{ \vartheta }{ 2 } } \, \groupC_{ \nicefrac{ \vartheta }{ 2 }, 0, r }
\\ & +
\nonumber
    3 \, ( | \groupC_{ \nicefrac{ \vartheta }{ 2 } } |^2 + \groupC_\vartheta ) 
+
  2 \, ( | \groupC_{ \nicefrac{ \vartheta }{ 2 } } |^2 + \groupC_\vartheta ) \, \groupC_r
  \Big[
    \groupC_{ -r, r }
    +
    \tfrac{
    \groupC_{ \vartheta, -r, r } \, T^{ ( 1 - \vartheta ) } \,
    }{ ( 1 - \vartheta ) }
    +
    \tfrac{
    \sqrt{6} \, \groupC_{ \nicefrac{\vartheta}{2}, -r, r } \, 
    T^{ ( 1 - \vartheta )/2 }
    }{
    \sqrt{ 1 - \vartheta }
    }
  \Big]
  \bigg)
  \bigg]
\\ & \cdot
\nonumber
  \bigg[
  \tfrac{
  |\groupC_{ \nicefrac{\rho}{2} }|^2
  }{
  T^{ \rho/2 }
  } \,
  | \varphi |_{ C^1_b( H, V ) }
  +
  \tfrac{
  |\groupC_0|^3 \, | \groupC_r |^2 \, \groupC_{ 0, r } \,
  | \groupC_{ \theta - \vartheta } |^3 \,
  | \groupC_{ \nicefrac{ ( \theta - \vartheta ) }{ 2 } } |^6 \,
  \max\{ 1, T^{ ( 1 - \vartheta ) } \}
  }{ ( 1 - \vartheta - r ) \, T^r } \, \varsigma_{ F, B } \, 
  \Big[
  \| \varphi \|_{ C^3_b( H, V ) }
  +
  \sup_{ \kappa \in ( 0, T ] }
  \big[
  c^{ ( \kappa ) }_{ -\vartheta } 
\\ & +
\nonumber
  c^{ ( \kappa ) }_{ -\vartheta, 0 }
+ 
  c^{ ( \kappa ) }_{ -\vartheta, 0, 0 } + c^{ ( \kappa ) }_{ -\vartheta, 0, 0, 0 } 
  +
  c^{ ( \kappa ) }_{ -\nicefrac{ \vartheta }{ 2 }, -\nicefrac{ \vartheta }{ 2 } }
  +
  c^{ ( \kappa ) }_{ -\nicefrac{ \vartheta }{ 2 }, -\nicefrac{ \vartheta }{ 2 }, 0 }
+
  c^{ ( \kappa ) }_{ -\nicefrac{ \vartheta }{ 2 }, -\nicefrac{ \vartheta }{ 2 }, 0, 0 }
  +
  \tilde{c}^{ ( \kappa ) }_{ -\nicefrac{ \vartheta }{ 2 }, -\nicefrac{ \vartheta }{ 2 }, 0, 0 }
  \big]
  \Big]
  \bigg]
  .
\end{align}
In the next step we 
note that Corollary~\ref{cor:initial_perturbation} 
together with Lebesgue's dominated convergence theorem yields that 
$
  \lim_{ \delta \to 0 }
  \ES\big[
    \varphi(
      \hat{X}^{ 0, \delta }_T
    )
  \big]
  =
  \ES\big[
    \varphi(
      X_T
    )
  \big]
$ 
and 
$
  \lim_{ \delta \to 0 }
  \ES\big[
    \varphi(
      \hat{Y}^{ 0, \delta }_T
    )
  \big]
  =
  \ES\big[
    \varphi(
      Y_T
    )
  \big]
$. 
Combining this with inequality~\eqref{eq:weak_convergence_conclude} proves the first 
inequality 
in~\eqref{eq:weak_convergence}. 
The second inequality 
in~\eqref{eq:weak_convergence}
follows from Andersson \& Jentzen~\cite{AnderssonJentzen2014}.
The proof of Proposition~\ref{prop:weak_convergence_irregular} is thus completed. 
\end{proof}

\begin{corollary}
\label{cor:weak_convergence_irregular}
Assume the setting in Section~\ref{sec:setting_weak_convergence_irregular} 
and let 
$ \rho \in ( 0, 1 - \theta ) \cap ( 6( \theta - \vartheta ), \infty ) $. 
Then it holds that 
$
  \ES\big[
  \| \varphi(X_T) \|_V
  +
  \| \varphi(Y_T) \|_V
  \big]
  < \infty
$
and 
\allowdisplaybreaks
\begin{align}
\label{eq:weak_convergence_simplified_rate}
&
  \big\|
    \ES\big[ 
      \varphi( X_T )
    \big]
    -
    \ES\big[ 
      \varphi( Y_T )
    \big]
  \big\|_V
\leq
  \left[
  \tfrac{
    57 
    \,
    | \groupC_0 |^{20} 
  }{
    |
      \!
      \min\{ T, \frac{ 1 }{ T } \} 
    |^{
      3 ( \rho + \theta )
    }
  }
  \right]
    h^{ 
      \left(
        \rho - 6 ( \theta - \vartheta ) 
      \right)
    }
\nonumber
\\ & \cdot
\nonumber
  \left[
    \max\{
    1
    ,
    \| X_0 \|_{ \lpn{5}{\P}{H} }
    \}
+
  \tfrac{
    \groupC_\theta 
    \, 
    \groupC_{ \nicefrac{\rho}{2} + \theta } 
    \,
    T^{ ( 1 - \theta ) }
    \,
    \| F \|_{ C^1_b( H, H_{ -\theta } ) }
  }{
    ( 1 - \theta - \nicefrac{\rho}{2} )
  }
  +
  \tfrac{
    \groupC_{ \nicefrac{\theta}{2} } \, \groupC_{ \nicefrac{( \rho + \theta )}{2} }
    \sqrt{ 10 \, T^{ ( 1 - \theta ) } }
    \,
    \| B \|_{ C^1_b( H, HS( U, H_{ -\nicefrac{\theta}{2} } ) ) }
  }{
    \sqrt{ 1 - \theta - \rho }
  }
  \right]^{10}
\\ & \cdot
\nonumber
  \left|
  \mathcal{E}_{ ( 1 - \theta ) }\!\left[
    \tfrac{
      \sqrt{ 2 }
      \,
      \groupC_0 \, \groupC_{ \theta }
      \,
      T^{ ( 1 - \theta ) }
      \,
      |
        F
      |_{
        C^1_b( H, H_{ - \theta } )
      }
    }{
      \sqrt{1 - \theta}
    }
    +
    2 \, \groupC_0 \, \groupC_{ 
      \theta / 2 
    }
    \sqrt{
      5 \, T^{ ( 1 - \theta ) }
    } \,
    |
      B
    |_{
      C^1_b( H, HS( U, H_{ - \nicefrac{\theta}{2} } ) )
    }
  \right]
  \right|^5
\\ & \cdot
  \Bigg[
    2^{(\rho+1)}
    +
    \tfrac{ T^{ ( 1 - \vartheta ) } }{ ( 1 - \vartheta - \rho ) }
  \bigg(
       2 \, \groupC_\vartheta + \groupC_{ \rho + \vartheta } + 2 \, | \groupC_{ \nicefrac{\vartheta}{2} } |^2
       +
       2 \, \groupC_{ \rho + \nicefrac{\vartheta}{2} } \, \groupC_{ \nicefrac{\vartheta}{2} }
       +
    \groupC_{ \vartheta, 0, \rho } + 2 \, \groupC_{ \nicefrac{ \vartheta }{ 2 } } \, \groupC_{ \nicefrac{ \vartheta }{ 2 }, 0, \rho }
\\ & +
\nonumber
    3 \, ( | \groupC_{ \nicefrac{ \vartheta }{ 2 } } |^2 + \groupC_\vartheta )
+
  2 \, ( | \groupC_{ \nicefrac{ \vartheta }{ 2 } } |^2 + \groupC_\vartheta ) \, \groupC_\rho
  \Big[
    \groupC_{ -\rho, \rho }
    +
    \tfrac{
    \groupC_{ \vartheta, -\rho, \rho } \, T^{ ( 1 - \vartheta ) } \,
    }{ ( 1 - \vartheta ) }
    +
    \tfrac{
    \sqrt{6} \, \groupC_{ \nicefrac{\vartheta}{2}, -\rho, \rho } \, 
    T^{ ( 1 - \vartheta )/2 }
    }{
    \sqrt{ 1 - \vartheta }
    }
  \Big]
  \bigg)
  \Bigg]
\\ & \cdot
\nonumber
  \bigg[
  \tfrac{
  |\groupC_{ \nicefrac{\rho}{2} }|^2
  }{
  T^{ \rho/2 }
  } \,
  | \varphi |_{ C^1_b( H, V ) }
  +
  \tfrac{
    |\groupC_0|^3 
    \, 
    | \groupC_\rho |^2 
    \, 
    \groupC_{ 0, \rho } 
    \,
    | \groupC_{ \theta - \vartheta } |^3 
    \,
    | \groupC_{ \nicefrac{ ( \theta - \vartheta ) }{ 2 } } |^6 
    \max\{ 
      1, T^{ ( 1 - \vartheta ) } 
    \}
    \, 
    \varsigma_{ F, B } 
  }{ 
    ( 1 - \vartheta - \rho ) \, T^\rho 
  } 
  \, 
  \Big(
  \| \varphi \|_{ C^3_b( H, V ) }
  +
  \sup_{ \kappa \in ( 0, T ] }
  \big[
  c^{ ( \kappa ) }_{ -\vartheta } 
\\ & +
\nonumber
  c^{ ( \kappa ) }_{ -\vartheta, 0 }
+ 
  c^{ ( \kappa ) }_{ -\vartheta, 0, 0 } + c^{ ( \kappa ) }_{ -\vartheta, 0, 0, 0 } 
  +
  c^{ ( \kappa ) }_{ -\nicefrac{ \vartheta }{ 2 }, -\nicefrac{ \vartheta }{ 2 } }
  +
  c^{ ( \kappa ) }_{ -\nicefrac{ \vartheta }{ 2 }, -\nicefrac{ \vartheta }{ 2 }, 0 }
+
  c^{ ( \kappa ) }_{ -\nicefrac{ \vartheta }{ 2 }, -\nicefrac{ \vartheta }{ 2 }, 0, 0 }
  +
  \tilde{c}^{ ( \kappa ) }_{ -\nicefrac{ \vartheta }{ 2 }, -\nicefrac{ \vartheta }{ 2 }, 0, 0 }
  \big]
  \Big)
  \bigg]
  < \infty.
\end{align}
\end{corollary}
\begin{proof}
First of all, we apply\footnote{with $ r = \rho $ in the notation of Proposition~\ref{prop:weak_convergence_irregular}} 
Proposition~\ref{prop:weak_convergence_irregular}
to obtain 
that 
$
  \ES\big[
  \| \varphi(X_T) \|_V
  +
  \| \varphi(Y_T) \|_V
  \big]
  < \infty
$
and 
\allowdisplaybreaks
\begin{align}
\label{eq:weak_convergence_simplified_rate_proof}
&
  \big\|
    \ES\big[ 
      \varphi( X_T )
    \big]
    -
    \ES\big[ 
      \varphi( Y_T )
    \big]
  \big\|_V
\leq
  \left[
    \tfrac{
      57 \,
      | \groupC_0 |^{20} \,
    }{
      | \min\{ T, \frac{ 1 }{ T } \} |^{ ( \rho + 3 ( \theta - \vartheta ) ) }
    } 
  \right]
  h^{
    \frac{ \rho^2 }{ ( \rho + 6 ( \theta - \vartheta ) ) } 
  }
\nonumber
\\ & \cdot
  \left[
    \max\{
    1
    ,
    \| X_0 \|_{ \lpn{5}{\P}{H} }
    \}
+
  \tfrac{
    \groupC_\theta \, \groupC_{ \nicefrac{\rho}{2} + \theta } \,
    T^{ ( 1 - \theta ) }
    \,
    \| F \|_{ C^1_b( H, H_{ -\theta } ) }
  }{
    ( 1 - \theta - \nicefrac{\rho}{2} )
  }
  +
  \tfrac{
    \groupC_{ \nicefrac{\theta}{2} } \, \groupC_{ \nicefrac{( \rho + \theta )}{2} }
    \sqrt{ 10 \, T^{ ( 1 - \theta ) } }
    \,
    \| B \|_{ C^1_b( H, HS( U, H_{ -\nicefrac{\theta}{2} } ) ) }
  }{
    \sqrt{ 1 - \theta - \rho }
  }
  \right]^{ 10 }
\nonumber
\\ & \cdot
  \left|
  \mathcal{E}_{ ( 1 - \theta ) }\!\left[
    \tfrac{
      \sqrt{ 2 }
      \,
      \groupC_0 \, \groupC_{ \theta }
      \,
      T^{ ( 1 - \theta ) }
      \,
      |
        F
      |_{
        C^1_b( H, H_{ - \theta } )
      }
    }{
      \sqrt{1 - \theta}
    }
    +
    2 \, \groupC_0 \, \groupC_{ 
      \theta / 2 
    }
    \sqrt{
      5 \, T^{ ( 1 - \theta ) }
    } \,
    |
      B
    |_{
      C^1_b( H, HS( U, H_{ - \nicefrac{\theta}{2} } ) )
    }
  \right]
  \right|^5
\nonumber
\\ & \cdot
  \Bigg[
    2^{(\rho+1)}
    +
    \tfrac{ T^{ ( 1 - \vartheta ) } }{ ( 1 - \vartheta - \rho ) }
  \bigg(
       2 \, \groupC_\vartheta + \groupC_{ \rho + \vartheta } + 2 \, | \groupC_{ \nicefrac{\vartheta}{2} } |^2
       +
       2 \, \groupC_{ \rho + \nicefrac{\vartheta}{2} } \, \groupC_{ \nicefrac{\vartheta}{2} }
       +
    \groupC_{ \vartheta, 0, \rho } + 2 \, \groupC_{ \nicefrac{ \vartheta }{ 2 } } \, \groupC_{ \nicefrac{ \vartheta }{ 2 }, 0, \rho }
\\ & +
\nonumber
    3 \, ( | \groupC_{ \nicefrac{ \vartheta }{ 2 } } |^2 + \groupC_\vartheta ) 
+
  2 \, ( | \groupC_{ \nicefrac{ \vartheta }{ 2 } } |^2 + \groupC_\vartheta ) \, \groupC_\rho
  \Big[
    \groupC_{ -\rho, \rho }
    +
    \tfrac{
    \groupC_{ \vartheta, -\rho, \rho } \, T^{ ( 1 - \vartheta ) } \,
    }{ ( 1 - \vartheta ) }
    +
    \tfrac{
    \sqrt{6} \, \groupC_{ \nicefrac{\vartheta}{2}, -\rho, \rho } \, 
    T^{ ( 1 - \vartheta )/2 }
    }{
    \sqrt{ 1 - \vartheta }
    }
  \Big]
  \bigg)
  \Bigg]
\\ & \cdot
\nonumber
  \bigg[
  \tfrac{
  |\groupC_{ \nicefrac{\rho}{2} }|^2
  }{
  T^{ \rho/2 }
  } \,
  | \varphi |_{ C^1_b( H, V ) }
  +
  \tfrac{
    | \groupC_0 |^3 \, 
    | \groupC_\rho |^2 \, 
    \groupC_{ 0, \rho } \,
    | \groupC_{ \theta - \vartheta } |^3 \,
    | \groupC_{ \nicefrac{ ( \theta - \vartheta ) }{ 2 } } |^6 
    \max\{ 
      1, T^{ ( 1 - \vartheta ) } 
    \}
    \, 
    \varsigma_{ F, B }
  }{ 
    ( 1 - \vartheta - \rho ) \, T^\rho 
  }  
  \, 
  \Big(
  \| \varphi \|_{ C^3_b( H, V ) }
  +
  \sup_{ \kappa \in ( 0, T ] }
  \big[
  c^{ ( \kappa ) }_{ -\vartheta } 
\\ & +
\nonumber
  c^{ ( \kappa ) }_{ -\vartheta, 0 }
+ 
  c^{ ( \kappa ) }_{ -\vartheta, 0, 0 } + c^{ ( \kappa ) }_{ -\vartheta, 0, 0, 0 } 
  +
  c^{ ( \kappa ) }_{ -\nicefrac{ \vartheta }{ 2 }, -\nicefrac{ \vartheta }{ 2 } }
  +
  c^{ ( \kappa ) }_{ -\nicefrac{ \vartheta }{ 2 }, -\nicefrac{ \vartheta }{ 2 }, 0 }
+
  c^{ ( \kappa ) }_{ -\nicefrac{ \vartheta }{ 2 }, -\nicefrac{ \vartheta }{ 2 }, 0, 0 }
  +
  \tilde{c}^{ ( \kappa ) }_{ -\nicefrac{ \vartheta }{ 2 }, -\nicefrac{ \vartheta }{ 2 }, 0, 0 }
  \big]
  \Big)
  \bigg]
  < \infty
  .
\end{align}
Next we note that 
\begin{equation}
\label{eq:simplified_rate}
\begin{split}
&
  h^{
    { \frac{ \rho^2 }{ ( \rho + 6 ( \theta - \vartheta ) ) } }
  }
  =
  h^{ 
      \rho 
      \left[
        \frac{ 1 }{ 1 + 6 ( \theta - \vartheta ) / \rho } 
        -
        1 
        + \frac{ 6 ( \theta - \vartheta ) }{ \rho }
      \right]
    }
    \,
    h^{ 
      \rho 
      \left[
        1 - \frac{ 6 ( \theta - \vartheta ) }{ \rho }
      \right]
    }
\\ & 
  \leq
  |\! \max\{ 1, T \} |^{
    \rho 
    \left[
      \frac{ 1 }{ 1 + 6 ( \theta - \vartheta ) / \rho } 
      -
      1 
      + \frac{ 6 ( \theta - \vartheta ) }{ \rho }
    \right]
  }
    \,
    h^{ 
      \left(
        \rho - 6 ( \theta - \vartheta ) 
      \right)
    }
\leq
  |\! \max\{ 1, T \} |^\rho
    \,
    h^{ 
      \left(
        \rho - 6 ( \theta - \vartheta ) 
      \right)
    }
    .
\end{split}
\end{equation}
Plugging \eqref{eq:simplified_rate} into \eqref{eq:weak_convergence_simplified_rate_proof} 
implies \eqref{eq:weak_convergence_simplified_rate}. 
This completes the proof of Corollary~\ref{cor:weak_convergence_irregular}.
\end{proof}

\section{Lower bounds for weak errors of Euler-type approximations for SPDEs}
\label{sec:lower_bound}

In this section a few specific lower bounds for weak approximation errors
of temporal numerical approximations are established in the case of concrete example SEEs. 
A few specific lower bounds for weak approximation errors of spatial spectral Galerkin approximations
can be found in Section~6 in Conus et al.~\cite{ConusJentzenKurniawan2014arXiv}.
Lower bounds for strong approximation errors for examples of SEEs and for whole classes of SEEs can be found in 
\cite{dg01,mr07a,mrw08} and the references mentioned therein. 
The article \cite{mrw08} and Section~5 in \cite{ConusJentzenKurniawan2014arXiv} study exclusively parabolic SEEs driven by additive noise.
The papers \cite{dg01,mr07a} investigate parabolic SEEs driven by possibly non-additive noise.
In this section we consider exclusively parabolic SEEs driven by additive noise.

\subsection{Setting}
\label{sec:setting_lower_bound}

Throughout Section~\ref{sec:lower_bound} the following setting is frequently used. 
Let 
$ 
  ( H, \left< \cdot, \cdot \right>_H, \left\| \cdot \right\|_H ) 
$
be a separable $ \R $-Hilbert space
with $ H \neq \{ 0 \} $,
let $ \set \subseteq H $ be an orthonormal basis of $ H $,
let 
$ h \in ( 0, \infty ) $, 
$ T \in \{ h, 2 h , 3 h , \dots \} $, 
$ \beta \in [ 0, \frac{ 1 }{ 2 } ) $,
$ 
  \angle 
  =
  \left\{
    ( t_1, t_2 ) \in [0,T]^2 \colon t_1 < t_2
  \right\}
$, 
let 
$ \lambda, \mu \colon \set \rightarrow \R $ be functions
such that 
$ 
  \sup_{ b \in \set } \lambda_b < 0
$ 
and 
$
  \sum_{ b \in \set }
  | \mu_b |^2
  \,
  | \lambda_b |^{ -2 \beta }
  < \infty
$, 
let 
$ A \colon D(A) \subseteq H \to H $ be a linear operator 
such that
$
  D( A ) 
  = 
  \{ 
    v \in H 
  \colon 
    \sum_{ b \in \set }
    \left| 
      \lambda_b
      \left< b, v \right>_H 
    \right|^2
  < \infty
  \}
$
and such that
for all $ v \in D(A) $ it holds that
$
  A v = 
  \sum_{ b \in \set }
  \lambda_b
  \left< b, v \right>_H
  b
$,
let
$  
  ( 
    H_r 
    ,
    \left< \cdot, \cdot \right>_{ H_r }
    ,$ $
    \left\| \cdot \right\|_{ H_r } 
  )
$,
$ r \in \R $,
be a family of interpolation spaces associated to $ - A $ 
(see, e.g., Theorem and Definition~2.5.32 in \cite{Jentzen2014SPDElecturenotes})), 
let $ ( \Omega, \mathcal{F}, \P, ( \mathcal{F}_t )_{ t \in [0,T] } ) $
be a stochastic basis,
let $ ( W_t )_{ t \in [0,T] } $ be a cylindrical
$ \operatorname{Id}_H $-Wiener process w.r.t.\ $ ( \mathcal{F}_t )_{ t \in [0,T] } $,
and 
let 
$
  B \in 
    HS( 
      H, 
      H_{ -\beta }
    ) 
$ 
satisfy that
for all $ v \in H $
it holds that
$
  B v = 
  \sum_{ b \in \set }
  \mu_b
  \left< b, v \right>_H
  b
$. 
The above assumptions 
ensure 
that there exist 
$ 
  X, Y_1, Y_2
  \in \mathcal{M}\big( \mathcal{F}, \mathcal{B}(H) \big)
$ 
which satisfy that
it holds $ \P $-a.s.\ that
$
  X
  = 
    \int_0^T e^{ ( T - s )A } \, B \, dW_s 
$, 
$
  Y_1
  = 
    \int_0^T e^{ ( T - \floor{ s }{ h } )A } \, B \, dW_s 
$, 
and 
$
  Y_2
  = 
    \int_0^T
    \big( \operatorname{Id}_H - hA \big)^{
      -
      (
      T - \floor{ s }{ h }
      )
      /
      h
    }
    \, B
    \, dW_s
$.

\subsection{Variance estimates for Euler-type approximations of SPDEs}

\begin{lemma}[Variance estimates for exponential Euler approximations]
\label{lem:variance_differece_expE}
Assume the setting in Section~\ref{sec:setting_lower_bound} 
and let $ b \in \mathbb{H} $. 
Then it holds that 
$
  \ES\big[
    |
      \langle
        b, X
      \rangle_H
    |^2
    +
    |
      \langle
        b, Y_1
      \rangle_H
    |^2
  \big]
  < \infty
$ 
and 
\begin{equation}
\label{eq:variance_differece_expE}
\begin{split}
  \operatorname{Var}\!\big(
  \big<
    b, X
  \big>_H
  \big)
  -
  \operatorname{Var}\!\big(
  \big<
    b, Y_1
  \big>_H
  \big)
\geq
  \frac{
    | \mu_b |^2
    \left(
      1 
      - 
      e^{
        - 2 | \lambda_b | T
      }
    \right)
    h
  }{
    4 \, 
    e^{
      | \lambda_b | h
    }
  }
  \geq
  0.
\end{split}
\end{equation}
\end{lemma}
\begin{proof}
First, observe that it holds $ \P $-a.s.\ that
\begin{equation}
\label{eq:variance_exponentialE}
\begin{split}
&
  \big< 
    b,
    Y_1
  \big>_H
=
  \left< 
    b,
    \int_0^T
    e^{ ( T - \floor{ s }{ h } )A } B \, dW_s
  \right>_H
=
  \int_0^T
  \left< 
    e^{ ( T - \floor{ s }{ h } )A }
    \, b,
    B \, dW_s
  \right>_H
\\ & =
  \int_0^T
  e^{ -|\lambda_b| \left( T - \floor{ s }{ h } \right) }
  \left< 
    b,
    B \, dW_s
  \right>_H
=
  \mu_b
  \int_0^T
  e^{ -|\lambda_b| \left( T - \floor{ s }{ h } \right) }
  \left< 
    b,
    dW_s
  \right>_H
  .
\end{split}
\end{equation}
This shows that
$
  \ES\big[
  |\langle
    b, X
  \rangle_H|^2
  +
  |\langle
    b, Y_1
  \rangle_H|^2
  \big]
  < \infty
$. 
It thus remains to prove \eqref{eq:variance_differece_expE}.
For this we combine \eqref{eq:variance_exponentialE}, It\^{o}'s isometry, the fact that 
$
  \forall \, s \in [ 0, T ]
 \colon 
  T
  -
  \floor{ T - s }{ h }
  =
  \ceil{ s }{ h }
$, 
and, e.g., Lemma~6.1 in~\cite{ConusJentzenKurniawan2014arXiv} 
to obtain that 
\begin{equation}
\label{eq:norm_square_lb_exponential_Euler}
\begin{split}
&
  \operatorname{Var}\!\big(
  \big<
    b, X
  \big>_H
  \big)
  -
  \operatorname{Var}\!\big(
  \big<
    b, Y_1
  \big>_H
  \big)
=
  | \mu_b |^2
  \int^T_0
  e^{
    - 2 | \lambda_b | s 
    }
  -
  e^{
   - 2 | \lambda_b | \ceil{ s }{ h } 
   }
  \,
  ds
\\ & =
  | \mu_b |^2
  \int^T_0
  e^{
  - 2 | \lambda_b | \ceil{ s }{ h } 
  }
  \left(
  e^{
    2 | \lambda_b | 
    ( \ceil{ s }{ h } - s ) 
    }
  -
  1
  \right)
  ds
=
  | \mu_b |^2
  \left(
  \sum^{ \nicefrac{T}{h} }_{ k = 1 }
  e^{
  - 2 | \lambda_b |
  k h 
  }
  \right)
  \int^h_0
  \left(
  e^{
    2 | \lambda_b | s
    }
  -
  1
  \right)
  ds
\\ & =
  | \mu_b |^2
  \left(
  \frac{
  1 - e^{
    - 2 | \lambda_b | T
    }
  }{
  1 - e^{
    - 2 | \lambda_b | h
    }
  }
  \right)
  \int^h_0
  \left(
  e^{
    - 2 | \lambda_b | s
    }
  -
  e^{
      - 2 | \lambda_b | h
      }
  \right)
  dt
  .
\end{split}
\end{equation}
Moreover, note that 
\begin{equation}
\label{eq:integral_approx}
\begin{split}
  \int^h_0
  \left(
  e^{
    - 2 | \lambda_b | s
    }
  -
  e^{
      - 2 | \lambda_b | h
      }
  \right)
  ds
& \geq
  \int^{ 
  \nicefrac{ h }{ 2 } 
  }_0
  \left(
  e^{
    - 2 | \lambda_b | s
    }
  -
  e^{
      - 2 | \lambda_b | h
      }
  \right)
  ds
\geq
  \frac{ h }{ 2 }
  \left(
  e^{
    - | \lambda_b | h
    }
  -
  e^{
      - 2 | \lambda_b | h
      }
  \right)
  .
\end{split}
\end{equation}
Combining~\eqref{eq:norm_square_lb_exponential_Euler} and \eqref{eq:integral_approx} yields that 
\begin{equation}
\begin{split}
  \operatorname{Var}\!\big(
  \big<
    b, X
  \big>_H
  \big)
  -
  \operatorname{Var}\!\big(
  \big<
    b, Y_1
  \big>_H
  \big)
\geq
  \frac{
    | \mu_b |^2
    \,
    e^{ -|\lambda_b| h }
    \left(
      1 - 
      e^{
        - 2 | \lambda_b | T
      }
    \right)
  }{ 2 }
  \left(
  \frac{
  1
  -
  e^{
      - | \lambda_b | h
      }
  }{
  1 - e^{
    - 2 | \lambda_b | h
    }
  }
  \right)
  h
  .
\end{split}
\end{equation}
This and the fact that 
$
  \forall \, x \in [ 0, 1 ) 
  \colon
  \nicefrac{
  ( 1 - x )
  }{
  ( 1 - x^2 )
  }
  =
  \nicefrac{ 1 }{
  ( 1 + x )
  }
  \geq
  \nicefrac{ 1 }{ 2 }
$ 
finish the proof of Lemma~\ref{lem:variance_differece_expE}.
\end{proof}

\begin{lemma}[Variance estimates for linear-implicit Euler approximations]
\label{lem:variance_difference_linear_implicitE}
Assume the setting in Section~\ref{sec:setting_lower_bound} 
and let 
$ b \in \mathbb{H} $. 
Then it holds that 
$
  \ES\big[
  |\langle
    b, X
  \rangle_H|^2
  +
  |\langle
    b, Y_2
  \rangle_H|^2
  \big]
  < \infty
$ 
and 
\begin{equation}
\label{eq:variance_difference_linear_implicitE}
\begin{split}
&
  \operatorname{Var}\!\big(
  \big<
    b, X
  \big>_H
  \big)
  -
  \operatorname{Var}\!\big(
  \big<
    b, Y_2
  \big>_H
  \big)
\geq
  \frac{
    | \mu_b |^2
    \left(
      1 - 
      e^{
        - 2 | \lambda_b | T
      }
    \right)
    h
  }{ 
    4
    \left( 
      1 + h | \lambda_b | 
    \right) 
  }
\geq
  \frac{
    | \mu_b |^2
    \left(
      1 
      - 
      e^{
        - 2 | \lambda_b | T
      }
    \right)
    h
  }{
    4 \, 
    e^{
      | \lambda_b | h
    }
  }
  \geq
  0
  .
\end{split}
\end{equation}
\end{lemma}

\begin{proof}
First, 
we observe that it holds $ \P $-a.s.\ that
\begin{equation}
\label{eq:variance_implicitE}
\begin{split}
&
  \big< 
    b,
    Y_2
  \big>_H
\\ & =
  \left< 
    b,
    \int_0^T
    \big(
    \operatorname{Id}_H
    -
    h A
    \big)^{ -( T - \floor{ s }{ h } )/h }
    B \, dW_s
  \right>_H
=
  \int_0^T
  \left< 
    \big(
    \operatorname{Id}_H
    -
    h A
    \big)^{ -( T - \floor{ s }{ h } )/h }
    \, b ,
    B \, dW_s
  \right>_H
\\ & =
  \int_0^T
  ( 1 + h|\lambda_b| )^{ -( T - \floor{ s }{ h } )/h }
  \left< 
    b,
    B \, dW_s
  \right>_H
=
  \mu_b
  \int_0^T
  ( 1 + h|\lambda_b| )^{ -( T - \floor{ s }{ h } )/h }
  \left< 
    b,
    dW_s
  \right>_H
  .
\end{split}
\end{equation}
This shows that
$
  \ES\big[
  |\langle
    b, X
  \rangle_H|^2
  +
  |\langle
    b, Y_2
  \rangle_H|^2
  \big]
  < \infty
$. 
It thus remains to prove \eqref{eq:variance_difference_linear_implicitE}.
For this we combine \eqref{eq:variance_implicitE}, It\^{o}'s isometry, the fact that 
$
  \forall \, s \in [ 0, T ] 
  \colon
  T
  -
  \floor{ T - s }{ h }
  =
  \ceil{ s }{ h }
$, 
and, e.g.,
Lemma 6.1 in \cite{ConusJentzenKurniawan2014arXiv}
to obtain that 
\begin{equation}
\label{eq:geometric_sum_approach_implicit}
\begin{split}
&
  \operatorname{Var}\!\big(
  \big<
    b, X
  \big>_H
  \big)
  -
  \operatorname{Var}\!\big(
  \big<
    b, Y_2
  \big>_H
  \big)
  =
  | \mu_b |^2
  \left[
    \frac{
      1
      -
      e^{ - 2 | \lambda_b | T }
    }{
      2 | \lambda_b |
    }
    -
    \int^T_0
    ( 1 + h | \lambda_b | )^{
      - 2 \ceil{ s }{ h } / h
    }
    \,
  ds
  \right]
\\ & =
  | \mu_b |^2
  \left[
    \frac{
      1
      -
      e^{ - 2 | \lambda_b | T }
    }{
      2 \, | \lambda_b |
    }
    -
    h
    \sum^{ \nicefrac{ T }{ h } }_{ k = 1 }
    \left( 
      1 + h | \lambda_b | 
    \right)^{
      - 2 k
    }
  \right]
=
  | \mu_b |^2
  \left[
    \frac{
      1
      -
      e^{ - 2 | \lambda_b | T }
    }{
      2 \, | \lambda_b |
    }
    -
    \frac{
      \left[
        1
        -
        ( 1 + h | \lambda_b | )^{
         - 2 T / h  
        }
      \right]
    }{
      | \lambda_b |
      \left( 2 + h | \lambda_b | \right)
    }
  \right]
  .
\end{split}
\end{equation}
The fact that 
$
  \forall \, x \in [ 0, \infty ) 
  \colon
  ( 1 + x )^{ -1 }
  \geq
  e^{ -x }
$ 
hence yields that 
\begin{equation}
\begin{split}
  \operatorname{Var}\!\big(
  \big<
    b, X
  \big>_H
  \big)
  -
  \operatorname{Var}\!\big(
  \big<
    b, Y_2
  \big>_H
  \big)
& \geq
  | \mu_b |^2
  \left(
      1
      -
      e^{ - 2 | \lambda_b | T }
  \right)
  \left[
  \frac{
  1
  }{
    2 \, | \lambda_b |
  }
  -
  \frac{
    1
  }{
    | \lambda_b | \,
    ( 2 + h | \lambda_b | )
  }
  \right]
\\ & =
  \frac{
    | \mu_b |^2
    \left(
      1
      -
      e^{ - 2 | \lambda_b | T }
    \right)
    h
  }{
    2 
    \left( 2 + h | \lambda_b | \right)
  }
  .
\end{split}
\end{equation}
This implies \eqref{eq:variance_difference_linear_implicitE}.
The proof of Lemma~\ref{lem:variance_difference_linear_implicitE} is thus completed.
\end{proof}

\subsection{Lower bounds for the squared norm as the test function}

\begin{proposition}\label{prop:variance_difference_concrete_lower_bound}
Assume the setting in Section~\ref{sec:setting_lower_bound}, 
let $ b \colon \N \to \set $ be a bijective function,
and let 
$ c, \rho \in ( 0, \infty ) $, 
$ 
  \delta \in \R 
$, 
$ i \in \{ 1, 2 \} $ 
satisfy that for all $ n \in \N $
it holds that
$
  \lambda_{ b_n }
  =
  - c \, n^{ \rho }
$ 
and 
$
  \mu_{ b_n }
  =
  \left| \lambda_{ b_n } \right|^{ \delta }
$. 
Then 
it holds that 
$
  B 
  \in
  \cap_{
    r \in 
    ( 
      - \infty , 
      -
      \frac{ 1 }{ 2 }
      \left[ 
        \nicefrac{ 1 }{ \rho } + 2 \delta 
      \right]
    )
  }
  HS( 
    H, H_r 
  )
$ 
and 
\begin{equation}
\label{eq:variance_difference_concrete_lower_bound}
\begin{split}
&
  \ES\big[ 
    \| X \|^2_H
  \big]
  -
  \ES\big[ 
    \| Y_i \|^2_H
  \big]
\geq
  \tfrac{
    \left(
      1 
      - 
      e^{
        - 2 c T
      }
    \right) 
    \,
    \left(
      1
      -
      e^{ - 1 }
    \right) 
    \,
    T^{ ( \nicefrac{1}{\rho} + 2 \delta )^+ } \,
    c^{ 2 \delta } \,
    h^{ 
      ( 
        1 - [ \nicefrac{1}{\rho} + 2 \delta ]^+
      ) 
    }
  }{
    4^{ ( 1 + \rho \delta^- ) } 
    \,
    e^{ 2^\rho e c T }
    \,
    \left( \rho + ( 1 + 2 \rho \delta )^- \right)
  }
  > 0
  .
\end{split}
\end{equation}
\end{proposition}
\begin{proof}
First of all, we observe that for all 
$
  r \in 
  ( -\infty, -\frac{1}{2} \, [ \nicefrac{1}{\rho} + 2 \delta ] )
$
it holds that 
$
  2 \rho ( r + \delta ) < -1
$ 
and 
\begin{equation}
  \|B\|_{ HS( H, H_r ) }
  =
  \sum^\infty_{ n = 1 }
  |\mu_{b_n}|^2 \, |\lambda_{b_n}|^{2r}
  =
  \sum^\infty_{ n = 1 }
  \left| \lambda_{ b_n } \right|^{ 2 ( r + \delta ) }
  =
  \sum^\infty_{ n = 1 }
  c^{ 2( r + \delta ) } \,
  n^{ 2 \rho ( r + \delta ) }
  < \infty
  .
\end{equation}
This proves that
$
  B 
  \in
  \cap_{
    r \in 
    ( 
      - \infty , 
      -
      \frac{ 1 }{ 2 }
      \left[ 
        \nicefrac{ 1 }{ \rho } + 2 \delta 
      \right]
    )
  }
  HS( 
    H, H_r 
  )
$.
It thus remains to prove \eqref{eq:variance_difference_concrete_lower_bound}.
For this we note that 
\begin{equation}
\label{eq:sum_lb}
\begin{split}
&
  \sum^\infty_{n=1}
  \frac{
    n^{ 2 \rho \delta }
  }{
    e^{
      c n^\rho h
    }
  }
  =
  \sum^{ \infty }_{ n = 0 }
  \int^{ n + 1 }_n
  \frac{
    ( n + 1 )^{ 2 \rho \delta }
  }{
    e^{ c ( n + 1 )^\rho h }
  }
  \, dx
\geq
  \sum^\infty_{n=1}
  \int^{ n + 1 }_n
  \frac{
    ( n + 1 )^{ 2 \rho \delta }
  }{
    e^{ c ( n + 1 )^\rho h }
  }
  \, dx
\\ & \geq
  \sum^\infty_{n=1}
  \int^{ n + 1 }_n
  \frac{
    2^{ - 2 \rho \delta^- }
    x^{ 2 \rho \delta }
  }{
    e^{ 2^\rho c x^\rho h }
  }
  \, dx
  =
  \frac{ 1 }{
    4^{ \rho \delta^- }
  }
  \int^\infty_1
  \frac{
    x^{ 2 \rho \delta }
  }{
    e^{ 2^\rho c x^\rho h }
  }
  \, dx
  .
\end{split}
\end{equation}
Combining \eqref{eq:sum_lb} with Lemma~\ref{lem:variance_differece_expE} and Lemma~\ref{lem:variance_difference_linear_implicitE} 
ensures that 
\begin{equation}
\label{eq:norm_sq_lb_1st_estimate}
\begin{split}
&
  \ES\big[ 
    \| X \|^2_H
  \big]
  -
  \ES\big[ 
    \| Y_i \|^2_H
  \big]
=
  \sum^{ \infty }_{ n = 1 }
  \Big[
    \operatorname{Var}\!\big(
    \big< b_n, X \big>_H
    \big)
    -
    \operatorname{Var}\!\big(
    \big< b_n, Y_i \big>_H
    \big)
  \Big]
\\ & \geq
  \frac{
    c^{ 2 \delta } 
    \left(
      1 - 
      e^{
        - 2 c T
      }
    \right) 
    h
  }{
    4
  }
  \,
  \sum^{ \infty }_{ n = 1 }
  \frac{
    n^{ 2 \rho \delta }
  }{
    e^{
      c n^\rho h
    }
  }
\geq
  \frac{
    c^{ 2 \delta } 
    \left(
      1 - 
      e^{ - 2 c T
      }
    \right) 
    h
  }{
    4^{
      \left( 
        1 + \rho \delta^{ - } 
      \right)
    }
  }
  \int^\infty_1
  \frac{
    x^{ 2 \rho \delta }
  }{
    e^{ 
      2^\rho c x^\rho h 
    }
  }
  \, dx
\\ & =
  \frac{
    \left(
      1 
      - 
      e^{
        - 2 c T
      }
    \right) 
    h^{ 
      \left( 
        1 - 2 \delta - \nicefrac{ 1 }{ \rho } 
      \right) 
    }
  }{
    2^{ ( 3 + 2 \rho \delta^{ + } ) } 
    \,
    \rho 
    \, 
    c^{ \nicefrac{ 1 }{ \rho } }
  }
  \int^\infty_{ 2^\rho ch }
  \frac{
    x^{ 
      ( \nicefrac{ 1 }{ \rho } + 2 \delta - 1 )
    }
  }{
    e^{ x }
  }
  \, dx
\\ & \geq
  \frac{
    \left(
      1 
      - 
      e^{
        - 2 c T
      }
    \right) 
    h^{ 
      \left( 
        1 - 2 \delta - \nicefrac{ 1 }{ \rho } 
      \right) 
    }
  }{
    2^{ ( 3 + 2 \rho \delta^{ + } ) } 
    \,
    \rho 
    \, 
    c^{ \nicefrac{ 1 }{ \rho } }
  }
  \max\!\left\{
    \int^{ 2^\rho e c h }_{ 2^\rho ch }
    \frac{
      x^{ 
        ( \nicefrac{ 1 }{ \rho } + 2 \delta - 1 )
      }
    }{
      e^{ x }
    }
    \, dx
    ,
    \int^{ 2^\rho e c T }_{ 2^\rho c T }
    \frac{
      x^{ 
        ( \nicefrac{ 1 }{ \rho } + 2 \delta - 1 )
      }
    }{
      e^{ x }
    }
    \, dx
  \right\}
\\ & \geq
  \frac{
    \left(
      1 
      - 
      e^{
        - 2 c T
      }
    \right) 
    h^{ 
      \left( 
        1 - 2 \delta - \nicefrac{ 1 }{ \rho } 
      \right) 
    }
  }{
    2^{ ( 3 + 2 \rho \delta^{ + } ) } 
    \,
    e^{ 2^\rho e c T }
    \,
    \rho 
    \, 
    c^{ \nicefrac{ 1 }{ \rho } }
  }
  \max\!\left\{
    \int^{ 2^\rho e c h }_{ 2^\rho ch }
      x^{ 
        ( \nicefrac{ 1 }{ \rho } + 2 \delta - 1 )
      }
    \, dx
    ,
    \int^{ 2^\rho e c T }_{ 2^\rho c T }
      x^{ 
        ( \nicefrac{ 1 }{ \rho } + 2 \delta - 1 )
      }
    \, dx
  \right\}
  .
\end{split}
\end{equation}
Next we observe that the fact that 
$
  \forall \, x \in ( 0, \infty ) 
  \colon
  \frac{ ( e^x - 1 ) }{ x } \geq 1
$ 
implies that 
for all 
$ r, q \in ( 0, \infty ) $
it holds that 
\begin{equation}
\begin{split}
&
  \int^{ er }_r
  x^{-1}
  \, dx
  =
  1
  ,
\qquad
  \int^{ er }_r
  x^{ ( q - 1 ) }
  \, dx
  =
  \frac{
    r^q \, ( e^q - 1 )
  }{ q }
\geq
  r^q,
\qquad
  \int^{ er }_r
  x^{ ( - q - 1 ) }
  \, dx
  =
  \frac{
    r^{ - q } \, ( 1 - e^{-q} )
  }{ q }
  .
\end{split}
\end{equation}
This and the fact that 
$
  \forall \, x \in (0,\infty)
  \colon
  \frac{ ( 1 - e^{-(1+x)} ) }{ ( 1 + x ) }
  \leq
  \frac{ ( 1 - e^{-x} ) }{ x }
  \leq 1
$ 
ensure that for all $ r \in ( 0, \infty )$, $ q \in \R $ it holds that 
\begin{equation}
\label{eq:integral_lb}
  \int^{ er }_r
  x^{ (q - 1) }
  \, dx
  \geq
  \frac{
    r^q
    \,
    ( 1 - e^{ -1-q^- } )
  }{
    ( 1 + q^- )
  }
  .
\end{equation}
In the next step we combine
\eqref{eq:norm_sq_lb_1st_estimate} and \eqref{eq:integral_lb} 
to obtain that 
\begin{equation}
\begin{split}
&
  \ES\big[ 
    \| X \|^2_H
  \big]
  -
  \ES\big[ 
    \| Y_i \|^2_H
  \big]
\\ & \geq 
  \frac{
    \left(
      1 
      - 
      e^{
        - 2 c T
      }
    \right) 
    \big(
      1
      -
      e^{ -1-( \nicefrac{1}{\rho} + 2 \delta )^- }
    \big) 
    \,
    h^{ 
      \left( 
        1 - 2 \delta - \nicefrac{ 1 }{ \rho } 
      \right) 
    }
  }{
    2^{ ( 3 + 2 \rho \delta^{ + } ) } 
    \,
    e^{ 2^\rho e c T }
    \,
    \rho 
    \,
    c^{ \nicefrac{ 1 }{ \rho } }
    \left( 1 + ( \nicefrac{ 1 }{ \rho } + 2 \delta )^- \right)
  }
  \max\!\left\{
    \left[ 
      2^{ \rho } c h
    \right]^{ 
      ( \nicefrac{ 1 }{ \rho } + 2 \delta )
    }
    ,
    \left[ 
      2^{ \rho } c T
    \right]^{ 
      ( \nicefrac{ 1 }{ \rho } + 2 \delta )
    }
  \right\}
\\ & \geq
  \frac{
    \left(
      1 
      - 
      e^{
        - 2 c T
      }
    \right) 
    \big(
      1
      -
      e^{ -1-( \nicefrac{1}{\rho} + 2 \delta )^- }
    \big) \,
    c^{ 2 \delta } \,
    h^{ 
      \left( 
        1 - 2 \delta - \nicefrac{ 1 }{ \rho } 
      \right) 
    }
  }{
    4^{ ( 1 + \rho \delta^- ) } 
    \,
    e^{ 2^\rho e c T }
    \left( \rho + ( 1 + 2 \rho \delta )^- \right)
  }
  \,
  \max\{
    h^{ ( \nicefrac{1}{\rho} + 2 \delta ) }
    ,
    T^{ ( \nicefrac{1}{\rho} + 2 \delta ) }
  \}
  .
\end{split}
\end{equation}
This together with the fact that 
$
  \max\{
    h^{ ( \nicefrac{1}{\rho} + 2 \delta ) }
    ,
    T^{ ( \nicefrac{1}{\rho} + 2 \delta ) }
  \}
  =
  T^{ ( \nicefrac{1}{\rho} + 2 \delta )^+ }
  \,
  h^{ -( \nicefrac{1}{\rho} + 2 \delta )^- }
$ 
implies \eqref{eq:variance_difference_concrete_lower_bound}. 
The proof of Proposition~\ref{prop:variance_difference_concrete_lower_bound} is thus completed.
\end{proof}

The next result, Corollary~\ref{cor:variance_difference_laplacian_lower_bound},
specialises 
Proposition~\ref{prop:variance_difference_concrete_lower_bound}
to the case where 
$ c = \pi^2 $
and 
$ \rho = 2 $
(Laplacian with Dirichlet boundary conditions on $ (0,1) $)
and slightly further estimates
the right hand side of \eqref{eq:variance_difference_concrete_lower_bound}.

\begin{corollary}
\label{cor:variance_difference_laplacian_lower_bound}
Assume the setting in Section~\ref{sec:setting_lower_bound}, 
let $ b \colon \N \to \set $ be a bijective function,
and let 
$ 
  \delta \in  \R 
$, 
$ i \in \{ 1, 2 \} $
satisfy that for all $ n \in \N $
it holds that
$
  \lambda_{ b_n }
  =
  - \pi^2 \, n^2
$ 
and 
$
  \mu_{ b_n }
  =
  \left| \lambda_{ b_n } \right|^{ \delta }
$. 
Then it holds that 
$
  B 
  \in
  \cap_{
    r \in 
    ( 
      - \infty , 
      - \delta - \nicefrac{ 1 }{ 4 }
    )
  }
  HS( 
    H, H_r 
  )
$ 
and 
\begin{equation}
\begin{split}
&
  \ES\big[ 
    \| X \|^2_H
  \big]
  -
  \ES\big[ 
    \| Y_i \|^2_H
  \big]
\geq
  \left[
  \tfrac{
    \left(
      1 - 
      e^{
        - T
      }
    \right)
    \,
    \left(
      1
      - 
      e^{
        - 1
      }
    \right)
    \,
    T^{ ( \nicefrac{ 1 }{ 2 } + 2 \delta )^+ } 
    \pi^{ 4 \delta }
  }{
    4^{ ( 1 + 2 \delta^- ) } 
    \,
    e^{ 12 \pi^2 T }
    \,
    \left(
      3 + 4 \delta^-
    \right)
  }
  \right]
    h^{ \min\{ \nicefrac{ 1 }{ 2 } - 2 \delta , 1 \} }
  > 0
  .
\end{split}
\end{equation}
\end{corollary}

\subsection{Lower bounds for a specific regular test function}

\begin{lemma}
\label{lem:weak_error_lower_bound_general}
Let 
$ 
  ( H, \left< \cdot, \cdot \right>_H, \left\| \cdot \right\|_H ) 
$ 
be a separable $ \R $-Hilbert space
with $ H \neq \{ 0 \} $, 
let 
$
  ( \Omega, \mathcal{F}, \P )
$ 
be a probability space, 
let $ \set \subseteq H $ be an orthonormal basis of $ H $, 
let $ \varphi \in \mathbb{M}( H, \R ) $ satisfy that 
for all $ v \in H $ 
it holds that 
$
  \varphi( v )
  =
  \operatorname{exp}\!\left(
    -\| v \|^2_H
  \right)
$, 
and let 
$ 
  X, Y \in 
  L^2( \P; H )
$ 
be such that 
$
  \left<
  b, X
  \right>_H
$, 
$ b \in \set $, 
is a family of independent centered Gaussian random variables, 
such that 
$
  \left<
  b, Y
  \right>_H
$, 
$ b \in \set $, 
is a family of independent centered Gaussian random variables, 
and such that for all $ b \in \mathbb{H} $ 
it holds that 
$
  \operatorname{Var}\!\left(
    \left<
      b, X
    \right>_H
  \right)
  \geq
  \operatorname{Var}\!\left(
    \left<
      b, Y
    \right>_H
  \right)
$.
Then 
it holds that
$ \varphi \in C^5_b( H, \R ) $ 
and
\begin{equation}
\label{eq:weak_error_lower_bound_general}
\begin{split}
&
  \ES\big[ 
    \varphi( Y )
  \big]
  -
  \ES\big[ 
    \varphi( X )
  \big]
\geq
  \frac{
    \big(
      \ES[
        \| X \|^2_H
      ]
      -
      \ES[
        \| Y \|^2_H
      ]
    \big)
  }{
    \exp\!\big(
      6
      \,
      \ES[ \| X \|^2_H ]
    \big)
  }
  .
\end{split}
\end{equation}
\end{lemma}

\begin{proof}
First of all, we observe that it is well-known that 
$ \varphi \in C^5_b( H, \R ) $ 
(see, e.g., (97)--(102) in~\cite{ConusJentzenKurniawan2014arXiv}). 
Next we assume w.l.o.g.\ that $ \set $ is a finite set
(\eqref{eq:weak_error_lower_bound_general} in the case 
where $ \set $ is an infinite set follows from 
Lebesgue's dominated convergence theorem 
and \eqref{eq:weak_error_lower_bound_general} in the case where $ \set $ is a finite set).
Next we note that the assumption that 
$
  \left<
  b, X
  \right>_H
$, 
$ b \in \set $, 
is a family of independent centered Gaussian random variables,
the assumption that
$
  \left<
  b, Y
  \right>_H
$, 
$ b \in \set $, 
is a family of independent centered Gaussian random variables, 
and, e.g., (103) in Conus et al.~\cite{ConusJentzenKurniawan2014arXiv}
imply that
\begin{equation}\label{eq:weak_error_decompose}
\begin{split}
&
  \ES\big[ 
    \varphi( Y )
  \big]
  -
  \ES\big[ 
    \varphi( X )
  \big]
\\ & =
  \prod_{ b \in \mathbb{H} }
  \left[
    1 +
    2 \operatorname{Var}\!\big(
    \big< b, Y \big>_H
    \big)
  \right]^{ -\nicefrac{ 1 }{ 2 } }
  -
  \prod_{ b \in \mathbb{H} }
  \left[
    1 +
    2 \operatorname{Var}\!\big(
    \big< b, X \big>_H
    \big)
  \right]^{ -\nicefrac{ 1 }{ 2 } }
\\ & =
  \left(
  \prod_{ b \in \mathbb{H} }
  \left[
    1 +
    2 \operatorname{Var}\!\big(
    \big< b, Y \big>_H
    \big)
  \right]^{ -\nicefrac{ 1 }{ 2 } }
  \right)
  \left(
    1 -
    \left[
    \prod_{ b \in \mathbb{H} }
    \frac{
  \left[
    1 +
    2 \operatorname{Var}\!\big(
    \big< b, Y \big>_H
    \big)
  \right]
    }{
  \left[
    1 +
    2 \, \operatorname{Var}\!\big(
    \big< b, X \big>_H
    \big)
  \right]
    }
    \right]^{ \nicefrac{ 1 }{ 2 } }
  \right)
\\ & =
  \ES\big[ 
    \varphi( Y )
  \big]
  \left(
    1 -
    \left[
    \prod_{ b \in \mathbb{H} }
    \frac{
  \left[
    1 +
    2 \operatorname{Var}\!\big(
    \big< b, Y \big>_H
    \big)
  \right]
    }{
  \left[
    1 +
    2 \operatorname{Var}\!\big(
    \big< b, X \big>_H
    \big)
  \right]
    }
    \right]^{ \nicefrac{ 1 }{ 2 } }
  \right)
  .
\end{split}
\end{equation}
Moreover, Jensen's inequality shows that 
\begin{equation}
\label{eq:weak_moment_numerics}
\begin{split}
&
  \ES\big[ 
    \varphi( Y )
  \big]
\geq
  \operatorname{exp}\!\left(
  -\ES[ \| Y \|^2_H ]
  \right)
  \geq
  \operatorname{exp}\!\left(
  -\ES[ \| X \|^2_H ]
  \right)
  .
\end{split}
\end{equation}
Next we observe that the facts that
$
  \forall \, b \in \set \colon
  2 
  \left[
    \operatorname{Var}\!\left(
      \left<
        b, X
      \right>_H
    \right)
    -
    \operatorname{Var}\!\left(
      \left<
        b, Y
      \right>_H
    \right)
  \right]
  \geq 0
$
and 
$
  \forall \, n \in \N \colon
  \forall \, x_1, \dots, x_n \in [0,\infty) \colon
  \prod_{ k = 1 }^n
  \left[ 
    1 + x_k
  \right]
  \geq
  1 +
  \sum_{ k = 1 }^n
  x_k
$
prove that
\allowdisplaybreaks
\begin{align}
\label{eq:weak_error_incomplete_lower_bound}
&
  1 
  -
  \left[
    \prod_{ b \in \mathbb{H} }
    \frac{
      \left[
        1 +
        2 
        \operatorname{Var}\!\big(
          \big< b, Y \big>_H
        \big)
      \right]
    }{
      \left[
        1 +
        2 
        \operatorname{Var}\!\big(
          \big< b, X \big>_H
        \big)
      \right]
    }
  \right]^{ 
    \nicefrac{ 1 }{ 2 } 
  }
  =
  1 
  -
  \left[
    \prod_{ b \in \mathbb{H} }
    \frac{
      \left[
        1 +
        2 
        \operatorname{Var}\!\big(
          \big< b, X \big>_H
        \big)
      \right]
    }{
      \left[
        1 +
        2 
        \operatorname{Var}\!\big(
          \big< b, Y \big>_H
        \big)
      \right]
    }
  \right]^{ 
    - \nicefrac{ 1 }{ 2 } 
  }
\\ & =
\nonumber
    1 -
    \left[
    \prod_{ b \in \mathbb{H} }
  \left[
    1 
    + 
    \frac{
      2 
      \left[
      \operatorname{Var}\!\big(
      \big< b, X \big>_H
      \big)
      -
      \operatorname{Var}\!\big(
      \big< b, Y \big>_H
      \big)
      \right]
    }{
  \left[
    1 +
    2 
    \operatorname{Var}\!\big(
    \big< b, Y \big>_H
    \big)
  \right]
    }
  \right]
    \right]^{ -\nicefrac{ 1 }{ 2 } }
\\ & \geq
\nonumber
  1 -
  \left[
    1 + 
    \sum_{ b \in \mathbb{H} }
    \frac{
      2
      \left[
        \operatorname{Var}\!\big(
          \big< b, X \big>_H
        \big)
        -
        \operatorname{Var}\!\big(
          \big< b, Y \big>_H
        \big)
      \right]
    }{
  \left[
    1 +
    2 \operatorname{Var}\!\big(
    \big< b, Y \big>_H
    \big)
  \right]
    }
    \right]^{ -\nicefrac{ 1 }{ 2 } }
  .
\end{align}
In addition, we note that the fundamental theorem 
of calculus ensures that 
$
  \forall \, x \in [0,\infty) 
  \colon
  1 - 
  \left[ 
    1 + x
  \right]^{ - 1 / 2 }
  =
  \frac{ 1 }{ 2 }
  \int_0^x
  \left[
    1 + y
  \right]^{ - 3 / 2 }
  dy
  \geq
  \frac{ 1 }{ 2 }
  x
  \left[
    1 + x
  \right]^{ - 3 / 2 }
$.
Hence, we obtain that
$
  \forall \, x \in [0,\infty) 
  \colon
  1 - 
  \left[ 
    1 + 2 x
  \right]^{ - 1 / 2 }
\geq
  x
  \left[
    1 + 2 x
  \right]^{ - 3 / 2 }
$.
Combinig this with \eqref{eq:weak_error_incomplete_lower_bound}
implies that
\begin{equation}
\label{eq:weak_error_incomplete_lower_bound2}
\begin{split}
&
  1 
  -
  \left[
    \prod_{ b \in \mathbb{H} }
    \frac{
      \left[
        1 +
        2 
        \operatorname{Var}\!\big(
          \big< b, Y \big>_H
        \big)
      \right]
    }{
      \left[
        1 +
        2 
        \operatorname{Var}\!\big(
          \big< b, X \big>_H
        \big)
      \right]
    }
  \right]^{ 
    \nicefrac{ 1 }{ 2 } 
  }
\\ & \geq
  1 -
  \left[
    1 + 
    2
    \sum_{ b \in \mathbb{H} }
    \frac{
        \operatorname{Var}\!\big(
          \big< b, X \big>_H
        \big)
        -
        \operatorname{Var}\!\big(
          \big< b, Y \big>_H
        \big)
    }{
  \left[
    1 +
    2 
    \operatorname{Var}\!\big(
    \big< b, Y \big>_H
    \big)
  \right]
    }
    \right]^{ -\nicefrac{ 1 }{ 2 } }
\\ & \geq
    \left[
    \sum_{ b \in \mathbb{H} }
    \frac{
    \operatorname{Var}\!\big(
    \big< b, X \big>_H
    \big)
    -
    \operatorname{Var}\!\big(
    \big< b, Y \big>_H
    \big)
    }{
  \left[
    1 +
    2 \operatorname{Var}\!\big(
    \big< b, Y \big>_H
    \big)
  \right]
    }
    \right]
  \left[
    1+
    2 \sum_{ b \in \mathbb{H} }
    \frac{
    \operatorname{Var}\!\big(
    \big< b, X \big>_H
    \big)
    -
    \operatorname{Var}\!\big(
    \big< b, Y \big>_H
    \big)
    }{
  \left[
    1 +
    2 
    \operatorname{Var}\!\big(
    \big< b, Y \big>_H
    \big)
  \right]
    }
  \right]^{ -\nicefrac{ 3 }{ 2 } } .
\end{split}
\end{equation}
In the next step we combine 
\eqref{eq:weak_error_decompose},
\eqref{eq:weak_moment_numerics},
and
\eqref{eq:weak_error_incomplete_lower_bound2}
with the fact that
$
  \forall \, b \in \set 
  \colon
  \operatorname{Var}\!\big(
    \big< b, Y \big>_H
  \big)
  \leq
  \operatorname{Var}\!\big(
    \big< b, X \big>_H
  \big)
  \leq
  \ES\big[ 
    \| X \|^2_H
  \big]
$
to obtain that 
\begin{equation}
\label{eq:weak_error_lower_bound_unrefined}
\begin{split}
  \ES\big[ 
    \varphi( Y )
  \big]
  -
  \ES\big[ 
    \varphi( X )
  \big]
& \geq
  \operatorname{exp}\!\left(
    - \ES\big[ \| X \|^2_H \big]
  \right)
  \left[
    \sum_{ b \in \mathbb{H} }
    \frac{
      \operatorname{Var}\!\big(
        \big< b, X \big>_H
      \big)
      -
      \operatorname{Var}\!\big(
        \big< b, Y \big>_H
      \big)
    }{
      \left[
        1 +
        2 
        \operatorname{Var}\!\big(
          \big< b, Y \big>_H
        \big)
      \right]
    }
  \right]
\\ & \quad  
  \cdot
  \left[
    1+
    2 \sum_{ b \in \mathbb{H} }
    \frac{
    \operatorname{Var}\!\big(
    \big< b, X \big>_H
    \big)
    -
    \operatorname{Var}\!\big(
    \big< b, Y \big>_H
    \big)
    }{
  \left[
    1 +
    2 \operatorname{Var}\!\big(
    \big< b, Y \big>_H
    \big)
  \right]
    }
  \right]^{ -\nicefrac{ 3 }{ 2 } }
\\ & \geq
  \exp\!\left(
    - 
    \ES\big[ \| X \|^2_H \big]
  \right)
  \left[
    \sum_{ b \in \mathbb{H} }
    \frac{
      \operatorname{Var}\!\big(
        \big< b, X \big>_H
      \big)
      -
      \operatorname{Var}\!\big(
        \big< b, Y \big>_H
      \big)
    }{
      \left[
        1 
        +
        2 
        \,
        \ES\big[ 
          \| X \|^2_H 
        \big]
      \right]
    }
  \right]
\\ & \quad  
  \cdot
  \left[
    1 +
    2 \sum_{ b \in \mathbb{H} }
    \operatorname{Var}\!\big(
    \big< b, X \big>_H
    \big)
  \right]^{ -\nicefrac{ 3 }{ 2 } }
\\ & =
  \exp\!\left(
    - 
    \ES\big[ 
      \| X \|^2_H 
    \big]
  \right)
  \left[
    1 
    +
    2 \,
    \ES\big[ \| X \|^2_H \big]
  \right]^{ - \nicefrac{ 5 }{ 2 } }
  \Big(
    \ES\big[ \| X \|^2_H \big]
    -
    \ES\big[ \| Y \|^2_H \big]
  \Big)
  .
\end{split}
\end{equation}
Combining \eqref{eq:weak_error_lower_bound_unrefined} with the fact that 
$
  \forall \, x \in [ 0, \infty )
  \colon
  ( 1 + x )^{ -1 }
  \geq
  e^{ -x }
$ 
implies \eqref{eq:weak_error_lower_bound_general}.
This completes the proof of Lemma~\ref{lem:weak_error_lower_bound_general}.
\end{proof}

The next result, Corollary~\ref{cor:weak_error_lower_bound_abstract},
is a direct consequence of Lemma~\ref{lem:variance_differece_expE}, 
of Lemma~\ref{lem:variance_difference_linear_implicitE}, 
and of Lemma~\ref{lem:weak_error_lower_bound_general}.

\begin{corollary}
\label{cor:weak_error_lower_bound_abstract}
Assume the setting in Section~\ref{sec:setting_lower_bound}
and 
let 
$ i \in \{ 1, 2 \} $, 
$ \varphi \in \mathbb{M}( H, \R ) $ 
satisfy that 
for all $ v \in H $ 
it holds that 
$
  \varphi( v )
  =
  \operatorname{exp}\!\left(
    -\| v \|^2_H
  \right)
$.
Then 
\begin{equation}
\begin{split}
&
  \ES\big[ 
    \varphi( Y_i )
  \big]
  -
  \ES\big[ 
    \varphi( X )
  \big]
\geq
  e^{
    - 6 \, \ES[ \| X \|^2_H ]
  }
  \left(
    \ES\big[
    \| X \|^2_H
    \big]
    -
    \ES\big[
    \| Y_i \|^2_H
    \big]
  \right)
  .
\end{split}
\end{equation}
\end{corollary}

The next result, Proposition~\ref{prop:weak_error_concrete_lower_bound},
is an immediate consequence of 
Proposition~\ref{prop:variance_difference_concrete_lower_bound} 
and of Corollary~\ref{cor:weak_error_lower_bound_abstract}.

\begin{proposition}
\label{prop:weak_error_concrete_lower_bound}
Assume the setting in Section~\ref{sec:setting_lower_bound}, 
let $ b \colon \N \to \set $ be a bijective function,
let 
$ i \in \{ 1, 2 \} $,
$ c, \rho \in ( 0, \infty ) $, 
$ 
  \delta \in \R 
$
satisfy that for all $ n \in \N $
it holds that
$
  \lambda_{ b_n }
  =
  - c \, n^{ \rho }
$ 
and 
$
  \mu_{ b_n }
  =
  \left| \lambda_{ b_n } \right|^{ \delta }
$, 
and let $ \varphi \in \mathbb{M}( H, \R ) $ satisfy that 
for all $ v \in H $ 
it holds that 
$
  \varphi( v )
  =
  \operatorname{exp}\!\left(
    -\| v \|^2_H
  \right)
$. 
Then it holds that 
$
  \varphi \in C^5_b( H, \R )
$, 
$
  B 
  \in
  \cap_{
    r \in 
    ( 
      - \infty , 
      -
      \frac{ 1 }{ 2 }
      \left[ 
        \nicefrac{ 1 }{ \rho } + 2 \delta 
      \right]
    )
  }
  HS( 
    H, H_r 
  )
$
and 
\begin{equation}
\label{eq:weak_error_concrete_lower_bound}
\begin{split}
&
  \ES\big[ 
    \varphi( Y_i )
  \big]
  -
  \ES\big[ 
    \varphi( X )
  \big]
\geq
  \left[
      \tfrac{
        \left(
          1 
          - 
          e^{
            - 2 c T
          }
        \right) 
        \,
        \left(
          1
          -
          e^{ - 1 }
        \right) 
        \,
        T^{ ( \nicefrac{ 1 }{ \rho} + 2 \delta )^+ } 
        \,
        c^{ 2 \delta } 
      }{
        4^{ ( 1 + \rho \delta^- ) } 
        \,
        \exp\left( 
          2^{ \rho } 
          e c T 
          + 
          6 \, \ES[ \| X \|^2_H ] 
        \right)
        \,
        \left( \rho + ( 1 + 2 \rho \delta )^- \right)
      }
  \right]
        h^{ 
          ( 
            1 - [ \nicefrac{1}{\rho} + 2 \delta ]^+
          ) 
        }
    > 0
      .
\end{split}
\end{equation}
\end{proposition}

In the next result, Corollary~\ref{cor:weak_error_laplacian_lower_bound},
we specialise Proposition~\ref{prop:weak_error_concrete_lower_bound}
to the case where $ c = \pi^2 $
and $ \rho = 2 $ (Laplacian with Dirichlet boundary conditions on
$ (0,1) $)
and slightly further estimates
the right hand side of \eqref{eq:weak_error_concrete_lower_bound}.

\begin{corollary}
\label{cor:weak_error_laplacian_lower_bound}
Assume the setting in Section~\ref{sec:setting_lower_bound}, 
let $ b \colon \N \to \set $ be a bijective function,
let 
$ 
  \delta \in \R 
$,
$ i \in \{ 1, 2 \} $,
assume that for all $ n \in \N $
it holds that
$
  \lambda_{ b_n }
  =
  - \pi^2 \, n^2
$ 
and 
$
  \mu_{ b_n }
  =
  \left| \lambda_{ b_n } \right|^{ \delta }
$, 
and let $ \varphi \in \mathbb{M}( H, \R ) $ satisfy that 
for all $ v \in H $ 
it holds that 
$
  \varphi( v )
  =
  \operatorname{exp}\!\left(
    -\| v \|^2_H
  \right)
$. 
Then it holds that 
$
  \varphi \in C^5_b( H, \R )
$, 
$
  B 
  \in
  \cap_{
    r \in 
    ( 
      - \infty 
      ,
      - \delta
      - 
      \nicefrac{ 1 }{ 4 }
    )
  }
  HS( 
    H, H_r
  )
$
and 
\begin{equation}
\begin{split}
&
  \ES\big[ 
    \varphi( Y_i )
  \big]
  -
  \ES\big[ 
    \varphi( X )
  \big]
\geq
  \left[
  \tfrac{
    \left(
      1 - 
      e^{
        - T
      }
    \right)
    \,
    \left(
      1
      - 
      e^{
        - 1 
      }
    \right)
    \,
    T^{ ( \nicefrac{ 1 }{ 2 } + 2 \delta )^+ } 
    \pi^{ 4 \delta }
  }{
    4^{ ( 1 + 2 \delta^- ) } 
    \,
    \left(
      3 + 4 \delta^-
    \right)
    \,
    \exp\left( 
      12 \pi^2 T 
      + 
      6 \, \ES[ \| X \|^2_H ] 
    \right)
  }
  \right]
  h^{ \min\{ \nicefrac{ 1 }{ 2 } - 2 \delta , 1 \} }
  > 0
    .
\end{split}
\end{equation}
\end{corollary}

\bibliographystyle{acm}
\bibliography{Bib/bibfile}
\end{document}